\documentclass[reqno,english]{amsart}
\usepackage[utf8]{inputenc}
\usepackage{amsmath,amssymb,amsthm,mathrsfs,color,times,textcomp,yfonts,mathtools,cases}
\usepackage[left=0.8in, right=0.8in, top=0.8in, bottom=0.8in]{geometry}

\allowdisplaybreaks[4]

\usepackage{bm}

\usepackage[T1]{fontenc} 

\usepackage{subcaption}
\usepackage[normalem]{ulem}
\usepackage[export]{adjustbox}
\usepackage{esint}
\usepackage{xcolor}
\usepackage{array}
\usepackage[colorlinks=true]{hyperref}
\hypersetup{urlcolor=blue, citecolor=blue, linkcolor=red}

\hypersetup{
	colorlinks=true,
	linkcolor=red
}

\usepackage[square,numbers]{natbib}

\usepackage{float}
\usepackage{ulem}
\usepackage{syntonly}
\usepackage{mathtools}
\usepackage{bm}
\usepackage{soul}
\usepackage{amsfonts,amsmath,latexsym,verbatim,amscd,mathrsfs,color,array}
\usepackage[colorlinks=true]{hyperref}

\usepackage{amsmath,amssymb,amsthm,amsfonts,graphicx,color}
\usepackage{amssymb}
\usepackage{epstopdf}

\newcommand{\1}{\mathbf{1}}

\newcommand{\pp}{ {\partial} }

\theoremstyle{plain}
\newtheorem{theorem}{Theorem}[section]
\newtheorem{lemma}[theorem]{Lemma}
\newtheorem{prop}[theorem]{Proposition}

\newtheorem{corollary}[theorem]{Corollary}
\newtheorem{remark}{Remark}[theorem]
\newcommand{\bremark}{\begin{remark} \em}
	\newcommand{\eremark}{\end{remark} }

\numberwithin{equation}{section}

\title[Singularity formation with multiple rates]{Co-existence of Type II blow-ups with multiple blow-up rates for five-dimensional heat equation  with critical nonlinear boundary conditions}

\author[J. Wei]{Juncheng Wei}
\author[Z. Ye]{Zikai Ye}
\author[X. Zeng]{Xiaoyu Zeng}
\author[Q. Zhang]{Qidi Zhang}

\address[J. Wei]{
	Department of Mathematics,
	Chinese University of Hong Kong,
	Shatin, NT, Hong Kong}
\email{\href{mailto:jcwei@math.ubc.ca}{jcwei@math.ubc.ca}}

\address[Z. Ye]{
	Department of Mathematics,
	University of British Columbia
	\newline\indent
	V6T 1Z2, Vancouver-BC, Canada}
\email{\href{mailto:yezikai@math.ubc.ca}{yezikai@math.ubc.ca}}

\address[Q. Zhang]{
Department of Mathematics, The University of Hong Kong, Hong Kong, China
	}
\email{\href{mailto:qdzhang@hku.hk}{qdzhang@hku.hk}}

\address[X. Zeng]{ Center for Mathematical Sciences,
  Wuhan University of Technology
  \newline \indent Wuhan 430070, P.R. China
	}
\email{\href{mailto:xyzeng@whut.edu.cn}{xyzeng@whut.edu.cn}}



\begin{document}

\begin{abstract}
    We consider the following five-dimensional heat equation with critical boundary condition
\begin{equation*}
	\partial_t u=\Delta u
	\mbox{ \ in \ }  \mathbb{R}_+^5\times (0,T) ,
	\quad
	-\partial_{x_5}u =|u|^\frac{2}{3}u   \mbox{ \ on \ } \pp \mathbb{R}^5_+ \times (0,T) .
\end{equation*}
Given $\mathfrak{o}$ distinct boundary points $q^{[i]} \in \partial \mathbb{R}_+^5$,
and
$\mathfrak{o}$ integers $l_i\in \mathbb{N}$ (possibly duplicated), $i=1,2,\dots, \mathfrak{o}$,
for $T>0$ sufficiently small,
we construct a finite-time blow-up solution $u$ with a type II blow-up rate $(T-t)^{-3l_i -3}$ for $x$ near $q^{[i]}$. This seems to be the first result of the co-existence of type II blowups with different blow-up rates. To accommodate highly unstable blowups with different blowup rates, we first develop a unified linear theory for the inner problem with more time decay in the blow-up scheme through restriction on the spatial growth of the right-hand side, and then use vanishing adjustment functions for deriving multiple rates at distinct points.
This paper is inspired by \cite{harada2019higher, FreeHMF2023, Zhang-Zhao2023}.

\end{abstract}
    \maketitle

{
  \hypersetup{linkcolor=black}
  \tableofcontents
}

\section{Introduction and main results}
In this paper we consider the heat equation with the critical boundary conditions
\begin{equation}\label{qd2023Nov21-1}
	\partial_t u=\Delta u
	\mbox{ \ in \ }  \mathbb{R}_+^{n}\times (0,T) ,
	\quad
	-\partial_{x_n}u =|u|^{\frac{2}{n-2} } u   \mbox{ \ on \ } \pp \mathbb{R}^{n}_+ \times (0,T)  ,
	\quad
	u(\cdot,0)=u_0  \mbox{ \ in \ }  \mathbb{R}^{n}_+ ,
\end{equation}
where the dimension $n\ge 3$, $\tilde{x}\in \mathbb{R}^{n-1}, x_n\in \mathbb{R}_+$, and $x=\left(\tilde{x},x_n\right) \in \mathbb{R}_+^n$.

The phenomena of finite-time singularity formation triggered by the superlinear right-hand side appear in many evolution equations. It is interesting to reveal how boundary conditions influence the behaviour of solutions for the evolution equations.

In this paper, we give a comprehensive study of the possible phenomenon of the finite-time blow-up of heat equations with the nonlinear boundary condition.
The regularity and existence problems of elliptic and parabolic equations with linear and nonlinear oblique derivatives are extensively studied. We refer to \cite[Chapter III-5, Chapter V-7]{parabolicBook1968}, \cite{ lieberman1996second, Lieberman2012CPAA, lieberman2013oblique-book}.

For an integer $n\ge 2$, $\alpha\in (0,1), p>0, -\infty\le t_0<t_1\le \infty$, the heat equation with algebraic power nonlinear boundary condition in $\mathbb{R}_+^n$ has the form
\begin{equation}\label{qd2023Nov19-1}
	\left\{
	\begin{aligned}
	&	\partial_t u =\Delta_{\tilde{x}} u + \frac{1-2\alpha}{x_n} \partial_{x_n} u + \partial_{x_n x_n} u
		=
		x_n^{-(1-2\alpha)} \nabla_x \cdot \left( x_n^{1-2\alpha} \nabla_x u \right)
		& &
		\mbox{ \ for \ } (x,t) \in \mathbb{R}_+^n \times (t_0,t_1) ,
		\\
	&	-\lim\limits_{x_n \downarrow 0} \frac{u(\tilde{x},x_n,t) - u(\tilde{x},0,t) }{x_n^{2\alpha} } =
 \lim\limits_{x_n \downarrow 0}
 \frac{-x_n^{1-2\alpha} \partial_{x_n} u(\tilde{x},x_n,t)}{2\alpha}
 =
 \left( \left|u\right|^{p-1} u \right) \left(\tilde{x},0,t\right)
		& &
		\mbox{ \ for \ }  (\tilde{x}, t) \in \mathbb{R}^{n-1} \times (t_0,t_1) ,
\\
&
u\left(x,t_0\right) = u_0\left(x\right)
& &
\mbox{ \ for \ } x\in \mathbb{R}_+^n ,
	\end{aligned}
\right.
\end{equation}
where the initial value is vacuum if $t_0 =-\infty$.
This equation appears in the extension form of $$\frac{\Gamma(-\alpha)}{4^\alpha \Gamma(\alpha)} \left(\partial_t-\Delta_{\tilde{x}} \right)^\alpha u(\tilde{x},0,t)=  (|u|^{p-1}u)(\tilde{x},0,t)$$ for $(\tilde{x},t) \in \mathbb{R}^{n-1} \times (-\infty,t_1)$ with the Gamma function $\Gamma(\cdot)$ according to \cite[Theorem 1.7, 1.8]{stinga2017regularity}.
\eqref{qd2023Nov21-1} is a special case of \eqref{qd2023Nov19-1}.

For $(t_0,t_1)=\mathbb{R}, p>0$,
\eqref{qd2023Nov19-1} is dilation and translation invariant in the sense that for $u$ solving \eqref{qd2023Nov19-1},
\begin{equation}\label{qd202325-1}
\lambda u \left(\lambda^{\frac{p-1}{2\alpha}} \left(\tilde{x} -\xi\right), \lambda^{\frac{p-1}{2\alpha}} x_n, \lambda^{\frac{p-1}{\alpha}} \left(t -s\right) \right)
\end{equation}
still satisfies \eqref{qd2023Nov19-1} with $\lambda>0, \xi\in \mathbb{R}^{n-1}$, $s\in \mathbb{R}$.
The energy associated with \eqref{qd2023Nov19-1} is
\begin{equation*}
	J\left[ u\right] :=
	\frac{1}{2}
	\int_{\mathbb{R}_+^n }   x_n^{1-2\alpha} \left| \nabla_x u \right|^2 dx
	-
	\frac{2\alpha}{p+1}
	\int_{\mathbb{R}^{n-1}}
	\left|u \right|^{p+1}\left(\tilde{x},0,t\right) d\tilde{x}  .
\end{equation*}
For $u$ with sufficient smoothness and spatial decay, $\frac{d}{d t} J\left[ u\right] = -\int_{\mathbb{R}_+^n} x_n^{1-2\alpha} \left|\partial_t u \right|^2 dx$.
Note that
\begin{equation*}
J\left[ \lambda u\left( \lambda^{\frac{p-1}{2\alpha} } \left(\tilde{x} -\xi\right), \lambda^{\frac{p-1}{2\alpha} } x_n, t \right) \right] = ( \lambda^{\frac{p-1}{2\alpha} } )^{\frac{4\alpha}{p-1}+2\alpha -(n-1) }
J\left[ u\left(\tilde{x},x_n,t\right) \right] .
\end{equation*}
For this reason, we define the energy critical exponent
\begin{equation*}
p_{\alpha,S} :=
\begin{cases}
	\infty
&
\mbox{ \ if \ }
n\le 1+2\alpha
\\
 \frac{n-1+2\alpha}{n-1-2\alpha}
 &
 \mbox{ \ if \ }
 n>1+2\alpha
\end{cases}
\end{equation*}
and
$p< (=, >)
 p_{\alpha, S}$ is called the energy subcritical (critical, supercritical) case.

\eqref{qd2023Nov19-1} has a close relationship with the well-known Fujita-type equation
\begin{equation}\label{2023Nov07-1}
	\partial_t u-\Delta u = |u|^{p-1} u \quad \mbox{ \ in \ } \Omega \times (0,T) ,
	\quad
u(\cdot,0) = u_0  \quad \mbox{ \ in \ } \Omega ,
\end{equation}
where $\Omega$ is domain in $\mathbb{R}^n$. There are three important exponents for \eqref{2023Nov07-1}: the Fujita exponent $p_F$, the Sobolev exponent $p_S$, and the Joseph-Lundgren exponent $p_{JL}$, which are defined respectively as
\begin{equation*}
	p_F:=1+\frac2n,\quad
	p_S:=\begin{cases}
		\infty ~&\mbox{ if }~n=1,2
		\\
		\frac{n+2}{n-2}~&\mbox{ if }~n\geq 3 ,
	\end{cases}
	\quad
	p_{JL}:=
	\begin{cases}
		\infty & \mbox{ if }~n\le 10
		\\
		1+\frac{4}{n-4-2\sqrt{n-1}}
		& \mbox{ if }~n\ge 11 .
	\end{cases}
\end{equation*}
The dynamics and colorful phenomena of \eqref{2023Nov07-1} are sensitive to the power $p$.

Given an initial value $u_0 \in L^\infty(\Omega)$, there is a unique solution of \eqref{2023Nov07-1} in $L^\infty(\Omega \times (0,t) )$ for $t\in (0,T)$ with a maximum life time $T\le \infty$.
We say that $u$ blows up in finite time if $T<\infty$, and in this case, $\limsup_{t\uparrow T} \| u\left(\cdot,t\right) \|_{L^\infty(\Omega)} =\infty$. For $p>1$, there are two types of blow-ups depending on the rates compared with the ODE solution of \eqref{2023Nov07-1}, $(p-1)^{-\frac{1}{p-1}} (T-t)^{-\frac{1}{p-1}}$,
\begin{align}
& \text{Type I}: \limsup_{t \uparrow T}(T-t)^{\frac{1}{p-1}}\|u(\cdot,t)\|_{L^\infty(\Omega) }<\infty , \label{typeI-def}
\\
& \text{Type II}: \limsup_{t \uparrow T }(T-t)^{\frac{1}{p-1}}\|u(\cdot,t)\|_{L^\infty(\Omega) }=\infty.  \label{typeII-def}
\end{align}
Indeed, once $u$ blows up at $T$, by comparison theorem, $\| u(\cdot,t)\|_{L^\infty(\Omega)} \ge (p-1)^{-\frac{1}{p-1}} (T-t)^{-\frac{1}{p-1}}$ with $p>1$ in bounded domain with zero Dirichlet boundary condition by \cite[Proposition 23.1]{Souplet19book}.
From the celebrated work \cite{Fujita66} of Fujita, \eqref{2023Nov07-1} possesses a global nontrivial solution $u\ge 0$ if and only if $p>p_F$.

For the subcritical case $1<p<p_S$, under the assumption of \eqref{typeI-def}, $1<p\le p_S$, Giga and Kohn \cite{giga1985asymptotically} gave the asymptotic behavior of solutions in the parabolic cylinder.
Giga and Kohn \cite{giga1987characterizing} proved \eqref{typeI-def} in a bounded convex domain or $\mathbb{R}^n$, provided that either $(i)$ $u_0\ge 0$ and $p$ satisfies $1<p<p_S$ or $n\le 2$, or
$(ii)$
$1<p<\frac{3n+8}{3n-4}$ or $n=1$.
Finally,
Giga, Matsui and Sasayama \cite{Giga04Indiana, Giga04-MathMethod} proved \eqref{typeI-def} in a bounded, convex domain or $\mathbb{R}^n$ without assumptions on the sign of solutions. Quittner \cite{Quittner2021-Duke} proved that $\partial_t u=\Delta u + u^p$ in $\mathbb{R}^n\times \mathbb{R}$ does not possess positive classical solutions and there is only type I blow-up for non-negative classical solutions of \eqref{2023Nov07-1} in arbitrary domain of $\mathbb{R}^n$ without convex assumption.
For $1<p\le p_S$, Souplet \cite[Theorem 1.1]{Souplet19-simplyproof} gave a much simpler proof when the initial value is bounded, nonnegative, radially symmetric, and nonincreasing.

For the critical case $p= p_S$,
Filippas, Herrero and Vel\'azquez \cite[Section 6]{FHV00} excluded the type II blow-up in the positive radial and monotonically decreasing class for $n\ge 3$.
The same result was obtained when the monotone assumption was removed by Matano and Merle \cite[Theorem 1.7]{MatanoMerle04},
and in higher dimensions $n\ge 7$ for positive solutions in bounded convex domains or $\mathbb{R}^n$ without radially symmetric assumptions by Wang and Wei \cite{wang2021refined}.

By formal asymptotic analysis in \cite{FHV00}, it is conjectured that there exist sign-changing type II blow-up solutions in lower dimensions $n=3,4,5,6$ and rigorous construction are given in a series of work \cite{Schweyer12JFA,type25D,ni4D,harada2019higher,harada6D,li2022slow}.
For $n=5$, del Pino, Musso, and the first author \cite{type25D} constructed the first rate.
Harada \cite{harada2019higher} adopted a self-similar variable and then adjusted eigenfunctions with cut-off functions to achieve the fast time decay for the outer problem at the blow-up point for deriving other rates. Zhang and Zhao \cite{Zhang-Zhao2023} first constructed finite-time blow-up solutions with multiple different rates by the gluing method. Zhang and Zhao first split the right-hand side of the outer problem into several parts. Second, they introduced the idea of using the solution of the heat equation to eliminate the derivatives for the influence between distinct blow-up points. Finally, applying Tylor expansion deduced the desired vanishing condition for the outer problem at the prescribed blow-up points.

In the case $p_S<p<p_{JL}$, Matano and Merle \cite{MatanoMerle04} excluded the type II blow-up in the radially symmetric class when $\Omega$ is a ball or $\mathbb{R}^n$ (under some additional requirement). In the non-radial case, del Pino, Musso, and the first author \cite{delMussoWei2021} constructed positive type II blow-up solutions in some special domain in $\mathbb{R}^n$ for $n\ge 7$, $p=\frac{n+1}{n-3} \in (p_S, p_{JL})$.
Du \cite[Theorem 1.3]{du2021finiteArx} gave counterpart classification results for finite-time blow-up solutions of rotational symmetric harmonic map heat flow in dimension $3\le n <7$.

For $p>p_S$, various characterizations of
type I and type II blow-ups are established in \cite{Merle09JFA}.

For $p=p_{JL}$, Seki \cite{Seki18} constructed a type II blow-up solution.

For $p>p_{JL}$, Herrero and Vel\'azquez \cite{HV1993} first constructed a type II blow-up solution in the radial case and Mizoguchi \cite{Mizoguchi04} gave a simpler proof.
Collot \cite{Collot2017} constructed non-radial type II blow-up solutions under some restrictions of the exponent $p>p_{JL}$.

Mizoguchi and Souplet \cite[Theorem 1]{Mizoguchi-Souplet19-typeILp} built the connection between type I blow-up and $L^{\frac{n(p-1)}{2}}$ norm blow-up.
We refer to the monograph \cite{Souplet19book} for various developments about Fujita-type equations.


Compared with the Fujita-type equations, there are much fewer studies about \eqref{qd2023Nov19-1} and variations of \eqref{qd2023Nov19-1}.    By Proposition \ref{2023Nov23-3-prop}, when $\left(t_0,t_1\right) = \left(-\infty,T\right)$, $p>1$, there exists a solution of \eqref{qd2023Nov19-1} with the form
\begin{equation}\label{qd2023Nov21-u1-def}
	u_1(x,t) = \frac{x_n^{2\alpha}}{4^\alpha \Gamma(\alpha)}
	\int_0^\infty
	\frac{ e^{-\frac{x_n^2}{4\tau}}  }
	{\tau^{1+\alpha} }
	\left(
	\pm
	C_{\alpha,p}  \right)
	\left(T-t+\tau\right)^{\frac{-\alpha}{p-1}}    d\tau ,
	\quad
	C_{\alpha,p} :=
	\left[
	\frac{1}{4^\alpha \Gamma(\alpha)}
	\int_1^\infty \frac{1-z^{\frac{-\alpha}{p-1}} }{\left(z-1\right)^{1+\alpha} } dz
	\right]^{\frac{1}{p-1}} >0
\end{equation}
independent of $\tilde{x}$ and satisfying
\begin{equation}\label{qd2023Nov21-2}
	u_1 \left(\tilde{x},0,t\right) =
	\pm C_{\alpha,p} \left(T-t\right)^{\frac{-\alpha}{p-1}}
	,
	\quad
	u_1(x,t) \sim
	\pm \left(
	\max\left\{	T-t , x_n^2 \right\} \right)^{\frac{-\alpha}{p-1}}
	\mbox{ \ for \ } (x,t)\in \overline{\mathbb{R}_+^{n} } \times \left(-\infty,T\right)
	,
\end{equation}
where ``$\sim$'' only depends on $\alpha,p$.
Compared with the time rate of \eqref{qd2023Nov21-u1-def},
we define the type I and type II blow-up for \eqref{qd2023Nov19-1} as
\begin{align}
	&	\text{Type I}: \limsup _{t \uparrow T}(T-t)^{\frac{\alpha}{p-1}}\|u(\cdot,t)\|_{L^\infty(\Omega) }<\infty ,
	\label{Boundary-typeI-def}
	\\
	&	\text{Type II}: \limsup _{t \uparrow T }(T-t)^{\frac{\alpha}{p-1}}\|u(\cdot,t)\|_{L^\infty(\Omega) }=\infty.
	\label{Boundary-typeII-def}
\end{align}

Fila and Quittner \cite[p.205]{fila1991blow} considered
\begin{equation*}
\partial_t u =\Delta u \mbox{ \ in \ } B\times (0,T),
\quad
\partial_{\nu} u =u^p \mbox{ \ on \ } \partial B \times (0,T),
\quad
u(\cdot,0) = u_0 \mbox{ \ in \ } \bar{B}
\end{equation*}
with an outer normal derivative $\partial_{\nu}$ and a unit ball $B\subset \mathbb{R}^n$ in the radial class and got the convergence result for blow-up solutions under some non-negative restrictions on $u_0$ and its derivatives.

As a counterpart of \cite{Quittner2021-Duke}, Quittner \cite{quittner2020optimal} proved that there is no positive classical bounded solution of \eqref{qd2023Nov19-1} with $\alpha=\frac{1}{2}$, $1<p<p_{\frac{1}{2}, S}$, and only type I blow-up \eqref{Boundary-typeI-def} can occur for positive classical solutions in a bounded smooth domain.
See
Quittner and Souplet \cite[Theorem 4.1]{Quittner-Souplet12} for type I blow-up results in the half space without the sign assumption about the solution.

We believe many results in Fujita-type equations can be realized in \eqref{qd2023Nov19-1} with more general nonlinear boundary conditions.

Now, we state our main results. By classical results of Li and Zhu \cite[Theorem 1.2]{LiZhu1995} (see also \cite{Escobar88, ChenLiOu2006}), for $n\ge 3$, all  nonzero nonnegative solutions to
\begin{equation}\label{U-eq}
	\Delta U =0
	\quad
	\mbox{ \ in \ } \mathbb{R}_+^n,
	\quad
	-\pp_{x_n} U = U^{\frac{n}{n-2} }
	\quad
	\mbox{ \ on \ } \pp \mathbb{R}_+^n
\end{equation}
 are given by
 \begin{equation}
 \label{Usol}
 c^{\frac{n-2}{2}} \big\{   |\tilde{x}- \tilde{x}_0|^2 + \left[x_n+ (n-2)^{-1}c \right]^2  \big\}^{-\frac{n-2}{2}}, \quad  x=(\tilde{x},x_n)\in \mathbb{R}_+^n, \ \tilde{x}_0\in \mathbb{R}^{n-1}, \ c>0.
 \end{equation}
We refer to \cite{Li-Li06-nonlinearYama, 03Li-Zhang-LiouHarn} for Liouville theorems for more Yamabe-type equations. Let
\begin{equation}\label{qd24Jan14-U}
	U(x)=
(n-2)^{\frac{n-2}{2}} \left[   |\tilde{x}|^2 + \left(1+x_n\right)^2  \right]^{-\frac{n-2}{2}}
	, \quad  x=(\tilde{x},x_n)\in \mathbb{R}_+^n.
\end{equation}

As in \cite[p.2977]{FHV00}, we set
\begin{equation}\label{eq-heat-1}
	\Theta_l(x,t)=-(T-t)^{l}
	\left(L_l^{\frac{n-2}{2}}(0)\right)^{-1} L_l^{\frac{n-2}{2}} \left(\frac{|x|^2}{4(T-t)} \right)  \quad \mbox{ \ for \ } n,l \in \mathbb{N}, \ n\ge 1
\end{equation}
with the modified Laguerre polynomials
\begin{equation*}
	L_l^{\frac{n-2}{2}}(r) := r^{-\frac{n-2}{2}} e^r \frac{d^l}{dr^l} \left(r^{\frac{n-2}{2}+l} e^{-r} \right) .
\end{equation*}
Obviously,
$L_l^{\frac{n-2}{2}}(r) = \sum_{i=0}^l c_i r^i$ with some constants $c_i \in \mathbb{R}$, and $c_0 = L_l^{\frac{n-2}{2}}(0) > 0$.
$\Theta_l(x,t)$
satisfies $\Theta_l(0,t) = -(T-t)^{l}$, and for $n\ge 2$,
\begin{equation}\label{eq-heat}
	\pp_t \Theta_l =\Delta  \Theta_l  \mbox{ \ in \ } \mathbb{R}^{n}_+ \times (-\infty,T) ,
	\quad
	-\pp_{x_n} \Theta_l =0
	\mbox{ \ on \ } \pp\mathbb{R}^{n}_+ \times (-\infty,T) ,
\end{equation}
where the first formula in \eqref{eq-heat} can be deduced by the second equation below \cite[p.2977 (A1)]{FHV00}.

Inspired by the work \cite{harada2019higher, FreeHMF2023, Zhang-Zhao2023}, we construct finite-time blow-up solutions at a finite number of prescribed blow-up points with multiple rates for \eqref{qd2023Nov21-1} with $n=5$. The main theorem is stated as follows.
\begin{theorem}\label{main}
	Consider
\begin{equation}\label{fractextend}
	\partial_t u=\Delta u
	\mbox{ \ in \ }  \mathbb{R}_+^5\times (0,T) ,
	\quad
	-\partial_{x_5}u =|u|^\frac{2}{3}u   \mbox{ \ on \ } \pp \mathbb{R}^5_+ \times (0,T) .
\end{equation}
Given an integer $\mathfrak{o}\ge 1$,  $\mathfrak{o}$ distinct boundary points $q^{[i]} \in \partial \mathbb{R}_+^5$,
and
$\mathfrak{o}$ integers $l_i\in \mathbb{N}$ (possibly duplicated), $i=1,2,\dots, \mathfrak{o}$, $\delta =  \min\limits_{1\leq i\neq j\leq \mathfrak{o}} |q^{[i]}-q^{[j]}| / 32$,
then
for $T>0$ sufficiently small, there exists a finite-time blow-up solution of the form
\begin{equation*}
	\begin{aligned}
		u(x,t) = \ &
		\sum_{i=1}^{\mathfrak{o}} \bigg\{
		3^{\frac{3}{2}} \mu_i^{-\frac{3}{2}}
		 \bigg[   \left|\frac{\tilde{x}-\xi^{[i]}}{\mu_i} \right|^2 + \left(1+\frac{x_5}{\mu_i}\right)^2  \bigg]^{-\frac{3}{2}}
		\eta\left(\frac{x-q^{[i]}}{2\delta} \right)
		+
		\Theta_{l_i}(x-q^{[i]},t)
		\eta\left(\frac{x-q^{[i]}}{\delta} \right)
		\\
		&
		+
		O\left(  |\ln{T}|^{2}(T-t)^{l_i} \eta\left( \frac{|x-q^{[i]}|}{2\sqrt{T-t}}  \right) \right)  \bigg\}
		+ O\left(|\ln{T}|^{-\frac{1}{5}} \langle x\rangle^{-2} \left(1-\sum_{i=1}^{\mathfrak{o}}\eta\left( \frac{|x-q^{[i]}|}{\sqrt{T-t}}  \right) \right)
  \right),
	\end{aligned}
\end{equation*}
where $x=(\tilde{x},x_5)$, $\Theta_{l_i}$ is defined in \eqref{eq-heat-1} with $n=5$, and $\mu_i = \mu_i(t) \in C^1( [0,T), \mathbb{R}_+), \xi^{[i]} = \xi^{[i]}(t) \in C^1( [0,T), \mathbb{R}^4)$ satisfy
\begin{equation*}
\mu_i =
\big( D_i + O(|\ln T|^{-\frac{1}{15}}) \big) (T-t)^{2l_i+2},
\quad
\left| (\xi^{[i]},0)-q^{[i]}\right|
\le |\ln T|^{-\frac{1}{15}}
 (T-t)^{2l_i+2}
\end{equation*}
with some constants $D_i>0$ independent of $T$.
The initial value $u(\cdot,0) \in C_c^{\infty}(\overline{\mathbb{R}_+^5})$ with the support in $\cup_{i=1}^{\mathfrak{o}} \overline{B_5^+(q^{[i]}, 4\delta )}$.
In particular, $\|u(\cdot,t)\|_{L^\infty(\overline{B_5^+(q^{[i]}, 4\delta )})} \sim (T-t)^{-3l_i -3}$ for $i=1,2,\dots, \mathfrak{o}$.

\end{theorem}

\begin{remark}
The cut-off functions
$\eta(\frac{x-q^{[i]}}{2 \delta}), \eta(\frac{x-q^{[i]}}{\delta})$ are for the purpose of avoiding the influence between distinct blow-up points and
the multiple relationship between $2\delta$ and $\delta$ is not essential.
The radial property of $\eta(x)$ helps simplify the error on the boundary a lot.
See \eqref{qd24Jan13-3}.
\end{remark}
\begin{remark}
It is possible to add the type I blow-up rate in a solution with multiple rates.
\end{remark}

The proof relied on the parabolic gluing method established in the pioneering work \cite{Green16JEMS, 2020HMF}. This method has the powerful ability to analyze concentration phenomena in finite-time and infinite-time cases and
localize the blow-up points to achieve various concentration phenomena. We refer to \cite{2dEuler2020, 3dEuler-filament2022, 17halfHMF, 18type2Yamabeflow}.

We develop the linear theory for the inner problem with more time decay in the blow-up scheme through spatial restriction on the right-hand sides. We illustrate necessary norms before stating the next proposition.

Given a non-negative function $\ell(\tau)$, we define the following norms for the linear theory for the inner problem. For $-\infty < \tau_0<\tau_1\le \infty$, $\varsigma\in (0,1)$,
\begin{equation}\label{qd24Jan25-10}
	\begin{aligned}
		&	\|g\|_{\sigma,2+a,\ell(\tau),\mathbb{R}^n_+,\tau_0,\tau_1} :=
		\inf \left\{
		C \ | \
		|g(y,\tau)| \le C \tau^{\sigma} \langle y\rangle^{-2-a} \1_{|y|\le \ell(\tau)} \mbox{ \ for \ } y\in \mathbb{R}_+^n, \tau_0<\tau<\tau_1
		\right\} ,
		\\
		&	\|h\|_{\sigma,1+a,\ell(\tau),\mathbb{R}^{n-1},\tau_0,\tau_1} :=
		\inf \left\{
		C \ | \
		|h(\tilde{y},\tau)| \le C \tau^{\sigma} \langle \tilde{y}\rangle^{-1-a} \1_{|\tilde{y}|\le \ell(\tau)}  \mbox{ \ for \ } \tilde{y} \in \mathbb{R}^{n-1},   \tau_0 < \tau <\tau_1
		\right\} ,
		\\
		&
		Q((\tilde{y},\tau),r) :=
		\left\{ (\tilde{z},s)\in \mathbb{R}^{n-1} \times (\tau_0,\tau_1) \ | \ |\tilde{z}-\tilde{y}|<r, \tau-r^2<s<\tau \right\} ,
		\\
		&
		[h]_{C^{\varsigma,\frac{\varsigma}{2}} ( Q((\tilde{y},\tau),|\tilde{y}|/2) ) }
		:=\sup\limits_{(\tilde{y}^{[1]} ,\tau_1), (\tilde{y}^{[2]} ,\tau_2)
			\in Q((\tilde{y},\tau),|\tilde{y}|/2) }
		\frac{\left| h(\tilde{y}^{[1]} ,\tau_1) - h(\tilde{y}^{[2]} ,\tau_2) \right| }{\left( \max\{|\tilde{y}^{[1]}-\tilde{y}^{[2]}| , |\tau_1-\tau_2|^{1/2} \} \right)^{\varsigma} }  ,
		\\
		&	[h]_{\sigma,1+a,\ell(\tau),\varsigma,\mathbb{R}^{n-1},\tau_0,\tau_1} :=
		\inf \left\{
		C \ | \
		[h]_{C^{\varsigma,\frac{\varsigma}{2}} (Q((\tilde{y},\tau),|\tilde{y}|/2) ) }
		\le C \tau^{\sigma} \langle \tilde{y}\rangle^{-1-a-\varsigma} \1_{|\tilde{y}|\le 2\ell(\tau) }  \mbox{ \ for \ } \tilde{y} \in \mathbb{R}^{n-1},   \tau_0 < \tau <\tau_1
		\right\} ,
		\\
		&
		\|h\|_{\sigma,1+a,\ell(\tau),\varsigma,\mathbb{R}^{n-1},\tau_0,\tau_1} :=
		\|h\|_{\sigma,1+a,\ell(\tau),\mathbb{R}^{n-1},\tau_0,\tau_1}
		+
		[h]_{\sigma,1+a,\ell(\tau),\varsigma,\mathbb{R}^{n-1},\tau_0,\tau_1} .
	\end{aligned}
\end{equation}

We arrive at the following linear theory for the inner problem:
\begin{prop}\label{prop-23Oct24-1}
	
	Given an integer $n\ge 5$, consider
	\begin{equation}\label{eqn-23Oct24-1}
		\pp_\tau\phi=\Delta \phi+ g
		\mbox{ \ in \ }  \mathbb{R}^{n}_+\times (\tau_0, \tau_1),
		\quad
		-\pp_{y_n} \phi =
		\frac{n}{n-2} U^{\frac{2}{n-2}}\phi+h
		\mbox{ \ on \ }\pp \mathbb{R}^n_+\times(\tau_0,\tau_1) ,
		\quad
		\phi(y,\tau_0)= C_{\phi} \tilde{Z}_0(y)
		\mbox{ \ in \ }  \mathbb{R}^{n}_+ .
	\end{equation}
	Suppose that $1\le \tau_0 <\tau_1\le \infty$, $\ell(\tau)$ satisfies $C_\ell^{-1} \tau^p\le \ell(\tau) \le C_\ell \tau^p$ with a constant $C_\ell\ge 1$,
	\begin{equation}\label{23Oct24-1}
		2<a<n-2,
		\quad
		a^{-1}<p\le \frac{1}{2},
		\quad
		\iota\in (0,\frac{1}{4}),
		\quad
		\sigma -p a  +2\iota n > 0  ,
		\quad
		\varsigma\in (0,1),
	\end{equation}    $\|g\|_{\sigma,2+a,\ell(\tau),\mathbb{R}^n_+,\tau_0,\tau_1} <\infty$, $\|h\|_{\sigma,1+a,\ell(\tau),\varsigma,\mathbb{R}^{n-1},\tau_0,\tau_1}<\infty$, and
	$g=g(y,\tau)$, $h=h(\tilde{y},\tau)$ satisfy the orthogonality conditions
	\begin{equation}\label{gh-ortho-Oct20}
		\int_{\mathbb{R}^n_+} g(y,\tau)Z_j(y)dy
		+
		\int_{\mathbb{R}^{n-1}} h(\tilde{y},\tau) Z_j(\tilde{y},0) d \tilde{y} =0
		\quad
		\mbox{ \ for \ } \tau\in (\tau_0, \tau_1),
		\quad j=1,2,\dots,n
	\end{equation}
 with $Z_j$ given in \eqref{qd2023Nov28-1},
	$\tilde{Z}_0(y) \in C^2( \mathbb{R}^{n}_+ )\cap C^{1,\varsigma}(\overline{\mathbb{R}^{n}_+ })$  satisfies
\begin{equation}\label{til-Z0-Oct20}
		\tilde{Z}_0(y)=0 \mbox{ \ for \ }
		|y|\ge C_0,
		\quad
		\int_{\mathbb{R}_+^n}  \tilde{Z}_0  Z_j dy = 0 \quad \mbox{ \ for \ } j=1,2,\dots,n,
		\quad
		\int_{\mathbb{R}_+^n}  \tilde{Z}_0  Z_0 dy \ne 0
	\end{equation}
	with $Z_0$ given in \eqref{Z0-eq} and a constant $C_0>0$,
	then there exist $\phi=\phi[g,h]$ and a constant $C_{\phi}=C_{\phi}[g,h]$ as linear mappings of $g,h$ solving
	\eqref{eqn-23Oct24-1} and satisfying
	\begin{equation}\label{23Oct24-2}
		\int_{\mathbb{R}^n_+}\phi(y,\tau)Z_j(y)dy=0
		\quad \mbox{ \ for \ }  \tau\in (\tau_0, \tau_1),
		\quad
		j=1,2, \dots,n,
	\end{equation}
and
	\begin{equation}\label{23Oct24-3}
		\begin{aligned}
			&
			|\phi|\lesssim
			\Big(
			\tau^\sigma \langle y \rangle^{-a}
			\1_{|y|\le \ell(\tau) }
			+
			\tau^{\sigma} \ell^{-a}(\tau) e^{- \iota \frac{|y|^2}{\tau}}
			\1_{|y|> \ell(\tau) }
			\Big) \Big( \|g\|_{\sigma,2+a,\ell(\tau),\mathbb{R}^n_+,\tau_0,\tau_1} + \|h\|_{\sigma,1+a,\ell(\tau),\mathbb{R}^{n-1},\tau_0,\tau_1} \Big),
			\\
			&
			|\nabla\phi|  \lesssim
			\Big(
			\tau^\sigma \langle y \rangle^{-1-a}
			\1_{|y|\le \ell(\tau) }
			+
			\tau^{\sigma} \ell^{-a}(\tau) |y|^{-1}
			\1_{\ell(\tau)<|y|\le \tau^{\frac{1}{2}} }
			+
			\tau^{\sigma-\frac{1}{2}} \ell^{-a}(\tau)  e^{-  \tilde{\iota}  \frac{|y|^2}{\tau}}
			\1_{ |y|> \tau^{\frac{1}{2}} }
			\Big)
			\\
			&
			\qquad  \qquad  \times
			\Big( \|g\|_{\sigma,2+a,\ell(\tau),\mathbb{R}^n_+,\tau_0,\tau_1} + \|h\|_{\sigma,1+a,\ell(\tau),\varsigma,\mathbb{R}^{n-1},\tau_0,\tau_1} \Big),
			\\
			&
			|C_{\phi}| \lesssim \tau_0^\sigma \Big( \|g\|_{\sigma,2+a,\ell(\tau),\mathbb{R}^n_+,\tau_0,\tau_1} + \|h\|_{\sigma,1+a,\ell(\tau),\mathbb{R}^{n-1},\tau_0,\tau_1} \Big)
		\end{aligned}
	\end{equation}
	with a constant $\tilde{\iota} \in (0,\iota)$, where all ``$\lesssim$'' are independent of $\tau_0, \tau_1, g, h$.
	
\end{prop}

\begin{remark}\label{qd24Jan14-9-rem}
	By Lemma \ref{z2023Nov10-lem}, there exists $\tilde{Z}_0(y) \in  C^{\infty}(\overline{\mathbb{R}^{n}_+ })$ satisfying the assumption \eqref{til-Z0-Oct20}.
\end{remark}

\begin{remark}
The orthogonality conditions \eqref{gh-ortho-Oct20} are initiated from \cite[Proposition 4.1, 4.2]{FreeHMF2023}.
\end{remark}

\textbf{Main difficulties and novelties.}

$\bullet$
Due to the lack of an ODE method for the linearized equations around steady state solution $U(x)$, the strategy for deriving the linear theory for the inner problem in \cite[Section 7]{Green16JEMS} does not work here. Instead, we resort to the parabolic blow-up argument to get a desired linear theory. The parabolic blow-up argument was first introduced in the linear theory of mode $1$ in \cite[Section 7.3]{2020HMF}, and this argument works for many heat flows of Schr\"odinger operator with nondegenerate property. See \cite[Section 5.1]{2019-fracInfinite} for instance. In the proof of \cite[Proposition A.2]{2022-gluingsurvey}, the authors modified the blow-up argument when the spatial variable goes to infinity.

For $\mathcal{T}_{\mathbb{R}^n }\left[\cdot\right]$ defined in \eqref{zz24Jan02-3} with $n>2$,
if $b \le 2$, $\lim\limits_{t\to \infty}t^{-a}\mathcal{T}_{\mathbb{R}^n }\left[ t^a \langle x \rangle^{-b} \right](0,t) =\infty$ by the similar lower bound estimate in \cite[Lemma A.1]{infi4D};
if $b\ge n$ or $a \le -1$, $\lim\limits_{t\to \infty}
\big\{
\big[(t^a \langle x \rangle^{2-b} )^{-1} \mathcal{T}_{\mathbb{R}^n }\left[ t^a \langle x \rangle^{-b} \right](x,t) \big]\big|_{|x|=\sqrt{t}}
\big\} =\infty $ by
similar lower bound estimates in  \cite[Lemma A.1]{infi4D} and \cite[Lemma A.2]{infi4D} respectively.
Thus
the parameter about the growth of time in \cite[Proposition A.2]{2022-gluingsurvey} should be optimal
in the algebraic power sense.

To obtain better time decay for the linear theory in a wider application, we impose restrictions on the spatial growth of the right-hand side. Additionally, we discover that the scaling argument when the spatial variable goes to infinity is not necessary in previous works. Instead, we utilize convolution estimates directly.

$\bullet$
For the outer problem, Harada \cite{harada2019higher,harada6D} used a linear combination of eigenfunctions corresponding $\Delta_z -\frac{z}{2} \cdot\nabla_z$ with cut-off functions as the modified functions to adjust the time vanishing at the blow-up point in the outer problem.
However, the algebraic growth of these eigenfunctions and the rough cut-off functions cause complexity in calculation.

We realized that the modified functions
need not be eigenfunctions corresponding $\Delta_z -\frac{z}{2} \cdot\nabla_z$.
The key point is that the modified functions are not orthogonal with these eigenfunctions, so we have more flexibility in the choice of modified functions and simplify the proof for the outer problem. See Corollary \ref{orth-lem} and the application in Section \ref{section-outer}.

$\bullet$
In order to derive multiple rates at distinct points, we introduce vanishing adjustment functions to eliminate derivatives at arbitrarily prescribed finitely many points.
 This method works for more general parabolic equations. See Proposition \ref{lemma-cons1}.

The method of adjusting the initial value for improving the vanishing condition of the outer problem was implemented in \cite[p.386]{2020HMF} for the case of one bubble with the first rate. Zhang and Zhao \cite[p.8]{Zhang-Zhao2023} introduced the idea of using the heat equation with some special initial value to improve the vanishing at the blow-up points for multiple rates. We make this process clearer and summarize it as the vanishing adjustment functions. A counterpart for the Schr\"odinger equation is given in \cite[Lemma 4.1]{BourgainWang97}.

$\bullet$
The compactness argument for the parameters in the initial value in the fixed-point argument in Subsection \ref{qd24Jan29-1-sec} deduces a smooth initial value $u(x,0)$ with compact support. This can not be derived from the parabolic regularity theory. Although for a fixed $t\in (0, T)$, $u(x,t)$ is smooth, the support of $u(x,t)$ is usually not compact.

$\bullet$
Convolution estimates in Appendix \ref{convo-sec} are in the spirit of  \cite[Appendix A]{infi4D} (See also \cite[Appendix B.1]{sun2021bubble}). We combine the comparison theorem to approach the best constant in the power of the exponential term. Appendix \ref{convo-sec} is prepared for the proof of Proposition \ref{prop-23Oct04-1} and also establishes the foundation for the long-time dynamics of \eqref{qd2023Nov21-1}.

{\bf Notations:}
\begin{itemize}
\item
Denote $\mathbb{R}_+ =\left(0,\infty\right)$,  $\overline{\mathbb{R}_+} =\left[0,\infty\right)$,  $\mathbb{R}_+^n = \mathbb{R}^{n-1} \times \mathbb{R}_+$,
	$\overline{\mathbb{R}_+^n} = \mathbb{R}^{n-1} \times \overline{\mathbb{R}_+}$.

\item
Denote the natural number set $\mathbb{N} = \left\{ 0,1,\dots \right\}$, the set of natural numbers n-tuples
 $\mathbb{N}^n = \underbrace{\mathbb{N}\times \mathbb{N} \times \dots \times \mathbb{N} }_{\mbox{ \rm n multiplicities} }$.
For $C_1, C_2 \in \mathbb{R}$, denote $C_1 \mathbb{N} + C_2 = \left\{ C_1 i + C_2 \ | \ i\in \mathbb{N} \right\}$. In particular, $2\mathbb{N} \ (2\mathbb{N} + 1)$ is the set of even (odd) numbers.

\item
For a $m\times n$ matrix $A=\left(A_{ij}\right)_{m\times n}$, denote $\| A \|_{\ell_1} =\sum_{i=1}^{m} \sum_{j=1}^n |A_{ij}|$.

\item
For $x\in \mathbb{R}^n$, denote the Japanese bracket $\langle x\rangle =\sqrt{1+|x|^2}$.

\item
Given $q\in \mathbb{R}^n$, $r>0$, denote $B_{n}(q,r) = \left\{ x\in \mathbb{R}^n \ | \ \left|x-q\right|<r \right\}$, $B_{n}^+(q,r) = B_{n}(q,r) \cap \mathbb{R}_+^n$.

\item For any $x\in \mathbb{R}^n$, we use $\tilde{x}$, $x_n$  to denote the first $n-1$ components and the nth component of $x$ respectively.

\item
$C(a,b,\dots)$ denotes a constant only depending on parameters $a,b,\dots$.

\item We write $a\lesssim b$ (resp. $a \gtrsim b$) if there exists a  constant $C > 0$ independent of $T$ such that $a \le  Cb$ (resp. $a \ge  Cb$). Set $a \sim b$ if $b \lesssim a \lesssim b$.
Given a non-negative function $g$, $O(g)$ denotes some function $f$ satisfying $|f|\lesssim g$.

\item We write $a\lesssim_{\alpha, \beta,\dots} b$ (resp. $a \gtrsim_{\alpha, \beta,\dots} b$) to emphasize that there exists a  constant $C(\alpha, \beta,\dots) > 0$ such that $a \le  C(\alpha, \beta,\dots) b$ (resp. $a \ge  C(\alpha, \beta,\dots) b$). Set $a \sim_{\alpha, \beta,\dots} b$ if $b \lesssim_{\alpha, \beta,\dots} a \lesssim_{\alpha, \beta,\dots} b$.

\item For constants $C_1, C_2>0$, the symbol $C_1 \ll C_2$ denotes that there exists a constant $c>0$ sufficiently small such that $C_1\le c C_2$.

\item  For $C\in \mathbb{R}$, the ceiling function $\lceil C \rceil $ denotes the smallest integer greater than or equal to $C$.

\item Given a positive integer $n$, a multi-index $\mathbf{m}=(m_1, m_2, \dots, m_n) \in \mathbb{N}^n$, and a sufficiently smooth function $f$ defined in a domain of $\mathbb{R}^n$, denote $D_{x}^{\mathbf{m}} f = \partial_{x_1}^{m_1}  \partial_{x_2}^{m_2} \cdots  \partial_{x_n}^{m_n} f$.

\item  $\eta(x)$ is a smooth radial cut-off function in $\mathbb{R}^n$ satisfying  $\eta(x)=1$ for $|x|\le 1$, $\eta(x)=0$ for $|x|\ge 2$, and $0\le \eta(x) \le 1$ in $ \mathbb{R}^n$.

\item Denote $\1_{\Omega}(x)$ as the indicator function with $\1_{\Omega}(x)=1$ if $x\in \Omega$ and $\1_{\Omega}(x)=0$ if $x\not\in \Omega$. We will use $\1_{\Omega}$ to denote $\1_{\Omega}(x)$ if no ambiguity.

\end{itemize}

The structure of this paper is as follows. Section \ref{Preliminary} is the preliminary to prepare some backgrounds and useful estimates. In Section \ref{Approximate solution and inner-outer gluing scheme}, we calculate the errors of the approximate solution and give the inner-outer gluing system. Section \ref{app-blowup} builds the linear theory of inner problem by blow-up argument. Section \ref{formal derivation} gives the formal calculation of $\mu_i$ and introduces the topology for the fixed-point argument. In Section \ref{section-outer}, we study the outer problem and obtain apriori estimates. Finally, in Section \ref{section-gluing}, we use the Schauder fixed-point theorem to solve the inner-outer gluing system away from $T$ and then deduce Theorem \ref{main}.

\section{Preliminary}\label{Preliminary}

\subsection{Kernels and the eigenfunction with negative eigenvalue}

For $n\ge 3$,
by dilation and translation invariance \eqref{qd202325-1}, the linearized operator of the steady equation of \eqref{qd2023Nov21-1} around $U$  has kernels
\begin{equation}\label{qd2023Nov28-1}
\begin{aligned}
&
Z_i(x):=\pp_{x_i} U
=
-(n-2)^{\frac{n}{2}} \left[
|\tilde{x}|^2 +\left(1+x_n\right)^2
\right]^{-\frac{n}{2}} x_i
,
\quad
i=1,2,\dots, n-1,
\\
&
Z_n(x):=\frac{n-2}{2}U+x\cdot\nabla U
=
2^{-1}
(n-2)^{\frac{n}{2}} \left[ |\tilde{x}|^2 + (1+x_n)^2\right]^{-\frac{n}{2}} \left( 1-|x|^2\right) ,
\end{aligned}
\end{equation}
which satisfy
\begin{equation}\label{Zj-eq}
	\Delta Z_i =0 \quad \mbox{ \ in \ } \mathbb{R}_+^n,
	\quad
-\pp_{x_n} Z_i = \frac{n}{n-2}  U^{\frac{2}{n-2} } Z_i
\quad
	\mbox{ \ on \ } \pp \mathbb{R}_+^n,
	\qquad
i=1,2,\dots, n.
\end{equation}	

By Proposition \ref{z24Jan25-6-prop}, the corresponding eigenvalue problem
\begin{equation}\label{Z0-eq}
				- \Delta Z_0=\lambda_0 Z_0
			\quad	\mbox{ \ in \ }  \mathbb{R}^{n}_+,
		\quad
				-\pp_{x_n} Z_0  =\frac{n}{n-2}U^{\frac{2}{n-2}}Z_0
			\quad
				 \mbox{ \ on \ }\pp \mathbb{R}^n_+
		\end{equation}
has only one negative eigenvalue $\lambda_0$ and $\lambda_0$ is simple with an eigenfunction  $Z_0(x) \in C^\infty(\overline{\mathbb{R}_+^n}) \cap H^1(\mathbb{R}_+^n)$ satisfying $\| Z_0 \|_{L^2(\mathbb{R}^n_+)} =1$,
$0<Z_0(x) \le C e^{-\nu|x|} $ in $\overline{\mathbb{R}_+^n}$ for all  $\nu \in \left[0, \sqrt{-\lambda_0} \right)$
with a constant $C$ depending on $n, \lambda_0, \nu$.

\subsection{Properties of the operator $A_z=\Delta_z-\frac{z}{2} \cdot \nabla_z$}\label{eigen}

Given a domain $\Omega \subset \mathbb{R}^n$,
define the weighted $L^2(\Omega)$ space by
\begin{equation}\label{z24Jan25-8}
L_\rho^2(\Omega):=\big\{f \ | \ \|f\|_{L_{\rho}^2(\Omega)} < \infty\big\}
\end{equation}
equipped with the inner product and the norm
\begin{equation*}
\left(f_1,f_2\right)_{L_{\rho}^2(\Omega)}:=\int_{\Omega}f_1(z)f_2(z)\rho(z)dz ,
\quad \rho(z) := e^{-\frac{|z|^2}{4}},
\quad
\left\|f\right\|_{L_{\rho}^2(\Omega)} :=
\left(f,f\right)_{L_{\rho}^2(\Omega)}^{1/2}
.
\end{equation*}

Define the weighted $H^1\left( \Omega\right)$ space by
\begin{equation}\label{z24Jan25-7}
H_\rho^1(\Omega):=\left\{f \ | \ f , \nabla f \in L^2_\rho(\Omega) \right\}
\end{equation}
equipped with the inner product and the norm
\begin{equation*}
\left(f_1,f_2\right)_{H_\rho^1(\Omega)} :=
\left(\nabla f_1, \nabla f_2\right)_{L_{\rho}^2(\Omega)}
+
\left(f_1,f_2\right)_{L_{\rho}^2(\Omega)}
,
\quad
\left\|f\right\|_{H_\rho^1(\Omega)} :=
\left(f,f\right)_{H_\rho^1(\Omega)}^{1/2} .
\end{equation*}

Denote $A_z=\Delta_z-\frac{z}{2} \cdot \nabla_z$. For $f,g \in C^2(\mathbb{R}_+^n)\cap C^1(\overline{\mathbb{R}_+^n})$,
\begin{equation}\label{qd23Dec11-2}
\begin{aligned}
	\left(A_z f, g\right)_{L_{\rho}^2(\mathbb{R}_+^n)}
	= \ &
-(\nabla f ,\nabla g)_{L_{\rho}^2(\mathbb{R}_+^n)}
+
\int_{\partial \mathbb{R}_+^n}
g(z) e^{-\frac{|z|^2}{4}} \left(-\partial_{z_n} f\right) dS
	\\
	= \ &
	\left(A_z g, f\right)_{L_{\rho}^2(\mathbb{R}_+^n)}+
	\int_{\partial \mathbb{R}_+^n}  g(z) e^{-\frac{|z|^2}{4}} \left(-\partial_{z_n} f\right)(z)
	dS
	-
	\int_{\partial \mathbb{R}_+^n}
	f(z)
	e^{-\frac{|z|^2}{4}} \left(-\partial_{z_n} g\right)(z)
	dS
	.
\end{aligned}
\end{equation}

Denote $\tilde{H}_\alpha(s) = H_\alpha(\frac{s}{2})$, $\alpha\in\mathbb{N}$, $s\in \mathbb{R}$, where $H_\alpha(r) = \left(-1\right)^\alpha e^{r^2} \frac{d^\alpha }{dr^\alpha} (e^{-r^2})$ is the Hermite polynomial. Two basic properties of $\tilde{H}_{\alpha}(x)$ are given in the following lemma without proof.
\begin{lemma}\label{z24Jan31-2-lem}

\begin{enumerate}
\item\label{qd24Jan31-1} $\tilde{H}_{\alpha}(x)$ is an even (odd) polynomial of order $\alpha$ provided that $\alpha$ is an even (odd) number. And $\tilde{H}_{\alpha}(0)\ne 0$ if $\alpha$ is even, and $\frac{d}{dx} \tilde{H}_{\alpha}(0) \ne 0$ if $\alpha$ is odd.

\item $ (\tilde{H}_{\alpha} )_{\alpha \in \mathbb{N}}$ is the basis of the eigenfunctions of $- \left(\pp_{ss} -\frac{s}{2}\pp_s \right)$ in $L_\rho^2(\mathbb{R})$ satisfying $- \left(\pp_{ss} -\frac{s}{2}\pp_s \right) \tilde{H}_{\alpha}(s)  = \frac{\alpha}{2} \tilde{H}_{\alpha}(s)$, and is an orthogonal basis in $L_\rho^2(\mathbb{R})$.

\end{enumerate}

\end{lemma}

Denote a multi-index $\bm{\alpha}=(\alpha_1, \alpha_2,\dots, \alpha_n)$ with each $\alpha_i\in\mathbb{N}$, and ${\bf{ \tilde{H} } }_{\bm{\alpha}}(z):= \prod_{i=1}^n \tilde{H}_{\alpha_i}(z_i ) $.
By separation of variables, $( {\bf{ \tilde{H} } }_{\bm{\alpha}} )_{ {\bm{\alpha}} \in \mathbb{N}^n }$ is the basis of the eigenfunctions of $-A_z$ in $L_\rho^2(\mathbb{R}^n)$ satisfying $-A_z {\bf{ \tilde{H} } }_{\bm{\alpha}}(z)  =  2^{-1} \|\bm{\alpha}\|_{\ell_1}  {\bf{ \tilde{H} } }_{\bm{\alpha}}(z)$, and is an orthogonal basis in $L_\rho^2(\mathbb{R}^n)$.

Given a domain $\Omega \subset \mathbb{R}^n$, denote
\begin{equation}\label{eigen3}
	\begin{aligned}
&
\mathcal{E}_{i/2}(\Omega) :=
\left\{ {\bf {\tilde{H} }}_{\bm{\alpha}}(z), z\in \Omega  \ | \  |\bm{\alpha}| = i
\right\} ,
\\
		&
		\mathcal{E}_{i/2,\rm{even}}(\Omega):=
		\left\{ {\bf \tilde{H}}_{\bm{\alpha} }(z) \in \mathcal{E}_{i/2}(\Omega)  \ | \ \alpha_n \text{ is even}
		\right\}
		,
	\quad
		\mathcal{E}_{i/2,\rm{odd}}(\Omega):=
		\left\{ {\bf \tilde{H}}_{\bm{\alpha}}(z) \in \mathcal{E}_{i/2}(\Omega)  \ | \ \alpha_n \text{ is odd}\right\} .
	\end{aligned}
\end{equation}

Some basic properties of Hermite polynomials are listed in the next lemma without proof.
\begin{lemma}\label{lem-Az}
\begin{enumerate}
\item\label{z24Jan31-5}
The eigenvalue problem
$ -A_z h =\lambda h $ in $\mathbb{R}^n$, $h \in L_\rho^2(\mathbb{R}^n)$ has the eigenvalue $\lambda=\frac{i}{2}$, $i\in \mathbb{N}$ with $\mathcal{E}_{i/2}(\mathbb{R}^n)$ as the orthogonal basis of the associated eigenspace.

\item
For $f_1, f_2 \in \cup_{j=0}^\infty\mathcal{E}_{j/2}(\mathbb{R}^n)$ and $f_1\ne f_2$, we have $\left(f_1,f_2\right)_{L_\rho^2(\mathbb{R}^n) }=0$.

\item\label{z24Jan31-6}
For $i\in \mathbb{N}$, denote $S_1 = \left\{ f\in H_\rho^1(\mathbb{R}^n) \ | \ (f,g)_{L_\rho^2(\mathbb{R}^n) }=0 \mbox{ \ for all
\ }
g\in \cup_{j=0}^i \mathcal{E}_{j/2}(\mathbb{R}^n) \right\}$, then
\begin{equation}\label{eigen3-2}
\inf_{ 0\not\equiv f\in S_1}
 \left\|f\right\|_{L_\rho^2(\mathbb{R}^n) }^{-2}
\left\|\nabla f\right\|_{L_\rho^2(\mathbb{R}^n) }^2
=
\frac{i+1}{2}
.
\end{equation}
\end{enumerate}

\end{lemma}

A counterpart in the half space $\mathbb{R}_+^n$ is the following.
\begin{lemma}\label{lemma-eigen}

	\begin{enumerate}
		\item\label{zz23Nov30-1}
		The eigenvalue problem with the Neumann (Dirichlet) boundary condition
	\begin{equation}\label{zz23Nov30-4}
-A_{z} h =\lambda h
\quad
\mbox{ \ in \ } \mathbb{R}_+^n,
\quad
-\partial_{z_n} h=0 \  ( h =0)
\quad
\mbox{ \ on \ }  \pp\mathbb{R}_+^n
	\end{equation}
in $L_\rho^2(\mathbb{R}_+^n)$
has the eigenvalue $\lambda=\frac{i}{2}$ with $i\in \mathbb{N}$,
and $\mathcal{E}_{i/2,\rm{even}}(\mathbb{R}_+^n) \ ( \mathcal{E}_{i/2,\rm{odd}}(\mathbb{R}_+^n) )$ is the orthogonal basis of the associated eigenspace.

\item\label{zz23Nov30-3}
For $f_1, f_2 \in \cup_{j=0}^\infty \mathcal{E}_{j/2,\rm{even}}(\mathbb{R}_+^n) \ ( \mathcal{E}_{j/2,\rm{odd}}(\mathbb{R}_+^n) )$ and $f_1\ne f_2$, we have $\left(f_1,f_2\right)_{L_\rho^2(\mathbb{R}_+^n)}=0$.

\item\label{zz23Nov30-2}
For any $i\in \mathbb{N}$, denote $S_2 = \left\{ f\in H_\rho^1(\mathbb{R}_+^n) \ | \
(f,g)_{L_\rho^2(\mathbb{R}_+^n)}=0  \mbox{ \ for all \ } g\in
\cup_{j=0}^i \mathcal{E}_{j/2,\rm{even}}(\mathbb{R}_+^n) \right\}$,
		\begin{equation}\label{eigen3-5}
			\inf_{ 0\not\equiv f\in S_2}
	 \left\|f\right\|_{L_\rho^2(\mathbb{R}_+^n)}^{-2}
		\left\|\nabla f\right\|_{L_\rho^2(\mathbb{R}_+^n)}^2
		= \frac{i+1}{2}.
		\end{equation}
	
	\end{enumerate}

\end{lemma}

\begin{proof}

Given any function $f$ in $\overline{\mathbb{R}_+^n}$, denote
\begin{equation*}
f_{\rm e}(z) :=
\begin{cases}
f(z) & \mbox{ \ if \ } z_n\ge 0
\\
f(\tilde{z}, -z_n) & \mbox{ \ if \ } z_n < 0
\end{cases}
\mbox{ \ and \ }
f_{\rm o}(z) :=
\begin{cases}
f(z) & \mbox{ \ if \ } z_n\ge 0
\\
-f(\tilde{z}, -z_n) & \mbox{ \ if \ } z_n < 0.
\end{cases}
\end{equation*}

\eqref{zz23Nov30-1}.
For the Neumann boundary condition, obviously $h_{\rm e}(z) \in L_{\rho}^2(\mathbb{R}^n)$ and it is straightforward to derive $-A_{z} h_{\rm e}(z)=\lambda h_{\rm e}(z) $ in $\mathbb{R}^n$. By Lemma \ref{lem-Az} \eqref{z24Jan31-5}, $\partial_{z_n} h_{\rm e} |_{z_n=0} =0$, and Lemma \ref{z24Jan31-2-lem} \eqref{qd24Jan31-1}, then $\lambda\in \frac{1}{2} \mathbb{N}$ and $h\in \mathcal{E}_{\lambda,\rm{even}}(\mathbb{R}_+^n)$. On the other hand, all elements in $ \mathcal{E}_{i/2,\rm{even}}(\mathbb{R}_+^n)$, $i \in \mathbb{N}$ satisfy \eqref{zz23Nov30-4} with $\lambda= i/2$.

For the Dirichlet boundary condition, the conclusion can be deduced by analyzing $h_{\rm o}(z)$ similarly.

\eqref{zz23Nov30-3}.	
For any $g_1, g_2$ in $\overline{\mathbb{R}_+^n}$, $\left(g_{1{\rm e}}, g_{2{\rm e}} \right)_{L_\rho^2(\mathbb{R}^n)} =2 \left(g_1,g_2\right)_{L_\rho^2(\mathbb{R}_+^n)} = \left(g_{1{\rm o}}, g_{2{\rm o}} \right)_{L_\rho^2(\mathbb{R}^n)}$. Then the result holds.
	
\eqref{zz23Nov30-2}. For any $f\in S_2$, we have $f_{\rm e}(z)\in H^1_\rho(\mathbb{R}^n)$;
$ \left(f_{\rm e}, g\right)_{L_\rho^2(\mathbb{R}^n)}=0 $ for all
$ g \in \mathcal{E}_{j/2,\rm{odd}}(\mathbb{R}^n), j\in \mathbb{N} $;
$(f_{\rm e}, g)_{L_\rho^2(\mathbb{R}^n)}=2(f,g)_{L_\rho^2(\mathbb{R}_+^n)}=0 $ for all $g\in \cup_{j=0}^{i} \mathcal{E}_{j/2,\rm{even}}(\mathbb{R}^n)$. Applying Lemma \ref{lem-Az} \eqref{z24Jan31-6} to $f_{\rm e}$, we get $\inf_{ 0\not\equiv f\in S_2}
	 \left\|f\right\|_{L_\rho^2(\mathbb{R}_+^n)}^{-2}
		\left\|\nabla f\right\|_{L_\rho^2(\mathbb{R}_+^n)}^2
		\ge \frac{i+1}{2}$. The equality sign can be attained by $\mathcal{E}_{(i+1)/2,\rm{even}}(\mathbb{R}_+^n)$.
\end{proof}

The next lemma gives a general localized modification method for approximate orthogonal functions.
\begin{lemma}\label{z2023Nov10-lem}
	Given a domain $\Omega\subset \mathbb{R}^n$ (possibly unbounded), a weight $w(z)$ (possibly sign-changing), and an integer $m \ge 1$, suppose that for
	$i,j = 1,2,\dots, m$, $\vartheta_i(z) \vartheta_j(z) w(z) \in L_{\rm loc}^1(\Omega)$, a function $\chi(z) \in L^\infty(\Omega)$ with compact support,
	\begin{equation*}
		\begin{aligned}
			&
			\sum\limits_{l=1,l\ne i}^m \left|
			\int_{\Omega} \vartheta_i(z)  \vartheta_l(z) w(z) \chi(z) dz \right| < \left|
			\int_{\Omega} \vartheta_i^2(z)  w(z) \chi(z) dz \right| <\infty
			\quad
			\mbox{ \ for \ }
			i =  1,2,\dots, m ,
		\end{aligned}
	\end{equation*}
	then
	there exist functions $\tilde{\vartheta}_i(z)$, $i=1,2,\dots,m$ of the form $\tilde{\vartheta}_i(z) = \sum\limits_{l=1}^{m} a_{il} \vartheta_l (z) \chi(z)$, where $\left(a_{il}\right)_{m\times m}$ is the inverse matrix of  $\left( \int_{\Omega} \vartheta_i(z)  \vartheta_l(z) w(z) \chi(z) dz \right)_{m\times m}$, such that
	\begin{equation*}
		\int_{\Omega} \tilde{\vartheta}_i(z) \vartheta_j(z) w(z) dz
		=
		\delta_{ij} \quad \mbox{ \ for \ } i, j =1,2,\dots, m.
	\end{equation*}
	
\end{lemma}

\begin{proof}
The conclusion is deduced by the non-singularity of the strictly diagonally dominant matrix.
\end{proof}

We arrange the eigenfunctions of \eqref{zz23Nov30-4} with the Neumann boundary condition according to non-decreasing eigenvalues and label them as $\left(e_i(z)\right)_{i}$, $i =0,1,\dots$.
Due to the multiplicity of eigenspaces, the arrangement appears to be non-unique.
We only fix one sequence $\left(e_i(z)\right)_{i}$ and $e_i(z)$ satisfies
\begin{equation}\label{qd2023Dec17-2}
	-A_{z} e_i(z) =\lambda_i e_i(z)
 \quad
\mbox{ \ in \ } \mathbb{R}_+^n,
	\quad
	-\partial_{z_n}e_i(z) =0
\mbox{ \ on \ }  \pp\mathbb{R}_+^n
 ,
\end{equation}
where the eigenvalues $\lambda_i$ are non-decreasing about $i$ and $\lambda_i \to \infty$ as $i\to \infty$. Define the eigenvalue counting function of $-A_{z}$ by
\begin{equation}\label{qd24Feb13-1}
N(C) := \# \left\{ i\in \mathbb{N} \ | \   \lambda_i \le C \right\}  , \quad C\in \mathbb{R} .
\end{equation}

\begin{corollary}\label{orth-lem}
Given $m\in \mathbb{N}$, there exists a constant $M>0$ sufficiently large and a non-singular symmetric constant real-valued matrix $(a_{il})_{(m+1)\times (m+1)}$ to  make
	\begin{equation}\label{eq2.16}
		\tilde{e}_i(z) := \sum\limits_{l=0}^{m} a_{il} e_l (z) \eta(\frac{z}{M}),\quad i=0,1, \dots, m
	\end{equation}
	satisfy
	\begin{equation}\label{eq2.17}
		\partial_{z_n}\tilde{e}_i=0
		\quad	\mbox{ \ on \ }\pp\mathbb{R}_+^n
		\quad
		\mbox{ \ and \ }
		\quad
		 \left( \tilde e_i , e_j  \right)_{L_\rho^2(\mathbb{R}_+^n)} =\delta_{ij }
		\quad \mbox{ \ for \ } i, j =0,1,\dots, m.
	\end{equation}
	
\end{corollary}

\begin{proof}
By Lemma \ref{lemma-eigen} \eqref{zz23Nov30-3}, and Lemma \ref{z2023Nov10-lem} in the case $\Omega = \mathbb{R}_+^n$, $w(z) = \rho(z)$, $(e_i(z))_{0\le i \le m}$, $\chi(z)= \eta(\frac{z}{M})$ with $M>0$ sufficiently large, we get the desired $\tilde{e}_i(z)$ satisfying the second property in \eqref{eq2.17}.
Since $\eta(x)$ is a smooth radial cut-off function and $\partial_{z_n}e_i =0$, $i\in \mathbb{N}$
on $\pp\mathbb{R}_+^n $, we have
$\partial_{z_n}\tilde{e}_i=0$ on $ \pp\mathbb{R}_+^n $.
\end{proof}

\subsection{Green's functions in $\mathbb{R}_+^n$}

The fundamental solution with the Neumann boundary condition in $\mathbb{R}_+^n$ has the form
\begin{equation}\label{qd23Dec22-1}
	G_n(x,t,z,s) =\left[ 4\pi (t-s)\right]^{-\frac{n}{2}}
	\left[
	e^{-\frac{|\tilde{x} - \tilde{z}|^2+ |x_n - z_n|^2 }{4(t-s)}}
	+
	e^{-\frac{|\tilde{x} -\tilde{z}|^2+ |x_n + z_n|^2 }{4(t-s)}}  \right]
	\mbox{ \ for \ } t>s .
\end{equation}
Obviously,
\begin{equation}\label{qd24Feb04-1}
\begin{aligned}
&
\partial_t^\iota D_{x}^{\mathbf{m}} G_n(x,t,z,s), \iota \in \mathbb{N},
\mathbf{m} \in \mathbb{N}^n
\mbox{ is odd (even) about $x_n$ if $m_n$ is odd (even) for } t>s.
\\
&
\partial_t^\iota D_{x}^{\mathbf{m}} G_n((\tilde{x},0),t,z,s) = 0
\mbox{ if $m_n$ is odd for } t>s.
\end{aligned}
\end{equation}

By Green's identities,  for $f$ with sufficient smoothness and decay, we have
\begin{equation}\label{qd2023Dec03-3}
	\begin{aligned}
		f(x,t) = \ &
		\int_{t_0}^{t}
		\int_{\mathbb{R}_+^n}
		G_n(x,t,z,s) \left( \partial_s f - \Delta_z f \right)(z,s) dz ds
		\\
		&
		+
		\int_{t_0}^{t}\int_{\mathbb{R}^{n-1}}
		G_n(x,t,(\tilde{z},0),s) 	\left[\left(-\pp_{z_n} f\right)(  ( \tilde{z},0),s ) \right]  d\tilde{z} ds
		+
		\int_{\mathbb{R}_+^n}
		G_n(x,t,z, t_0) f(z,t_0)  dz.
	\end{aligned}
\end{equation}
\eqref{qd2023Dec03-3} can be used for representing solutions in some weak sense with right-hand sides, boundary value and initial value in some Sobolev spaces.

\begin{lemma}\label{z24Feb13-2-lem}
		
		Let $n\ge 2$ be an integer, $-\infty < t_1 < t_2 < \infty$,
		\begin{equation}\label{qd2023Dec11-1}
				\pp_t \bar{\theta} = \Delta_x \bar{\theta}
				+
				g_1(x,t) \mbox{ \ in \ }
				\mathbb{R}_+^n \times (t_1,t_2),
	\quad
				-\partial_{x_n } \bar{\theta}
				=
				g_2(\tilde{x} ,t )
	 \mbox{ \ on \ }  \partial\mathbb{R}_+^n \times (t_1,t_2),
	\quad
				\bar{\theta}(x ,t_1 )  = g_3 (x )
				 \mbox{ \ in \ } \mathbb{R}_+^n .
		\end{equation}
		Given $t_3\in [t_2,\infty)$, $q\in \partial \mathbb{R}_+^n$, set
		\begin{equation*}
		\begin{aligned}
&
			s = s(t) = - \ln (t_3-t ) \in (s_1,s_2),
\quad
s_i = s(t_i) = - \ln (t_3-t_i), \ i=1,2 ,
			\quad
			z = z(x ,t ) = (t_3-t )^{-\frac 1 2 } \left(x -q\right),
\\
&
\mbox{that is, }
\quad
t  = t(s) = t_3-e^{-s} ,
\quad
t_i =t(s_i) =t_3-e^{-s_i} , \ i=1,2 ,
\quad
x = x(z,s) = e^{-\frac{s}{2}} z + q ;
\\
&
\bar{\theta}(x,t):= \theta \big( (t_3-t)^{-\frac 1 2 } \left(x -q\right) , - \ln (t_3-t) \big)
,
\quad
\mbox{ that is, }
\quad
\theta(z,s) = \bar{\theta}\big( e^{-\frac{s}{2}} z +q, t_3-e^{-s} \big)  .
		\end{aligned}
		\end{equation*}
Then \eqref{qd2023Dec11-1} is equivalent to
		\begin{equation*}
			\begin{cases}
				\partial_s \theta
				=   \Delta_z \theta
				-
				\frac{z}{2} \cdot \nabla_z \theta
				+
				e^{-s}
				g_1(e^{-\frac{s}{2}} z + q,t_3-e^{-s})
				&
				\mbox{ \ for \ }
				(z,s) \in  \mathbb{R}_+^n \times (s_1,s_2),
				\\
			-\partial_{z_n} \theta =
				e^{-\frac{s}{2}} g_2(e^{-\frac{s}{2}} \tilde{z} + \tilde{q},t_3-e^{-s})
				&
				\mbox{ \ for \ }
				(z,s) \in  \partial\mathbb{R}_+^n \times (s_1,s_2),
				\\
				\theta(z,s_1)
				=
				g_3\big( (t_3-t_1 )^{\frac{1}{2}} z+q \big)
				&
				\mbox{ \ for \ }
				z \in  \mathbb{R}_+^n  .
			\end{cases}
		\end{equation*}
		
	\end{lemma}
	
	\begin{proof}
Plug $\bar{\theta}(x,t):= \theta \big( (t_3-t)^{-\frac 1 2 } \left(x -q\right) , - \ln (t_3-t) \big) $ into \eqref{qd2023Dec11-1}, then
		\begin{equation*}
			\begin{aligned}
				& 2^{-1} \left(t_3-t \right)^{-\frac{3}{2}}
(x-q) \cdot \left(\nabla_z \theta\right)\big( (t_3-t )^{-\frac 1 2 } \left(x -q\right) , - \ln (t_3-t ) \big)
				\\
				& +
				\left(t_3-t \right)^{-1} \left(\partial_s \theta\right)\big( (t_3-t )^{-\frac 1 2 } \left(x -q\right) , - \ln (t_3-t ) \big)
				\\
				= \ &
				\left(t_3-t \right)^{-1} \left( \Delta_z \theta\right)\big( (t_3-t )^{-\frac 1 2 } \left(x -q\right) , - \ln (t_3-t ) \big)
				+
				g_1(x ,t )  \mbox{ \ in \ }
				\mathbb{R}_+^n \times (t_1,t_2) ,
\\
&
\left(t_3-t \right)^{-\frac{1}{2}}
\left( -\partial_{z_n} \theta\right)\big( (t_3-t )^{-\frac 1 2 } \left(x -q\right) , - \ln (t_3-t ) \big) =
g_2(\tilde{x} ,t )  \mbox{ \ on \ }  \partial\mathbb{R}_+^n \times (t_1,t_2) .
			\end{aligned}
		\end{equation*}
Changing the variables deduces the conclusion.
	\end{proof}

\begin{lemma}\label{qd23Dec04-2-lem}
	
	Let $n\ge 2$ be an integer, $-\infty<s_1<s_2 \le \infty$,
	\begin{equation}\label{qd2023Dec03-2}
		\pp_{s}\theta  = \Delta_z \theta - \frac z 2  \cdot \nabla_z \theta
		+f_1(z,s)  \mbox{ \ in \ }  \mathbb{R}_+^n \times (s_1,s_2),
		\quad
		-\partial_{z_n} \theta =
		f_2(\tilde{z},s)  \mbox{ \ on \ }  \partial \mathbb{R}_+^n \times (s_1,s_2),
		\quad
		\theta(z,s_1) =  f_3(z)  \mbox{ \ in \ }  \mathbb{R}_+^n .
	\end{equation}
Given $t_3\in \mathbb{R}$, $q\in \partial \mathbb{R}_+^n$,  set
	\begin{equation*}
\begin{aligned}
&
		t = t(s) := t_3-e^{-s} \in (t_1,t_2),
		\quad
		t_{i} = t_{i}(s_i) := t_3-e^{-s_i}, \ i=1,2,
		\quad
		x = x(z,s) := e^{-\frac{s}{2}} z + q,
\\
&
\mbox{that is, \ }
s = s(t) = - \ln (t_3-t),
\quad
s_i = s(t_i ) = - \ln (t_3-t_i), \ i=1,2 ,
\quad
z = z(x,t) = (t_3-t)^{-\frac 1 2 } (x-q) ;
\\
&
\bar{\theta}(x,t):= \theta \big( (t_3-t)^{-\frac 1 2 } (x-q), - \ln (t_3-t) \big)
,
\quad
\mbox{ that is, }
\quad
\theta(z,s) = \bar{\theta}\big(e^{-\frac{s}{2}} z + q, t_3-e^{-s} \big)  .
\end{aligned}
	\end{equation*}
Then \eqref{qd2023Dec03-2} is equivalent to
	\begin{equation}\label{qd2023Dec03-1}
		\begin{cases}
			\pp_{t} \bar{\theta} = \Delta_{x} \bar{\theta}
			+
			(t_3-t)^{-1} f_1\big((t_3-t)^{-\frac 1 2 } (x-q), - \ln (t_3-t) \big)  & \mbox{ \ for \ }
			(x,t) \in  \mathbb{R}_+^n \times (t_1,t_2),
			\\
		-\partial_{x_n} \bar{\theta}
			=
			(t_3-t)^{-\frac{1}{2}}
			f_2\big( (t_3-t)^{-\frac 1 2 } (\tilde{x} -\tilde{q} ),
			- \ln (t_3-t)
			\big)
			& \mbox{ \ for \ }
			(x,t) \in  \partial \mathbb{R}_+^n \times (t_1,t_2),
			\\
			\bar{\theta}(x,t_{1})  = f_3 (e^{\frac{s_1}{2}} (x-q)) & \mbox{ \ for \ } x\in \mathbb{R}_+^n .
		\end{cases}
	\end{equation}
	Moreover, the solution of \eqref{qd2023Dec03-2} is given by
	\begin{equation}\label{qd23Dec04-1}
		\begin{aligned}
			&
			\theta (z, s)
			=
			\int_{s_1}^{s}
			\int_{\mathbb{R}_+^n}
			H_n\left(z,s,w,\sigma\right)  f_1(w, \sigma) dw d\sigma
			\\
			&
			+
			\int_{s_1}^{s}
			\int_{\mathbb{R}^{n-1}}
			H_n\left(z,s,(\tilde{w},0),\sigma\right)  f_2(\tilde{w}, \sigma) d \tilde{w} d\sigma
			+
			\int_{\mathbb{R}_+^n}
			H_n\left(z,s,w,s_1\right)  f_3 (w)  dw   ,
		\end{aligned}
	\end{equation}
	where
	\begin{equation*}
		\begin{aligned}
			&
			H_n\left(z,s,w,\sigma\right) :=
			e^{-\sigma \frac{n}{2} }
			G_n(e^{-\frac{s}{2}} z,t_3-e^{-s},e^{-\frac{\sigma}{2}} w,t_3 -e^{-\sigma})
			=
			\left[ 4\pi \left( 1-e^{\sigma-s}\right)\right]^{-\frac{n}{2}}
			\times
			\\
			&
			\left[
			\exp \left(
			-\frac{\left| e^{\frac{\sigma -s}{2}} \tilde{z} -\tilde{w}\right|^2 + \left| e^{\frac{\sigma -s}{2}} z_n -w_{n} \right|^2 }{4\left(1-e^{\sigma-s} \right)}
			\right)+
			\exp \left(
			-\frac{\left| e^{\frac{\sigma -s}{2}} \tilde{z} -\tilde{w}\right|^2 + \left| e^{\frac{\sigma -s}{2}} z_n +w_{n} \right|^2 }{4\left(1-e^{\sigma-s} \right)}
			\right)
			\right].
		\end{aligned}
	\end{equation*}

\end{lemma}

\begin{proof}
	
	By direct calculation,
	\begin{equation*}
		\begin{aligned}
			&
			\bar{\theta}(x,t_{1})
			=
			f_3 (e^{\frac{s_1}{2}} (x-q))
			,
			\quad
			\left( -\partial_{x_{n}} \bar{\theta} \right)\left( (\tilde{x},0),t \right)
			=
			(t_3-t)^{-\frac{1}{2}}
			f_2\big( (t_3-t)^{-\frac 1 2 } (\tilde{x} -\tilde{q} ),
			- \ln (t_3-t)
			\big)  ,
			\\
			&
			\pp_{s} \theta =
			\nabla_{x} \bar{\theta} \cdot z e^{-\frac{s}{2}} (-\frac{1}{2})
			+
			\pp_{t} \bar{\theta} e^{-s},
			\quad
			\pp_{z_i} \theta =\pp_{x_{i}} \bar{\theta} e^{-\frac{s}{2}},
			\quad
			\pp_{z_i z_i} \theta =\pp_{x_{i} x_{i}} \bar{\theta} e^{-s} .
		\end{aligned}
	\end{equation*}
Then \eqref{qd2023Dec03-2} is equivalent to \eqref{qd2023Dec03-1}. By \eqref{qd2023Dec03-3}, the solution of \eqref{qd2023Dec03-1} is given by
	\begin{equation*}
		\begin{aligned}
			&	\bar{\theta}(x,t) =
			\int_{t_{1}}^{t}
			\int_{\mathbb{R}_+^n}
			G_n(x,t,v,\vartheta) (t_3-\vartheta)^{-1} f_1((t_3-\vartheta)^{-\frac 1 2 } (v-q), - \ln (t_3-\vartheta) ) dv d\vartheta
			\\
			&
			+
			\int_{t_{1}}^{t}\int_{\mathbb{R}^{n-1}}
			G_n(x,t,(\tilde{v},0),\vartheta) 	(t_3-\vartheta)^{-\frac{1}{2}}
			f_2\big( (t_3-\vartheta)^{-\frac 1 2 } (\tilde{v} -\tilde{q}) ,
			- \ln (t_3-\vartheta)
			\big)  d\tilde{v} d\vartheta
			\\
			&
			+
			\int_{\mathbb{R}_+^n}
			G_n(x,t,v, t_{1}) f_3 (e^{\frac{s_1}{2}} (v-q))  dv  .
		\end{aligned}
	\end{equation*}
	Then
	\begin{equation*}
		\begin{aligned}
			&
			\theta (z, s) = \bar{\theta}\big(e^{-\frac{s}{2}} z + q, t_3-e^{-s} \big)
			\\
			= \ &
			\int_{t_3-e^{-s_1}}^{t_3-e^{-s}}
			\int_{\mathbb{R}_+^n}
			G_n(e^{-\frac{s}{2}} z+ q,t_3-e^{-s},v,\vartheta) (t_3-\vartheta)^{-1} f_1((t_3-\vartheta)^{-\frac 1 2 } (v-q), - \ln (t_3-\vartheta) ) dv d\vartheta
			\\
			&
			+
			\int_{t_3-e^{-s_1}}^{t_3-e^{-s}}\int_{\mathbb{R}^{n-1}}
			G_n(e^{-\frac{s}{2}} z+ q,t_3-e^{-s},(\tilde{v},0),\vartheta) 	(t_3-\vartheta)^{-\frac{1}{2}}
			f_2\big( (t_3-\vartheta)^{-\frac 1 2 } (\tilde{v} - \tilde{q}),
			- \ln (t_3-\vartheta)
			\big)  d\tilde{v} d\vartheta
			\\
			&
			+
			\int_{\mathbb{R}_+^n}
			G_n(e^{-\frac{s}{2}} z+ q,t_3-e^{-s},v, t_3-e^{-s_1}) f_3 (e^{\frac{s_1}{2}} (v-q) )  dv
			\\
			= \ &
			\int_{s_1}^{s}
			\int_{\mathbb{R}_+^n}
			e^{-\sigma \frac{n}{2} }
			G_n(e^{-\frac{s}{2}} z,t_3-e^{-s},e^{-\frac{\sigma}{2}} w,t_3 -e^{-\sigma})  f_1(w, \sigma) dw d\sigma
			\\
			&
			+
			\int_{s_1}^{s}
			\int_{\mathbb{R}^{n-1}}
			e^{-\sigma \frac{n}{2} }
			G_n(e^{-\frac{s}{2}} z,t_3-e^{-s},(e^{-\frac{\sigma}{2}} \tilde{w} ,0),t_3 -e^{-\sigma})  f_2(\tilde{w}, \sigma) d \tilde{w} d\sigma
			\\
			&
			+
			\int_{\mathbb{R}_+^n}
			e^{-s_1 \frac{n}{2}}
			G_n(e^{-\frac{s}{2}} z,t_3-e^{-s},e^{-\frac{s_1}{2}} w, t_3-e^{-s_1})  f_3 (w)  dw  ,
		\end{aligned}
	\end{equation*}
where we used $G_n(x+q,t,z+q,s) = G_n(x,t,z,s)$ with $q\in \partial \mathbb{R}_+^n$ for the last step.
\end{proof}

\subsection{Some barrier functions}

\begin{lemma}\label{lemma-barrier1}

Let $n\ge 2$ be an integer.
For $\tilde y\in \mathbb{R}^{n-1}$, $y_n\ge 0$, $y=(\tilde y, y_n)$, $\vartheta> 0$, $a>0$, denote
$
P(y) =
\big|(\tilde y,y_n+1+\vartheta|\tilde y|)\big|^{-a} $, then
$
\pp_{y_n}P
\sim_{a,\vartheta}
- \langle y \rangle^{-a-1} $,
$
\left|
y\cdot\nabla P
\right| \lesssim_{a,\vartheta}
\langle y \rangle^{-a} $.
Moreover, if $a\in (0,n-2)$, $0< \vartheta <C(n,a)$ with a sufficiently small constant $C(n,a)>0$ only depending on $n,a$,
then $ \Delta P \lesssim_{n,a,\vartheta} -  \langle y \rangle^{-a-2} $.

\end{lemma}

\begin{proof}
	
Denote $ g(y)=\big|(\tilde y,y_n+1+\vartheta|\tilde y|)\big|^2=(1+\vartheta^2)|\tilde y|^2+(y_n+1)^2+2\vartheta (y_n+1)|\tilde y|  $.
Obviously, $g(y) \sim_{\vartheta} \langle y \rangle^2$, $P(y)=g^{-\frac{a}{2}}(y)$.
By direct calculation, for $i=1,2,\dots, n-1$,
\begin{equation*}
\begin{aligned}
\pp_{y_i}P = \ & -a
g^{-\frac{a}{2}-1}
\left[(1+\vartheta^2)y_i+\vartheta(y_n+1) y_i |\tilde y|^{-1} \right]  ,
\\
\pp_{y_iy_i}P =  \ &  a(a+2)
g^{-\frac{a}{2}-2}
\left[ 1+\vartheta^2 + \vartheta(y_n+1) |\tilde y|^{-1} \right]^2 y_i^2 \\
&
-a
g^{-\frac{a}{2}-1}
\left[ 1+\vartheta^2 +
\vartheta(y_n+1) |\tilde y|^{-1}
-
\vartheta (y_n+1)y_i^2
|\tilde y|^{-3}
\right],
\\
\pp_{y_n}P =  \ &  -a
g^{-\frac{a}{2}-1}
\left(\vartheta|\tilde y|+y_n+1\right),
\\
\pp_{y_ny_n}P = \ &
a(a+2)
g^{-\frac{a}{2}-2}
\left(\vartheta|\tilde y|+y_n+1 \right)^2-a g^{-\frac{a}{2}-1} ,
\end{aligned}
\end{equation*}
which implies $\pp_{y_n}P
\sim_{a,\vartheta}
- \langle y \rangle^{-a-1}$,
$ \left|
y\cdot\nabla P
\right| \lesssim_{a,\vartheta}
\langle y \rangle^{-a} $. Moreover,
\begin{equation*}
\begin{aligned}
\Delta P = \ & a(a+2) g^{-\frac{a}{2}-2} \left\{ \left[(1+\vartheta^2)^2+\vartheta^2\right]|\tilde y|^2+2\vartheta(\vartheta^2+2)(y_n+1)|\tilde y|
+
(\vartheta^2+1)(y_n+1)^2\right\}\\
&-a
g^{-\frac{a}{2}-1}
\left[
(\vartheta^2+1)(n-1)+1+
\vartheta(n-2)(y_n+1) |\tilde y|^{-1}
\right]\\
= \ &
a g^{-\frac{a}{2}-1} \Big\{ g^{-1} (a+2)\left[(\vartheta^4+3\vartheta^2+1)|\tilde y|^2+2\vartheta(\vartheta^2+2)(y_n+1)|\tilde y|+(\vartheta^2+1)(y_n+1)^2\right]
\\
&- (n-1)\vartheta^2 -n-
\vartheta(n-2)(y_n+1) |\tilde y|^{-1}
\Big\}
:=a g^{-\frac{a}{2}-1} I(t),
\end{aligned}
\end{equation*}
where we set $t=(y_n+1) |\tilde y|^{-1}\in (0,\infty)$ and
\begin{equation*}
\begin{aligned}
&
I(t):=
\left(a+2\right) \frac{\vartheta^4+3\vartheta^2+1 + 2\vartheta\left(\vartheta^2+2\right)t + \left(\vartheta^2+1\right) t^2}{\vartheta^2+1 +2\vartheta t + t^2}
-
\left(n-1\right) \vartheta^2 -n -\left(n-2\right)\vartheta t
\\
= \ &
-\frac{f(t)}{\vartheta^2+1 +2\vartheta t + t^2}  ,
\end{aligned}
\end{equation*}
where
\begin{equation*}
\begin{aligned}
f(t) := \ & t^3\vartheta(n-2)
+
t^2
\left[ \vartheta^2 \left(3n-a-7\right)
+n-a-2
\right]
+t \vartheta \left[
\vartheta^2 \left(3n -2a -8\right)
+3n -4a -10
\right]
\\
&
+
\vartheta^4 \left(n-a-3\right)
+
\vartheta^2 \left( 2n-3a-7\right)
+n-a-2  .
\end{aligned}
\end{equation*}
For $a<n-2$, by the discriminant of the quadratic polynomial, there exists a constant $C(n, a)>0$ sufficiently small such that for any $0< \vartheta <C(n, a)$, we have
\begin{equation*}
	 \vartheta^2 \left(3n-a-7\right)
	+n-a-2>0
\mbox{ \quad and \quad }
f(t) - t^3\vartheta(n-2) >0 \mbox{ \ for \ } t\ge 0 ,
\end{equation*}
which implies $I(t)<0$ for $t\ge 0$. Note that $\lim\limits_{t\to \infty} I(t) <0$. Thus, $\sup\limits_{t\ge 0} I(t) \le - C_1(n,a,\vartheta)$ with a constant $C_1(n,a,\vartheta)>0$. By $a>0$, we conclude the second result.
\end{proof}

\begin{lemma}\label{qd23Dec06-2-lem}
	Given $t_2\ge t_1> 0$, $a_1>-1$, $a_2\in \mathbb{R}$,
	\begin{equation*}
		\int_0^{t_1} \left(t_1-s\right)^{a_1} \left(t_2-s\right)^{a_2} ds \lesssim_{a_1,a_2}
		\begin{cases}
			t_2^{a_1+a_2+1}, & \mbox{ \ if \ } a_1+a_2>-1
			\\
			1+
			\ln \left( \frac{t_2}{t_2 -\left[ t_1 -\left(t_2-t_1\right)\right]_+}\right)
			, & \mbox{ \ if \ } a_1+a_2=-1
			\\
			\left(t_2-t_1\right)^{a_1+a_2+1} ,
			& \mbox{ \ if \ } a_1+a_2<-1  .
		\end{cases}
	\end{equation*}
	
\end{lemma}

\begin{proof}
	\begin{equation*}
		\begin{aligned}
			&
			\bigg(	\int_0^{\left[ t_1-\left(t_2-t_1\right)\right]_+} + \int_{\left[ t_1-\left(t_2-t_1\right)\right]_+}^{t_1} \bigg) \left(t_1-s\right)^{a_1} \left(t_2-s\right)^{a_2} ds
			\\
		&	\sim_{a_1,a_2}
			\int_0^{\left[ t_1-\left(t_2-t_1\right)\right]_+}  \left(t_2-s\right)^{a_1+a_2} ds
			+
			\left(t_2-t_1\right)^{a_2}
			\int_{\left[ t_1-\left(t_2-t_1\right)\right]_+}^{t_1}
			\left(t_1-s\right)^{a_1}  ds .
		\end{aligned}
	\end{equation*}
Using $a_1>-1$, we conclude this lemma.
\end{proof}

\begin{lemma}\label{outerbarrieri1}
	Given an integer $n\ge 2$ and constants $s_1\in \mathbb{R}$, $C_f\ge 0$,
	consider
	\begin{equation}\label{barrier9}
 \left\{
 \begin{aligned}
		& \pp_ s \Phi = \Delta_z \Phi-\frac{z}{2} \cdot \nabla_z \Phi+ f_1(z,s)  \mbox{ \ in \ } \mathbb{R}^n_+\times ( s_1,\infty),
  \\
& -\pp_{z_n}\Phi = f_2(\tilde{z},s)
		\mbox{ \ on \ }
		\pp\mathbb{R}^n_+\times ( s_1,\infty),
		\quad
		\Phi(z,s_1)=f_3(z)
		\mbox{ \ in \ } \mathbb{R}^n_+ ,
   \end{aligned}
   \right.
	\end{equation}
	where $\Phi$ is given by the formula \eqref{qd23Dec04-1} formally.
	\begin{enumerate}
		\item\label{qd23Dec06-1}

If
$
\left| f_1(z,s)\right| \le C_f e^{\alpha_2 s } \left\langle e^{\alpha_1 s }
z \right\rangle^{-2-a},
\left| f_2(\tilde{z},s)\right| \le C_f e^{(\alpha_2 -\alpha_1) s } \left\langle e^{\alpha_1 s }
\tilde{z} \right\rangle^{-1-a}
,
f_3(z) \equiv 0
$
with $ (2+a)\alpha_1-\alpha_2 >0$, $a\in \left(0,n-2\right)$,
then there exists a constant $C(n,a,\alpha_1,\alpha_2)>0$ such that
\begin{equation*}
	|\Phi(z, s)|\le   C_f C(n,a,\alpha_1,\alpha_2)\left\{e^{(\alpha_2 - 2\alpha_1 )s} \langle e^{\alpha_1 s } z \rangle^{-a} +
	e^{-[ (2+a) \alpha_1 - \alpha_2] s_1}\right\}.
\end{equation*}

		\item\label{outerheatkerneli2}
		If $
		f_1(z,s) \equiv 0 ,
		f_2(\tilde{z},s)  \equiv 0 	$,
		then there exists a constant $C(n)>0$ such that
		\begin{equation*}
			\left| \Phi(z ,s) \right| \le C(n)
			\left( 1-e^{s_1-s}\right)^{-\frac{n}{4}}
			e^{  \frac{ |z|^2 }{4\left(1+e^{s-s_1}\right)} }
			\|f_3\|_{L_\rho^2(\mathbb{R}_+^n)} .
		\end{equation*}
		Moreover, if $|f_3(z)|\le  C_f |z|^a$ with $a\ge 0$, then there exists a constant $C(n,a)>0$ such that
		\begin{equation*}
			\left|\Phi(z,s) \right|\le C_f C(n,a) \left[ e^{a \frac{s_1-s}{2}} |z|^{a}
			+
			\left(1-e^{s_1-s}\right)^{\frac{a}{2}}
			\right].
		\end{equation*}
		
		\item\label{lemma-barrier3}
		If $
		\left| f_1(z,s)\right| \le C_f e^{\alpha s},
		f_2(\tilde{z},s)  \equiv 0
		,
		f_3(z) \equiv 0 $
		with $\alpha\in \mathbb{R}$, then
		\begin{equation*}
			\left| \Phi(z,s) \right|
			\le
			C_f
			\begin{cases}
				\alpha^{-1} \left(
				e^{\alpha s } -e^{\alpha s_1 } \right),
				& \mbox{ \ if \ } \alpha\ne 0
				\\
				s-s_1  , & \mbox{ \ if \ } \alpha= 0 .
			\end{cases}
		\end{equation*}
		
		\item\label{lemma-barrier4-0}
		If $ f_1(z,s) \equiv 0 ,
		\left|f_2(\tilde{z},s) \right| \le C_f e^{\alpha s}  \left| \tilde z \right|^a
		,
		f_3(z) \equiv 0 $
		with $\alpha\in \mathbb{R}$, $a\ge 0$,
		then there exists a constant $C(n,a,\alpha) >0$ such that
		\begin{equation*}
			\begin{aligned}
				\left| \Phi(z,s) \right|
				\le \ &
				C_f C(n,a,\alpha)
				\bigg[
				e^{-\frac{a}{2} s}
				\left|\tilde{z}\right|^a
				\begin{cases}
					e^{(\alpha+\frac{a}{2}) s_1 }, & \mbox{ \ if \ }
					\alpha+\frac{a}{2} <0
					\\
					\big\langle \ln \left( 1 -\left( 1-2e^{s_1-s}\right)_+\right)
					\big\rangle
					, & \mbox{ \ if \ } \alpha+\frac{a}{2}=0
					\\
					e^{(\alpha+\frac{a}{2}) s } ,
					& \mbox{ \ if \ } \alpha+\frac{a}{2}>0
				\end{cases}
				\\
				&
				+
				\begin{cases}
					e^{\alpha s_1}, & \mbox{ \ if \ } \alpha<0
					\\
					\big\langle \ln \left( 1 -\left( 1-2e^{s_1-s}\right)_+\right)
					\big\rangle
					, & \mbox{ \ if \ } \alpha=0
					\\
					e^{\alpha s} ,
					& \mbox{ \ if \ } \alpha>0
				\end{cases}
				\bigg].
			\end{aligned}
		\end{equation*}

	\end{enumerate}
	
\end{lemma}

\begin{proof}
	
	Denote
	\begin{equation*}
		T_1=e^{- s_1},
		\quad
		\psi(x,t):= \Phi \big((T_1-t)^{-\frac 1 2 } x, - \ln (T_1-t) \big)
		,
		\quad
		\mbox{ that is, }
		\quad
		\Phi(z,s) = \psi( e^{-\frac{s}{2}} z,T_1-e^{-s} )
		.
	\end{equation*}
	Then by Lemma \ref{qd23Dec04-2-lem}, \eqref{barrier9} is equivalent to
	\begin{equation}\label{qd2023Dec03-10}
		\begin{cases}
			\pp_{t} \psi = \Delta_{x} \psi
			+
			(T_1-t)^{-1} f_1\big((T_1-t)^{-\frac 1 2 } x, - \ln (T_1-t) \big)  & \mbox{ \ for \ }
			(x,t) \in  \mathbb{R}_+^n \times (0,T_1),
			\\
			\left(-\partial_{x_{n}} \psi \right)((\tilde{x},0),t)
			=
			(T_1-t)^{-\frac{1}{2}}
			f_2\big( (T_1-t)^{-\frac 1 2 } \tilde{x},
			- \ln (T_1-t)
			\big)
			& \mbox{ \ for \ }
			(\tilde{x},t) \in \mathbb{R}^{n-1} \times (0,T_1),
			\\
			\psi(x,0)  = f_3 (e^{\frac{s_1}{2}} x) & \mbox{ \ for \ } x \in \mathbb{R}_+^n .
		\end{cases}
	\end{equation}
	
	\eqref{qd23Dec06-1}.
	In this case, we have $f_3 (e^{\frac{s_1}{2}} x)\equiv 0$,
	\begin{equation*}
\begin{aligned}
&
			(T_1-t)^{-1} \left| f_1\big((T_1-t)^{-\frac 1 2 } x, - \ln (T_1-t) \big) \right|
			\le
			C_f
			\left(T_1-t\right)^{-\alpha_2 -1}
			\langle
			\left(T_1-t\right)^{-\alpha_1 -\frac{1}{2}} x
			\rangle^{-2-a} ,
\\
&
	(T_1-t)^{-\frac{1}{2}}
\left|f_2\big( (T_1-t)^{-\frac 1 2 } \tilde{x},
- \ln (T_1-t)
\big)
\right| \le C_f
\left(T_1-t\right)^{-(\alpha_2-\alpha_1) -\frac{1}{2}}
\langle
\left(T_1-t\right)^{-\alpha_1 -\frac{1}{2}} \tilde{x}
\rangle^{-1-a}   .
\end{aligned}
	\end{equation*}
	By Lemma \ref{lemma-barrier1}, for $a\in \left(0,n-2\right)$, we set
	\begin{equation*}
		P(y)= \left|(\tilde y,y_n+1+\vartheta|\tilde y|)\right|^{-a}, \quad y=\left(T_1-t\right)^{-(\alpha_1+\frac 12)} x,
	\end{equation*}
	where $\vartheta= \vartheta(n,a) >0$ is a small constant to make $\Delta_y P(y) \lesssim_{n,a,\vartheta} -  \langle y \rangle^{-a-2}$.

	Set $\psi_1(x,t)= 2 (T_1-t)^{2\alpha_1-\alpha_2} P(y)$ with $y=\left(T_1-t\right)^{-\alpha_1-\frac 12} x$, then
	\begin{equation*}
		-\pp_{x_n}\psi_1\big|_{x_n=0}=-2\left(T_1-t\right)^{\alpha_1-\alpha_2-\frac 12} \pp_{y_n}P(y)\big|_{y_n=0}
		\sim_{a,\vartheta}
		\left(T_1-t\right)^{\alpha_1-\alpha_2-\frac 12}
		\langle \tilde y\rangle^{-a-1}
		,
	\end{equation*}
	and
	\begin{equation*}
		\pp_t\psi_1-\Delta_x\psi_1 =
		-\left(T_1-t\right)^{-\alpha_2 -1}  \Delta_y P (y)
		+
		g_1(x,t)  ,
	\end{equation*}
	where
	\begin{equation*}
		g_1(x,t)
		:=  \left(T_1-t\right)^{-\alpha_2 -1}  \left[-\Delta_y P (y)
-\left(4\alpha_1-2\alpha_2\right)
\left(T_1-t\right)^{2\alpha_1} P(y)
+
\left(2\alpha_1+1\right)
\left(T_1-t\right)^{2\alpha_1}
y\cdot \nabla_y P(y)
		\right].
	\end{equation*}
By the estimates of $P(y)$ in Lemma \ref{lemma-barrier1}, we have
	\begin{equation*}
		\begin{aligned}
			&
			-\left(T_1-t\right)^{- \alpha_2 -1}  \Delta_y P (y) \gtrsim_{n,a,\vartheta} \left(T_1-t\right)^{-\alpha_2 -1} \langle y \rangle^{-a-2} ,
			\\
		&	g_1(x,t)
			\ge
			\left(T_1-t\right)^{- \alpha_2 -1}
			C_1(n,a,\vartheta)
			\left[
			\langle y \rangle^{-a-2}
			- C_2(n,a,\vartheta,\alpha_1,\alpha_2)
			\left(T_1-t\right)^{2 \alpha_1 } \langle y \rangle^{-a}
			\right]
			\\
		&\qquad	\quad \ge
			-C_3(n,a,\vartheta,\alpha_1,\alpha_2)
			\left(T_1-t\right)^{\alpha_1 a +2\alpha_1-\alpha_2-1}
		\end{aligned}
	\end{equation*}
with some positive constants $C_i(*,*,\dots)$, $i=1,2,3$.
	In order to control $g_1(x,t)$, we set $\psi_2(x,t)=T_1^{\alpha_1 a +2\alpha_1-\alpha_2}-(T_1-t)^{\alpha_1 a +2\alpha_1-\alpha_2}$, which satisfies $\partial_{x_n} \psi_2 =0$,
	$(\pp_t-\Delta_x)\psi_2=
	\left( \alpha_1 a +2\alpha_1-\alpha_2 \right)
(T_1-t)^{\alpha_1 a +2\alpha_1-\alpha_2 - 1} $.
	Since $ \alpha_1 a +2\alpha_1-\alpha_2 >0$, $ C_f C_5(n,a,\vartheta,\alpha_1,\alpha_2)\big(\psi_1+ C_4(n,a,\vartheta,\alpha_1,\alpha_2)\psi_2 \big)$ is a supersolution of \eqref{qd2023Dec03-10} for some  large constants $C_i(n,a,\vartheta,\alpha_1,\alpha_2)>0$,  $i=4,5$. Hence, \eqref{qd23Dec06-1} is deduced by the comparison theorem and plugging $\Phi(z,s) = \psi( e^{-\frac{s}{2}} z,T_1-e^{-s} )$.
	
	\medskip
	
	\eqref{outerheatkerneli2}.
Using the even extension of $f_3$ in $z_n$ variable, we can deduce the first result by the following calculation, which is similar to the proof of \cite[Lemma 2.2]{Harada17system}.
For $g \in L_{\rho}^2(\mathbb{R}^n)$,
	\begin{equation*}
\begin{aligned}
&
	\left\{4\pi \left[1 -e^{-(s-s_1)} \right]\right\}^{-\frac{n}{2}} \int_{\mathbb{R}^n}
\exp\bigg(  -\frac{\big|e^{-\frac{s-s_1}{2}} z-w \big|^2}{4\big[1 -e^{-(s-s_1)} \big]} \bigg)
		g (w) dw
\\
\le \ &
\left\{4\pi \left[1 -e^{-(s-s_1)} \right]\right\}^{-\frac{n}{2}}
\bigg[
\int_{\mathbb{R}^n}
\exp\bigg(  -\frac{\big|e^{-\frac{s-s_1}{2}} z-w\big|^2}{2\left[1 -e^{-(s-s_1)} \right]}
\bigg)
e^{\frac{|w|^2}{4}} dw \bigg]^{\frac{1}{2}}
\| g\|_{L_{\rho}^2(\mathbb{R}^n)}
\\
= \ &
2^{-\frac{n}{2}} \pi^{-\frac{n}{4}}
\left[ 1+e^{-(s-s_1)}\right]^{-\frac{n}{4}}
\left[ 1-e^{-(s-s_1)}\right]^{-\frac{n}{4}}
e^{\frac{ |z|^2 }{4\left(1+e^{s-s_1}\right)}}
\| g\|_{L_{\rho}^2(\mathbb{R}^n)} .
\end{aligned}
	\end{equation*}
	
	For the second result, without loss of generality, one taking $C_f=1$, then	
	\begin{equation*}
		\begin{aligned}
			&
			|\psi(x,t)|\le \int_{\mathbb{R}^n_+}G_n(x,t,v,0)
			\left|f_3 (e^{\frac{s_1}{2}} v) \right| dv
			\le
			e^{a \frac{s_1}{2}}
			\int_{\mathbb{R}^n}
			(4\pi t)^{-\frac{n}{2} }
			e^{-\frac{|x-v|^2}{4t}}
			|v|^a
			dv
			\\
			= \ &
			e^{a \frac{s_1}{2}}
			(4\pi )^{-\frac{n}{2} } \int_{\mathbb{R}^n}e^{-\frac{|w|^2}{4}}
			|x-\sqrt{t}w|^a
			dw
			\le C(n,a)
			e^{a \frac{s_1}{2}}
			\int_{\mathbb{R}^n}e^{-\frac{|w|^2}{4}}
			\left(
			|x|^a+t^\frac{a}{2}|w|^a \right)dw
			\le C(n,a)
			e^{a \frac{s_1}{2}}
			\left(|x|^a+t^\frac{a}{2}  \right)  ,
		\end{aligned}
	\end{equation*}
	where we used $a\ge 0$ for the third ``$\le$''. Then we get the second result.

	\eqref{lemma-barrier3}.
	Take $C_f=1$.
	\begin{equation*}
		\begin{aligned}
			&
			\left|\psi(x,t) \right| \le
			\int_{0}^{t}
			\int_{\mathbb{R}_+^n}
			G_n(x,t,v,\vartheta)
			(T_1-\vartheta)^{-1}
			\left|	 f_1((T_1-\vartheta)^{-\frac 1 2 } v, - \ln (T_1-\vartheta) )  \right| dv d\vartheta
			\\
			\le \ &
			\int_{0}^{t} \int_{\mathbb{R}^n}
			\left[4\pi\left(t-\vartheta\right) \right]^{-\frac{n}{2}} e^{-\frac{|x-v|^2}{4(t-\vartheta)}}
			\left(T_1-\vartheta\right)^{-1-\alpha} dv d\vartheta
			\\
			= \ &
			\int_{0}^{t} \left(T_1-\vartheta\right)^{-1-\alpha}  d\vartheta
			=
			\begin{cases}
				\alpha^{-1} \left[
				\left(T_1-t\right)^{-\alpha}
				-T_1^{-\alpha}
				\right], & \mbox{ \ if \ } \alpha\ne 0
				\\
				\ln (\frac{T_1}{T_1-t} )  , & \mbox{ \ if \ } \alpha= 0 .
			\end{cases}
		\end{aligned}
	\end{equation*}
	Using the relationship $\Phi(z,s) = \psi( e^{-\frac{s}{2}} z,T_1-e^{-s} )$, we get the conclusion.
	
	\medskip
	
	\eqref{lemma-barrier4-0}.
	Set $C_f=1$.
	\begin{equation*}
		\begin{aligned}
			&	\left|\psi(x,t) \right| \le
			\int_{0}^{t}\int_{\mathbb{R}^{n-1}}
			G_n(x,t,(\tilde{v},0),\vartheta) 	(T_1-\vartheta)^{-\frac{1}{2}}
			\left|	f_2\big( (T_1-\vartheta)^{-\frac 1 2 } \tilde{v},
			- \ln (T_1-\vartheta)
			\big) \right|  d\tilde{v} d\vartheta
			\\
			\le \ &
			\int_{0}^{t}\int_{\mathbb{R}^{n-1}}
			2\left[4\pi\left(t-\vartheta\right)\right]^{-\frac{n}{2}}
			e^{-\frac{|\tilde{x}-\tilde{v}|^2 +x_n^2}{4(t-\vartheta)}}
			\left(T_1-\vartheta \right)^{-\frac 1 2 -\alpha-\frac{a}{2}}
			\left|\tilde{v}\right|^a	   d\tilde{v} d\vartheta
			\\
			= \ &
			\pi^{-\frac{n}{2}}
			\int_{0}^{t}
			\left(t-\vartheta\right)^{-\frac{1}{2}}
			\left(T_1-\vartheta\right)^{-\frac{1}{2} -\alpha -\frac{a}{2}}
			e^{-\frac{x_n^2}{4\left(t-\vartheta\right)}}
			\int_{\mathbb{R}^{n-1}}
			e^{-\left|\tilde{w}\right|^2}
			\left|\tilde{x} -2\sqrt{t-\vartheta} \tilde{w}\right|^a d\tilde{w} d \vartheta
			\\
			\le \ &
			C(n,a)
			\int_{0}^{t}
			\left(t-\vartheta\right)^{-\frac{1}{2}}
			\left(T_1-\vartheta\right)^{-\frac{1}{2} -\alpha -\frac{a}{2}}
			\left[ \left|\tilde{x}\right|^a
			+
			\left( t-\vartheta \right)^{\frac{a}{2}}
			\right]  d \vartheta
			\\
			\le \ &
			C(n,a,\alpha)
			\bigg[
			\left|\tilde{x}\right|^a
			\begin{cases}
				T_1^{-\alpha-\frac{a}{2}}, & \mbox{ \ if \ }
				\alpha+\frac{a}{2} <0
				\\
				1+
				\ln \left( \frac{T_1}{T_1 -\left[ t -\left(T_1-t\right)\right]_+}\right)
				, & \mbox{ \ if \ } \alpha+\frac{a}{2}=0
				\\
				\left(T_1-t\right)^{-\alpha-\frac{a}{2}} ,
				& \mbox{ \ if \ } \alpha+\frac{a}{2}>0
			\end{cases}
			+
			\begin{cases}
				T_1^{-\alpha}, & \mbox{ \ if \ } \alpha<0
				\\
				1+
				\ln \left( \frac{T_1}{T_1 -\left[ t -\left(T_1-t\right)\right]_+}\right)
				, & \mbox{ \ if \ } \alpha=0
				\\
				\left(T_1-t\right)^{-\alpha} ,
				& \mbox{ \ if \ } \alpha>0
			\end{cases}
			\bigg],
		\end{aligned}
	\end{equation*}
	where we used $a\ge 0$ for the third ``$\le$'' and Lemma \ref{qd23Dec06-2-lem} for the fourth ``$\le$''. Then the conclusion holds.
	
\end{proof}

\subsection{Vanishing adjustment functions}

In order to derive multiple rates at distinct points, we introduce vanishing adjustment functions to eliminate derivatives at arbitrarily prescribed finitely many points.
This method works for more general parabolic equations.

\begin{lemma}\label{qd23Dec22-4-lem}
	
	Suppose that $n\ge 1$, $l\in \mathbb{N}$, $u_0(x) \in C^{l}(\mathbb{R}^n)$ satisfies $D_{x}^{\mathbf{k}} u_0(x) \to 0$ as $|x|\to \infty$ for all $\mathbf{k} \in \mathbb{N}^n$ with $\| \mathbf{k}\|_{\ell_1} \le l$, $t_0\in \mathbb{R}$, denote $u(x,t) =
 \left[4\pi \left(t-t_0\right)\right]^{-\frac{n}{2}}
	\int_{\mathbb{R}^n}  e^{-\frac{|x-y|^2}{4\left(t-t_0\right)}} u_0(y) dy$,
	then for $\iota\in \mathbb{N}$, $\mathbf{m} \in \mathbb{N}^n$ with $ 2\iota +
	\|\mathbf{m}\|_{\ell_1} \le l$, we have $\partial_t^{\iota}	D_{x}^{\mathbf{m}} u\in C(\mathbb{R}^n\times [t_0,\infty) )$,
	\begin{equation*}
		\partial_t^{\iota}	D_{x}^{\mathbf{m}} u(x,t)
		\to \Delta_x^{\iota}    D_{x}^{\mathbf{m}} u_0(x)
		\mbox{ \ in \ }
		L^\infty(\mathbb{R}^n)
		\mbox{ \ as \ } t\downarrow t_0  ,
		\quad
		\| \partial_t^{\iota} D_{x}^{\mathbf{m}} u(\cdot,t) \|_{L^\infty(\mathbb{R}^n)}
		\le  \| \Delta_x^{\iota}    D_{x}^{\mathbf{m}} u_0 \|_{L^\infty(\mathbb{R}^n)}
  \mbox{ \ for \ } t \ge t_0 .
	\end{equation*}
	
\end{lemma}

\begin{proof}
	
	For $t>t_0$,
	\begin{equation*}
		\begin{aligned}
			&
			D_{x}^{\mathbf{m}}
			u(x,t)
			=
			\left[4\pi \left(t-t_0\right)\right]^{-\frac{n}{2}}
			\int_{\mathbb{R}^n}
			D_{x}^{\mathbf{m}}
			\left[ e^{-\frac{|x-y|^2}{4\left(t-t_0\right)}} \right] u_0(y) dy
			\\
			= \ &
			\left[4\pi \left(t-t_0\right)\right]^{-\frac{n}{2}}
			(-1)^{\|\mathbf{m}\|_{\ell_1}}
			\int_{\mathbb{R}^n}
			D_{y}^{\mathbf{m}}
			\left[ e^{-\frac{|x-y|^2}{4\left(t-t_0\right)}}   \right]
			u_0(y) dy
			=
			\left[4\pi \left(t-t_0\right)\right]^{-\frac{n}{2}}
			\int_{\mathbb{R}^n}     e^{-\frac{|x-y|^2}{4\left(t-t_0\right)}}
			D_{y}^{\mathbf{m}}  u_0(y) dy  ,
			\\
			&
			\partial_t^{\iota} D_{x}^{\mathbf{m}}
			u(x,t)
			=
			\int_{\mathbb{R}^n}
			\Delta_x^{\iota}
			\left\{
			\left[4\pi \left(t-t_0\right)\right]^{-\frac{n}{2}}
			e^{-\frac{|x-y|^2}{4\left(t-t_0\right)}} \right\}
			D_{y}^{\mathbf{m}}  u_0(y) dy
			=
			\int_{\mathbb{R}^n}
			\left[4\pi \left(t-t_0\right)\right]^{-\frac{n}{2}}
			e^{-\frac{|x-y|^2}{4\left(t-t_0\right)}}
			\Delta_y^{\iota}
			D_{y}^{\mathbf{m}}  u_0(y) dy  .
		\end{aligned}
	\end{equation*}
We conclude this lemma by the property of the heat kernel.
\end{proof}

\begin{lemma}\label{23Dec23-1-lem}
	
	Given $n\ge 2$, $u_0(x) \in L^\infty(\mathbb{R}_+^n)$, $t_0\in \mathbb{R}$, denote
	\begin{equation*}
		u(x,t) := \int_{\mathbb{R}_+^n}
		G_n(x,t,z, t_0) u_0(z)  dz
		=
		\left[4\pi (t-t_0)\right]^{-\frac{n}{2}}
		\int_{\mathbb{R}^n}
		e^{-\frac{|\tilde{x} - \tilde{z}|^2+ |x_n - z_n|^2 }{4(t-t_0)}}
		\left(u_0(z)\1_{z_n\ge 0} +
		u_0(\tilde{z},-z_n)  \1_{z_n< 0}
		\right)
		dz
	\end{equation*}
	with $G_n$ given in \eqref{qd23Dec22-1}, then for $t>t_0$, $\iota\in \mathbb{N}$, $\mathbf{m} \in \mathbb{N}^n$, $\partial_t^{\iota} D_{x}^{\mathbf{m}} u(x,t)$ is odd (even) about $x_n\in \mathbb{R}$ if $m_n$ is odd (even).
	In particular, $\partial_t^\iota D_{x}^{\mathbf{m}} u((\tilde{x},0),t) = 0$ if  $m_n$ is odd.
	
Under the additional assumptions that $u_0(x) \in C^{l}(\overline{\mathbb{R}_+^n} )$, $l\in \mathbb{N}$ satisfies that for $\mathbf{k}\in \mathbb{N}^n$ with $\|\mathbf{k}\|_{\ell_1} \le l$, $D_{x}^{\mathbf{k}} u_0(x) \to 0$ as $|x|\to \infty$ in $\overline{\mathbb{R}_+^n}$ and $D_{x}^{\mathbf{k}} u_0(x)|_{x_n=0}=0$ if $k_n $ is odd, then
	for $\iota\in \mathbb{N}$, $\mathbf{m} \in \mathbb{N}^n$ with $2\iota + \|\mathbf{m}\|_{\ell_1} \le l$, we have $\partial_t^{\iota} D_{x}^{\mathbf{m}} u \in C( \overline{\mathbb{R}_+^n} \times [t_0,\infty) ) $,
	\begin{equation*}
		\partial_t^{\iota} D_{x}^{\mathbf{m}} u(x,t) \to \Delta_x^{\iota} D_{x}^{\mathbf{m}} u_0(x) \mbox{ \ in \ } L^\infty( \mathbb{R}_+^n )
		\mbox{ \ as \ } t\downarrow t_0 ,
		\quad
		\| \partial_t^{\iota} D_{x}^{\mathbf{m}} u(\cdot,t) \|_{L^\infty(\mathbb{R}_+^n)}
		\le
		\| \Delta_x^{\iota} D_{x}^{\mathbf{m}} u_0  \|_{L^\infty(\mathbb{R}_+^n)} \mbox{ \ for \ } t \ge t_0 .
	\end{equation*}
	
\end{lemma}

\begin{proof}
	
\eqref{qd24Feb04-1} deduces the first result.
Under the additional assumptions of $u_0$, $u_0(x)\1_{x_n\ge 0} +
	u_0(\tilde{x},-x_n)  \1_{x_n< 0}$ satisfies the assumption of Lemma \ref{qd23Dec22-4-lem}, which implies the second result.
\end{proof}

We will give vanishing adjustment functions in the next proposition.
First, we will find a basis for derivatives at the prescribed finite number of points and then use the continuity in a short time to derive a basis for derivatives at these points at $t=T\ll 1$.

\begin{prop}\label{lemma-cons1}
	Suppose that $\mathfrak{l}$ is a positive integer, $p^{[1]}, p^{[2]}, \dots, p^{[\mathfrak{l}]}$ be arbitrary distinct points on $\pp\mathbb{R}^n_+$, $n\ge 2$, and a constant $d$ satisfies $0< d < \min_{1\le i\not=j\le \mathfrak{l}}|p^{[i]} -p^{[j]}| / 4$, $N_0\in \mathbb{N}$, then for $T\ll 1$, there exist  $V_{p^{[i]} , \mathbf{m}}(x,t)$, $i=1,2,\dots, \mathfrak{l}$ with $\mathbf{m}\in \mathbb{N}^n$, $\| \mathbf{m} \|_{\ell_1} \le N_0$, $m_n \in 2 \mathbb{N}$,  solving
	\begin{equation*}
		\pp_t V_{p^{[i]}, \mathbf{m} }=\Delta V_{p^{[i]}, \mathbf{m} }
		\mbox{ \ in \ }  \mathbb{R}^{n}_+\times (0,T] ,
		\quad
		-\pp_{x_n} V_{p^{[i]}, \mathbf{m} }=0
		\mbox{ \ on \ } \pp \mathbb{R}^{n}_+ \times (0,T] ,
\quad
V_{p^{[i]}, \mathbf{m} }(x,0) =
V_{p^{[i]},\mathbf{m},0}(x)
\mbox{ \ in \ }  \mathbb{R}^{n}_+ ,
	\end{equation*}
	and the following properties hold:
	\begin{enumerate}
		\item  $V_{p^{[i]},\mathbf{m},0}(x)$
		is smooth in $\overline{\mathbb{R}_+^n}$ and $V_{p^{[i]},\mathbf{m},0}(x) = 0$ in $\overline{\mathbb{R}^{n}_+  \backslash B_n^+(p^{[i]} ,2d)}$.
		
		\item
		$\partial_t^\iota D_{x}^{\mathbf{k}} V_{p^{[i]} , \mathbf{m}}((\tilde{x},0),t) = 0$ for $\tilde{x}\in \mathbb{R}^{n-1}$, $t\in [0,T]$, and $\iota\in \mathbb{N}$, $\mathbf{k}\in \mathbb{N}^n$, $k_n\in 2\mathbb{N}+1$.
		
		\item
		$D_x^{\mathbf{k}} V_{p^{[i]} , \mathbf{m} }(p^{[j]},T) = \delta_{\mathbf{m}, \mathbf{k}} \delta_{p^{[i]}, p^{[j]}}$ for $\mathbf{k}\in \mathbb{N}^n$, $\| \mathbf{k} \|_{\ell_1} \le N_0$, $j =1,2,\dots, \mathfrak{l}$.
		
		\item
		$\|\partial_t^{\iota} D_x^{\mathbf{k}} V_{p^{[i]} , \mathbf{m}}   \|_{L^\infty(\mathbb{R}_+^n \times [0,T] ) } \le C$ for $\iota\in\mathbb{N}$, $\mathbf{k}\in \mathbb{N}^n$, $2\iota + \| \mathbf{k} \|_{\ell_1} \le N_0$ with a constant $C>0$ only depending on $\mathfrak{l}$, $d$, $N_0$, $p^{[1]}, p^{[2]}, \dots, p^{[\mathfrak{l}]}$.
		
	\end{enumerate}
	
\end{prop}

\begin{proof}
	
	Denote
	$ \phi_{\mathbf{m}}(x)
	=
	\eta(x/d)
	\prod_{j=1}^{n}
	\left( m_j ! \right)^{-1}
	x_j^{m_j}
	$,
	and
	$g_{p^{[i]} , \mathbf{m},0}(x) = \phi_{\mathbf{m}}(x-p^{[i]} )$. Since $\eta$ is radial, for $\mathbf{m} \in \mathbb{N}^n$ with $m_n\in 2\mathbb{N}$, $\mathbf{k}\in \mathbb{N}^n$ with $k_n\in 2\mathbb{N}+1$,
	$D_{x}^{\mathbf{k}} g_{p^{[i]} , \mathbf{m},0}$ is odd about $x_n$. In particular, $D_{x}^{\mathbf{k}} g_{p^{[i]} , \mathbf{m},0}(\tilde{x},0) \equiv 0$.
	For $\mathbf{m}, \mathbf{k}\in \mathbb{N}^n$,
	$D_{x}^{\mathbf{k}} g_{p^{[i]} , \mathbf{m},0}(p^{[j]}) =\delta_{\mathbf{m}, \mathbf{k}} \delta_{p^{[i]}, p^{[j]}}$.
	
	Denote $g_{p^{[i]} , \mathbf{m}}(x,t) = \int_{\mathbb{R}_+^n}
	G_n(x,t,z,0) g_{p^{[i]} , \mathbf{m},0}(z)  dz$.
	By Lemma \ref{23Dec23-1-lem}, $\| \partial_t^{\iota} D_x^{\mathbf{k}}
	g_{p^{[i]} , \mathbf{m}}(\cdot,t) \|_{L^\infty(\mathbb{R}_+^n)} $ $\le \| \Delta_x^{\iota} D_{x}^{\mathbf{k}} g_{p^{[i]} , \mathbf{m},0} \|_{L^{\infty}(\mathbb{R}_+^n)} $ $\lesssim_{\iota,\mathbf{k},\mathbf{m}} 1$;
	$\partial_t^\iota D_{x}^{\mathbf{k}} g_{p^{[i]} , \mathbf{m}}((\tilde{x},0),t) = 0$ for $\tilde{x}\in \mathbb{R}^{n-1}$, $t\in [0,\infty)$ if  $k_n\in 2\mathbb{N}+1$, $m_n\in 2\mathbb{N}$;
	and	
	given $\epsilon>0$, for all $  \mathbf{m}, \mathbf{k} \in \mathbb{N}^n$ satisfying $  \| \mathbf{m} \|_{\ell_{1}}, \| \mathbf{k} \|_{\ell_{1}} \le N_0$, $m_n\in 2\mathbb{N}$, and all $i,j=1,2, \dots, \mathfrak{l}$, there exists $T\ll 1$ such that $\big|D_x^{\mathbf{k}} g_{p^{[i]} , \mathbf{m}}(p^{[j]},T)  - \delta_{\mathbf{m}, \mathbf{k}} \delta_{p^{[i]}, p^{[j]}} \big| <\epsilon$.
	
	One taking $\epsilon$ sufficiently small, by the non-singularity of the strictly diagonally dominant matrix, then for $i=1,2,\dots, \mathfrak{l}$,
	there exist
	$V_{p^{[i]} , \mathbf{m}}(x,t)$ as a linear combination of $g_{p^{[i]} , \mathbf{k} }$, $\| \mathbf{k}\|_{\ell_1}\le N_0$ satisfying the conclusion.
\end{proof}

\subsection{Derivative estimate in the vanishing region}

In the next lemma, we only require $T$ to have a fixed upper bound, but it does not need to be small.
\begin{lemma}\label{lem-smooth}
	Suppose that $n\ge 2$ is an integer, $0<T\le C_T$, $g_1, g_2, g_3$ satisfy
	\begin{equation*}
		\left|g_1(x,t) \right|
		\le C_g (T-t)^{a_1} |x-q|^{b_1} \1_{|x-q|\ge r} ,
		\quad
		\left|  g_2(\tilde{x},t) \right|
		\le C_g (T-t)^{a_2} |\tilde{x}-\tilde{q}|^{b_2} \1_{|\tilde{x}-\tilde{q}|\ge r},
		\quad
		\left| g_3(x) \right|  \le C_g |x-q|^{b_3} \1_{|x-q|\ge r}
	\end{equation*}
for all $x=(\tilde{x},x_n) \in \mathbb{R}_+^n$, $t\in (0,T)$, where $C_T, C_g , r \in \mathbb{R}_+$, $a_1, a_2, b_1, b_2, b_3 \in \mathbb{R}$, $q=(\tilde  q, 0)\in \partial\mathbb{R}_+^n$, denote
	\begin{equation*}
		\begin{aligned}
			&
			f_1(x,t) =
			\int_{0}^{t}
			\int_{\mathbb{R}_+^n}
			G_n(x,t,z,s) g_1(z,s) dz ds ,
			\quad
			f_2(x,t) =
			\int_{0}^{t}\int_{\mathbb{R}^{n-1}}
			G_n(x,t,(\tilde{z},0),s) 	g_2(\tilde{z},s)  d\tilde{z} ds ,
			\\
			&
			f_3(x,t) =
			\int_{\mathbb{R}_+^n}
			G_n(x,t,z, 0) g_3(z)  dz
		\end{aligned}
	\end{equation*}
with $G_n$ given in \eqref{qd23Dec22-1}, then $f_1, f_2, f_3 \in C^\infty\big( \overline{B^+_{n}(q,r/2)}\times [0,T]\big)$, and
	for $\iota\in \mathbb{N}$, $\mathbf{m}\in \mathbb{N}^n$, $(x,t)\in \overline{B^+_{n}(q,r/2)} \times [0,T]$, we have
	\begin{equation*}
		\left| \pp_t^\iota D_x^{\mathbf{m}}
		f_i(x,t) \right| \lesssim_{n, \iota, r, a_i, b_i, C_T, \| \mathbf{m} \|_{\ell_1} }
		C_g
		t^3 e^{-\frac{r^2}{22 t}} , \ i=1,2,
		\quad
		\left| \pp_t^\iota D_x^{\mathbf{m}}
		f_3(x,t) \right|
		\lesssim_{n,\iota, r, b_3, C_T, \| \mathbf{m} \|_{\ell_1} }
		C_g
		t e^{-\frac{r^2}{21 t}}  .
	\end{equation*}
	Moreover, if $m_n$ is odd, then for
	$(\tilde{x},t)\in \overline{B_{n-1}(\tilde{q},r/2)} \times [0,T]$,
	$\pp_t^\iota D_x^{\mathbf{m}}
	f_i((\tilde{x}, 0),t) = 0$, $i=1,2, 3$.
\end{lemma}

\begin{proof}
	Due to the property of supports of $g_1, g_2, g_3$, by the parabolic regularity theory, then $f_1, f_2, f_3 \in C^\infty\big( \overline{B^+_{n}(q,r/2)}\times [0,T]\big)$.
	Notice that
	\begin{equation}\label{qd24Feb07-1}
		\begin{aligned}
			&
			|x-z|\le |x-(\tilde{z},-z_n)|
			\mbox{ \ for \ }
			x,z \in \overline{\mathbb{R}_+^n} ;
			\quad
			\mbox{for \ }
			|x-q|\le r/2, \
			|z-q|\ge r  , \mbox{ \  then \ }
			|x-z|\ge |z-q|/2 ,
			\\
			&
			\Big| \pp_t^\iota D_x^{\mathbf{m}} \Big[(t-s)^{-\frac{n}{2}}
			e^{-\frac{|\tilde{x} - \tilde{z}|^2+ |x_n \pm z_n|^2 }{4(t-s)}} \Big] \Big| \lesssim_{n,\iota,\| \mathbf{m} \|_{\ell_1} }
			(t-s)^{-\frac{n}{2} -\iota -\frac{\| \mathbf{m} \|_{\ell_1} }{2} }
			e^{-\frac{|\tilde{x} - \tilde{z}|^2+ |x_n \pm z_n|^2 }{5(t-s)}}
			\lesssim_{n,\iota,r,C_T, \| \mathbf{m} \|_{\ell_1} }
			e^{-\frac{|z-q|^2}{21(t-s)}}  .
		\end{aligned}
	\end{equation}
	For an integer $n_1 \ge 1$, $C_0 \in \mathbb{R}_+$, $t_1\in (0,C_T]$, $b\in \mathbb{R}$,
	\begin{equation}\label{qd24Feb07-2}
		\int_{\mathbb{R}^{n_1}}
		e^{-C_0 \frac{|z|^2}{t_1}}
		|z|^b \1_{|z|\ge r} dz
		\sim_{C_0, n_1}
		t_1^{\frac{b}{2}+\frac{n_1}{2}}
		\int_{C_0 \frac{r^2}{t_1}}^\infty e^{-y} y^{\frac{b}{2}+\frac{n_1}{2}-1} dy
		\sim_{C_0, n_1, r, b, C_T}
		t_1 e^{- C_0 \frac{r^2}{t_1} }  ,
	\end{equation}
	where we used Lemma \ref{23Dec23-3-lem}\eqref{23Oct10-2} for the last ``$\sim$''.
	Hereafter in this proof, we always assume $(x,t)\in \overline{B^+_{n}(q,r/2)} \times (0,T]$, and
 \eqref{qd24Feb07-1}, \eqref{qd24Feb07-2} will be used repetitively in calculation.
	\begin{equation*}
		\left| 	\pp_t^\iota D_x^{\mathbf{m}}
		f_3(x,t) \right|
		\lesssim_{n,\iota,r, C_T, \| \mathbf{m} \|_{\ell_1} }
		C_g
		\int_{\mathbb{R}_+^n}
		e^{-\frac{|z-q|^2}{21 t}}
		|z-q|^{b_3} \1_{|z-q|\ge r} dz
		\lesssim_{n,\iota, r, b_3, C_T, \| \mathbf{m} \|_{\ell_1} }
		C_g
		t e^{-\frac{r^2}{21 t}}  .
	\end{equation*}
	\begin{equation*}
		\begin{aligned}
			&\quad
			\Big| \int_{\mathbb{R}_+^n}
			\pp_t^{\bar{\iota}} D_x^{\overline{\mathbf{m} }} G_n(x,t,z,s) g_1(z,s) dz \Big|
			\lesssim_{n, \bar{\iota}, r, C_T, \|\overline{\mathbf{m} }\|_{\ell_1} }
			C_g
			(T-s)^{a_1}
			\int_{\mathbb{R}_+^n}
			e^{-\frac{|z-q|^2 }{21 (t-s)}}
			|z-q|^{b_1}
			\1_{|z-q|\ge r} dz
			\\
			&
			\lesssim_{n, \bar{\iota}, r, b_1, C_T, \|\overline{\mathbf{m} }\|_{\ell_1} }
			C_g
			(T-s)^{a_1}
			(t-s)
			e^{-\frac{r^2 }{21 (t-s)}}
		\end{aligned}
	\end{equation*}
for all $\bar{\iota}\in \mathbb{N}, \overline{\mathbf{m} } \in \mathbb{N}^n$.
	In particular, $\lim\limits_{s\uparrow t}  \int_{\mathbb{R}_+^n}
	\pp_t^{\bar{\iota}} D_x^{\overline{\mathbf{m} }} G_n(x,t,z,s) g_1(z,s) dz =0$. Thus,
	\begin{equation*}
		\begin{aligned}
			& \quad
			\left| 	\pp_t^\iota D_x^{\mathbf{m}}
			f_1(x,t) \right|
			=
			\Big|
			\int_{0}^{t} \int_{\mathbb{R}_+^n}
			\pp_t^\iota D_x^{\mathbf{m}}
			\left(
			G_n(x,t,z,s) \right) g_1(z,s) dz ds
			\Big|
			\\
			&
			\lesssim_{n,\iota,r,C_T, \| \mathbf{m} \|_{\ell_1} }
			C_g
			\int_{0}^{t}
			(T-s)^{a_1} \int_{\mathbb{R}_+^n}
			e^{-\frac{|z-q|^2}{21(t-s)}}
			|z-q|^{b_1} \1_{|z-q|\ge r} dz ds
			\\
			&
			\lesssim_{n,\iota,r, b_1, C_T, \| \mathbf{m} \|_{\ell_1} }
			C_g
			\int_{0}^{t}
			(T-s)^{a_1} (t-s)
			e^{-\frac{r^2}{21(t-s)}}   ds
			\\
			&
			\lesssim_{n, \iota, r, a_1, b_1, C_T, \| \mathbf{m} \|_{\ell_1} }
			C_g
			\int_{0}^{t}
			(t-s)
			e^{-\frac{r^2}{22(t-s)}}
			ds
			\sim_{n, \iota, r, a_1, b_1, C_T, \| \mathbf{m} \|_{\ell_1} }
			C_g
			t^3 e^{-\frac{r^2}{22 t}} ,
		\end{aligned}
	\end{equation*}
	where we used Lemma \ref{23Dec23-3-lem}\eqref{23Oct10-2} for the last ``$\sim$''.
	\begin{equation*}
		\begin{aligned}
			&  \quad
			\Big| \int_{\mathbb{R}^{n-1}} \pp_t^{\bar{\iota}} D_x^{\overline{\mathbf{m} }}
			G_n(x,t,(\tilde{z},0),s) 	g_2(\tilde{z},s)  d\tilde{z} \Big|
			\lesssim_{n, \bar{\iota}, r,C_T, \|\overline{\mathbf{m} }\|_{\ell_1}}
			C_g
			(T-s)^{a_2}
			\int_{\mathbb{R}^{n-1}}
			e^{-\frac{|\tilde{z}-\tilde{q}|^2}{21(t-s)}}
			|\tilde{z}-\tilde{q}|^{b_2}
			\1_{|\tilde{z}-\tilde{q}|\ge r}  d\tilde{z}
			\\
			&
			\lesssim_{n, \bar{\iota}, r,b_2,C_T, \|\overline{\mathbf{m} }\|_{\ell_1} }
			C_g
			(T-s)^{a_2} (t-s)
			e^{-\frac{r^2}{21 (t-s)}}
		\end{aligned}
	\end{equation*}
 for all $\bar{\iota}\in \mathbb{N}, \overline{\mathbf{m} } \in \mathbb{N}^n$.
	In particular,
	$\lim\limits_{s\uparrow t} \int_{\mathbb{R}^{n-1}}
	\pp_t^{\bar{\iota}} D_x^{\overline{\mathbf{m} }} G_n(x,t,(\tilde{z},0),s) 	g_2(\tilde{z},s)  d\tilde{z} = 0$.
	Thus,
	\begin{equation*}
		\begin{aligned}
			& \quad
			\left| \pp_t^\iota D_x^{\mathbf{m}}
			f_2(x,t) \right|
			=
			\Big|
			\int_{0}^{t}\int_{\mathbb{R}^{n-1}}
			\pp_t^\iota D_x^{\mathbf{m}}
			\left(
			G_n(x,t,(\tilde{z},0),s)
			\right)	g_2(\tilde{z},s)  d\tilde{z} ds
			\Big|
			\\
			&
			\lesssim_{n, \iota, r, C_T, \| \mathbf{m} \|_{\ell_1} }
			C_g
			\int_{0}^{t}
			(T-s)^{a_2}
			\int_{\mathbb{R}^{n-1}}
			e^{-\frac{|\tilde{z}-\tilde{q}|^2 }{21 (t-s)}}
			|\tilde{z}-\tilde{q}|^{b_2}
			\1_{|\tilde{z}-\tilde{q}|\ge r}
			d\tilde{z} ds
			\\
			&
			\lesssim_{n, \iota, r, b_2, C_T, \| \mathbf{m} \|_{\ell_1} }
			C_g
			\int_{0}^{t}
			(T-s)^{a_2}
			(t-s)
			e^{-\frac{r^2}{21(t-s)}}
			ds
			\lesssim_{n, \iota, r, a_2, b_2, C_T, \| \mathbf{m} \|_{\ell_1} }
			C_g
			t^3 e^{-\frac{r^2}{22 t}} .
		\end{aligned}
	\end{equation*}

If $m_n$ is odd, by \eqref{qd24Feb04-1}, $\pp_t^\iota D_x^{\mathbf{m}}
	f_i((\tilde{x}, 0),t) = 0$ for $t\in (0,T]$, $i=1,2, 3$. Combining $f_1, f_2, f_3 \in C^\infty\big( \overline{B^+_{n}(q,r/2)}\times [0,T]\big)$, we complete the proof.
\end{proof}

\subsection{Scaling argument}

\begin{lemma}\label{24Jan08-1-lem}
	Let $n\ge 2$ be an integer, $-\infty <t_0 <t_1\le \infty$, $v$ be a weak solution of
	\begin{equation}\label{z24Jan08-1}
		\partial_t v = \Delta v + h_1(x,t) \mbox{ \ in \ } \mathbb{R}_+^n \times (t_0,t_1),
		\quad
		\partial_{x_n} v + \beta^0(\tilde{x}, t) v = h_2(\tilde{x},t)
		\mbox{ \ on \ } \partial \mathbb{R}_+^n \times (t_0,t_1),
		\quad
		v(x,t_0) = v_0(x)
		\mbox{ \ in \ } \mathbb{R}_+^n .
	\end{equation}
	Given $x_*\in \overline{\mathbb{R}_+^{n} } , t_* \in (t_0,t_1), \rho\in\mathbb{R}_+$, suppose $\rho
	\| \beta^0 \|_{L^\infty  \big( B_{n-1}(\tilde{x}_*,2 \rho ) \times \big( \max\{ t_0,t_*-4\rho^2\} , t_* \big] \big) } \le C_{\beta^0,1}$, $v, h_1\in L_{\rm{loc}}^{\infty} ( \overline{\mathbb{R}_+^n} \times [t_0,t_1) )$,
$h_2 \in L_{\rm{loc}}^{\infty} (\mathbb{R}^{n-1} \times [t_0,t_1) )$,
then
there exist positive constants $C_1, \alpha_0$ only depending on $n$ such that if $\alpha \in (0,\alpha_0], v_0 \in C_{\rm{loc}}^{\alpha} ( \overline{\mathbb{R}_+^n} )$, we have
	\begin{equation}\label{qd24Jan08-2}
		\begin{aligned}
			&
			[v]_{C^{\alpha,\alpha/2} \big( B_n^+(x_*,\rho)
				\times \big( \max\{t_0, t_*-\rho^2 \}, t_* \big]
				\big) }
			\\
			\le \ &
			C_1 \langle C_{\beta^0,1} \rangle
			\rho^{-\alpha}
			\Big[
			\| v \|_{L^\infty
				\big(  B_n^+(x_*,2\rho)  \times \big( \max\{ t_0,t_*-4\rho^2\} , t_* \big] \big) }
			+
			\rho^2 \|h_1\|_{L^\infty
				\big( B_n^+(x_*,2\rho) \times \big( \max\{ t_0,t_*-4\rho^2\} , t_* \big] \big) }
			\\
			&
			+
			\1_{x_{*n} \le 4\rho}
			\
			\rho
			\| h_2 \|_{L^\infty  \big( B_{n-1}(\tilde{x}_*,2 \rho ) \times \big( \max\{ t_0,t_*-4\rho^2\} , t_* \big] \big) }
			+
			\1_{\sqrt{t_*-t_0} \le 4\rho} \
			\left( \| v_0 \|_{L^{\infty} ( B_n^+(x_*,2 \rho ) ) }
			+
			\rho^\alpha [ v_0 ]_{C^{\alpha } (B_n^+(x_*,2\rho )) }
			\right)	
			\Big].
		\end{aligned}
	\end{equation}
Under the additional assumption that for $\gamma \in (0,1)$,
\begin{equation}\label{qd24Jan07-2}
	\rho \|\beta^0\|_{L^{\infty} \big( B_{n-1}(\tilde{x}_*,2\rho) \times \big(\max\{t_0, t_* -4\rho^2 \}, t_* \big] \big) }
	+
	\rho^{1+\gamma} [\beta^0]_{C^{\gamma, \gamma/2} \big( B_{n-1}(\tilde{x}_*,2\rho) \times \big(\max\{t_0, t_* -4\rho^2 \}, t_* \big] \big) }
	\le C_{\beta^0,2} ,
\end{equation}
$h_2 \in C_{\rm{loc}}^{\gamma, \gamma/2} (\mathbb{R}^{n-1} \times [t_0,t_1) )$,
$v_0 \in C_{\rm{loc}}^{1+\gamma} ( \overline{\mathbb{R}_+^n} )$,
then we have
\begin{equation}\label{qd24Jan08-5}
	\begin{aligned}
		&
		\rho
		\| \nabla_x v \|_{L^\infty
			\big( B_n^+(x_*,\rho) \times \big( \max\{ t_0,t_*-\rho^2\} , t_* \big] \big) }
		+
		\rho^{1+\gamma} [\nabla_x v]_{C^{\gamma,\gamma/2} \big( B_n^+(x_*,\rho) \times \big( \max\{ t_0 , t_*-\rho^2\}, t_* \big] \big) }
		\\
		&
		+
		\rho^{1+\gamma}
		\sup\limits_{ x\in B_n^+(x_*,\rho) , \  \tau_1, \tau_2 \in ( \max\{ t_0, t_*-\rho^2\} , t_* ] }
		\frac{|v(x,\tau_1) - v(x,\tau_2) |}{|\tau_1-\tau_2 |^{(1+\gamma)/2} }
		\\
		\le \ &
		C_2
		\Big[
		\| v \|_{L^\infty
			\big( B_n^+(x_*,2\rho) \times \big( \max\{ t_0,t_*-4\rho^2\} , t_* \big] \big) }
		+
		\rho^2 \|h_1\|_{L^\infty
			\big( B_n^+(x_*,2\rho) \times \big( \max\{ t_0,t_*-4\rho^2\} , t_* \big] \big) }
		\\
		&
		+
		\1_{x_{*n} \le 4\rho} \
		\Big(
		\rho \|h_2\|_{L^{\infty} \big( B_{n-1}(\tilde{x}_*,2\rho) \times \big(\max\{t_0, t_* -4\rho^2 \}, t_* \big] \big) }
		+
		\rho^{1+\gamma} [h_2]_{C^{\gamma, \gamma/2} \big( B_{n-1}(\tilde{x}_*,2\rho) \times \big(\max\{t_0, t_* -4\rho^2 \}, t_* \big] \big) }
		\Big)
		\\
		&
		+
		\1_{ \sqrt{t_*-t_0} \le 4\rho } \
		\Big( \| v_0 \|_{L^{\infty} ( B_n^+(x_*,2\rho ) ) }
		+
		\rho^{1+\gamma} [ \nabla_x v_0 ]_{C^{\gamma } (B_n^+(x_*,2\rho )) } \Big) \Big]
	\end{aligned}
\end{equation}
with a positive constant $C_2$ only depending on $n, \gamma, C_{\beta^0,2}$.

\end{lemma}

\begin{proof}
	Set $ \tilde{v}(z,s) = v(x_*+ \rho  z , t_* + \rho^2 s)  $, that is, $
	v(x, t) =
	\tilde{v}\big( \rho^{-1}(x-x_*) , \rho^{-2}(t-t_*) \big) $. By direct calculation,
	\begin{equation*}
			\pp_{t}v =
			\rho^{-2} \left(\pp_{s}\tilde{v}\right)( \rho^{-1}(x-x_*) , \rho^{-2}(t-t_*) ) ,
			\quad
			D_x^{\mathbf{m}} v =
			\rho^{-\| \mathbf{m}\|_{\ell_1}}  ( D_z^{\mathbf{m}} \tilde{v} )( \rho^{-1}(x-x_*) , \rho^{-2}(t-t_*) )
	\end{equation*}
for $\mathbf{m} \in \mathbb{N}^n, \| \mathbf{m}\|_{\ell_1} \le 2$.
	Then \eqref{z24Jan08-1} is equivalent to
	\begin{equation*}
		\left\{
		\begin{aligned}
			&
			\pp_{s}\tilde{v}
			=  \Delta_z \tilde{v}
			+
			\tilde{h}_1(z,s)
			\mbox{ \ for \ }
			(z,s) \in Q	,
			\quad
			\pp_{z_n} \tilde{v}
			+
			\tilde{\beta}^0(\tilde{z},s) \tilde{v} =  \tilde{h}_2(\tilde{z},s)
			\mbox{ \ for \ } (z,s) \in S Q ,
			\\
			&
			\tilde{v}  = v_0(x_*+\rho z )
			\mbox{ \ for \ }
			(z,s) \in B Q,
		\end{aligned}
		\right.
	\end{equation*}
	where
	\begin{equation*}
		\begin{aligned}
			&
			\tilde{\beta}^0(\tilde{z},s)
			:= \rho \beta^0(\tilde{x}_*+\rho \tilde{z},t_*+\rho^2 s),
			\quad
			\tilde{h}_1(z,s):= \rho^2 h_1(x_*+ \rho  z , t_* + \rho^2 s),
			\quad
			\tilde{h}_2(\tilde{z},s):=\rho h_2(\tilde{x}_*+\rho \tilde{z},t_*+\rho^2 s )  ,
			\\
			&
			Q := \left\{ (z,s) \ | \  z\in A_1, \
			s\in \left(\rho^{-2}(t_0-t_*),
			\rho^{-2}(t_1-t_*)
			\right) \right\} ,
\quad
A_1 := \{ z \ | \ \tilde{z}\in\mathbb{R}^{n-1},
\
z_n> -\rho^{-1}x_{*n} \} ,
\\
			&
			S Q := \left\{ \big( (\tilde{z} , -\rho^{-1}x_{*n} ),s \big) \ | \
			(\tilde{z},s)\in A_2 \right\} ,
\quad
A_2 := \left\{ ( \tilde{z},s) \ | \ \tilde{z}\in\mathbb{R}^{n-1},
\
s\in \left(\rho^{-2}(t_0-t_*),
\rho^{-2}(t_1-t_*)
\right) \right\} ,
\\
			&
			B Q := \left\{ \big(z, \rho^{-2}(t_0-t_*) \big) \ | \ z\in A_1  \right\}  .
		\end{aligned}
	\end{equation*}
Obviously,
\begin{equation}\label{qd24Jan13-1}
	\begin{aligned}
&
(z,s)\in Q \Leftrightarrow x_*+\rho z \in \mathbb{R}_+^n, t_* + \rho^2 s \in (t_0,t_1);
\quad
(\tilde{z},s) \in A_2 \Leftrightarrow
\tilde{x}_*+\rho \tilde{z} \in \mathbb{R}^{n-1}, t_* + \rho^2 s \in (t_0,t_1) ;
\\
&
z \in A_1
\Leftrightarrow x_*+\rho z \in \mathbb{R}_+^n  .
\end{aligned}
\end{equation}

For any $r\in \mathbb{R}_+$, for $r>0$, denote $Q_n(r) := B_n(0,r) \times (-r^2,0]$. For brevity,
denote $\|f_1\|_{L^\infty  (Q_n(r) \cap S Q) } := \|f_1\|_{L^\infty  \big( \big\{ (\tilde{z},s) \ \big| \  \big( (\tilde{z} , -\rho^{-1}x_{*n} ),s \big) \in Q_n(r) \cap S Q \big\} \big) }$, $\| f_2 \|_{C^{\alpha} (Q_n(r)\cap B Q) } := \| f_2 \|_{C^{\alpha} \big( \big\{ z \ \big| \  \big(z, \rho^{-2}(t_0-t_*) \big) \in Q_n(r)\cap B Q \big\}
\big) }  $, and \\
$\|f_3 \|_{C^{\gamma,\gamma/2} (Q_n(r) \cap S Q) }$,
$\| f_4 \|_{C^{1+\gamma} (Q_n(r)\cap B Q) }$ are defined similarly.
Since
\begin{equation}\label{qd24Jan08-9}
		\|\tilde{\beta}^0\|_{L^\infty  (Q_n(2) \cap S Q) }
		\le
		\rho
		\| \beta^0(\tilde{x}_*+\rho \tilde{z},t_*+\rho^2 s ) \|_{L^\infty  (  Q_{n-1}(2) \cap A_2 ) }
		=
		\rho
		\| \beta^0 \|_{L^\infty  \big( B_{n-1}(\tilde{x}_*,2 \rho ) \times \big( \max\{ t_0,t_*-4\rho^2\} , t_* \big] \big) } \le C_{\beta^0,1} ,
\end{equation}
by \cite[p.140 Theorem 6.44]{lieberman1996second} (see also \cite[Corollary 7.6]{Lieberman2001-nonsmooth}), there exist positive constants $C_1, \alpha_0$ determined only by $n$ such that if $\alpha \in (0,\alpha_0], v_0 \in C_{\rm{loc}}^{\alpha} ( \overline{\mathbb{R}_+^n} )$, we have
	\begin{equation*}
		[\tilde{v}]_{C^{\alpha,\alpha/2} ( Q_n(1) \cap Q) }
		\le
C_1 \langle C_{\beta^0,1} \rangle
		\Big(
		\| \tilde{v} \|_{L^\infty
			( Q_n(2) \cap Q ) }
		+
		\|\tilde{h}_1 \|_{L^\infty
			(Q_n(2) \cap Q) }
		+
		\|\tilde{h}_2 \|_{L^\infty  (Q_n(2) \cap S Q) }
		+
		\| v_0(x_*+\rho z ) \|_{C^{\alpha} (Q_n(2)\cap B Q) } \Big)  .
	\end{equation*}
We will handle the above term by term and \eqref{qd24Jan13-1} will be used repetitively.
	\begin{equation}\label{qd24Jan08-4}
		\begin{aligned}
			&
			[\tilde{v}]_{C^{\alpha,\alpha/2} ( Q_n(1) \cap Q) }
			=
			[ v(x_*+ \rho  z , t_* + \rho^2 s) ]_{C^{\alpha,\alpha/2} ( Q_n(1) \cap Q) }
			\\
			= \ &
			\rho^{\alpha}
			\sup\limits_{(z^{[1]}, s_1), (z^{[2]}, s_2) \in  Q_n(1) \cap Q}
			\frac{\big| v(x_*+ \rho  z^{[1]} , t_* + \rho^2 s_1) - v(x_*+ \rho  z^{[2]} , t_* + \rho^2 s_2) \big|}{\big(\max\{ | (x_*+ \rho  z^{[1]} ) - (x_*+ \rho  z^{[2]}) |, | (t_* + \rho^2 s_1) - (t_* + \rho^2 s_2)|^{1/2} \} \big)^{\alpha} }
			\\
			= \ &
			\rho^{\alpha} [v]_{C^{\alpha,\alpha/2} \big( B_n^+(x_*,\rho) \times \big( \max\{t_0, t_*-\rho^2 \}, t_* \big] \big) }  ,
		\end{aligned}
	\end{equation}
where we used \eqref{qd24Jan13-1} for the last step.
	\begin{equation}\label{qd24Jan08-3}
		\begin{aligned}
			&
			\| \tilde{v} \|_{L^\infty
				( Q_n(2) \cap Q ) } = \| v \|_{L^\infty
				\big( B_n^+(x_*,2\rho) \times \big( \max\{ t_0,t_*-4\rho^2\} , t_* \big] \big) } ,
			\\
			&
			\|\tilde{h}_1 \|_{L^\infty
				(Q_n(2) \cap Q) } =
			\rho^2 \|h_1\|_{L^\infty
				\big( B_n^+(x_*,2\rho) \times \big( \max\{ t_0,t_*-4\rho^2\} , t_* \big] \big) } .
		\end{aligned}
	\end{equation}

	$\|\tilde{h}_2 \|_{L^\infty  (Q_n(2) \cap S Q) }$ is vacuum if $-\rho^{-1}x_{*n}<-4$. For $-\rho^{-1}x_{*n}\ge -4 \Leftrightarrow x_{*n} \le 4\rho$, similar to \eqref{qd24Jan08-9},
	\begin{equation*}
			\|\tilde{h}_2 \|_{L^\infty  (Q_n(2) \cap S Q) }
			\le
			\rho
			\| h_2 \|_{L^\infty  \big( B_{n-1}(\tilde{x}_*,2 \rho ) \times \big( \max\{ t_0,t_*-4\rho^2\} , t_* \big] \big) } .
	\end{equation*}

	$\| v_0(x_*+\rho z ) \|_{C^{\alpha} (Q_n(2)\cap B Q) }$ is vacuum if $\rho^{-2}(t_0-t_*) <-16$. For $\rho^{-2}(t_0-t_*) \ge -16 \Leftrightarrow \sqrt{t_*-t_0} \le 4\rho$,
\begin{equation}\label{z24Jan07-8}
	\begin{aligned}
		&
		[ v_0(x_*+\rho z ) ]_{C^{\alpha} (Q_n(2)\cap B Q) }
		\le
		\rho^\alpha
		\sup\limits_{x_*+\rho z^{[1]} , x_*+\rho z^{[2]} \in B_n^+(x_*,2\rho ) }
		\frac{|v_0(x_*+\rho z^{[1]} ) - v_0(x_*+\rho z^{[2]} )|}{| (x_*+\rho z^{[1]} ) - (x_*+\rho z^{[2]} ) |^{\alpha}}
		=
		\rho^\alpha [ v_0 ]_{C^{\alpha } (B_n^+(x_*,2\rho )) }  ,
		\\
		&
		\| v_0(x_*+\rho z ) \|_{L^{\infty} (Q_n(2)\cap B Q) }
		\le
		\| v_0 \|_{L^{\infty} ( B_n^+(x_*,2\rho ) ) }  .
	\end{aligned}
\end{equation}
As a result, we get \eqref{qd24Jan08-2}.
	
\medskip

Under the additional assumption,  we have
$ \|\tilde{\beta}^0\|_{C^{\gamma,\gamma/2} (Q_n(2) \cap S Q ) }
	\le C_{\beta^0,2} $ since \eqref{qd24Jan07-2} holds and
\begin{equation}\label{qd24Jan13-2}
	\begin{aligned}
		&
		[\tilde{\beta}^0]_{C^{\gamma,\gamma/2} (Q_n(2) \cap S Q ) }
		\le
		\rho
		\left[ \beta^0(\tilde{x}_*+\rho \tilde{z},t_*+\rho^2 s ) \right]_{C^{\gamma,\gamma/2} (  Q_{n-1}(2)\cap A_2 ) }
		\\
		= \ &
		\rho^{1+\gamma}
		\sup_{ (\tilde{z}^{[1]},s_1), (\tilde{z}^{[2]} ,s_2) \in Q_{n-1}(2)\cap A_2   }
		\frac{\big| \beta^0(\tilde{x}_*+\rho \tilde{z}^{[1]},t_*+\rho^2 s_1 )  -
		\beta^0(\tilde{x}_*+\rho \tilde{z}^{[2]},t_*+\rho^2 s_2 )
			\big|}{\big(\max\{ | (\tilde{x}_*+\rho \tilde{z}^{[1]}) -
			(\tilde{x}_*+\rho \tilde{z}^{[2]})|,
			| (t_*+\rho^2 s_1) -(t_*+\rho^2 s_2)|^{1/2} \} \big)^{\gamma} }
		\\
		= \ &
		\rho^{1+\gamma} [\beta^0]_{C^{\gamma, \gamma/2} \big( B_{n-1}(\tilde{x}_*,2\rho) \times \big(\max\{t_0, t_* -4\rho^2 \}, t_* \big] \big) } ,
		\\
		&
		\|\tilde{\beta}^0\|_{L^{\infty} (Q_n(2) \cap S Q ) }
		\le
		\rho \|\beta^0\|_{L^{\infty} \big( B_{n-1}(\tilde{x}_*,2\rho) \times \big(\max\{t_0, t_* -4\rho^2 \}, t_* \big] \big) }  .
	\end{aligned}
\end{equation}
	
Hence, by \cite[p.69 Theorem 4.21, and p.79 Theorem 4.30]{lieberman1996second}, we have
	\begin{equation*}
		\begin{aligned}
			\| \tilde{v} \|_{C^{1+\gamma, (1+\gamma)/2 } ( Q_n(1) \cap Q) }
			\le \ &
			C_{21}
			\Big(
			\| \tilde{v} \|_{L^\infty
				( Q_n(2) \cap Q ) }
			+
			\|\tilde{h}_1 \|_{L^\infty
				(Q_n(2) \cap Q) }
			\\
			&
			+
			\|\tilde{h}_2 \|_{C^{\gamma,\gamma/2} (Q_n(2) \cap S Q) }
			+
			\| v_0(x_*+\rho z ) \|_{C^{1+\gamma} (Q_n(2)\cap B Q) } \Big)
		\end{aligned}
	\end{equation*}
with a constant $C_{21}>0$ only depending on $n, \gamma, C_{\beta^0,2}$.
We will handle the above term by term.
	\begin{equation*}
		\| \tilde{v} \|_{C^{1+\gamma, (1+\gamma)/2 } ( Q_n(1) \cap Q) } =
		\| \tilde{v} \|_{L^{\infty} ( Q_n(1) \cap Q) }
		+
		\| \nabla_z \tilde{v} \|_{L^{\infty} ( Q_n(1) \cap Q) }
		+
		[ \nabla_z \tilde{v} ]_{C^{\gamma,\gamma/2} ( Q_n(1) \cap Q) }
		+
		[\tilde{v} ]_{C_t^{(1+\gamma)/2} ( Q_n(1) \cap Q) }  .
	\end{equation*}
	Therein, by \eqref{qd24Jan13-1},
	\begin{equation*}
		\| \nabla_z \tilde{v} \|_{L^{\infty} ( Q_n(1) \cap Q) }
		=
		\rho \| (\nabla_x v)(x_*+ \rho  z , t_* + \rho^2 s) \|_{L^{\infty} ( Q_n(1) \cap Q) }
		=
		\rho
		\| \nabla_x v \|_{L^\infty
			\big( B_n^+(x_*,\rho) \times \big( \max\{ t_0,t_*-\rho^2\} , t_* \big] \big) }  .
	\end{equation*}
	Similar to \eqref{qd24Jan08-4},
	\begin{equation*}
		[ \nabla_z \tilde{v} ]_{C^{\gamma,\gamma/2} ( Q_n(1) \cap Q) }
		=
		\rho
		[ (\nabla_x v)(x_*+ \rho  z , t_* + \rho^2 s) ]_{C^{\gamma,\gamma/2} ( Q_n(1) \cap Q) }
		=
		\rho^{1+\gamma} [\nabla_x v]_{C^{\gamma,\gamma/2} \big( B_n^+(x_*,\rho) \times \big( \max\{ t_0 , t_*-\rho^2\}, t_* \big] \big) }  .
	\end{equation*}
	\begin{equation*}
		\begin{aligned}
			&
			[\tilde{v} ]_{C_t^{(1+\gamma)/2} ( Q_n(1) \cap Q) }
			=
			\rho^{1+\gamma}  \sup\limits_{ (z,s_1),(z,s_2) \in Q_n(1) \cap Q }
			\frac{|v(x_*+ \rho  z , t_* + \rho^2 s_1) - v(x_*+ \rho  z , t_* + \rho^2 s_2) |}{| (t_* + \rho^2 s_1)- (t_* + \rho^2 s_2) |^{(1+\gamma)/2} }
			\\
			= \ &
			\rho^{1+\gamma}
			\sup\limits_{ x\in B_n^+(x_*,\rho) , \  \tau_1, \tau_2 \in ( \max\{ t_0, t_*-\rho^2\} , t_* ] }
			\frac{|v(x,\tau_1) - v(x,\tau_2) |}{|\tau_1-\tau_2 |^{(1+\gamma)/2} }  .
		\end{aligned}
	\end{equation*}
	
	$\| \tilde{v} \|_{L^\infty
		( Q_n(2) \cap Q ) } ,
	\|\tilde{h}_1 \|_{L^\infty
		(Q_n(2) \cap Q) }$ have been handled in \eqref{qd24Jan08-3}.
	
	$\|\tilde{h}_2 \|_{C^{\gamma,\gamma/2} (Q_n(2) \cap S Q) }$ is vacuum if $-\rho^{-1} x_{*n}<-4$. For $-\rho^{-1} x_{*n} \ge -4 \Leftrightarrow x_{*n} \le 4\rho$, similar to \eqref{qd24Jan13-2},
	\begin{equation*}
		\begin{aligned}
			&
			[\tilde{h}_2]_{C^{\gamma,\gamma/2} (Q_n(2) \cap S Q ) }
			\le
			\rho^{1+\gamma} [h_2]_{C^{\gamma, \gamma/2} \big( B_{n-1}(\tilde{x}_*,2\rho) \times \big(\max\{t_0, t_* -4\rho^2 \}, t_* \big] \big) } ,
			\\
			&
			\|\tilde{h}_2\|_{L^{\infty} (Q_n(2) \cap S Q ) }
			\le
			\rho \|h_2\|_{L^{\infty} \big( B_{n-1}(\tilde{x}_*,2\rho) \times \big(\max\{t_0, t_* -4\rho^2 \}, t_* \big] \big) }  .
		\end{aligned}
	\end{equation*}

	$\| v_0(x_*+\rho z ) \|_{C^{1+\gamma} ( Q_n(2) \cap B Q )}$ is vacuum if $\rho^{-2}(t_0-t_*) < -16 $.
	For $\rho^{-2}(t_0-t_*) \ge -16 \Leftrightarrow \sqrt{t_*-t_0} \le 4\rho$, similar to \eqref{z24Jan07-8}, we have
	\begin{equation*}
		\begin{aligned}
			&
			\| v_0(x_*+\rho z ) \|_{C^{1+\gamma} ( Q_n(2) \cap B Q ) }
			\sim_{\gamma}
			\| v_0(x_*+\rho z ) \|_{L^{\infty} (  Q_n(2) \cap B Q ) }
			+
			\rho [ (\nabla_x v_0)(x_*+\rho z ) ]_{C^{\gamma} ( Q_n(2) \cap B Q ) }
			\\
			\le \ &
			\| v_0 \|_{L^{\infty} ( B_n^+(x_*,2\rho ) ) }
			+
			\rho^{1+\gamma} [ \nabla_x v_0 ]_{C^{\gamma } (B_n^+(x_*,2\rho )) }  .
		\end{aligned}
	\end{equation*}
	In sum, we obtain \eqref{qd24Jan08-5}.	
\end{proof}

\section{Approximate solution and inner-outer gluing scheme}\label{Approximate solution and inner-outer gluing scheme}

The finite-time blow-up solutions of heat equations with the critical boundary condition are expected to behave like the dilation of steady-state solutions. Given an integer $\mathfrak{o}\ge 1$, we define
\begin{equation*}
	U_{\mu_i,\xi^{[i]}}(x) :=\mu_i^{-\frac{n-2}{2}}U\left(\frac{\tilde{x}-\xi^{[i]}}{\mu_i},\frac{x_n}{\mu_i}\right) ,
	\quad \tilde{x}\in\mathbb{R}^{n-1}, \ x_n\in \mathbb{R}_+, \ x=\left(\tilde{x}, x_n\right),
	\quad i=1,2,\dots,\mathfrak{o}
\end{equation*}
with $\mu_i=\mu_i(t) \in C^{1}\left([0,T),\mathbb{R}_+ \right)$, $\xi^{[i]}=\xi^{[i]}(t)
=
\big( \xi_1^{[i]}(t) ,\xi_2^{[i]}(t),\dots, \xi_{n-1}^{[i]}(t)  \big) \in C^{1}\left([0,T), \mathbb{R}^{n-1} \right)$ to be determined later.
Given arbitrary $\mathfrak{o}$ distinct points $q^{[i]}$ on $\partial \mathbb{R}_+^n$, $i=1,2,\dots, \mathfrak{o}$, we define
\begin{equation*}
x^{[i]}:= x-q^{[i]} ,
\quad
	y^{[i]}:=\frac{x-(\xi^{[i]}, 0) }{\mu_i},
	\quad
	z^{[i]} :=\frac{x-q^{[i]}}{\sqrt{T-t}},
	\quad
	\delta:= \frac{1}{32} \min\limits_{1\leq i\neq j\leq \mathfrak{o}} |q^{[i]}-q^{[j]}| ,
\end{equation*}
and some cut-off functions
\begin{equation}\label{cut-off}
\eta_{R}(y^{[i]}) :=\eta(\frac{y^{[i]}}{R})
\mbox{ \ with \ }
R=|\ln{T}|
,
    \quad
\eta_{C \delta}(x^{[i]}):=\eta(\frac{x^{[i]}}{C \delta}) \mbox{ \ for \ } C>0.
\end{equation}

Given $\mathfrak{o}$ integers $l_i\in \mathbb{N}$ (could  be duplicated), $i=1,2,\dots, \mathfrak{o}$, denote $l_{\rm{max}} := \max\limits_{i=1,2,\dots, \mathfrak{o}} l_i $.
We look for the solution of the form
\begin{equation*}
	u=\sum_{i=1}^{\mathfrak{o}}\left(
	U_{\mu_i,\xi^{[i]}} (x)
\eta_{2\delta}(x^{[i]})
	+
	\Theta_{l_i}(x^{[i]},t)\eta_{\delta}(x^{[i]})+
	\mu_{i}^{-\frac{n-2}{2}} \phi_i(y^{[i]},t) \eta_{R}(y^{[i]}) \right)+\psi(x,t),
\end{equation*}
where $\Theta_{l_i}$ is given in \eqref{eq-heat-1} satisfying \eqref{eq-heat} and $\phi_i, \psi$ will be determined later.
Denote
\begin{equation*}
	{\bm{\mu}}=(\mu_1, \mu_2,\dots,\mu_{\mathfrak{o}}),
	\quad
	\bm{\xi} =
	(\xi^{[1]}, \xi^{[2]},\dots, \xi^{[\mathfrak{o}]}) ,
	\quad
	\bm{\phi} =
	(\phi_1,\phi_2,\dots,\phi_{\mathfrak{o}}) .
\end{equation*}

We introduce  the error operators as
\begin{equation*} \mathcal{E}_1[u]:=-\pp_t u+\Delta_x u
	\mbox{ \ in \ }  \mathbb{R}^{n}_+\times (0,T) ,
	\qquad
	\mathcal{E}_2[u]:=
	\pp_{x_n}u +|u|^\frac{2}{n-2}u
	\mbox{ \ on \ } \partial \mathbb{R}^{n}_+ \times (0,T) .
\end{equation*}
Solving \eqref{qd2023Nov21-1} is equivalent to making $\mathcal{E}_1[u] =0$ and $\mathcal{E}_2[u] =0$. By \eqref{eq-heat},
direct calculations give that \begin{equation*}
	\begin{aligned}
		& \mathcal{E}_1[u] =
		\left(-\partial_t + \Delta_x \right) \psi
		+
		\sum_{i=1}^{\mathfrak{o}}
		\Big[
		\mathcal{E}_1\left[U_{\mu_i,\xi^{[i]}} (x)\right] \eta_{2\delta}(x^{[i]})
		+
		\mathcal{E}_{U,i}^{\rm{cut}}
		+
		\mathcal{E}_{\Theta,i}^{\rm{cut}}
		\\
		& +
		\mu_i^{-\frac{n+2}{2}}
		\eta_{R}(y^{[i]})
		\left(
		-  \mu_i^2  \pp_t \phi_i(y^{[i]},t)
		+
		  \Delta_{y^{[i]}} \phi_i(y^{[i]},t)
		  \right)
		+
		\Lambda_{1,i}[\phi_i,\mu_i,\xi^{[i]}] + \Lambda_{2,i}[\phi_i,\mu_i,\xi^{[i]}]
		\Big],
	\end{aligned}
\end{equation*}
where
\begin{equation}\label{qd23Dec11-3}
\begin{aligned}
&
	\mathcal{E}_1\left[U_{\mu_i,\xi^{[i]}} (x)\right] =
	\dot{\mu}_i
	\mu_i^{-\frac{n}{2}}  Z_n(y^{[i]})
	+
	\mu_i^{-\frac{n}{2}}
	\dot{\xi}^{[i]} \cdot  \left( \nabla_{\tilde{y}^{[i]}} U \right)(\tilde{y}^{[i]},y_n^{[i]}) ,
\\
&
\mathcal{E}_{U,i}^{\rm{cut}} :=
2
\nabla_x \left(
U_{\mu_i,\xi^{[i]}} (x) \right)
\cdot
\nabla_x \left(
\eta_{2\delta}(x^{[i]}) \right)
+
U_{\mu_i,\xi^{[i]}} (x)
\Delta_x\left(
\eta_{2\delta}(x^{[i]}) \right) ,
\\
&
\mathcal{E}_{\Theta,i}^{\rm{cut}}:=
2\nabla_x\left(\Theta_{l_i}(x^{[i]},t) \right) \cdot \nabla_x\left( \eta_{\delta}(x^{[i]}) \right)
+
\Theta_{l_i}(x^{[i]},t)
\Delta_x\left(
\eta_{\delta}(x^{[i]}) \right)  ,
\\
&
\Lambda_{1,i}[\phi_i,\mu_i,\xi^{[i]}] :=
\mu_i^{-\frac{n-2}{2}}
\left[
\frac{\dot{\xi}^{[i]}}{\mu_i R}  \cdot \left( \nabla_{\tilde{x}} \eta\right)(\frac{\tilde{y}^{[i]}}{R},\frac{y_n^{[i]}}{R})
+
\frac{(\mu_i R)'}{\mu_i R}
\frac{y^{[i]}}{R} \cdot \left( \nabla_{x} \eta\right)(\frac{y^{[i]}}{R} )
\right] \phi_i(y^{[i]},t)
\\
& \qquad
+
\mu_i^{-\frac{n+2}{2}}
R^{-2} \phi_i(y^{[i]},t) \left(\Delta_x \eta\right)(\frac{y^{[i]}}{R})
+
2 \mu_i^{-\frac{n+2}{2}}
R^{-1} \nabla_{y^{[i]}} \phi_i(y^{[i]},t) \cdot \left(\nabla_x \eta\right)(\frac{y^{[i]}}{R}) ,
\\
&
\Lambda_{2,i}[\phi_i,\mu_i,\xi^{[i]}] :=
\mu_i^{-\frac{n-2}{2}}
\eta_{R}(y^{[i]})
\left[ \mu_i^{-1} \dot{\xi}^{[i]} \cdot \nabla_{\tilde{y}^{[i]}} \phi_i(y^{[i]},t)
+ \dot{\mu}_i \mu_i^{-1}
\left(
\frac{n-2}{2}  \phi_i(y^{[i]},t)
+
y^{[i]} \cdot \nabla_{y^{[i]}} \phi_i(y^{[i]},t)
\right)
\right].
\end{aligned}
\end{equation}

Since $\eta(x)$ is a smooth radial cut-off function, we have $\partial_{x_n} f \big|_{x_n=0} = 0$ on $\partial \mathbb{R}^{n}_+ \times (0,T)$ for $f= \eta_{R}(y^{[i]}) , \eta_{\delta}(x^{[i]} ), \eta_{2\delta}(x^{[i]})$. Combining \eqref{U-eq} for $U_{\mu_i,\xi^{[i]}}$ and \eqref{eq-heat} for $\Theta_{l_i}$, we have
\begin{equation}\label{qd24Jan13-3}
	\begin{aligned}
		&
		\mathcal{E}_2[u] =\sum_{i=1}^{\mathfrak{o}}
		\Big[
		- \left(U_{\mu_i,\xi^{[i]}} (x) \right)^{\frac{n}{n-2}}
		\eta_{2\delta}(x^{[i]})
		+
		\mu_i^{-\frac{n}{2}} \eta_{R}(y^{[i]}) \pp_{y_n^{[i]}} \phi_i(y^{[i]},t)
		\Big] +
		\partial_{x_n} \psi
		+|u|^\frac{2}{n-2}u
		\\
		= \ &
		\sum_{i=1}^{\mathfrak{o}}
		\Big[
		\left(U_{\mu_i,\xi^{[i]}} (x) \right)^{\frac{n}{n-2}}
		\left(
		\eta_{2\delta}^{\frac{n}{n-2}}(x^{[i]})
		-
		\eta_{2\delta}(x^{[i]})
		\right)
		+
		\mu_i^{-\frac{n}{2}} \eta_{R}(y^{[i]}) \pp_{y_n^{[i]}} \phi_i(y^{[i]},t)
		\Big] +
		\partial_{x_n} \psi
		\\
		& +
		\sum_{i=1}^{\mathfrak{o}}
		\frac{n}{n-2}
		\left(U_{\mu_i,\xi^{[i]}} (x) \right)^{\frac{2}{n-2}}
		\eta_{2\delta}^{\frac{2}{n-2}}(x^{[i]})
		\left(
		\Theta_{l_i}(x^{[i]},t)\eta_{\delta}(x^{[i]})+
		\mu_{i}^{-\frac{n-2}{2}} \phi_i(y^{[i]},t) \eta_{R}(y^{[i]}) +\psi(x,t)
		\right)
		+ \mathcal{N}\left[\psi,\bm{\phi},\bm{\mu},\bm{\xi}\right]
	\end{aligned}
\end{equation}
with the nonlinear term
\begin{equation*}
	\begin{aligned}
		&
		\mathcal{N}\left[\psi,\bm{\phi},\bm{\mu},\bm{\xi}\right]
		:=
		|u|^\frac{2}{n-2}u -\sum_{i=1}^{\mathfrak{o}} \left(U_{\mu_i,\xi^{[i]}} (x) \right)^{\frac{n}{n-2}}
		\eta_{2\delta}^{\frac{n}{n-2}}(x^{[i]})
		\\
		&
		-\sum_{i=1}^{\mathfrak{o}}
		\frac{n}{n-2}
		\left(U_{\mu_i,\xi^{[i]}} (x) \right)^{\frac{2}{n-2}}
		\eta_{2\delta}^{\frac{2}{n-2}}(x^{[i]})
		\left(
		\Theta_{l_i}(x^{[i]},t)\eta_{\delta}(x^{[i]})+
		\mu_{i}^{-\frac{n-2}{2}} \phi_i(y^{[i]},t) \eta_{R}(y^{[i]}) +\psi(x,t)
		\right).
	\end{aligned}
\end{equation*}

In order to make $\mathcal{E}_1[u]=0$ in $\mathbb{R}_+^n \times \left(0,T\right)$ and  $\mathcal{E}_2[u]=0$ on $\partial \mathbb{R}_+^n \times \left(0,T\right)$, it suffices to solve the following inner-outer gluing system.

{\textbf{The inner problems:}} For $i=1,2,\dots, \mathfrak{o}$,
\begin{equation}\label{cyl-inner}
\begin{cases}
\mu_i^2  \pp_t \phi_i
=
\Delta_{y^{[i]}} \phi_i
+
\mathcal{H}_{1,i} [ \mu_i,\xi^{[i]} ]
& \mbox{ \ for \ }
t\in \left(0,T\right) ,
y^{[i]}  \in B_n^+(0,2R) ,
\\
-\pp_{y_n^{[i]}} \phi_i
=
\frac{n}{n-2}
U^{\frac{2}{n-2}}(y^{[i]}) \phi_i
+
\mathcal{H}_{2,i} [ \psi, \mu_i, \xi^{[i]}]
& \mbox{ \ for \ }
t\in \left(0,T\right) ,
y^{[i]}  \in B_{n-1}(0,2R) \times \{ 0 \},
\end{cases}
\end{equation}
where
\begin{equation}\label{24Jan04-1}
	\begin{aligned}
		&
		\mathcal{H}_{1,i} [ \mu_i,\xi^{[i]} ]:=
  \eta\Big(\frac{y^{[i]}}{4R}\Big)
		\mu_i^{\frac{n+2}{2}}
		\mathcal{E}_1\left[U_{\mu_i,\xi^{[i]}} (x)\right]
		=
		\eta\Big(\frac{y^{[i]}}{4R}\Big)
		\Big(
		\dot{\mu}_i
		\mu_i  Z_n(y^{[i]})
		+
		\mu_i
		\dot{\xi}^{[i]} \cdot  \left( \nabla_{\tilde{y}^{[i]}} U \right)(\tilde{y}^{[i]},y_n^{[i]})
		\Big)
		,
		\\
		&
		\mathcal{H}_{2,i} [ \psi, \mu_i, \xi^{[i]} ]:=
		\frac{n}{n-2}
		\mu_i^{\frac{n}{2}-1}
		U^{\frac{2}{n-2}}(\tilde{y}^{[i]} , 0)
		\eta\Big(\frac{\tilde{y}^{[i]}}{4R} ,0\Big)
		\Big(
		\Theta_{l_i}\big( (\mu_i \tilde{y}^{[i]} + \xi^{[i]} -\tilde{q}^{[i]},0) ,t \big)
		+\psi\big( ( \mu_i \tilde{y}^{[i]} + \xi^{[i]},0),t\big)
		\Big).
	\end{aligned}
\end{equation}

{\textbf{The outer problem:}}
\begin{equation}\label{cyl-outer}
\partial_t \psi = \Delta_x \psi  +  \mathcal{G}_1 [ \bm{\phi},\bm{\mu},\bm{\xi} ]
\mbox{ \ in \ } \mathbb{R}_+^n \times\left(0,T\right) ,
\quad
-
\partial_{x_n} \psi
=  \mathcal{G}_2 [\psi, \bm{\phi},\bm{\mu},\bm{\xi} ]
\mbox{ \ on \ }
\partial\mathbb{R}_+^n \times\left(0,T\right)  ,
\end{equation}
where
\begin{equation}\label{G1G2}
\begin{aligned}
&
\mathcal{G}_1 [ \bm{\phi},\bm{\mu},\bm{\xi} ]
:=
\sum_{i=1}^{\mathfrak{o}}
\Big[
\mathcal{E}_1\left[U_{\mu_i,\xi^{[i]}} (x)\right] \left(\eta_{2\delta}(x^{[i]})
-
\eta_{R}(y^{[i]})
\right)
+
\mathcal{E}_{U,i}^{\rm{cut}}
+
\mathcal{E}_{\Theta,i}^{\rm{cut}}
+
\Lambda_{1,i}[\phi_i,\mu_i,\xi^{[i]}] + \Lambda_{2,i}[\phi_i,\mu_i,\xi^{[i]}]
\Big],
\\
&
\mathcal{G}_2 [\psi, \bm{\phi},\bm{\mu},\bm{\xi} ]
:=
\bigg[  \mathcal{N}\left[\psi,\bm{\phi},\bm{\mu},\bm{\xi}\right]
+
\sum_{i=1}^{\mathfrak{o}}
\Big\{
\left(U_{\mu_i,\xi^{[i]}} (x) \right)^{\frac{n}{n-2}}
\left(
\eta_{2\delta}^{\frac{n}{n-2}}(x^{[i]})
-
\eta_{2\delta}(x^{[i]})
\right)
\\
& +
\frac{n}{n-2}
\left(U_{\mu_i,\xi^{[i]}} (x) \right)^{\frac{2}{n-2}}
\Big[
\left(
\eta_{\delta}(x^{[i]})
-
\eta_{R}(y^{[i]})
\right)
\Theta_{l_i}(x^{[i]},t)
+
\left(
\eta_{2\delta}^{\frac{2}{n-2}}(x^{[i]})
-
\eta_{R}(y^{[i]})
\right)
\psi(x,t)
\\
&
+
\left(
\eta_{2\delta}^{\frac{2}{n-2}}(x^{[i]}) - 1
\right)  \eta_{R}(y^{[i]})
\mu_{i}^{-\frac{n-2}{2}} \phi_i(y^{[i]},t)
\Big]
\Big\} \bigg] \bigg|_{x_n=0} .
\end{aligned}
\end{equation}

In order to avoid the difficulty of the compactness argument due to the singularity of right-hand sides as $t\uparrow T$, we will solve \eqref{cyl-outer}, \eqref{cyl-inner} for $t\in (0, T_{\sigma_0})$ with $T_{\sigma_0}:=T-\sigma_0, \sigma_0\in (0, T)$ instead of $\left(0, T\right)$. Details will be given in Section \ref{section-gluing}. For applying Proposition \ref{prop-23Oct24-1} in $\mathbb{R}_+^n$ for solving
\eqref{cyl-inner}, we impose the cut-off functions $\eta(\frac{y^{[i]}}{4R}), \eta(\frac{\tilde{y}^{[i]}}{4R} ,0)$ to restrict the spatial growth of $\mathcal{H}_{1,i}, \mathcal{H}_{2,i} $ respectively.

\section{Proof of Proposition \ref{prop-23Oct24-1}}\label{app-blowup}

In this section, we will prove Proposition \ref{prop-23Oct24-1}, which will be used for solving the inner problems.
The following nondegeneracy lemma is prepared for the forthcoming blow-up argument.
\begin{lemma}\label{qd23Nov28-2-lem}
	For an integer $n\ge 3$, all bounded solutions of the equation $\Delta\phi=0$ in $\mathbb{R}^{n}_+$, $-\pp_{x_n} \phi  = \frac{n}{n-2}U^{\frac{2}{n-2}}\phi$ on $\partial \mathbb{R}^n_+$ are the linear combination of $Z_j(x)$, $j=1,2,\dots,n$, given in \eqref{qd2023Nov28-1}.
	
\end{lemma}

\begin{proof}
	
	Since $  (-\Delta_{\tilde{x}})^{\frac12} \phi(\tilde{x},0) - \frac{n}{n-2} \left(U^{\frac{2}{n-2}} \phi \right) (\tilde{x},0)   =0$ in $\mathbb{R}^{n-1}$ and \cite[Theorem 1.1]{frac-nondege2013}, we have $\phi(\tilde{x},0)=\sum\limits_{j=1}^n c_jZ_j(\tilde{x},0)$ for some  $c_j\in \mathbb{R}$, and thus $\phi(x)=\sum\limits_{j=1}^n c_jZ_j(x)$  in $\mathbb{R}^n_+$, where we used the fact that the bounded solution of the equation $\Delta u= 0$ in $\mathbb{R}_+^n$,
	$u =0$ on $\pp\mathbb{R}_+^n$ is zero.
\end{proof}

Recall the norms defined in \eqref{qd24Jan25-10}.
\begin{prop}\label{prop-23Oct04-1}
Given an integer $n\ge 5$, consider
	\begin{equation}\label{23Oct04-1}
		\pp_\tau\phi=\Delta \phi+ g
		\mbox{ \ in \ }  \mathbb{R}^{n}_+\times (\tau_0, \tau_1),
		\quad
		-\pp_{y_n} \phi =
		\frac{n}{n-2}U^{\frac{2}{n-2}}\phi+h
		\mbox{ \ on \ }\pp \mathbb{R}^n_+\times(\tau_0,\tau_1) ,
		\quad
		\phi(y,\tau_0)= 0
		\mbox{ \ in \ }  \mathbb{R}^{n}_+ ,
	\end{equation}
	where $\phi$ is given by
	\begin{equation}\label{23Oct28-1}
		\phi(y,\tau) =
		\int_{\tau_0}^{\tau}
		\int_{\mathbb{R}_+^n}
		G_n(y,\tau,z,s) g(z,s) dz ds
		+
		\int_{\tau_0}^{\tau}\int_{\mathbb{R}^{n-1}}
		G_n(y,\tau,(\tilde{z},0),s) 	\left[\frac{n}{n-2}\left(U^{\frac{2}{n-2}} \phi\right) (  ( \tilde{z},0),s )+ h(\tilde{z},s ) \right]  d\tilde{z} ds
	\end{equation}
with $G_n$ given in \eqref{qd23Dec22-1}.
	Suppose that $1\le \tau_0 <\tau_1\le \infty$, $\ell(\tau)$ satisfies $C_\ell^{-1} \tau^p\le \ell(\tau) \le C_\ell \tau^p$ with a constant $C_\ell\ge 1$,
	\begin{equation}\label{23Oct22-2}
		2<a<n-2,
		\quad
		a^{-1}<p\le \frac{1}{2},
		\quad
		\iota\in (0,\frac{1}{4}),
		\quad
		\sigma -p a  +2\iota n > 0  ,
		\quad
		\varsigma\in (0,1),
	\end{equation}    $\|g\|_{\sigma,2+a,\ell(\tau),\mathbb{R}^n_+,\tau_0,\tau_1} <\infty$, $\|h\|_{\sigma,1+a,\ell(\tau),\varsigma,\mathbb{R}^{n-1},\tau_0,\tau_1}<\infty$, and
	$g=g(y,\tau)$, $h=h(\tilde{y},\tau)$ satisfy the orthogonality conditions
	\begin{equation}\label{gh-ortho}
		\int_{\mathbb{R}^n_+} g(y,\tau)Z_j(y)dy
		+
		\int_{\mathbb{R}^{n-1}} h(\tilde{y},\tau) Z_j(\tilde{y},0) d \tilde{y} =0
		\quad
		\mbox{ \ for \ } \tau\in (\tau_0, \tau_1),
		\quad j=0,1,\dots,n
	\end{equation}
 with $Z_j$ given in \eqref{qd2023Nov28-1} and \eqref{Z0-eq}, then there exists a unique solution $\phi$ in $L^\infty\big(\mathbb{R}^n_+\times (\tau_0,\tilde \tau)\big)$ for all $\tilde{\tau} \in (\tau_0,\tau_1] \cap (\tau_0,\infty)$ and satisfying
	\begin{equation}\label{phi-ortho-23Oct}
		\int_{\mathbb{R}^n_+}\phi(y,\tau)Z_j(y)dy=0
		\quad \mbox{ \ for \ }  \tau \in (\tau_0,\tau_1) ,
		\quad
		j=0,1, \dots,n,
	\end{equation}
	and
	\begin{equation}\label{23Oct04-4}
		\begin{aligned}
			|\phi|\lesssim \ &
			\Big(
			\tau^\sigma \langle y \rangle^{-a}
			\1_{|y|\le \ell(\tau) }
			+
			\tau^{\sigma} \ell^{-a}(\tau) e^{- \iota \frac{|y|^2}{\tau}}
			\1_{|y|> \ell(\tau) }
			\Big) \Big( \|g\|_{\sigma,2+a,\ell(\tau),\mathbb{R}^n_+,\tau_0,\tau_1} + \|h\|_{\sigma,1+a,\ell(\tau),\mathbb{R}^{n-1},\tau_0,\tau_1} \Big),
			\\
			|\nabla\phi|  \lesssim \ &
			\Big(
			\tau^\sigma \langle y \rangle^{-1-a}
			\1_{|y|\le \ell(\tau) }
			+
			\tau^{\sigma} \ell^{-a}(\tau) |y|^{-1}
			\1_{\ell(\tau)<|y|\le \tau^{\frac{1}{2}} }
			+
			\tau^{\sigma-\frac{1}{2}} \ell^{-a}(\tau)  e^{-  \tilde{\iota}  \frac{|y|^2}{\tau}}
			\1_{ |y|> \tau^{\frac{1}{2}} }
			\Big)
			\\
			&\times
			\Big( \|g\|_{\sigma,2+a,\ell(\tau),\mathbb{R}^n_+,\tau_0,\tau_1} + \|h\|_{\sigma,1+a,\ell(\tau),\varsigma,\mathbb{R}^{n-1},\tau_0,\tau_1} \Big)
		\end{aligned}
	\end{equation}
	with a constant $\tilde{\iota} \in (0,\iota)$, where both ``$\lesssim$'' are independent of $\tau_0, \tau_1, g, h$.

\end{prop}

\medskip
We first use Proposition \ref{prop-23Oct04-1} to deduce Proposition \ref{prop-23Oct24-1}.
\begin{proof}[Proof of Proposition \ref{prop-23Oct24-1}]
	Set $\phi(y,\tau) =\phi_1(y,\tau)+ c(\tau) \tilde{Z}_0(y)$. It suffices to consider
	\begin{equation*}
		\pp_\tau\phi_1 =\Delta \phi_1 + g_1
		\mbox{ \ in \ }  \mathbb{R}^{n}_+\times (\tau_0, \tau_1),
		\quad
		-\pp_{y_n} \phi_1 =
		\frac{n}{n-2} U^{\frac{2}{n-2}}\phi_1+h_1
		\mbox{ \ on \ }\pp \mathbb{R}^n_+\times(\tau_0,\tau_1) ,
		\quad
		\phi_1(y,\tau_0)= 0
		\mbox{ \ in \ }  \mathbb{R}^{n}_+ ,
	\end{equation*}
	where
	\begin{equation*}
		\begin{aligned}
			&
			g_1=g_1(y,\tau):= g + c(\tau) \Delta \tilde{Z}_0(y) -c'(\tau) \tilde{Z}_0(y),
			\\
			&
			h_1=h_1(\tilde{y},\tau):=
			h+
			c(\tau)
			\frac{n}{n-2} \left(U^{\frac{2}{n-2}}  \tilde{Z}_0 \right)(\tilde{y},0)
			+
			c(\tau)
			\left(\pp_{y_n} \tilde{Z}_0\right)(\tilde{y},0) .
		\end{aligned}
	\end{equation*}
	For $j=0,1,\dots,n$, by \eqref{Zj-eq}, \eqref{Z0-eq}, \eqref{gh-ortho-Oct20}, orthogonality conditions in \eqref{til-Z0-Oct20}, we have
	\begin{equation*}
		\begin{aligned}
			&
			\int_{\mathbb{R}^n_+} g_1(y,\tau)Z_j(y)dy
			+
			\int_{\mathbb{R}^{n-1}} h_1(\tilde{y},\tau) Z_j(\tilde{y},0) d \tilde{y}
			\\
			= \ &
			\int_{\mathbb{R}_+^n} g(y,\tau)Z_j(y) dy
			+
			\int_{\mathbb{R}^{n-1}} h(\tilde{y},\tau) Z_j(\tilde{y},0) d\tilde{y}
			+
			c(\tau)
			\int_{\mathbb{R}^{n-1}}
			\tilde{Z}_0(\tilde{y},0)
			\left(
			\pp_{y_n} Z_j
			+
			\frac{n}{n-2} U^{\frac{2}{n-2}} Z_j
			\right)(\tilde{y},0) d\tilde{y}
			\\
			& +
			c(\tau) \int_{\mathbb{R}_+^n} \tilde{Z}_0 \Delta Z_j dy
			-
			c'(\tau)
			\int_{\mathbb{R}_+^n}  \tilde{Z}_0  Z_j dy
			\\
			= \ &
			\delta_{j0}
			\left(
			\int_{\mathbb{R}_+^n} g(y,\tau)Z_0(y) dy
			+
			\int_{\mathbb{R}^{n-1}} h(\tilde{y},\tau) Z_0(\tilde{y},0) d\tilde{y}
			-\lambda_0
			c(\tau) \int_{\mathbb{R}_+^n} \tilde{Z}_0  Z_0 dy
			-
			c'(\tau)
			\int_{\mathbb{R}_+^n}  \tilde{Z}_0  Z_0 dy
			\right)
		\end{aligned}
	\end{equation*}
with $\lambda_0<0$.	Since $\int_{\mathbb{R}_+^n}  \tilde{Z}_0  Z_0 dy \ne 0$, we can take
	\begin{equation*}
		c(\tau) =
		-\bigg( \int_{\mathbb{R}_+^n}  \tilde{Z}_0  Z_0 dy \bigg)^{-1} e^{-\lambda_0 \tau}
		\int_\tau^{\tau_1}
		e^{\lambda_0 s}
		\bigg( \int_{\mathbb{R}_+^n} g(y,s)Z_0(y) dy
		+
		\int_{\mathbb{R}^{n-1}} h(\tilde{y},s) Z_0(\tilde{y},0) d\tilde{y} \bigg)
		ds ,
	\end{equation*}
	which makes $g_1, h_1$ satisfy \eqref{gh-ortho}.
	Since $Z_0(y)$ decays exponentially, by \eqref{23Oct07-1}, we have
	\begin{equation*}
		|c(\tau)| + |c'(\tau)| \lesssim \tau^\sigma \left( \|g\|_{\sigma,2+a,\ell(\tau),\mathbb{R}^n_+,\tau_0,\tau_1} + \|h\|_{\sigma,1+a,\ell(\tau),\mathbb{R}^{n-1},\tau_0,\tau_1}\right) .
	\end{equation*}
	Combining $\tilde{Z}_0(y) \in C^2( \mathbb{R}^{n}_+ )\cap C^{1,\varsigma}(\overline{\mathbb{R}^{n}_+ })$, $\tilde{Z}_0(y)=0$ for
	$|y|\ge C_0$, for $C_1= C_\ell \max\left\{ C_0,1\right\}$, we have
	\begin{equation*}
 \begin{aligned}
&
		\|g_1\|_{\sigma,2+a,C_1\ell(\tau),\mathbb{R}^n_+,\tau_0,\tau_1}
  +
  \|h_1\|_{\sigma,1+a,C_1\ell(\tau),\mathbb{R}^{n-1},\tau_0,\tau_1}
		\lesssim
 \|g\|_{\sigma,2+a,\ell(\tau),\mathbb{R}^n_+,\tau_0,\tau_1} + \|h\|_{\sigma,1+a,\ell(\tau),\mathbb{R}^{n-1},\tau_0,\tau_1},
  \\
&  \|h_1\|_{\sigma,1+a,C_1\ell(\tau),\varsigma,\mathbb{R}^{n-1},\tau_0,\tau_1}
		\lesssim
\|g\|_{\sigma,2+a,\ell(\tau),\mathbb{R}^n_+,\tau_0,\tau_1}
+
\|h\|_{\sigma,1+a,\ell(\tau),\varsigma,\mathbb{R}^{n-1},\tau_0,\tau_1}.
  \end{aligned}
	\end{equation*}
Thus, we can use Proposition \ref{prop-23Oct04-1} to find $\phi_1$ satisfying \eqref{phi-ortho-23Oct}, \eqref{23Oct04-4}. Set $C_{\phi} = c(\tau_0)$. Combining \eqref{til-Z0-Oct20}, we complete the proof.		
\end{proof}

\begin{proof}[Proof of Proposition \ref{prop-23Oct04-1}]

By the contraction mapping theorem, it is easy to get the existence and uniqueness for \eqref{23Oct28-1} in $L^{\infty}(\mathbb{R}^n \times (\tau_0,\tilde{\tau}))$ for all $\tilde{\tau} \in (\tau_0, \tau_1] \cap (\tau_0,\infty)$ and \eqref{23Oct04-1} holds in the weak sense. By the parabolic regularity theory, $\phi\in C(\overline{\mathbb{R}_+^n}\times [\tau_0,\tilde{\tau}) )$.
	
We will use the comparison theorem for parabolic equations with the oblique derivative.
See
\cite[p.122, 13.5 Theorem]{Daners-Koch1992} for instance.
Set $\tilde{\phi}(y,\tau) =  e^{-\kappa \frac{|\tilde{y}|^2+\left(y_n+1\right)^2}{\tau}} $.
Using the supports of $g,h$, we have
	\begin{equation*}
		\begin{aligned}
			&
			\left( \pp_\tau -\Delta \right) \tilde{\phi} - g
			=
			\tau^{-2}\left\{
			2\kappa n \tau + \kappa\left(1-4\kappa\right)
			\left[
			|\tilde{y}|^2+\left(y_n+1\right)^2
			\right]
			\right\}  \tilde{\phi}
			\quad
			\mbox{ \ in \ } |y|> \ell(\tau),
			\\
			&
			\left(-\pp_{y_n} - \frac{n}{n-2}U^{\frac{2}{n-2}} \right) \tilde{\phi} -h
			=
			\left(
			2\kappa \tau^{-1}  - \frac{n}{n-2}U^{\frac{2}{n-2}}\right)
			\tilde{\phi}
			\quad
			\mbox{ \ on \ } |\tilde{y}|> \ell(\tau) , y_n=0 .
		\end{aligned}
	\end{equation*}
	$\kappa\in (0,\frac{1}{4}]$ is a necessary condition to make $C_1 \tilde{\phi}$ be a barrier function. We take $\kappa =\frac{1}{4}$.
	For any $\tilde{\tau} \in (\tau_0,\tau_1] \cap (\tau_0,\infty)$, $C_1  \tilde{\phi}$ is a barrier function in   $\left\{
	(y,\tau) \ | \ |y|\ge C_2 \tau^{\max\left\{ \frac{1}{2}, p \right\}}, \ \tau\in (\tau_0,\tilde{\tau})
	\right\}$ with sufficiently large constants $C_1$ depending on $\tilde{\tau}$ and $C_2>C_\ell$. By $\|g\|_{\sigma,2+a,\ell(\tau),\mathbb{R}^n_+,\tau_0,\tau_1} <\infty , \|h\|_{\sigma,1+a,\ell(\tau),\varsigma,\mathbb{R}^{n-1},\tau_0,\tau_1}<\infty$, and \eqref{qd24Jan08-5} in Lemma \ref{24Jan08-1-lem}, $\nabla \phi(\cdot,\tau)$ has exponential decay in space.

	For $j=0,1,\dots,n$, we test \eqref{23Oct04-1} with $Z_j(y)\eta(M^{-1} y)$ and integrate by parts,
	\begin{equation*}
		\begin{aligned}
			&
			\int_{\mathbb{R}_+^n} \phi(y,\tau) Z_j(y) \eta(M^{-1} y) dy
			-
			\int_{\mathbb{R}_+^n} \phi(y,\tau_0) Z_j(y) \eta(M^{-1} y) dy
			\\
			= \ &
			\int_{\tau_0}^\tau \int_{\mathbb{R}^{n-1}}
			\left(-\pp_{y_n} \phi Z_j \eta(M^{-1} \cdot) \right) ((\tilde{y},0),s)  d\tilde{y} ds
			+
			\int_{\tau_0}^\tau
			\int_{\mathbb{R}^{n-1}}
			\left[\phi \pp_{y_n}\left(Z_j \eta(M^{-1} \cdot)\right)\right]
			((\tilde{y},0),s) d\tilde{y} ds
			\\
			& +
			\int_{\tau_0}^\tau
			\int_{\mathbb{R}_+^n}
			\phi \Delta \left(Z_j (y) \eta(M^{-1} y)\right) dy ds
			+
			\int_{\tau_0}^\tau
			\int_{\mathbb{R}_+^n} g Z_j (y) \eta(M^{-1} y) dy ds
			\\
			= \ &
			\int_{\tau_0}^\tau \int_{\mathbb{R}^{n-1}}
			\left[
			\left( \frac{n}{n-2}U^{\frac{2}{n-2}}\phi+h \right)
			Z_j \eta(M^{-1} \cdot) \right] ((\tilde{y},0),s)  d\tilde{y} ds
			\\
			&
			+
			\int_{\tau_0}^\tau
			\int_{\mathbb{R}^{n-1}}
			\left[ \phi Z_j \pp_{y_n}\left(\eta(M^{-1} \cdot)\right)
			+
			\phi
			\eta(M^{-1} \cdot) \pp_{y_n} Z_j
			\right]
			((\tilde{y},0),s) d\tilde{y} ds
			\\
			& +
			\int_{\tau_0}^\tau
			\int_{\mathbb{R}_+^n}
			\phi
			\left[
			\eta(M^{-1} y) \Delta Z_j (y)
			+
			2\nabla Z_j (y) \cdot
			\nabla \left( \eta(M^{-1} y)\right)
			+
			Z_j (y) \Delta \left( \eta(M^{-1} y)\right)
			\right]
			dy ds
			\\
			&
			+
			\int_{\tau_0}^\tau
			\int_{\mathbb{R}_+^n} g Z_j (y) \eta(M^{-1} y) dy ds  .
		\end{aligned}
	\end{equation*}
	Taking $M\to \infty$ gives
	\begin{equation*}
		\begin{aligned}
			&
			\int_{\mathbb{R}_+^n} \phi(y,\tau) Z_j(y)  dy
			-
			\int_{\mathbb{R}_+^n} \phi(y,\tau_0) Z_j(y)  dy
			\\
			= \ &
			\int_{\tau_0}^\tau \int_{\mathbb{R}^{n-1}}
			\left( \frac{n}{n-2}U^{\frac{2}{n-2}}\phi Z_j +h Z_j + \phi  \pp_{y_n} Z_j  \right)  ((\tilde{y},0),s)  d\tilde{y} ds
			+
			\int_{\tau_0}^\tau
			\int_{\mathbb{R}_+^n}
			\phi   \Delta Z_j (y)
			dy ds
			+
			\int_{\tau_0}^\tau
			\int_{\mathbb{R}_+^n} g Z_j (y)  dy ds  .
		\end{aligned}
	\end{equation*}
	Using  \eqref{Zj-eq}, \eqref{Z0-eq}, \eqref{gh-ortho}, and $\int_{\mathbb{R}_+^n} \phi(y,\tau_0) Z_j(y) dy=0$, we have
	\begin{equation*}
		\int_{\mathbb{R}_+^n} \phi(y,\tau) Z_j(y)  dy
		=
		-
		\delta_{j0}
		\lambda_0
		\int_{\tau_0}^\tau \int_{\mathbb{R}_+^n} \phi Z_0 dy ,
	\end{equation*}
	which implies \eqref{phi-ortho-23Oct}.
	Define
	\begin{equation*}
		\|f\|_{\sigma,a,p,\tau_0,\tau_1}^{\#}
		:=
		\inf \left\{
		C \ | \
		|f(y,\tau)| \le C  w(y,\tau)
		\mbox{ \ for \ } y\in \mathbb{R}_+^n, \tau_0 < \tau <\tau_1
		\right\}  ,
	\end{equation*}
	where
	\begin{equation*}
		w(y,\tau):=
		\tau^\sigma \langle y \rangle^{-a}
		\1_{|y|\le \tau^{p} }
		+
		\tau^{\sigma-p a} e^{-\iota \frac{|y|^2}{\tau}}
		\1_{|y|> \tau^{p}} ,
		\quad \iota\in (0,\frac{1}{4}] .
	\end{equation*}
	The weight $w(y,\tau)$ is partially determined by the forthcoming estimate \eqref{23Oct15-1}, \eqref{23Oct16-1}. By the barrier function $C_1  \tilde{\phi}$, for $1\le \tau_0<\tau_1<\infty$, we have $e^{- \frac{|\tilde{y}|^2+\left(y_n+1\right)^2}{4 \tau}} \le C_3 w(y,\tau)$ for $y\in \mathbb{R}_+^n$, $\tau_0 < \tau <\tau_1$ with a constant $C_3>0$ depending on $\tau_1$, which implies
	\begin{equation*}
		\|\phi\|_{\sigma,a,p,\tau_0,\tau_1}^{\#} <\infty .
	\end{equation*}

	Claim: For all $1\le \tau_0<\tau_1<\infty$,  there exists a constant $C_5>0$ independent of $\tau_0,\tau_1, g, h$ such that
	\begin{equation}\label{23Oct14-2}
		\|\phi\|_{\sigma,a,p,\tau_0,\tau_1}^{\#} \le C_5 \left(  \|g\|_{\sigma,2+a,\ell(\tau),\mathbb{R}_+^n,\tau_0,\tau_1 }
		+
		\|h\|_{\sigma,1+a,\ell(\tau),\mathbb{R}^{n-1},\tau_0,\tau_1 } \right) .
	\end{equation}
	Indeed, for $p\in [0,\frac{1}{2}]$, one taking $\tau_1\rightarrow \infty$, \eqref{23Oct14-2} deduces the estimate of $\phi$ in \eqref{23Oct04-4} for the case $\tau_1=\infty$. Combining $\|g\|_{\sigma,2+a,\ell(\tau),\mathbb{R}^n_+,\tau_0,\tau_1} <\infty, \|h\|_{\sigma,1+a,\ell(\tau),\varsigma,\mathbb{R}^{n-1},\tau_0,\tau_1}<\infty$, and \eqref{qd24Jan08-5} in Lemma \ref{24Jan08-1-lem}, we get the desired bound of $\nabla \phi$ in \eqref{23Oct04-4}.

	We prove \eqref{23Oct14-2} by contradiction argument. Suppose that there exists a sequence $\left(\phi_k, g_k,h_k,\tau_{0k},\tau_{1k}\right)_k$ satisfying
	\begin{equation}\label{23Oct14-1}
		\begin{cases}
			\pp_\tau\phi_k=\Delta \phi_k+g_k
			\mbox{ \ in \ }  \mathbb{R}^{n}_+\times(\tau_{0k},\tau_{1k}),
			\quad
			-\pp_{y_n}\phi_k =\frac{n}{n-2}U^{\frac{2}{n-2}}\phi_k+h_k  \mbox{ \ on \ }\pp \mathbb{R}^n_+\times(\tau_{0k},\tau_{1k}),
			\\
			\phi_k(y,\tau_{0k})=0
			\mbox{ \ in \ }  \mathbb{R}^{n}_+
			,
			\quad
			\int_{\mathbb{R}^n_+}\phi_k(y,\tau) Z_j(y) dy=0 \mbox{ \ for all \ } \tau\in (\tau_{0k},\tau_{1k}),
			\quad
			j=0,1,\dots,n,
		\end{cases}
	\end{equation}
	and
	\begin{equation}\label{23Oct14-3}
		\|\phi_k\|_{\sigma,a,p,\tau_{0k},\tau_{1k}}^{\#}=1,
		\quad \|g_k\|_{\sigma,2+a,\ell(\tau),\mathbb{R}_+^n,\tau_{0k},\tau_{1k}} =o_k(1),
		\quad
		\|h_k\|_{\sigma,1+a,\ell(\tau),\mathbb{R}^{n-1},\tau_{0k},\tau_{1k}} =o_k(1) ,
	\end{equation}
	where $o_k(1) \to 0$
	as  $k\to \infty$. Therein,
	$\tau_{1k} \to \infty$ otherwise $\|\phi_k\|_{\sigma,a,p,\tau_{0k},\tau_{1k}}^{\#} \to 0$.
	Thus, there exists a sequence $\left( y_k,\tau_{2k}\right)_k$, $y_k\in \overline{\mathbb{R}_+^n}$, $\tau_{0k} < \tau_{2k} < \tau_{1k}$ such that
	\begin{equation}\label{23Oct14-5}
		\left( w \left(y_k,\tau_{2k}\right) \right)^{-1}
		\left|\phi_k(y_k,\tau_{2k})\right| \ge  \delta_1>0
	\end{equation}
with a small constant $\delta_1>0$ independent of $k$, where $\tau_{2k} \rightarrow \infty$ otherwise \eqref{23Oct14-5} fails.
	
	For $(y,\tau) \in \mathbb{R}^n \times [\tau_{0k},\tau_{1k})$, denote
	\begin{equation*}
			\mathcal{T}_1\left[f_1\right]:=
			\int_{\tau_{0k}}^{\tau}
			\int_{\mathbb{R}_+^n}
			G_n(y,\tau,z,s) f_1(z,s) dz ds,
   \quad
			\mathcal{T}_2\left[f_2\right]:=
			\int_{\tau_{0k}}^{\tau}\int_{\mathbb{R}^{n-1}}
			G_n(y,\tau,(\tilde{z},0),s)  f_2(\tilde{z},s )   d\tilde{z} ds
	\end{equation*}
for some admissible functions $f_1, f_2$. Then
 \begin{equation*}
		\phi_k(y,\tau) =   \mathcal{T}_1\left[g_k\right] + \mathcal{T}_2\left[
		\frac{n}{n-2}\left(U^{\frac{2}{n-2}} \phi_k\right) ( (\tilde{y},0),\tau) + h_k \right].
	\end{equation*}
	
	Hereafter, all general constants, like $C_i$, $C(*,*,\dots)$, are independent of $k$.
	
	By \eqref{23Oct14-3}, Lemma \ref{2023Oct01-1-lem},
	\begin{equation}\label{23Oct15-1}
		\left| \mathcal{T}_1\left[ g_k \right] \right|
		\le C_6
		o_k(1)
		\mathcal{T}_1\left[ \tau^{\sigma} \langle y \rangle^{-2-a} \1_{|y|\le \ell(\tau) } \right]
		\le C_6
		o_k(1)
		\tau^\sigma
		\left(
		\langle y \rangle^{-a} \1_{|y|\le \tau^p}
		+
		\tau^{-p a} e^{-\frac{|y|^2}{4\tau}}
		\1_{|y|>\tau^p }
		\right)
	\end{equation}
	under the assumption
	\begin{equation}\label{23Oct15-2}
		p\ge 0,
		\quad
		0<a<n-2,
		\quad
		\begin{cases}
			p \le \frac{1}{2},
			\sigma-p a+\frac{n}{2} \ge 0 ,
			&
			\mbox{ \ if \ }
			\sigma+p (n-2-a) \ne -1
			\\
			p <\frac{1}{2}
			\left(\Leftrightarrow
			\sigma-p a+\frac{n}{2} > 0 \right),
			&
			\mbox{ \ if \ }
			\sigma+p (n-2-a)=-1 ,
		\end{cases}
	\end{equation}
	where $C_6$ varies from line to line.
	
	By \eqref{23Oct14-3}, Lemma \ref{bd-ann-23Lem},
	\begin{equation}\label{23Oct16-1}
		\left|\mathcal{T}_2\left[h_k \right] \right|
		\le C_6
		o_k(1) \mathcal{T}_2\left[ \tau^{\sigma} \langle \tilde{y} \rangle^{-1-a} \1_{|\tilde{y}|\le \ell(\tau) } \right]
		\le
		C_6
		o_k(1) \tau^\sigma
		\left(
		\langle y \rangle^{-a} \1_{|y|\le \tau^p }
		+
		\tau^{-p a} e^{-\frac{|y|^2}{4\tau}}
		\1_{|y|>\tau^p }
		\right)
	\end{equation}
	under the assumption \eqref{23Oct15-2}.
	
	By \eqref{23Oct14-3}, Lemma \ref{bd-ann-23Lem}, Lemma \ref{lem:far-23Oct16},
	\begin{equation*}
		\begin{aligned}
			&
			\left| \mathcal{T}_2\left[
			\left(U^{\frac{2}{n-2}} \phi_k\right)( (\tilde{y},0),\tau) \right] \right|
			\le C_6
			\mathcal{T}_2\left[
			\tau^\sigma \langle \tilde{y} \rangle^{-a-2}
			\1_{|\tilde{y}|\le \tau^p }
			+
			\tau^{\sigma-p a}
			\langle \tilde{y} \rangle^{-2}
			e^{- \iota \frac{ |\tilde{y}|^2}{\tau}}
			\1_{|\tilde{y}|> \tau^p} \right]
			\\
			\le \ &
			C_6
			\mathcal{T}_2\left[
			\tau^\sigma \langle \tilde{y} \rangle^{-a-1-\epsilon}
			\1_{|\tilde{y}|\le \tau^p }
			+
			\tau^{\sigma-p a}
			|\tilde{y}|^{-1-\epsilon}
			\1_{ \tau^p < |\tilde{y}| \le \tau^{\frac{1}{2}} }
			+
			\tau^{\sigma-p a}
			|\tilde{y}|^{-1-\epsilon}
			e^{- \iota \frac{|\tilde{y}|^2}{\tau}}
			\1_{|\tilde{y}|> \tau^{\frac{1}{2}}}
			\right]
			\\
			\le \ &
			C_6
			\bigg[
			\tau^{\sigma}
			\left( \langle y \rangle^{-a-\epsilon}
			\1_{|y|\le \tau^p}
			+
			\tau^{-p (a+\epsilon)}
			e^{-\frac{|y|^2}{4\tau}}
			\1_{|y|> \tau^p}
			\right)
			\\
			& +
			\tau^{\sigma-p a}
			\left(
			\tau^{-p \epsilon}\1_{|y|\le \tau^p}
			+
			|y|^{-\epsilon}
			\1_{\tau^p<|y|\le \tau^{\frac{1}{2}} }
			+
			\tau^{-\frac{\epsilon}{2}}
			e^{-\frac{|y|^2}{4\tau}}
			\1_{|y|>\tau^{\frac{1}{2}} }
			\right)
			+
			\tau^{\sigma-p a -\frac{\epsilon}{2}} e^{-\iota \frac{|y|^2}{\tau}}
			\bigg]
		\end{aligned}
	\end{equation*}
	for any constant $\epsilon \in \left(0, \min\left\{n-2-a, 1\right\} \right)$, under the assumption
	\begin{equation*}
		\begin{aligned}
			&
			0<a<n-2,
			\quad
			p \ge 0,
			\quad
			\begin{cases}
				p  \le \frac{1}{2},
				\sigma-p (a+\epsilon)+\frac{n}{2} \ge 0 ,
				&
				\mbox{ \ if \ }
				\sigma+p (n-2-a-\epsilon)\ne -1
				\\
				p <\frac{1}{2}
				,
				&
				\mbox{ \ if \ }
				\sigma+p (n-2-a-\epsilon)=-1  ,
			\end{cases}
			\\
			& \iota\in (0,\frac{1}{4}),
			\quad
			\sigma -p a -\frac{\epsilon}{2} +2 \iota n \ge 0  .
		\end{aligned}
	\end{equation*}
	Thus,
	\begin{equation}\label{23Oct16-2}
		\left| \mathcal{T}_2\left[
		\left(U^{\frac{2}{n-2}} \phi_k\right)( (\tilde{y},0),\tau) \right] \right|
		\le C_6 \tau^{\sigma}
		\left( \langle y \rangle^{-a-\epsilon}
		\1_{|y|\le \tau^p}
		+
		\tau^{-p (a+\epsilon)}
		e^{-\iota \frac{|y|^2}{\tau}}
		\1_{|y|> \tau^p}
		\right).
	\end{equation}
	\eqref{23Oct15-1}, \eqref{23Oct16-1}, \eqref{23Oct16-2} imply
	\begin{equation*}
		\left| \phi_k(y,\tau)
		\right|
		\le C_6 w(y,\tau)
		\left(
		o_k(1)
		+
		\langle y \rangle^{-\epsilon}
		\1_{|y|\le \tau^p}
		+
		\tau^{-p \epsilon }
		\1_{|y|> \tau^p}
		\right).
	\end{equation*}
	Using $p >0$, $\tau_{2k} \rightarrow \infty$, and \eqref{23Oct14-5},  we conclude that there exists a constant $C_7>0$ such that $|y_k|\le C_7$ for all $k$.
	By \eqref{23Oct14-5}, for $k$ large,
	there exists a positive constant $C_8$ such that
	\begin{equation}\label{23Oct22-6}
		\tau_{2k}^{-\sigma} \left|\phi_k(y_k,\tau_{2k}) \right|\geq C_8 >0 .
	\end{equation}
	Set
	\begin{equation*}
		\tilde{\phi}_k(y,t) = \tau_{2k}^{-\sigma} \phi_k(y,\tau_{2k} + t),
		\quad
		\tilde{g}_k(y,t) = \tau_{2k}^{-\sigma} g_k(y,\tau_{2k} + t),
		\quad
		\tilde{h}_k(\tilde{y},t) = \tau_{2k}^{-\sigma} h_k(\tilde{y},\tau_{2k} + t)
		.
	\end{equation*}
	Then,
	\begin{equation}\label{23Oct22-7}
		|\tilde{\phi}_k(y_k,0)|\ge C_8>0 .
	\end{equation}
	
	The following argument is in the same spirit as the proof of \cite[Lemma 7.7]{2020HMF}.
	
	{\bf Case 1:} There exists a subsequence, (without loss of generality, we still use $k$ as the serial number), such that  $\tau_{2k} \ge 9\tau_{0k}$.
	By \eqref{23Oct14-1}, we have
	\begin{equation}\label{23Oct17-1}
		\begin{cases}
			\pp_t \tilde{\phi}_k=\Delta \tilde{\phi}_k+\tilde{g}_k
			\quad
			\mbox{ \ in \ }  \mathbb{R}^{n}_+\times (\tau_{0k}- \frac{\tau_{2k}}{2},0] ,
			\quad
			-\pp_{y_n}\tilde{\phi}_k =\frac{n}{n-2}U^{\frac{2}{n-2}}\tilde{\phi}_k+\tilde{h}_k
			\quad	\mbox{ \ on \ }\pp \mathbb{R}^n_+\times (\tau_{0k}- \frac{\tau_{2k}}{2},0],
			\\
			\int_{\mathbb{R}^n_+}\tilde{\phi}_k(y,t) Z_j(y) dy=0
			\quad
			\mbox{ \ for all \ } t\in  (\tau_{0k}- \frac{\tau_{2k}}{2},0],
			\quad
			j=0,1,\dots,n.
		\end{cases}
	\end{equation}
	And \eqref{23Oct14-3} implies
	\begin{equation}\label{23Oct22-1}
		\begin{aligned}
			&
			|\tilde{\phi}_k(y,t) |\le C(\sigma,a,p,\iota)
			\left(
			\langle y\rangle^{-a}
			\1_{|y| <  \tau_{2k}^p }
			+
			\tau_{2k}^{-p a}
			e^{-\iota \frac{|y|^2}{\tau_{2k}}}
			\1_{|y|\ge   \tau_{2k}^p }
			\right),
			\\
			&
			|\tilde{g}_k (y,t) |\le o_k(1) C(\sigma)
			\langle y\rangle^{-2-a} \1_{|y|\le C_\ell \tau_{2k}^p } ,
			\quad
			|\tilde{h}_k (\tilde{y},t) |\le o_k(1) C(\sigma)
			\langle \tilde{y}\rangle^{-1-a} \1_{|\tilde{y}|\le C_\ell \tau_{2k}^p }
		\end{aligned}
	\end{equation}
	holding in $t \in (\tau_{0k}- \frac{\tau_{2k}}{2},0]$, $y\in \mathbb{R}_+^n$, $\tilde{y}\in \mathbb{R}^{n-1}$. Since $\tau_{2k} \ge 9\tau_{0k}$, $\tau_{2k}\to \infty$ implies $\tau_{0k}- \frac{\tau_{2k}}{2} \to -\infty$.
	
	{\bf Case 2:} $\tau_{0k}<\tau_{2k} < 9\tau_{0k}$ holds for all $k$.
	Similarly, by \eqref{23Oct14-1}, we have
	\begin{equation}\label{23Oct22-4}
		\begin{cases}
			\pp_t \tilde{\phi}_k=\Delta \tilde{\phi}_k+\tilde{g}_k
			\quad
			\mbox{ \ in \ }  \mathbb{R}^{n}_+\times (\tau_{0k}- \tau_{2k},0] ,
			\quad
			-\pp_{y_n}\tilde{\phi}_k =\frac{n}{n-2}U^{\frac{2}{n-2}}\tilde{\phi}_k+\tilde{h}_k
			\quad	\mbox{ \ on \ }\pp \mathbb{R}^n_+\times (\tau_{0k}- \tau_{2k},0],
			\\
			\tilde{\phi}_k(y,\tau_{0k}- \tau_{2k})=0
			\mbox{ \ in \ }
			\mathbb{R}_+^n,
			\quad
			\int_{\mathbb{R}^n_+}\tilde{\phi}_k(y,t) Z_j(y) dy=0
			\quad \mbox{ \ for all \ } t\in  (\tau_{0k}- \tau_{2k},0],
			\quad
			j=0,1,\dots,n
		\end{cases}
	\end{equation}
	and \eqref{23Oct14-3} makes \eqref{23Oct22-1} hold for all $t \in (\tau_{0k}- \tau_{2k},0]$, $y\in \mathbb{R}_+^n$, $\tilde{y}\in \mathbb{R}^{n-1}$. \eqref{23Oct22-7} implies $\tau_{0k}- \tau_{2k}\to -\infty$.
	
	We will handle the above two cases in a unified way. Denote $t_k:=
	\begin{cases}
		\tau_{0k}- \frac{\tau_{2k}}{2},
		&
		\mbox{ \ for Case 1}
		\\
		\tau_{0k}- \tau_{2k},
		&
		\mbox{ \ for Case 2}
	\end{cases}$.
	Since $t_k \to -\infty$,
	by \eqref{23Oct22-1} and the parabolic regularity theorem (see \cite[p.2418]{Lieberman2012CPAA}), up to a subsequence, we have
	\begin{equation}\label{23Oct23-1}
		\tilde{\phi}_k\to\tilde{\phi}
		\mbox{ \ in \ } C_{\mathrm{loc}}^{\alpha,\frac{\alpha}{2}}\left( \overline{\mathbb{R}_+^n} \times (-\infty, 0] \right)
		\mbox{ \ with a constant \ }
		\alpha\in (0,1).
	\end{equation}
	Combining \eqref{23Oct22-7} with $|y_k|\le C_7$, \eqref{23Oct22-1}, we have
	\begin{equation}\label{23Oct17-5}
		\tilde{\phi}\not\equiv 0 ,
		\mbox{ \ and \ }
		|\tilde{\phi}(y,t) |\le
		C(\sigma,a,p,\iota)
		\langle y\rangle^{-a}
		\mbox{ \ for \ }(y,t)\in  \mathbb{R}_+^n\times (-\infty,0]    .
	\end{equation}

	Given $t\in (-\infty,0]$,
	for any $\epsilon_1>0$,  taking a large constant $R_1$, then for all $j=0,1,\dots,n$, when $k$ is sufficiently large, we have
	\begin{equation*}
		\begin{aligned}
			&
			\bigg| \int_{\mathbb{R}_+^n\cap \{|y|\ge R_1 \} } \tilde{\phi}_k(y,t) Z_j(y) dy \bigg|
			\le
			C(\sigma,a,p,\iota,n)
			\int_{|y|\ge R_1}
			\left(
			|y|^{2-n-a}
			\1_{|y| <  \tau_{2k}^p }
			+
			|y|^{2-n}
			\tau_{2k}^{-p a}
			e^{-\iota \frac{|y|^2}{\tau_{2k}}}
			\1_{|y|\ge   \tau_{2k}^p }
			\right)  dy
			\\
			\le \ &
			C(\sigma,a,p,\iota,n)
			\bigg(
			R_1^{2-a}
			+
			\tau_{2k}^{1-p a}
			\int_{\tau_{2k}^{2p-1}}^\infty e^{-\iota z} dz
			\bigg) <\epsilon_1 ,
		\end{aligned}
	\end{equation*}
	where in the last step, we require  the assumption
	\begin{equation*}
		a>2,
		\quad
		pa>1.
	\end{equation*}
	Similarly, for $a>2$, we have
	$
	\big| \int_{\mathbb{R}_+^n\cap \{|y|\ge R_1 \}} \tilde{\phi}(y,t) Z_j(y) dy \big|<\epsilon_1 $.
	By \eqref{23Oct23-1}, we have
	\begin{equation*}
		\lim\limits_{k\to \infty}
		\int_{\mathbb{R}_+^n\cap \{|y|< R_1 \}} \tilde{\phi}_k(y,t) Z_j(y) dy
		=
		\int_{\mathbb{R}_+^n\cap \{|y|< R_1 \}} \tilde{\phi}(y,t) Z_j(y) dy .
	\end{equation*}
	Thus, the orthogonality conditions in \eqref{23Oct17-1}, \eqref{23Oct22-4} yields
	\begin{equation*}
		\int_{\mathbb{R}_+^n}\tilde{\phi}(y, t)  Z_j(y)dy = 0 \quad \text{ for all } t \in (-\infty, 0],
		\quad j= 0, 1,\dots, n .
	\end{equation*}

	By \eqref{23Oct17-1}, \eqref{23Oct22-4}, we have
	\begin{equation*}
		\begin{aligned}
			& \tilde{\phi}_k(y,t) =
			\int_{\mathbb{R}_+^n}
			G_n(y,t,z,t_k) \tilde{\phi}_k(z,t_k)  dz
			+
			\int_{t_k}^{t}
			\int_{\mathbb{R}_+^n}
			G_n(y,t,z,s) \tilde{g}_k(z,s) dz ds
			\\
			&
			+
			\int_{t_k}^{t}\int_{\mathbb{R}^{n-1}}
			G_n(y,t,(\tilde{z},0),s) 	\left[ \frac{n}{n-2} \left(U^{\frac{2}{n-2}}\tilde{\phi}_k\right)(  ( \tilde{z},0),s )+\tilde{h}_k(\tilde{z},s) \right]  d\tilde{z} ds   .
		\end{aligned}
	\end{equation*}
	Taking $k\to \infty$, using \eqref{23Oct22-1}, $a>0$, \cite[Lemma A.3]{infi4D} for the first part in Case 1, $
	\int_0^{\infty} \left(4\pi t\right)^{-\frac{n}{2}} e^{-\frac{|x|^2}{4t}} dt
	=
	|x|^{2-n}
	\pi^{-\frac{n}{2}}
	2^{-1} (n-2)^{-1}
	\Gamma(\frac{n}{2}) $ with $n>2$ for the second and third parts,
	we have
	\begin{equation}\label{23Oct17-2}
		\tilde{\phi}(y,t) =
		\int_{-\infty}^{t}\int_{\mathbb{R}^{n-1}}
		G_n(y,t,(\tilde{z},0),s)  \frac{n}{n-2} \left(U^{\frac{2}{n-2}}\tilde{\phi} \right)(  ( \tilde{z},0),s )   d\tilde{z} ds   ,
	\end{equation}
	which satisfies
	\begin{equation}\label{23Oct17-3}
		\begin{cases}
			\pp_t \tilde{\phi}=\Delta \tilde{\phi}
			\quad
			\mbox{ \ in \ }  \mathbb{R}^{n}_+\times (-\infty,0] ,
			\quad
			-\pp_{y_n}\tilde{\phi} =\frac{n}{n-2}U^{\frac{2}{n-2}}\tilde{\phi}
			\quad
			\mbox{ \ on \ }\pp \mathbb{R}^n_+\times (-\infty,0] ,
			\\
			\int_{\mathbb{R}^n_+}\tilde{\phi}(y,t) Z_j(y) dy=0
			\quad
			\mbox{ \ for all \ } t\in  (-\infty,0] ,
			\quad j=0,1,\dots,n .
		\end{cases}
	\end{equation}
	
	Using \eqref{23Oct17-2}, there exists a constant $C_9$ varying from line to line , such that
	\begin{equation}\label{23Oct23-3}
		| \tilde{\phi} |
		\le C_9 \langle y \rangle^{2-n} ,
	\end{equation}
	and $\tilde{\phi}$ is smooth by the parabolic regularity theory (See \cite{Lieberman1986, Lieberman2012CPAA}). By the scaling argument, we have
	\begin{equation}\label{23Oct23-4}
		\langle y\rangle^{-1} | D \tilde{\phi}| + |\tilde{\phi}_t| + | D^2 \tilde{\phi}|\le C_9 \langle y \rangle^{-n}.
	\end{equation}
	Differentiating \eqref{23Oct17-3}, we get
	\begin{equation}\label{23Oct17-4}
		\begin{cases}
			\pp_t \tilde{\phi}_t=\Delta \tilde{\phi}_t
			\quad
			\mbox{ \ in \ }  \mathbb{R}^{n}_+\times (-\infty,0] ,
			\quad
			-\pp_{y_n}\tilde{\phi}_t =\frac{n}{n-2}U^{\frac{2}{n-2}}\tilde{\phi}_t
			\quad
			\mbox{ \ on \ }\pp \mathbb{R}^n_+\times (-\infty,0] ,
			\\
			\int_{\mathbb{R}^n_+}\tilde{\phi}_t(y,t) Z_j(y) dy=0
			\quad
			\mbox{ \ for all \ } t\in  (-\infty,0]
			,
			\quad
			j=0,1,\dots,n
		\end{cases}
	\end{equation}
	and then the scaling argument gives
	\begin{equation*}
		\langle y \rangle^{-1}| D \tilde{\phi}_t| + |\tilde{\phi}_{tt}| + | D^2 \tilde{\phi}_t|\le C_9 \langle y \rangle^{-n-2} .
	\end{equation*}
	
	Moreover, multiplying \eqref{23Oct17-4} by $\tilde{\phi}_{t}$ and integrating by parts, we get
	\begin{equation*}
		\frac{1}{2}\pp_t\int_{\mathbb{R}^n_+}|\tilde \phi_t|^2 dy+B[\tilde\phi_t,\tilde\phi_t]=0,
	\end{equation*}
	where
	$B[u,v]:=\int_{\mathbb{R}^n_+}\nabla u\cdot\nabla v dy-\frac{n}{n-2}\int_{\mathbb{R}^{n-1}} \left(U^{\frac{2}{n-2}} u v \right)(\tilde{y},0) d\tilde{y}$.

	Then $B[\tilde{\phi}_t,
	\tilde{\phi}_t ]\ge 0$ by $\int_{\mathbb{R}_+^n}  \tilde{\phi}_t(y, t)  Z_0(y)dy = 0$ since $Z_0$ is the only eigenfucntion of \eqref{Z0-eq} with negative eigenvalue. Thus, $\partial_t\int_{\mathbb{R}_+^n}|\tilde{\phi}_t|^2 dy \le 0$.

	Multiplying \eqref{23Oct17-3}  by $\tilde{\phi}_t$ and integrating by parts, we have
	\begin{equation*}
		\int_{\mathbb{R}_+^n}|\tilde{\phi}_t|^2 dy = -\frac{1}{2}\partial_t B[\tilde{\phi}, \tilde{\phi}].
	\end{equation*}
	
	By \eqref{23Oct23-3}, \eqref{23Oct23-4}, for $n>2$, $|B[\tilde\phi,\tilde\phi](t)|$ is uniformly bounded for $t\in(-\infty,0]$. Thus, we have
	\begin{equation*}
		\int_{-\infty}^0 \int_{\mathbb{R}_+^n}|\tilde{\phi}_t|^2 dy dt < \infty.
	\end{equation*}
	Hence $\tilde{\phi}_t = 0$, that is, $\tilde{\phi}=\tilde{\phi}(y)$ is independent of $t$.
	
	By \eqref{23Oct17-3} and Lemma \ref{qd23Nov28-2-lem}, we have $\tilde{\phi}\equiv 0$, which contradicts $\tilde{\phi}\not\equiv 0$ in \eqref{23Oct17-5}.
	As a result, \eqref{23Oct14-2} holds.
	Due to the arbitrarily small choice of $\epsilon$, we conclude the parameters restriction \eqref{23Oct22-2}.
\end{proof}	

\section{Leading term of $\mu_i$ and topology of the inner and outer problems}\label{formal derivation}

Hereafter, we focus on the case $n=5$ unless otherwise specified.

In order to apply Proposition \ref{prop-23Oct24-1} to solve the inner problems \eqref{cyl-inner}, suitable $\bm{\mu}, \bm{\xi}$ will be taken to satisfy the orthogonality conditions \eqref{gh-ortho-Oct20}.

Recall \eqref{eq-heat-1}. $\Theta_{l_i}(0,t)=-(T-t)^{l_i}$ is a leading role in the orthogonal equations. Roughly speaking, $\psi$ is a smaller term compared with $\Theta_{l_i}(0,t)$ in the inner problem. As the leading term of $\mu_i$, $\mu_{i,0}$ is determined by
\begin{equation}\label{qd23Jan01-2}
\begin{aligned}
&
\dot{\mu}_{i,0}
\mu_{i,0} \int_{\mathbb{R}_+^n}
Z_n^2(y) \eta\Big(\frac{y}{4R}\Big) dy
+
\int_{\mathbb{R}^{n-1}}
\frac{n}{n-2}
\mu_{i,0}^{\frac{n}{2}-1}
U^{\frac{2}{n-2}}(\tilde{y},0)
\eta\Big(\frac{\tilde{y}}{4R},0\Big) \Theta_{l_i}(0,t) Z_n(\tilde{y},0) d\tilde{y} =0
\\
\Leftrightarrow \ &
\dot{\mu}_{i,0} \mu_{i,0}^{2-\frac{n}{2}}
=
-\left(T-t\right)^{l_i} A_R
\end{aligned}
\end{equation}
with
\begin{equation}\label{24Jan01-6}
\begin{aligned}
&
A_R := -\Big(\int_{\mathbb{R}_+^n} Z_n^2(y) \eta\Big(\frac{y}{4R}\Big) dy \Big)^{-1} \int_{\mathbb{R}^{n-1}} \frac{n}{n-2}
U^{\frac{2}{n-2}}(\tilde{y},0)
\eta\Big(\frac{\tilde{y}}{4R} , 0 \Big)  Z_n(\tilde{y},0) d\tilde{y}
\\
= \ &
\frac{n-2}{2} \Big(\int_{\mathbb{R}_+^n} Z_n^2(y) dy \Big)^{-1} \int_{\mathbb{R}^{n-1}} U^{\frac{n}{n-2}} (\tilde{y},0) d\tilde{y} + O(R^{-1})  ,
\end{aligned}
\end{equation}
where we used \eqref{qd2023Nov28-1} and $n=5$ for the last step of $A_R$.
We take
\begin{equation}\label{24Jan01-7}
\mu_{i,0}(t) =
\left[ A_R \frac{6-n}{2} \left(l_i+1\right)^{-1}
\right]^{\frac{2}{6-n}}
\left(T-t\right)^{\frac{2}{6-n} \left(l_i+1\right)} .
\end{equation}

Let $\mu_{i,1}(t):=\mu_{i}(t)-\mu_{i,0}(t)$ be the minor term of $\mu_{i}(t)$. Denote
\begin{equation*}
		\bm{\mu}_{,0}
		=
		(\mu_{1,0}, \mu_{2,0},\dots , \mu_{\mathfrak{o},0}  ),
		\quad
		\bm{\mu}_{,1}
		=
		(\mu_{1,1}, \mu_{2,1},\dots , \mu_{\mathfrak{o},1}  )  .
\end{equation*}

For $T\ll 1$, one plugging $n=5$, then there exists a constant $C_{\mu_{i,0}}> 9$ sufficiently large such that
\begin{equation}\label{qd24Jan02-1}
	9 C_{\mu_{i,0}}^{-1} \left(T-t\right)^{2 l_i+2}
	\le
	\mu_{i,0}(t) \le 9^{-1} C_{\mu_{i,0}} \left(T-t\right)^{2l_i+2},
	\quad
	\left|\dot{\mu}_{i,0}(t) \right|
	\le
	9^{-1}
	C_{\mu_{i,0}} \left(T-t\right)^{2l_i + 1 }  .
\end{equation}
We make the ansatz that
\begin{equation}\label{outer-assumption2}
	\begin{aligned}
		&
		C_{\mu_{i,0}}^{-1} \left(T-t\right)^{2 l_i+2}
		\le  \mu_i \le
		C_{\mu_{i,0}} (T-t)^{2l_i+2},
		\quad
		|\dot{\mu}_i|\le
		C_{\mu_{i,0}} (T-t)^{2l_i+1},
		\\
		&
		\left| \xi^{[i]}-\tilde{q}^{[i]}\right|
		\le
		C_{\mu_{i,0}} (T-t)^{2l_i+2}  ,
		\quad
		|\dot{\xi}^{[i]}|\le  C_{\mu_{i,0}} (T-t)^{2l_i+1} .
	\end{aligned}
\end{equation}
Under the ansatz \eqref{outer-assumption2},
$\left|\dot{\mu}_i
\mu_i  Z_5(y^{[i]}) \right| \eta(\frac{y^{[i]}}{4R}) \lesssim \left(T-t\right)^{4 l_i+3} \langle y^{[i]} \rangle^{-3} \eta(\frac{y^{[i]}}{4R}) \lesssim
R^{\frac{3}{2}} \left(T-t\right)^{4 l_i+3} \langle y^{[i]} \rangle^{-\frac{9}{2}} \eta(\frac{y^{[i]}}{4R}) $. Using Proposition \ref{prop-23Oct24-1} with $n=5$ formally, for $i=1,2,\dots, \mathfrak{o}$, we define $T_{\sigma_0} := T- \sigma_0 $ with  $\sigma_0\in [0,T)$, and the spaces
\begin{equation}
B_{{\rm in},\sigma_0}  :=
\Big\{
(f_1,f_2,\dots, f_{\mathfrak{o}})
\ \big| \
f_i,\nabla f_i \in L^{\infty}\left(B_5^+(0,2R)\times(0,T_{\sigma_0})
\right),  i=1,2,\dots,\mathfrak{o},
\quad
\sup_{i=1,2,\dots,\mathfrak{o}}
\| f_i \|_{{\rm{in}},l_i,\sigma_0} \le 1
\Big\}
\end{equation}
with the norm
\begin{equation*}
\| f \|_{{\rm{in}},l_i,\sigma_0} := \sup_{(y,t) \in B_5^+(0,2R) \times  (0,T_{\sigma_0} ) }
\big( R^{2}(T-t)^{4l_i+3}
\langle y \rangle^{-\frac{5}{2}} \big)^{-1}
\left( |f(y,t)|+\langle y \rangle  |\nabla f(y,t) | \right)  ,
\end{equation*}
where we used $R^{2}$ instead of $R^{\frac{3}{2}}$ in the norm $\|\cdot\|_{{\rm{in}},l_i,\sigma_0}$ for the final fixed point argument.

\medskip

In order to get the vanishing property around the blow-up points for the outer problem, we adopt the ideas in \cite[p.318]{harada2019higher}, \cite[p.13]{Zhang-Zhao2023} to define the following space for the outer problem
\begin{equation}\label{qd23Jan04-6}
\mathcal{X}_{\delta_0,\sigma_0}
: =
\left\{
f\in L^\infty(\overline{B_5^+(0,\sigma_0^{-1}) } \times (0,T_{\sigma_0} ) )
\ | \ \| f \|_{\mathcal{X},\sigma_0 }  \le \delta_0
\right\}
\end{equation}
with the norm
\begin{equation*}
	\| f \|_{\mathcal{X},\sigma_0 }
	:= \sup_{(x,t)\in \overline{B_5^+(0,\sigma_0^{-1}) } \times (0,T_{\sigma_0})}
	\Big[
	\langle x\rangle^{-2}
	\1_{ \cap_{i=1}^{\mathfrak{o}}
		\big\{ |z^{[i]}| > e^{\frac{l_i s}{2 l_i +2}  } \big\} }
	+
	\sum\limits_{i=1}^{\mathfrak{o}}
	(T-t)^{l_i} \langle z^{[i]} \rangle^{2l_i+2}
	\1_{|z^{[i]}| \le e^{\frac{l_i s}{2 l_i +2}  } }
	\Big]^{-1}
	|f(x,t)|
\end{equation*}
with $z^{[i]} =\frac{x-q^{[i]}}{\sqrt{T-t}}$ and $s=-\ln(T-t)$, where we admit $B_5^+(0,\sigma_0^{-1})\big|_{\sigma_0=0} = \mathbb{R}_+^5$ and the constant $\delta_0 \in (0,1)$ will be determined later.
Note that  $|z^{[i]}| \le e^{\frac{l_i s}{2 l_i +2}  } \Leftrightarrow |x-q^{[i]}|\le (T-t)^{\frac{1}{2 l_i+2}}$, which implies $\{ |z^{[i]}| \le e^{\frac{l_i s}{2 l_i +2}  }  \} \cap \{ |z^{[j]}| \le e^{\frac{l_j s}{2 l_j +2}  }  \} = \emptyset$ for $i\ne j$ with $T\ll 1$.
Given $q\in \mathbb{R}_+^5$, $\langle x\rangle \sim_{q} \langle x-q\rangle$.

\section{Priori estimate of the outer problem}\label{section-outer}

Set $N_{\rm{max}} = \max\limits_{i=1,2,\dots, \mathfrak{o}} N(\lceil 5l_i/3 \rceil+1 )$ with $N(\cdot)$ defined in \eqref{qd24Feb13-1}. We set $\tilde{e}_j(z)$ by Corollary \ref{orth-lem} satisfying the properties
\begin{equation}\label{qd23Dec11-1}
\begin{aligned}
&
\tilde{e}_j(z) = \sum\limits_{\iota=0}^{N_{\rm{max}}} a_{j\iota} e_\iota (z) \eta(\frac{z}{C_{\tilde{e}}})
\mbox{ \ with a constant \ } C_{\tilde{e}}>0 ,
\quad  j=0,1, \dots, N_{\rm{max}}  ,
\\
&
\partial_{z_n}\tilde{e}_i=0
\quad	\mbox{ \ on \ }\pp\mathbb{R}_+^n
\quad
\mbox{ \ and \ }
\quad
\left( \tilde e_i , e_j  \right)_{L_\rho^2(\mathbb{R}_+^n)} =\delta_{ij }
\quad \mbox{ \ for \ } i, j =0,1,\dots, N_{\rm{max}}
\end{aligned}
\end{equation}
with a constant matrix $(a_{il})_{(N_{\rm{max}} +1)\times (N_{\rm{max}} +1)}$. Denote
\begin{equation}\label{24Dec01-1}
	\bm{\tilde e}_i(z) :=\left(\tilde e_{0}(z), \tilde e_{1}(z) , \dots, \tilde e_{N(\lceil 5l_i/3 \rceil+1 )} (z)  \right)  .
\end{equation}

We emphasize the following remark before further analysis.
\begin{remark}\label{qd24Feb23-1rk}
The property of $\bm{\mu} (\bm{\xi})$ being the original function of $\dot{\bm{\mu}} (\dot{\bm{\xi} })$  plays no role in this section, for which we can regard $\bm{\mu}, \dot{\bm{\mu}}, \bm{\xi}, \dot{\bm{\xi}} $ as four independent functions in this section.

\end{remark}

In order to find a solution for the outer problem \eqref{cyl-outer} with fast time decay near the blow-up points, we consider the following equation with a suitable initial value
\begin{equation}\label{outer-1}
\left\{
	\begin{aligned}
		&
		\partial_t \psi=\Delta_x \psi+\mathcal{G}_{1}[\bm{\phi},\bm{\mu},\bm{\xi}]
		\text{ \ in \ } \mathbb{R}_+^5 \times (0,T),
		\quad
		-\partial_{x_5}\psi=\mathcal{G}_{2}[ \psi,\bm{\phi},\bm{\mu},\bm{\xi}]
		\text{ \ on \ } \partial\mathbb{R}_+^5 \times (0,T),
\\
& \psi(x,0)=\sum\limits_{i=1}^{\mathfrak{o}}\bm b_i\cdot \bm{ \tilde{e}}_i\big( T^{-\frac{1}{2}}(x-q^{[i]}) \big)+\varphi_0(x)
		\text{ \ in \ } \mathbb{R}_+^5 ,
	\end{aligned}
\right.
\end{equation}
where $\varphi_0(x)\in C_{\rm c}^\infty(\overline{\mathbb{R}^5_+})$ and  $\bm{b}_i=(b_{i,0},b_{i,1},\dots,b_{i,N(\lceil 5l_i/3 \rceil+1 )})\in \mathbb{R}^{N(\lceil 5l_i/3 \rceil+1 )}$ will be determined later. By \eqref{qd2023Dec03-3}, \eqref{outer-1} is rewritten as
\begin{equation}\label{qd24Jan01-2}
	\begin{aligned}
& \psi(x,t) = \mathcal{T}^{\rm {out}}[ \psi,\bm{\phi},\bm{\mu},\bm{\xi}] := \int_0^{t}
		\int_{\mathbb{R}_+^5}
		G_5(x,t,z,s)
		\mathcal{G}_{1}[\bm{\phi},\bm{\mu},\bm{\xi}](z,s) dz ds
		\\
		&
		+
		\int_0^{t}\int_{\mathbb{R}^{4}}
		G_5(x,t,(\tilde{z},0),s)  \mathcal{G}_{2}[ \psi,\bm{\phi},\bm{\mu},\bm{\xi}](  ( \tilde{z},0),s )   d\tilde{z} ds
		+
		\int_{\mathbb{R}_+^5}
		G_5(x,t,z,0)
		\Big[
		\sum\limits_{i=1}^{\mathfrak{o}}\bm b_i\cdot \bm{ \tilde{e}}_i\big( T^{-\frac{1}{2}}(x-q^{[i]}) \big)+\varphi_0(z)
		\Big] dz.
	\end{aligned}
\end{equation}

Adopting the idea in \cite[p.14]{Zhang-Zhao2023} about the distribution of the right-hand side,
we decompose $\psi=\sum_{i=1}^{\mathfrak{o}+1} \psi_i$, where $\psi_i$  satisfy the following equations respectively. For $i=1,2,\dots, \mathfrak{o}$,
\begin{equation}\label{outeri}
        \partial_t \psi_i=\Delta_x \psi_i+\mathcal{G}_{1,i}
        \text{ \ in \ } \mathbb{R}_+^5 \times (0,T) ,
        \quad
    -\partial_{x_5}\psi_i=\mathcal{G}_{2,i}
    \text{ \ on \ } \partial\mathbb{R}_+^5 \times (0,T)
    ,
    \quad
    \psi_i(x,0) ={\bm{b}}_i\cdot {\bm{\tilde{e} }}_i\big( T^{-\frac{1}{2}}(x-q^{[i]}) \big)
    \text{ \ in \ } \mathbb{R}_+^5,
\end{equation}
and for $i=\mathfrak{o}+1$,
\begin{equation}\label{qd23Dec23-3}
        \partial_t \psi_{\mathfrak{o}+1}=\Delta_x \psi_{\mathfrak{o}+1}\text{ \ in \ } \mathbb{R}_+^5 \times (0,T) ,
        \quad
    -\partial_{x_5}\psi_{\mathfrak{o}+1}=\mathcal{G}_{2,\mathfrak{o}+1}
    \text{ \ on \ } \partial\mathbb{R}_+^5 \times (0,T),
    \quad
    \psi_{\mathfrak{o}+1}(x,0)=\varphi_0(x)
    \text{ \ in \ } \mathbb{R}_+^5,
\end{equation}
where for $i=1,2,\dots,\mathfrak{o}$,
\begin{equation}\label{G1i-G2i}
	\begin{aligned}
	&	\mathcal{G}_{1,i} =\mathcal{G}_{1,i}[\phi_i,\mu_i,\xi^{[i]}]  :=
		\Lambda_{1,i}[\phi_i,\mu_i,\xi^{[i]}] + \Lambda_{2,i}[\phi_i,\mu_i,\xi^{[i]}]
		+ \mathcal{E}_1\left[U_{\mu_i,\xi^{[i]}} (x)\right] \left(\eta_{2\delta}(x^{[i]})
		-
		\eta_{R}(y^{[i]})
		\right)
		+
		\mathcal{E}_{U,i}^{\rm{cut}}
		+
		\mathcal{E}_{\Theta,i}^{\rm{cut}}
		,
		\\
	&	\mathcal{G}_{2,i} = \mathcal{G}_{2,i}\left[\psi,\bm{\phi},\bm{\mu},\bm{\xi}\right] := \Big\{
		\eta_{4\delta}(x^{[i]}) \mathcal{N}\left[\psi,\bm{\phi},\bm{\mu},\bm{\xi}\right]
		+
		\left(U_{\mu_i,\xi^{[i]}} (x) \right)^{\frac{5}{3}}
		\left(
		\eta_{2\delta}^{\frac{5}{3}}(x^{[i]})
		-
		\eta_{2\delta}(x^{[i]})
		\right)
		\\
		& \quad  +
		\frac{5}{3}
		\left(U_{\mu_i,\xi^{[i]}} (x) \right)^{\frac{2}{3}}
		\Big[
		\left(
		\eta_{\delta}(x^{[i]})
		-
		\eta_{R}(y^{[i]})
		\right)
		\Theta_{l_i}(x^{[i]},t)
		+
		\left(
		\eta_{2\delta}^{\frac{2}{3}}(x^{[i]})
		-
		\eta_{R}(y^{[i]})
		\right)
		\psi(x,t)
		\Big] \Big\} \Big|_{x_5=0} ,
		\\
	&	\mathcal{G}_{2,\mathfrak{o}+1}
	=
	\mathcal{G}_{2,\mathfrak{o}+1}[\psi] :=\mathcal{G}_2-\sum_{i=1}^{\mathfrak{o}}\mathcal{G}_{2,i}=
		\Big\{
		\Big(1-\sum_{i=1}^{\mathfrak{o}} \eta_{4\delta}(x^{[i]}) \Big) \mathcal{N}\left[\psi,\bm{\phi},\bm{\mu},\bm{\xi}\right] \Big\} \Big|_{x_5=0}
		\\
		&
		\quad
	=
	\Big\{
	\Big(1-\sum_{i=1}^{\mathfrak{o}} \eta_{4\delta}(x^{[i]}) \Big) |\psi|^{\frac{2}{3}} \psi \Big\} \Big|_{x_5=0}
		 .
	\end{aligned}
\end{equation}
Here we used \eqref{outer-assumption2}  and $T\ll 1$ for $\mathcal{G}_{2,\mathfrak{o}+1}$ and
	$\left(
	\eta_{2\delta}^{\frac{2}{3}}(x^{[i]}) - 1
	\right)  \eta_{R}(y^{[i]})
	\mu_{i}^{-\frac{3}{2}} \phi_i(y^{[i]},t) \equiv 0 $ for $\mathcal{G}_{2,i}$.

For $i=1,2, \dots, \mathfrak{o}$, set
\begin{equation}\label{qd23Dec21-2}
\psi_i(x,t) = \Psi_i \big( (T-t)^{-\frac 1 2 } \left(x -q^{[i]}\right) , - \ln (T-t) \big)
	,
	\quad
	\mbox{ that is, }
	\quad
\Psi_i(z^{[i]},s) = \psi_i\big( e^{-\frac{s}{2}} z^{[i]} +q^{[i]}, T-e^{-s} \big)  .
\end{equation}
By Lemma \ref{z24Feb13-2-lem}, \eqref{outeri} is rewritten as
\begin{equation*}
\left\{
\begin{aligned}
&
		\partial_s \Psi_i
		=   A_{z^{[i]}} \Psi_i
		+
		g_{1,i}(z^{[i]} ,s)
		\quad
		\mbox{ \ in \ }  \mathbb{R}_+^5 \times  (s_0,\infty  ),
\quad
	-\partial_{z_5^{[i]}} \Psi_i =
	g_{2,i}(\tilde{z}^{[i]},s)
		\quad	\mbox{ \ on \ }  \partial\mathbb{R}_+^5 \times  (s_0,\infty),
\\
&
	\Psi_i\big( z^{[i]},s_0\big)
		=
		\bm{b}_i\cdot \bm{\tilde e }_i ( z^{[i]} )
			\quad
		\mbox{ \ in \ }  \mathbb{R}_+^5  ,
\end{aligned}
\right.
\end{equation*}
where
\begin{equation}\label{qd23Dec12-3}
s_0:=-\ln T,
\quad
g_{1,i}(z^{[i]} ,s) := e^{-s}
\mathcal{G}_{1,i}(e^{-\frac{s}{2}} z^{[i]} + q^{[i]},T-e^{-s}),
\quad
g_{2,i}(\tilde{z}^{[i]},s) := e^{-\frac{s}{2}} \mathcal{G}_{2,i}(e^{-\frac{s}{2}} \tilde{z}^{[i]} + \tilde{q}^{[i]},T-e^{-s}) .
\end{equation}
 In order to find a solution   $\Psi_i(z^{[i]} ,s)$ of the form
\begin{equation}\label{qd23Dec21-1}
	\Psi_i(z^{[i]} ,s)=\bm{d}_i(s)\cdot\bm{\tilde{e}}_i(z^{[i]}) +
	\Phi_i(z^{[i]} ,s)
\end{equation}
with $\bm{d}_i(s)=(d_{i,0}(s), d_{i,1}(s),\dots,d_{i,N(\lceil 5l_i/3 \rceil+1 )}(s))\in C^1([s_0,\infty))$ to be determined later, one using $\partial_{z_n}\tilde{e}_i=0 $
on $\pp\mathbb{R}_+^n $ in \eqref{qd23Dec11-1}, it suffices to solve
\begin{equation}\label{outeri2}
		\partial_s \Phi_i=A_{z^{[i]}} \Phi_i+\tilde g_{1,i}(z^{[i]} ,s)
		\text{ \ in \ } \mathbb{R}_+^5 \times (s_0,\infty) ,
		\quad
		-\partial_{z^{[i]}_5}\Phi_i=g_{2,i}(\tilde z^{[i]} ,s)
		\text{ \ on \ } \partial\mathbb{R}_+^5 \times (s_0,\infty),
		\quad
		\Phi_i(\cdot,s_0)=0
		\text{ \ in \ } \mathbb{R}_+^5 ,
\end{equation}
where we set
\begin{equation}\label{gi10}
	\bm{b}_i:=\bm{d}_i(s_0),
	\quad
	\tilde{g}_{1,i}(z^{[i]} ,s):=g_{1,i}(z^{[i]} ,s)+\sum_{j=0}^{N(\lceil 5l_i/3 \rceil+1 )} \left( d_{i,j}(s)A_{z^{[i]}}\tilde e_{j}(z^{[i]} )- \dot{d}_{i,j}(s)\tilde e_{j}(z^{[i]} )\right) .
\end{equation}
In order to recover the time decay of the right-hand sides of \eqref{outeri2} for $\Phi_i$, we impose the orthogonality conditions
\begin{equation}\label{qd2023Dec17-1}
	\int_{\mathbb R^5_+}\tilde {g}_{1,i}(z^{[i]},s)e_j(z^{[i]}) e^{-\frac{|z^{[i]}|^2}{4}} dz^{[i]}  +\int_{\mathbb R^4} g_{2,i}(\tilde{z}^{[i]},s)e_j(\tilde{z}^{[i]},0)e^{-\frac{|\tilde{z}^{[i]}|^2}{4}}d\tilde{z}^{[i]} =0, \quad j=0,1,\dots, N(\lceil 5l_i/3 \rceil+1 ) .
\end{equation}
One using \eqref{qd23Dec11-2}, \eqref{qd2023Dec17-2},
 and \eqref{qd23Dec11-1},
then \eqref{qd2023Dec17-1} is equivalent to
\begin{equation}\label{eq-5.7}
	\dot{d}_{i,j}(s) + \lambda_j d_{i,j}(s)
	=\int_{\mathbb{R}^5_+}g_{1,i}(z^{[i]},s)e_{j}(z^{[i]} )e^{-\frac{|z^{[i]}|^2}{4}}  dz^{[i]}  +
	\int_{\mathbb{R}^4}g_{2,i}(\tilde{z}^{[i]} ,s) e_j(\tilde{z}^{[i]},0 )e^{-\frac{|\tilde{z}^{[i]}|^2}{4}} d\tilde{z}^{[i]}
\end{equation}
for $j=0,1,\dots , N(\lceil 5l_i/3 \rceil+1 )$.

\subsection{Estimates of $\mathcal{G}_{2,\mathfrak{o}+1}$, $\mathcal{G}_{1,i}$, $\mathcal{G}_{2,i}$, $i=1,2,\dots, \mathfrak{o}$ in \texorpdfstring{\eqref{G1i-G2i}}{$(\ref{G1i-G2i})$}, and $d_{i,j}(s)$}

\begin{lemma}\label{G1G2est-lem}

Suppose that $\bm{\mu}, \dot{\bm{\mu}}, \bm{\xi}, \dot{\bm{\xi}} $ satisfy \eqref{outer-assumption2}, $\psi\in \mathcal{X}_{\delta_0, 0}$, and $T\ll 1$, then
\begin{equation}\label{Gk+12estimate}
	\left| \mathcal{G}_{2,\mathfrak{o}+1}
	\right|
	\le
	\delta_0^{\frac{5}{3}}
	\langle \tilde{x} \rangle^{-\frac{10}{3}}
	\1_{ \cap_{i=1}^{\mathfrak{o}} \{ |\tilde{x}^{[i]}|\ge 4\delta \} } .
\end{equation}
Under the additional assumption that $\bm{\phi} \in B_{{\rm in},0}$, then there exists a constant $C$ independent of $T, \delta_0$ such that for $i=1,2,\dots,\mathfrak{o}$, $\mathcal{G}_{1,i}$, $\mathcal{G}_{2,i}$  have the pointwise upper bounds
\begin{equation}\label{Gi1estimate}
	\begin{aligned}
	|\mathcal{G}_{1,i}|
	\le \ &
	C
	\Big\{
		R^{-\frac{1}{4}}
		\mu_i^{-2} e^{-l_is}
		\left(
		\langle y^{[i]}\rangle^{-\frac{9}{4}}
		\mathbf{1}_{|y^{[i]}|\le 2R }
		+
		\langle y^{[i]} \rangle^{-\frac{11}{4}}
		\1_{|y^{[i]}| > R}
		\right)
		\1_{|z^{[i]} | \le 1}
		\\
		&
	+
	\left[
		e^{-(3l_i +\frac{1}{2})s} |z^{[i]}|^{-3}
		+
		\1_{\delta e^{\frac{s}{2}} \le |z^{[i]}|\le 2\delta e^{\frac{s}{2}} }
		\right]
	\1_{1<|z^{[i]}|\le 4\delta e^{\frac{s}{2}} }
	\Big\},
\end{aligned}
\end{equation}
\begin{equation}\label{Gi2estimate}
\begin{aligned}
	|\mathcal{G}_{2,i}|
	\le \ &
	C
	\Big[
		\left(
		R^{-\frac{1}{4}}
		\mu_i^{-1} e^{-l_i s} \langle \tilde{y}^{[i]} \rangle^{-\frac{7}{4}}
		+
		e^{-\frac{5}{3} l_i s}
		\right)
		\1_{|\tilde{z}^{[i]} |\le 1}
+
		e^{-\frac{5}{3} l_i s}
		| \tilde{z}^{[i]} |^{\frac{10}{3} l_i+ \frac{10}{3} }	
		\1_{1<|\tilde{z}^{[i]} |\le e^{\frac{l_i s}{2l_i+2} }}
		\\
		& +
		\left(
		e^{-\frac{5}{3} l_i s}
		| \tilde{z}^{[i]} |^{\frac{10}{3} l_i }
		+
		\delta_0^{\frac{5}{3}}
		\right)
		\1_{e^{\frac{l_i s}{2l_i+2} } <	|\tilde{z}^{[i]}| \le 8\delta e^{\frac{s}{2}} }
\Big].
\end{aligned}
\end{equation}

\end{lemma}

\begin{proof}
	
The estimate of $\mathcal{G}_{2,\mathfrak{o}+1}$ is straightforward.
By \eqref{outer-assumption2}, we have
\begin{equation}\label{qd23Dec15-1}
\langle y^{[i]} \rangle
\sim
2 C_{\mu_i,0}^2 + |y^{[i]}|
\sim
\langle \mu_i^{-1} x^{[i]} \rangle
=
\langle \mu_i^{-1} \left(T-t\right)^{\frac{1}{2}}
z^{[i]} \rangle
\sim
\langle e^{(2l_i +\frac{3}{2})s}
z^{[i]} \rangle .
\end{equation}
	
	\textbf{Step 1: Estimate of $\mathcal{G}_{1,i}$.}  Recall \eqref{qd23Dec11-3}.
One using $R=|\ln{T}|$, \eqref{outer-assumption2}, $\| \phi_i \|_{{\rm{in}},l_i,0} \le 1$ (by $\bm{\phi} \in B_{{\rm in},0}$), then
\begin{equation*}
\begin{aligned}
&	|\Lambda_{1,i}+\Lambda_{2,i}|
 \lesssim \mu_i^{-2}\mu_i^{-\frac{3}{2}}\left(
R^{-2}|\phi_i|
+
	R^{-1} |\nabla_{y^{[i]}} \phi_i|
+
\mu_i|\dot\mu_i||\phi_i|+\mu_i
R^{-1} |\dot\xi^{[i]}|
|\phi_i|\right)\mathbf{1}_{R\le |y^{[i]}|\le 2R}\\
	&\qquad \qquad \qquad +\mu_i^{-2}\left[
	|\dot\mu_i|\mu_i^{-\frac{1}{2}}
	\left(
	\left|\phi_i  \right|
	+
	\left|y^{[i]} \cdot \nabla_{y^{[i]}} \phi_i  \right|
	\right)
	+\mu_i^{-\frac{1}{2}}|\dot\xi^{[i]}||\nabla_{\tilde y^{[i]}}\phi_i|\right]\mathbf{1}_{|y^{[i]}|\le 2R}
	\\
	\lesssim \ &
\mu_i^{-2} \left(T-t\right)^{l_i}
R^{-\frac{5}{2}}
\mathbf{1}_{R\le |y^{[i]}|\le 2R}
+
\mu_i^{-2}
R^2 \left(T-t\right)^{5l_i+3}
\langle y^{[i]} \rangle^{-\frac{5}{2}}
\mathbf{1}_{|y^{[i]}|\le 2R}
\lesssim
R^{-\frac{1}{4}}	\mu_i^{-2} e^{-l_is}
	\langle y^{[i]}\rangle^{-\frac{9}{4}}
	\mathbf{1}_{|y^{[i]}|\le 2R}.
\end{aligned}
\end{equation*}
Note that $|y^{[i]}|\le 2R$ implies $|z^{[i]}| \le 1$.
\begin{equation*}
\begin{aligned}
&
	\left|
	\mathcal{E}_1\left[U_{\mu_i,\xi^{[i]}} (x)\right] \left(\eta_{2\delta}(x^{[i]})
	-
	\eta_{R}(y^{[i]})
	\right)
	\right|
	\lesssim
	\left( \1_{|x^{[i]}| \le 4\delta} - \1_{|y^{[i]}| \le R} \right)
	\left(T-t\right)^{-3l_i -4}
	\langle y^{[i]} \rangle^{-3}
\\
= \ &
\1_{|x^{[i]}| \le 4\delta, \  |y^{[i]}| > R}
\left(T-t\right)^{-3l_i -4}
\langle y^{[i]} \rangle^{-3}
\lesssim
\mu_i^{-2} e^{-l_i s}
\langle y^{[i]} \rangle^{-3}
\1_{|z^{[i]}| \le 1, \  |y^{[i]}| > R}
+
e^{-(3l_i +\frac{1}{2})s} |z^{[i]}|^{-3}
\1_{1<|z^{[i]}|\le 4\delta e^{\frac{s}{2}} }  ,
\end{aligned}
\end{equation*}
where we used \eqref{qd23Dec15-1} for the last ``$\lesssim$''.
\begin{equation*}
\left| \mathcal{E}_{U,i}^{\rm{cut}} \right|
\lesssim
\left(\mu_i^{-\frac{3}{2}} \langle y^{[i]} \rangle^{-3}
+\mu_i^{-\frac{5}{2}} \langle y^{[i]} \rangle^{-4}
  \right) \1_{2\delta\le |x^{[i]}| \le 4\delta }
\sim
e^{-(3l_i +\frac{3}{2})s} |z^{[i]}|^{-3}
\1_{2\delta e^{\frac{s}{2}} \le |z^{[i]}|\le 4\delta e^{\frac{s}{2}} }  ,
\end{equation*}
where we used \eqref{qd23Dec15-1} for the last ``$\sim$''. Using \eqref{eq-heat-1}, then
\begin{equation*}
\left| \mathcal{E}_{\Theta,i}^{\rm{cut}} \right| \lesssim
\1_{\delta\le |x^{[i]}| \le 2\delta }
=
\1_{\delta e^{\frac{s}{2}} \le |z^{[i]}|\le 2\delta e^{\frac{s}{2}} }  .
\end{equation*}
Summarizing all the estimates reaches the upper bound of $\mathcal{G}_{1, i}$.

\medskip

\textbf{Step2: Estimate of $\mathcal{G}_{2,i}$.} By \eqref{qd23Dec15-1},
\begin{equation*}
\begin{aligned}
&
\left| \left(U_{\mu_i,\xi^{[i]}} (x) \right)^{\frac{5}{3}}
\left(
\eta_{2\delta}^{\frac{5}{3}}(x^{[i]})
-
\eta_{2\delta}(x^{[i]})
\right)   \right|
\lesssim
\mu_i^{-\frac{5}{2}}
\langle y^{[i]} \rangle^{-5} \1_{2\delta\le |x^{[i]}| \le 4\delta }
\sim
e^{-(5l_i +\frac{5}{2}) s }
|z^{[i]}|^{-5}
\1_{2\delta e^{\frac{s}{2}} \le |z^{[i]}|\le 4\delta e^{\frac{s}{2}} }  ,
\\
&
	\left|
\left(U_{\mu_i,\xi^{[i]}} (x) \right)^{\frac{2}{3}}
\left(
\eta_{\delta}(x^{[i]})
-
\eta_{R}(y^{[i]})
\right)
\Theta_{l_i}(x^{[i]},t)
\right|
\lesssim
\mu_i^{-1} \langle y^{[i]} \rangle^{-2}
e^{-l_i s} \langle z^{[i]} \rangle^{2 l_i}
\1_{|z^{[i]}|\le 2\delta e^{\frac{s}{2} }, \ |y^{[i]}|>R }
\\
\sim \ &
\mu_i^{-1} \langle y^{[i]} \rangle^{-2}
e^{-l_i s}
\1_{|z^{[i]}|\le 1, \ |y^{[i]}|>R }
+
e^{-(3l_i+1) s} |z^{[i]}|^{2 l_i-2}
\1_{1< |z^{[i]}| \le 2\delta e^{\frac{s}{2} } }   .
\end{aligned}
\end{equation*}
For $\psi\in \mathcal{X}_{\delta_0,0}$,
\begin{equation*}
\begin{aligned}
&
\left|
\left(U_{\mu_i,\xi^{[i]}} (x) \right)^{\frac{2}{3}}
\left(
\eta_{2\delta}^{\frac{2}{3}}(x^{[i]})
-
\eta_{R}(y^{[i]})
\right)
\psi
\right|
\\
\lesssim \ &
\mu_i^{-1} \langle y^{[i]} \rangle^{-2}
\1_{|x^{[i]}|\le 4\delta, \ |y^{[i]}|>R } \
\delta_0
\Big[
	(T-t)^{l_i} \langle z^{[i]} \rangle^{2l_i+2} \1_{|z^{[i]} |\le e^{\frac{l_is}{2l_i+2}}}
+
\langle x \rangle^{-2} \1_{|z^{[i]} | > e^{\frac{l_i s}{2l_i+2}} , \ |x^{[i]}|\le 4\delta }
\Big]
\\
\lesssim \ &
\delta_0
\Big[
	\mu_i^{-1} \langle y^{[i]} \rangle^{-2}
	e^{-l_i s} \1_{|z^{[i]} |\le 1 , |y^{[i]}|>R}
+
	e^{-(3 l_i+1) s} |z^{[i]}|^{2l_i}
 \1_{1<|z^{[i]} |\le e^{\frac{l_i s}{2l_i+2}}}
+
e^{-\big( 2l_i +2-\frac{1}{l_i+1} \big) s }
	\langle x^{[i]} \rangle^{-2}
 \1_{e^{\frac{l_i s}{2l_i+2} } < |z^{[i]} | \le 4\delta e^{\frac{s}{2}}}
 \Big],
\end{aligned}
\end{equation*}
where we used \eqref{qd23Dec15-1} and $\langle x \rangle\sim \langle x^{[i]} \rangle$ for the last ``$\lesssim $''.

For $|x^{[i]}|\le 8\delta$, $u=
	U_{\mu_i,\xi^{[i]}} (x)
\eta_{2\delta}(x^{[i]})
	+
	\Theta_{l_i}(x^{[i]},t)\eta_{\delta}(x^{[i]})+
	\mu_{i}^{-\frac{n-2}{2}} \phi_i(y^{[i]},t) \eta_{R}(y^{[i]}) +\psi$.
By the elementary inequality
$ \left||a+b|^{p-1}(a+b)-|a|^{p-1}a-p|a|^{p-1}b \right|\le p |b|^p
$
for $p\in (1,2]$, $a,b\in\mathbb{R}$, we have
\begin{equation*}
\begin{aligned}
&
\left|
\eta_{4\delta}(x^{[i]}) \mathcal{N}\left[\psi,\bm{\phi},\bm{\mu},\bm{\xi}\right]
\right|
=
\eta_{4\delta}(x^{[i]})
\Big|
|u|^\frac{2}{3}u -  \left(U_{\mu_i,\xi^{[i]}} (x) \right)^{\frac{5}{3}}
\eta_{2\delta}^{\frac{5}{3}}(x^{[i]})
\\
& \qquad
-
\frac{5}{3}
\left(U_{\mu_i,\xi^{[i]}} (x) \right)^{\frac{2}{3}}
\eta_{2\delta}^{\frac{2}{3}}(x^{[i]})
\left(
\Theta_{l_i}(x^{[i]},t)\eta_{\delta}(x^{[i]})+
\mu_{i}^{-\frac{3}{2}} \phi_i(y^{[i]},t) \eta_{R}(y^{[i]}) +\psi(x,t)
\right)
\Big|
\\
\lesssim \ &
\eta_{4\delta}(x^{[i]})
\left|
\Theta_{l_i}(x^{[i]},t)\eta_{\delta}(x^{[i]})+
\mu_{i}^{-\frac{3}{2}} \phi_i(y^{[i]},t) \eta_{R}(y^{[i]}) +\psi(x,t)
\right|^{\frac{5}{3}}
\\
\lesssim \ &
e^{-\frac{5}{3} l_i s}
\langle z^{[i]} \rangle^{\frac{10}{3} l_i }
\1_{|z^{[i]}|\le 2\delta e^{\frac{s}{2}} }
+
R^{\frac{10}{3}}
e^{-\frac{5}{3} l_i s}
\langle y^{[i]} \rangle^{-\frac{25}{6}}
\1_{|y^{[i]}|\le 2R}
\\
&
+
\delta_0^{\frac{5}{3}}
e^{-\frac{5}{3} l_i s}
\langle z^{[i]} \rangle^{\frac{10}{3} l_i+ \frac{10}{3} }
\1_{|z^{[i]} |\le e^{\frac{l_i s}{2l_i+2} }}
+
\delta_0^{\frac{5}{3}}
\langle x^{[i]} \rangle^{-\frac{10}{3}}
\1_{e^{\frac{l_i s}{2l_i+2} } <	|z^{[i]}| \le 8\delta e^{\frac{s}{2}}}  .
\end{aligned}
\end{equation*}
Combining the estimates above, we have the estimate of $\mathcal{G}_{2, i}$.
\end{proof}

To exploit the spectrum properties of the $-A_z$, we give the following estimates.
\begin{lemma}\label{g1g2-L2rho-lem}
	Under all assumptions in Lemma \ref{G1G2est-lem}, then for $i=1,2,\dots,\mathfrak{o}$, $s\in[s_0,\infty)$, it holds that
	\begin{equation*}
		\| g_{1,i}(\cdot,s) \|_{L_{\rho}^2(\mathbb{R}_+^5)}
		\le C e^{-(2l_i+\frac{3}{4}) s},
		\quad
		\| g_{2,i}(\cdot,s) \|_{L_\rho^2(\mathbb R^4)}\le C
		e^{-(\frac{5}{3} l_i + \frac{1}{2} ) s}
	\end{equation*}
with a constant $C$ independent of $T, \delta_0$ and varying from line to line.
For $j\in \mathbb{N}$,
\begin{equation}\label{dij}
	d_{i,j}(s)=e^{-\lambda_j s}\int_{s_{0*}}^s
	e^{\lambda_j \sigma}
	\bigg[
	\int_{\mathbb{R}^5_+}g_{1,i}(z^{[i]},\sigma)e_{j}(z^{[i]} )e^{-\frac{|z^{[i]}|^2}{4}}  dz^{[i]}  +
	\int_{\mathbb{R}^4}g_{2,i}(\tilde{z}^{[i]} ,\sigma) e_j(\tilde{z}^{[i]},0 )e^{-\frac{|\tilde{z}^{[i]}|^2}{4}} d\tilde{z}^{[i]}
	\bigg]
	d\sigma
\end{equation}
with $
s_{0*}:=
\begin{cases}
s_0  &  \mbox{ \ if \ } \lambda_j \ge  \frac{5}{3} l_i + \frac{1}{2}
\\
\infty
&
\mbox{ \ if \ } \lambda_j < \frac{5}{3} l_i + \frac{1}{2}
\end{cases}
$
are integrable and satisfy \eqref{eq-5.7}. Moreover, given $k\in \mathbb{N}$, for $j=0,1,\dots, k$,
\begin{align}\label{bikcik}
&
	|d_{i,j}(s)| +
	| \dot{d}_{i,j}(s)|\le C \begin{cases}
	e^{-(\frac{5}{3}l_i+\frac{1}{2})s}
	&  \mbox{ \ if \ } \lambda_j \ne  \frac{5}{3} l_i + \frac{1}{2}
	\\
	s e^{-(\frac{5}{3}l_i+\frac{1}{2})s}
	&  \mbox{ \ if \ } \lambda_j =  \frac{5}{3} l_i + \frac{1}{2}  ,
\end{cases}
	\\
&
\label{diestimate}
\sum_{j=0}^{k}
\big(
|\dot{d}_{i,j}(s) |
+
| d_{i,j}(s)|
\big)
\left(
\left|A_{z^{[i]}}\tilde {e}_{j}(z^{[i]} )
\right|
+
\left|\tilde{e}_{j}(z^{[i]} )\right|
\right)
\le C
\1_{|z^{[i]}|\le 2 C_{\tilde{e}}}
\begin{cases}
	e^{-(\frac{5}{3}l_i+\frac{1}{2})s}
	&  \mbox{ \ if \ } \lambda_k <  \frac{5}{3} l_i + \frac{1}{2}
	\\
	s e^{-(\frac{5}{3}l_i+\frac{1}{2})s}
	&  \mbox{ \ if \ } \lambda_k \ge  \frac{5}{3} l_i + \frac{1}{2}
\end{cases}
\end{align}
with $\tilde{e}_j$ given in \eqref{qd23Dec11-1}.
In particular,
\begin{equation}
\| \tilde{g}_{1,i}(\cdot,s) \|_{L_{\rho}^2(\mathbb{R}_+^5)}
		\le C s e^{-(\frac{5}{3} l_i + \frac{1}{2} ) s}.
\end{equation}

\end{lemma}

\begin{proof}
	
	By \eqref{qd23Dec12-3}, Lemme \ref{G1G2est-lem}, \eqref{qd23Dec15-1},
	\begin{equation*}
		\begin{aligned}
			&
			\| g_{1,i}(\cdot,s) \|_{L_{\rho}^2(\mathbb{R}_+^5)}
			\lesssim
			e^{-s}
			R^{-\frac{1}{4}}
			\mu_i^{-2} e^{-l_is}
			\big\|
			\langle \mu_i^{-1} \left(T-t\right)^{\frac{1}{2}}
			z^{[i]} \rangle^{-\frac{9}{4}}
			\mathbf{1}_{|z^{[i]}|\le 4R \mu_i (T-t)^{-\frac{1}{2}} }
			\\
			&
			+
			\langle \mu_i^{-1} \left(T-t\right)^{\frac{1}{2}}
			z^{[i]} \rangle^{-\frac{11}{4}}
			\1_{2^{-1} R \mu_i (T-t)^{-\frac{1}{2}} < |z^{[i]}| \le 1}  \big\|_{L_{\rho}^2(\mathbb{R}_+^5)}
			\\
			&
			+
			e^{-s}
			\big\|
			e^{-(3l_i +\frac{1}{2})s} |z^{[i]}|^{-3}
			\1_{1<|z^{[i]}|\le 4\delta e^{\frac{s}{2}}}
			+
			\1_{\delta e^{\frac{s}{2}} \le |z^{[i]}|\le 2\delta e^{\frac{s}{2}} }
			\big\|_{L_{\rho}^2(\mathbb{R}_+^5)}
			\\
			\lesssim \ &
			e^{-(2l_i+\frac{3}{4}) s}
			+
			e^{-s} \left[
			e^{-(3l_i+\frac{1}{2}) s}
			+e^{-\frac{1}{16} \delta^2 e^s }
			\right]
			\sim
			e^{-(2l_i+\frac{3}{4}) s} ,
		\end{aligned}
	\end{equation*}
	where we used
	\begin{equation*}
		\begin{aligned}
			&
			\big\|
			\langle \mu_i^{-1} \left(T-t\right)^{\frac{1}{2}}
			z^{[i]} \rangle^{-\frac{9}{4}}
			\mathbf{1}_{|z^{[i]}|\le 4R \mu_i (T-t)^{-\frac{1}{2}} }
			\big\|_{L_{\rho}^2(\mathbb{R}_+^5)}
			\\
			\sim \ &
			\big\|
			\mathbf{1}_{|z^{[i]}|\le \mu_i (T-t)^{-\frac{1}{2}} }
			+
			\mu_i^{\frac{9}{4}} \left(T-t\right)^{-\frac{9}{8}}
			|z^{[i]}|^{-\frac{9}{4}}
			\mathbf{1}_{\mu_i (T-t)^{-\frac{1}{2}}<|z^{[i]}|\le 4R \mu_i (T-t)^{-\frac{1}{2}} }
			\big\|_{L_{\rho}^2(\mathbb{R}_+^5)}
			\\
			\lesssim \ &
			\mu_i^{\frac{5}{2}} (T-t)^{-\frac{5}{4}}
			+
			R^{\frac{1}{4}} \mu_i^{\frac{5}{2}} (T-t)^{-\frac{5}{4}}
			\sim
			R^{\frac{1}{4}}
			e^{-(5 l_i+\frac{15}{4}) s }
			,
			\\
			&
			\big\|
			\langle \mu_i^{-1} \left(T-t\right)^{\frac{1}{2}}
			z^{[i]} \rangle^{-\frac{11}{4}}
			\1_{2^{-1} R \mu_i (T-t)^{-\frac{1}{2}} < |z^{[i]}| \le 1}  \big\|_{L_{\rho}^2(\mathbb{R}_+^5)}
			\\
			\sim \ &
			\mu_i^{\frac{11}{4}}
			(T-t)^{-\frac{11}{8}}
			\big\|
			|z^{[i]}|^{-\frac{11}{4}}
			\1_{2^{-1} R \mu_i (T-t)^{-\frac{1}{2}} < |z^{[i]}| \le 1}  \big\|_{L_{\rho}^2(\mathbb{R}_+^5)}
			\sim
			R^{-\frac{1}{4}}
			\mu_i^{\frac{5}{2}} (T-t)^{-\frac{5}{4}}
			\sim
			R^{-\frac{1}{4}}
			e^{-(5 l_i+\frac{15}{4}) s }
			.
		\end{aligned}
	\end{equation*}
	\begin{equation*}
		\| g_{2,i}(\cdot,s) \|_{L_\rho^2(\mathbb R^4)}\lesssim
		e^{-\frac{s}{2}}
		\left[
		R^{-\frac{1}{4}}
		\mu_i^{-1} e^{-l_i s}
		e^{-(\frac{7}{2} l_i +\frac{21}{8}) s }
		+
		e^{-\frac{5}{3} l_i s}
		+
		\left(
		e^{-\frac{5}{3} l_i s}
		+
		\delta_0^{\frac{5}{3}}
		\right)
		e^{-\frac{1}{16} e^{\frac{l_i}{l_i +1} s} }
		\right]
		\sim
		e^{-(\frac{5}{3} l_i + \frac{1}{2} ) s}  ,
	\end{equation*}
	where we used \eqref{qd23Dec15-1} and
	\begin{equation*}
		\begin{aligned}
			&
			\| \langle \tilde{y}^{[i]} \rangle^{-\frac{7}{4}}
			\1_{|\tilde{z}^{[i]} |\le 1} \|_{L_\rho^2(\mathbb R^4)}
			\sim
			\| \langle \mu_i^{-1} \left(T-t\right)^{\frac{1}{2}}
			\tilde{z}^{[i]} \rangle^{-\frac{7}{4}}
			\1_{|\tilde{z}^{[i]} |\le 1} \|_{L^2(\mathbb R^4)}
			\\
			\sim \ &
			\|
			\1_{|\tilde{z}^{[i]} |\le \mu_i \left(T-t\right)^{-\frac{1}{2}} }
			+
			\mu_i^{\frac{7}{4}} \left(T-t\right)^{-\frac{7}{8}}
			|\tilde{z}^{[i]}|^{-\frac{7}{4}}
			\1_{\mu_i \left(T-t\right)^{-\frac{1}{2}} <|\tilde{z}^{[i]} |\le 1}
			\|_{L^2(\mathbb R^4)}
			\\
			\lesssim \ &
			\mu_i^2 (T-t)^{-1}
			+
			\mu_i^{\frac{7}{4}} \left(T-t\right)^{-\frac{7}{8}}
			\sim
			e^{-(\frac{7}{2} l_i +\frac{21}{8}) s } .
		\end{aligned}
	\end{equation*}

It follows that
\begin{equation*}
		\bigg|
		\int_{\mathbb{R}^5_+}g_{1,i}(z^{[i]},s)e_{j}(z^{[i]} )e^{-\frac{|z^{[i]}|^2}{4}}  dz^{[i]}
		\bigg| + \bigg|
\int_{\mathbb{R}^4}g_{2,i}(\tilde{z}^{[i]} ,s) e_j(\tilde{z}^{[i]},0 )e^{-\frac{|\tilde{z}^{[i]}|^2}{4}} d\tilde{z}^{[i]} \bigg|
		\lesssim  e^{- (\frac{5}{3} l_i + \frac{1}{2} ) s} .
\end{equation*}
Now it is easy to get that $d_{i,j}(s)$ given in \eqref{dij} are well-defined and satisfy \eqref{bikcik}.
Due to the support of $\tilde{e}_j$, we deduce \eqref{diestimate}.
The bound of $\| \tilde{g}_{1,i}(\cdot,s) \|_{L_{\rho}^2(\mathbb{R}_+^5)}$ is deduced by \eqref{gi10}, the estimate of $\| g_{1,i}(\cdot,s) \|_{L_{\rho}^2(\mathbb{R}_+^5)}$, and
\eqref{diestimate}.
\end{proof}

\subsection{Estimates of $\Phi_{i}$ satisfying \texorpdfstring{\eqref{outeri2}}{$(\ref{outeri2})$}, $i=1,2,\dots,\mathfrak{o}$}

This subsection is inspired by \cite{harada2019higher, Zhang-Zhao2023}. We will derive the estimate of $\|\Phi_{i}(\cdot,s)\|_{L_\rho^2(\mathbb R^5_+)}$ and then give the pointwise estimate of $\Phi_i$.
\begin{lemma}\label{Phii2norm}
	Under all assumptions in Lemma \ref{G1G2est-lem}, then there exists a constant $C>0$ independent of $T, \delta_0$ such that
\begin{equation*}
\|\Phi_{i}(\cdot,s)\|_{L_\rho^2(\mathbb R^5_+)} \le C s e^{-(\frac{5}{3} l_i +\frac{1}{2}) s } .
\end{equation*}
\end{lemma}

\begin{proof}
We first assume $\mathcal{G}_{1,i}, \mathcal{G}_{2,i}$ are smooth about $z^{[i]}, s$ and $\tilde{z}^{[i]}, s$ respectively and the derivatives of $\mathcal{G}_{1,i}, \mathcal{G}_{2,i}$ belong to $L^{\infty}(\mathbb{R}_+^5 \times (s_0,s))$ and $L^{\infty}(\mathbb{R}^4 \times (s_0,s))$ respectively for all $s\in (s_0,\infty)$. Then $\Phi_{i}$ is smooth and the derivatives of $\Phi_{i}$ are bounded for any fixed $s\in (s_0,\infty)$.
For $j=0,1,\dots, N(\lceil 5l_i/3 \rceil+1 )$,	testing the equation \eqref{outeri2} by $e_j(z^{[i]} ) \rho(z^{[i]})$, integrating by parts as \eqref{qd23Dec11-2},  and using \eqref{qd2023Dec17-2}, \eqref{qd2023Dec17-1}, we obtain that
\begin{equation*}
	\begin{aligned}
	&	\partial_s(\Phi_{i}(\cdot,s),e_j )_{L_{\rho}^2(\mathbb{R}_+^5)}
		=
		(\Phi_i(\cdot,s),A_{z^{[i]} }e_j )_{{L_{\rho}^2(\mathbb{R}_+^5)}}
		\\
		&
		+
		\int_{\mathbb R^5_+}\tilde g_{1,i}(z^{[i]} ,s)e_j(z^{[i]} )e^{-\frac{|z^{[i]} |^2}{4}} dz^{[i]}
		+\int_{\mathbb R^4}g_{2,i}(\tilde z^{[i]} ,s)e_j(\tilde z^{[i]} ,0)e^{-\frac{|\tilde z^{[i]} |^2}{4}}  d\tilde z^{[i]}
		=-\lambda_j
		(\Phi_{i}(\cdot,s),e_j )_{L_{\rho}^2(\mathbb{R}_+^5)} .
	\end{aligned}
\end{equation*}
By $\Phi_i(\cdot,s_0)=0$, we have
\begin{equation}\label{qd2023Dec17-4}
(\Phi_{i}(\cdot,s) ,e_j )_{L_{\rho}^2(\mathbb{R}_+^5)} =0,
\quad
j=0,1,\dots, N(\lceil 5l_i/3 \rceil+1 )
\quad
\mbox{ \ for all \ } s\ge s_0.
\end{equation}
	
	Testing the equation \eqref{outeri2} by $\Phi_{i} \rho(z^{[i]})$ and integrating by parts as \eqref{qd23Dec11-2}, we obtain
\begin{equation*}
	\begin{aligned}
 &
		\frac{1}{2}\partial_s
		\left(
		\|\Phi_{i}(\cdot,s) \|^2_{L_\rho^2(\mathbb R^5_+)}  \right)
		= -\|\nabla\Phi_{i}(\cdot,s) \|_{L_\rho^2(\mathbb R^5_+)}^{2}+\int_{\mathbb R^5_+}\tilde g_{1,i}(z^{[i]} ,s)\Phi_{i}(z^{[i]} ,s)e^{-\frac{|z^{[i]} |^2}{4}} dz^{[i]}
  \\
  &
		+\int_{\mathbb R^4}g_{2,i}(\tilde z^{[i]} ,s)\Phi_{i}( (\tilde z^{[i]} ,0),s)e^{-\frac{|\tilde z^{[i]} |^2}{4}}  d\tilde {z}^{[i]}
		\le -\|\nabla\Phi_{i}(\cdot,s) \|_{L_\rho^2(\mathbb R^5_+)}^{2}
		+C f_1(s) \|\Phi_i(\cdot,s) \|_{H_{\rho}^1(\mathbb R^5_+)}
	\end{aligned}
\end{equation*}
with  $f_1(s):= \| g_{1,i}(\cdot,s) \|_{L_{\rho}^2(\mathbb{R}_+^5)}
		+
		\| g_{2,i}(\cdot,s) \|_{L_\rho^2(\mathbb R^4)} + \sum_{j=0}^{N(\lceil 5l_i/3 \rceil+1 )} \left( |d_{i,j}(s)| + |\dot{d}_{i,j}(s)|\right)$,
    where we used \eqref{gi10} and Lemma \ref{Sob-lem1} for the last inequality.
Then by \eqref{qd2023Dec17-4}, Lemma \ref{lemma-eigen} \eqref{zz23Nov30-2}, we have
\begin{equation*}
	\begin{aligned}
		\frac{1}{2}\partial_s \left(\|\Phi_{i}(\cdot,s)\|^2_{L_\rho^2(\mathbb R^5_+)}  \right)
		\le -
		\left(1-\epsilon\right)\left(\lceil 5 l_i/3 \rceil+ 3/2 \right)
		\| \Phi_{i}(\cdot,s) \|_{L_\rho^2(\mathbb R^5_+)}^{2}+\frac{C(\epsilon)}{2} (f_1(s))^2 ,
	\end{aligned}
\end{equation*}
where $\epsilon>0$ can be close to $0$ arbitrarily. By $\Phi_i(\cdot,s_0)=0$,
\begin{equation}\label{z24Feb19-1}
\|\Phi_{i}(\cdot,s)\|^2_{L_\rho^2(\mathbb R^5_+)}
\le e^{-(1-\epsilon)\left(2 \lceil 5 l_i/3 \rceil+ 3 \right) s} \int_{s_0}^s C(\epsilon) (f_1(a))^2 e^{(1-\epsilon)\left(2 \lceil 5 l_i/3 \rceil+ 3 \right) a } da .
\end{equation}

Generally, when $\mathcal{G}_{1,i}, \mathcal{G}_{2,i}$ are not smooth, we apply the mollifier to $\mathcal{G}_{1,i}, \mathcal{G}_{2,i}$ in spatial and time variables. Then repeating the above process and taking the limitation will deduce \eqref{z24Feb19-1}.
Using Lemma \ref{g1g2-L2rho-lem} and
taking $\epsilon$ sufficiently small to make $\left(1-\epsilon\right)\left(2\lceil 5 l_i/3 \rceil+ 3 \right) >10 l_i/3+1$, we get the desired estimate.
\end{proof}

The representation formula of $\Phi_i$ in \eqref{outeri2} is given by \eqref{qd23Dec04-1} of the form
\begin{equation*}
	\Phi_{i}(z^{[i]} ,s)
		=
		\int_{s_0}^{s}
		\int_{\mathbb{R}_+^5}
		H_5\left(z^{[i]},s,w,\sigma\right)  \tilde{g}_{1,i}(w, \sigma) dw d\sigma
		+
		\int_{s_0}^{s}
		\int_{\mathbb{R}^{4}}
		H_5\left(z^{[i]},s,(\tilde{w},0),\sigma\right)  g_{2,i}(\tilde{w}, \sigma) d \tilde{w} d\sigma .
\end{equation*}

\begin{lemma}\label{greenestimatei}
	Under all assumptions in Lemma \ref{G1G2est-lem}, then there exists a constant $C>0$ independent of $T, \delta_0$ such that for any $z^{[i]}\in \overline{\mathbb{R}^5_+}$, $s>s_1\ge s_0$, we have
\begin{equation*}
	\begin{aligned}
		&
		\int_{s_1}^{s}
		\int_{\mathbb{R}_+^5}
		H_5\left(z^{[i]},s,w,\sigma\right)  \left|\tilde{g}_{1,i}(w, \sigma) \right| dw d\sigma
		+
		\int_{s_1}^{s}
		\int_{\mathbb{R}^{4}}
		H_5\left(z^{[i]},s,(\tilde{w},0),\sigma\right)  \left|g_{2,i}(\tilde{w}, \sigma) \right| d \tilde{w} d\sigma
		\\
		\le  \ &
		C
		\Big[
		R^{-\frac{1}{4}}
		e^{-l_i s} \langle e^{(2l_i +\frac{3}{2}) s} z^{[i]} \rangle^{-\frac{1}{4}}
		+
		R^{-\frac{1}{4}}
		e^{-( l_i +\frac{3}{8}) s_1 }
		+
		\min\Big\{
		e^{-s_1} ,
		e^{-s_1} e^{-(l_i+1)s} |z^{[i]}|^{2l_i+2}
		\Big\}
		\\
		& 	+
		\min\Big\{
		e^{-\frac{s_1}{2}},
		e^{-(\frac{5}{3} l_i +\frac{1}{2}) s}
		|\tilde{z}^{[i]}|^{\frac{10}{3} l_i +\frac{10}{3}}
		\Big\}  \Big] .
	\end{aligned}
\end{equation*}
	
\end{lemma}

\begin{proof}
	
	By \eqref{gi10}, \eqref{qd23Dec12-3}, \eqref{diestimate}, Lemma \ref{G1G2est-lem},
	\begin{equation*}
		\begin{aligned}
			&
			\left| \tilde{g}_{1,i}(z^{[i]} ,s) \right|
			\lesssim
			e^{-s}
			\left|\mathcal{G}_{1,i} \right|
			+
			s e^{-(\frac{5}{3}l_i+\frac{1}{2})s}
			\1_{|z^{[i]}|\le 2 C_{\tilde{e}}}
			\lesssim
			R^{-\frac{1}{4}}
			\mu_i^{-2} e^{-(l_i+1)s}
			\langle y^{[i]}\rangle^{-\frac{9}{4}}
			+
			s e^{-(\frac{5}{3}l_i+\frac{1}{2})s}
			+
			e^{-s} \1_{\delta e^{\frac{s}{2}} \le |z^{[i]}|\le 2\delta e^{\frac{s}{2}} }  ,
			\\
			&
			\left| g_{2,i}(\tilde{z}^{[i]},s) \right| \lesssim  e^{-\frac{s}{2}} \left| \mathcal{G}_{2,i} \right|
			\lesssim
			R^{-\frac{1}{4}}
			\mu_i^{-1} e^{-(l_i + \frac{1}{2} ) s} \langle \tilde{y}^{[i]} \rangle^{-\frac{7}{4}}
			+
			e^{- (\frac{5}{3} l_i + \frac{1}{2} ) s}
			+
e^{-\frac{s}{2}}
\min\Big\{1, e^{-\frac{5}{3} l_i  s}
			| \tilde{z}^{[i]} |^{\frac{10}{3} l_i+ \frac{10}{3} }
 \Big\}    ,
		\end{aligned}
	\end{equation*}
	where we used
	$\delta_0\in (0,1)$.
	By \eqref{outer-assumption2}, \eqref{qd23Dec15-1},
	\begin{equation*}
		R^{-\frac{1}{4}}
		\mu_i^{-2} e^{-(l_i+1)s}
		\langle y^{[i]}\rangle^{-\frac{9}{4}}
		\sim
		R^{-\frac{1}{4}}
		e^{(3l_i +3)s}
		\langle e^{(2l_i +\frac{3}{2})s}
		z^{[i]} \rangle^{-\frac{9}{4}},
		\quad
		R^{-\frac{1}{4}}
		\mu_i^{-1} e^{-(l_i + \frac{1}{2} ) s} \langle \tilde{y}^{[i]} \rangle^{-\frac{7}{4}}
		\sim
		R^{-\frac{1}{4}}
		e^{(l_i+\frac{3}{2}) s }
		\langle e^{(2l_i +\frac{3}{2})s}
		\tilde{z}^{[i]} \rangle^{-\frac{7}{4}}  .
	\end{equation*}
	
	By Lemma \ref{outerbarrieri1} \eqref{qd23Dec06-1},
	\begin{equation*}
		\begin{aligned}
			&
			\int_{s_1}^{s}
			\int_{\mathbb{R}_+^5}
			H_5\left(z^{[i]},s,w,\sigma\right)
			R^{-\frac{1}{4}}
			e^{(3l_i +3)\sigma}
			\langle e^{(2l_i +\frac{3}{2}) \sigma}
			w \rangle^{-\frac{9}{4}}
			dw d\sigma
			\\
			&
			+
			\int_{s_1}^{s}
			\int_{\mathbb{R}^{4}}
			H_5\left(z^{[i]},s,(\tilde{w},0),\sigma\right)
			R^{-\frac{1}{4}}
			e^{(l_i+\frac{3}{2}) \sigma }
			\langle e^{(2l_i +\frac{3}{2}) \sigma}
			\tilde{w} \rangle^{-\frac{7}{4}}
			d \tilde{w} d\sigma
			\\
			\lesssim \ &
			R^{-\frac{1}{4}}
			\left[
			e^{-l_i s} \langle e^{(2l_i +\frac{3}{2}) s} z^{[i]} \rangle^{-\frac{1}{4}}
			+
			e^{-(\frac{3}{2} l_i +\frac{3}{8}) s_1 }
			+
			e^{-l_i s} \langle e^{(2l_i +\frac{3}{2}) s} z^{[i]} \rangle^{-\frac{3}{4}}
			+
			e^{-(\frac{5}{2} l_i +\frac{9}{8}) s_1 }
			\right]
			\\
			\sim \ &
			R^{-\frac{1}{4}}
			\left[
			e^{-l_i s} \langle e^{(2l_i +\frac{3}{2}) s} z^{[i]} \rangle^{-\frac{1}{4}}
			+
			e^{-(\frac{3}{2} l_i +\frac{3}{8}) s_1 }
			\right]
			.
		\end{aligned}
	\end{equation*}

By Lemma \ref{outerbarrieri1} \eqref{lemma-barrier3}, for an arbitrarily small constant $\epsilon\in (0,1/4)$,
\begin{equation*}
		\int_{s_1}^{s}
		\int_{\mathbb{R}_+^5}
		H_5\left(z^{[i]},s,w,\sigma\right)
		\sigma e^{-(\frac{5}{3}l_i+\frac{1}{2}) \sigma}
		dw d\sigma
\lesssim_{\epsilon}
\int_{s_1}^{s}
\int_{\mathbb{R}_+^5}
H_5\left(z^{[i]},s,w,\sigma\right)     e^{-(\frac{5}{3}l_i+\frac{1}{2} -\epsilon ) \sigma}
dw d\sigma
\lesssim_{\epsilon}
e^{-(\frac{5}{3} l_i +\frac{1}{2} -\epsilon) s_1}.
\end{equation*}

By Lemma \ref{outerbarrieri1} \eqref{lemma-barrier3} \eqref{outerheatkerneli2},
\begin{equation*}
	\begin{aligned}
		&
		\int_{s_1}^{s}
		\int_{\mathbb{R}_+^5}
		H_5\left(z^{[i]},s,w,\sigma\right)
		\left( e^{-\sigma} \1_{\delta e^{\frac{\sigma}{2}} \le |w|\le 2\delta e^{\frac{\sigma}{2}} }
		\right)
		dw d\sigma
\\
\lesssim \ &
\int_{s_1}^{s}
\int_{\mathbb{R}_+^5}
H_5\left(z^{[i]},s,w,\sigma\right)
\min\Big\{e^{-\sigma} ,
e^{-(l_i+2)\sigma} |w|^{2l_i+2} \Big\}
dw d\sigma
\\
\lesssim \ &
\min\Big\{
e^{-s_1} ,
\int_{s_1}^s
\left[
e^{-\sigma} e^{-(l_i+1)s} |z^{[i]}|^{2l_i+2}
+
e^{-(l_i+2)\sigma}
\right] d \sigma
\Big\}
\lesssim
\min\Big\{
e^{-s_1} ,
e^{-s_1} e^{-(l_i+1)s} |z^{[i]}|^{2l_i+2}
+
e^{-(l_i+2) s_1}
\Big\} .
	\end{aligned}
\end{equation*}

By Lemma \ref{outerbarrieri1} \eqref{lemma-barrier4-0},
\begin{equation*}
\begin{aligned}
&
	\int_{s_1}^{s}
	\int_{\mathbb{R}^{4}}
	H_5\left(z^{[i]},s,(\tilde{w},0),\sigma\right)
	\left[ e^{- (\frac{5}{3} l_i + \frac{1}{2} ) \sigma}
	+
	\min\Big\{ e^{-\frac{\sigma}{2}}, e^{- (\frac{5}{3} l_i + \frac{1}{2} )  \sigma}
	| \tilde{w} |^{\frac{10}{3} l_i+ \frac{10}{3} }
	\Big\}
	\right] d \tilde{w} d\sigma
	\\
\lesssim \ &
\min\Big\{
e^{-\frac{s_1}{2}},
e^{-(\frac{5}{3} l_i +\frac{1}{2}) s}
|\tilde{z}^{[i]}|^{\frac{10}{3} l_i +\frac{10}{3}}
+
e^{-(\frac{5}{3} l_i +\frac{1}{2}) s_1}
\Big\} .
\end{aligned}
\end{equation*}
	
Combining the above estimates and using $C_1+\min\{ C_2,C_3\} = \min\{ C_1+C_2, C_1+C_3\}$, then we get the desired estimate.	
\end{proof}

\begin{prop}\label{Phii2pointwise1}
Under all assumptions in Lemma \ref{G1G2est-lem}, then there exists a constant $C>0$ independent of $T, \delta_0$ such that for any $z^{[i]} \in \overline{\mathbb{R}^5_+}$, $s>s_0 $, we have
\begin{equation}\label{esti-8.23}
	\begin{aligned}
	\left| \Phi_{i} (z^{[i]},s) \right|
	\le \ &
C \Big[
	\Big(
	R^{-\frac{1}{4}}
	e^{-l_i s} \langle e^{(2l_i +\frac{3}{2}) s} z^{[i]} \rangle^{-\frac{1}{4}}
	+
	T^{\frac{3}{8}}
	R^{-\frac{1}{4}}
	e^{-l_i s }
	\langle |z^{[i]}|^{2l_i} \rangle
	+
	e^{-(l_i+\frac{1}{2})s} |z^{[i]}|^{2l_i+2}
	\Big) \1_{|z^{[i]} | \le e^{\frac{l_i s}{2l_i+2} }}
\\
&
	+
	R^{-\frac{1}{4}}
	\1_{|z^{[i]} | > e^{\frac{l_i s}{2l_i+2} }} \Big].
	\end{aligned}
\end{equation}

\end{prop}

\begin{proof}
Applying Lemma \ref{greenestimatei} with $s_1=s_0$, we get $|\Phi_i| \lesssim R^{-\frac{1}{4}}$ in $\overline{\mathbb{R}_+^5}\times (s_0,\infty)$.
For the more delicate estimate, we separate the domain into the following four parts and estimate separately.
	
{ \textbf{ Case 1: Fix $z^{[i]} \in \overline{\mathbb{R}^5_+}$ and $s\in (s_0,s_0+2]$. } } By Lemma \ref{greenestimatei} with $s_1=s_0$,
\begin{equation*}
\big| \Phi_{i} (z^{[i]}, s) \big|
\lesssim
R^{-\frac{1}{4}}
e^{-l_i s} \langle e^{(2l_i +\frac{3}{2}) s} z^{[i]} \rangle^{-\frac{1}{4}}
+
R^{-\frac{1}{4}}
e^{-( l_i +\frac{3}{8}) s_0 }
+
e^{-s_0} e^{-(l_i+1)s} |z^{[i]}|^{2l_i+2}
+
e^{-(\frac{5}{3} l_i +\frac{1}{2}) s}
|z^{[i]}|^{\frac{10}{3} l_i +\frac{10}{3}}  .
\end{equation*}
	
{\textbf{Case 2: Fix $|z^{[i]} |\le 2$ and $s\in (s_0+2, \infty)$. } }
	By uniqueness, we split $\Phi_{i}=\Phi_{i}^{(1)}+\Phi_{i}^{(2)}$ and rewrite \eqref{outeri2} as
	\begin{equation*}
\begin{aligned}
&
		\left\{
		\begin{aligned}
			&
			\partial_{s_*}\Phi_{i}^{(1)}=A_{w }\Phi_{i}^{(1)}+\tilde {g}_{1,i}(w,s_*)
			\mbox{ \ for \  }
			(w,s_*) \in
			\mathbb{R}^5_+\times (s-1,\infty),
	\\
	&
		-\partial_{w_{5}}\Phi_{i}^{(1)}((\tilde{w},0),s_*)=g_{2,i}(\tilde{w} ,s_*)
			\mbox{ \ for \ }
			(\tilde{w},s_*) \in
			\mathbb{R}^4\times (s-1,\infty)
			,
		\quad
			\Phi_{i}^{(1)}(\cdot ,s-1) =0
			\mbox{ \ in \ }
			\mathbb{R}^5_+  ,
		\end{aligned}
	\right.
\\
&
\left\{
\begin{aligned}
	&
	\partial_{s_*}\Phi_{i}^{(2)}=A_{w}\Phi_{i}^{(2)}
	\mbox{ \ for \ }
	(w,s_*)\in
	\mathbb{R}^5_+\times (s -1,\infty),
\\
&
-\partial_{w_{5}}\Phi_{i}^{(2)}((\tilde{w},0),s_*)=0
	\mbox{ \ for \  }
	(\tilde{w},s_*)
	\in \mathbb{R}^4 \times (s -1,\infty)
	,
\quad
	\Phi_{i}^{(2)}(\cdot,s-1)=\Phi_{i}(\cdot,s-1)
	\mbox{ \ in \ }
	\mathbb{R}^5_+.
\end{aligned}
\right.
\end{aligned}
	\end{equation*}

For $\Phi_{i}^{(1)}$, applying Lemma \ref{greenestimatei} with $s_1=s-1$ at the point $(z^{[i]},s)$, we have
\begin{equation*}
\big| \Phi_{i}^{(1)}(z^{[i]},s) \big| \lesssim
R^{-\frac{1}{4}}
e^{-l_i s} \langle e^{(2l_i +\frac{3}{2}) s} z^{[i]} \rangle^{-\frac{1}{4}}
+
R^{-\frac{1}{4}}
e^{-( l_i +\frac{3}{8}) s }
+
e^{-s} e^{-(l_i+1)s} |z^{[i]}|^{2l_i+2}
+
e^{-(\frac{5}{3} l_i +\frac{1}{2}) s}
|z^{[i]}|^{\frac{10}{3} l_i +\frac{10}{3}}  .
\end{equation*}

For $\Phi_{i}^{(2)}$, by Lemma \ref{outerbarrieri1} \eqref{outerheatkerneli2} with $s_1=s-1$, $|z^{[i]} |\le 2$, and Lemma \ref{Phii2norm}, then
\begin{equation*}
\big| \Phi_{i}^{(2)}(z^{[i]},s) \big|
\lesssim
\| \Phi_{i}(\cdot,s-1) \|_{L_\rho^2(\mathbb{R}_+^5)}
\lesssim
s e^{-(\frac{5}{3} l_i +\frac{1}{2}) s }  .
\end{equation*}

Since $s e^{-(\frac{5}{3} l_i +\frac{1}{2}) s } \lesssim R^{-\frac{1}{4}}
e^{-( l_i +\frac{3}{8}) s }$, we have
\begin{equation*}
	\left| \Phi_{i} (z^{[i]},s) \right|
\lesssim
	R^{-\frac{1}{4}}
	e^{-l_i s} \langle e^{(2l_i +\frac{3}{2}) s} z^{[i]} \rangle^{-\frac{1}{4}}
	+
	R^{-\frac{1}{4}}
	e^{-( l_i +\frac{3}{8}) s }
	+
	e^{-s}
 e^{-(l_i+1)s} |z^{[i]}|^{2l_i+2}
	+
	e^{-(\frac{5}{3} l_i +\frac{1}{2}) s}
	|z^{[i]}|^{\frac{10}{3} l_i +\frac{10}{3}}  .
\end{equation*}

{\textbf{ Case 3: Fix $2<|z^{[i]} |<e^{\frac{s-s_0}{2}}$ and $s \in (s_0+2,\infty)$.} }
We choose $s_1\in (s_0, s-1)$ such that $|z^{[i]} |=e^{\frac{s-s_1}{2}}$.
Similar to Case 2, we split $\Phi_{i}=\Phi_{i}^{(3)}+\Phi_{i}^{(4)}$ and the equation \eqref{outeri2} is rewritten as
	\begin{equation*}
\begin{aligned}
&
\left\{
\begin{aligned}
&
			\partial_{s_*}\Phi_{i}^{(3)}=A_{w}\Phi_{i}^{(3)}+\tilde {g}_{1,i}(w,s_*)
			\mbox{ \ for \ }
			(w,s_*) \in  \mathbb{R}^5_+\times (s_1,\infty),
			\\
		&
	-\partial_{w_{5}}\Phi_{i}^{(3)} ( (\tilde{w},0) ,s_*)
		=g_{2,i}(\tilde{w} ,s_*)
			\mbox{ \ for \  }
			(\tilde{w} ,s_*)
			\in
			\mathbb{R}^4\times (s_1,\infty)  ,
\quad
			\Phi_{i}^{(3)}(\cdot,s_1)=0
			\mbox{ \ in \ }
			\mathbb{R}^5_+  ,
\end{aligned}
\right.
\\
&
\left\{
\begin{aligned}
&
\partial_{s_*}\Phi_{i}^{(4)}=A_{w}\Phi_{i}^{(4)}
\mbox{ \ for \ } (w,s_*)\in \mathbb{R}^5_+\times (s_1,\infty),
\\
&
-\partial_{w_{5}}\Phi_{i}^{(4)}( (\tilde{w},0) ,s_*)
=0
\mbox{ \ for \  }
(\tilde{w},s_*)
\in \mathbb{R}^4 \times (s_1,\infty)
,
\quad
\Phi_{i}^{(4)}(\cdot,s_1)=\Phi_{i}(\cdot,s_1)
\mbox{ \ in \ }
\mathbb{R}^5_+.
\end{aligned}
\right.
\end{aligned}
	\end{equation*}

	By Lemma \ref{greenestimatei}, we have	
\begin{equation*}
\big| \Phi_{i}^{(3)} (z^{[i]}, s) \big|
\lesssim
T^{\frac{3}{8}}
R^{-\frac{1}{4}}
e^{-l_i s } |z^{[i]}|^{2l_i}
+
T e^{-(l_i+1)s } |z^{[i]}|^{2l_i+2}
+
e^{-(\frac{5}{3} l_i +\frac{1}{2}) s}
|z^{[i]}|^{\frac{10}{3} l_i +\frac{10}{3}}  ,
\end{equation*}
where we used
\begin{equation*}
\begin{aligned}
&
e^{-s_1 } = e^{-s} |z^{[i]}|^2,
\quad
R^{-\frac{1}{4}}
e^{-l_i s} \langle e^{(2l_i +\frac{3}{2}) s} z^{[i]} \rangle^{-\frac{1}{4}}
\lesssim
R^{-\frac{1}{4}}
e^{-l_i s_1}
 e^{-(\frac{l_i}{2} +\frac{3}{8}) s}
| z^{[i]} |^{-\frac{1}{4}}
\lesssim
T^{\frac{3}{8}}
R^{-\frac{1}{4}}
e^{-l_i s} |z^{[i]}|^{2l_i} ,
\\
&
R^{-\frac{1}{4}}
e^{-( l_i +\frac{3}{8}) s_1 }
=
e^{-\frac 38 s_1 }
R^{-\frac{1}{4}} e^{-l_i s} |z^{[i]}|^{2l_i}
\le T^{\frac{3}{8}}
R^{-\frac{1}{4}}
e^{-l_i s} |z^{[i]}|^{2l_i} .
\end{aligned}
\end{equation*}
	
By Lemma \ref{outerbarrieri1} \eqref{outerheatkerneli2}, $|z^{[i]} |=e^{\frac{s-s_1}{2}}$, $s_1\le  s-1$, and Lemma \ref{Phii2norm}, we obtain that
\begin{equation*}
\big| \Phi_{i}^{(4)} (z^{[i]}, s) \big|
\lesssim
\|  \Phi_{i} (\cdot, s_{1}) \|_{L_\rho^2(\mathbb{R}_+^5)}
\lesssim
s_1 e^{-(\frac{5}{3} l_i +\frac{1}{2}) s_1 }
\lesssim
(-\ln T) T^{\frac{1}{2}} e^{-l_i s} |z^{[i]}|^{2l_i}  .
\end{equation*}

Since $(-\ln T) T^{\frac{1}{2}}\lesssim T^{\frac{3}{8}}
R^{-\frac{1}{4}}$, we have
\begin{equation*}
	\big| \Phi_{i} (z^{[i]}, s) \big|
	\lesssim
	T^{\frac{3}{8}}
	R^{-\frac{1}{4}}
	e^{-l_i s} |z^{[i]}|^{2l_i}
	+
	T e^{-(l_i+1)s} |z^{[i]}|^{2l_i+2}
	+
	e^{-(\frac{5}{3} l_i +\frac{1}{2}) s}
	|z^{[i]}|^{\frac{10}{3} l_i +\frac{10}{3}}  .
\end{equation*}

{\textbf{ Case 4: Fix $|z^{[i]} |\ge e^{\frac{s-s_0}{2}}$ and $s\in(s_0+2,\infty)$. } } By Lemma \ref{greenestimatei} with $s_1=s_0$,
\begin{equation*}
\left| \Phi_i(z^{[i]}, s) \right|
\lesssim
T^{\frac{3}{8}}
R^{-\frac{1}{4}}
e^{-l_i s} |z^{[i]}|^{2 l_i}
+
T e^{-(l_i+1)s} |z^{[i]}|^{2l_i+2}
+
e^{-(\frac{5}{3} l_i +\frac{1}{2}) s}
|z^{[i]}|^{\frac{10}{3} l_i +\frac{10}{3}}  ,
\end{equation*}
where we used
\begin{equation*}
\begin{aligned}
&
e^{-s_0} \le e^{-s} |z^{[i]}|^2,
\quad
R^{-\frac{1}{4}}
e^{-l_i s} \langle e^{(2l_i +\frac{3}{2}) s} z^{[i]} \rangle^{-\frac{1}{4}}
\lesssim
R^{-\frac{1}{4}}
e^{-l_i s_0}
e^{-(\frac{l_i}{2} + \frac{3}{8}) s}
|z^{[i]}|^{-\frac{1}{4}}
\lesssim
T^{\frac{3}{8}}
R^{-\frac{1}{4}}
e^{-l_i s} |z^{[i]}|^{2 l_i}  ,
\\
&
R^{-\frac{1}{4}}
e^{-( l_i +\frac{3}{8}) s_0 }
\lesssim
T^{\frac{3}{8}}
R^{-\frac{1}{4}}
e^{-l_i s} |z^{[i]}|^{2 l_i}   .
\end{aligned}
\end{equation*}

Combining the above four cases, we get
\begin{equation*}
	\left| \Phi_{i} (z^{[i]},s) \right|
	\lesssim
	R^{-\frac{1}{4}}
	e^{-l_i s} \langle e^{(2l_i +\frac{3}{2}) s} z^{[i]} \rangle^{-\frac{1}{4}}
	+
	T^{\frac{3}{8}}
	R^{-\frac{1}{4}}
	e^{-l_i s }
	\langle |z^{[i]}|^{2l_i} \rangle
	+
	T
	e^{-(l_i+1)s} |z^{[i]}|^{2l_i+2}
	+
	e^{-(\frac{5}{3} l_i +\frac{1}{2}) s}
	|z^{[i]}|^{\frac{10}{3} l_i +\frac{10}{3}} .
\end{equation*}
In particular,
\begin{equation*}
	\left| \Phi_{i} (z^{[i]},s) \right|
	\lesssim
	R^{-\frac{1}{4}}
	e^{-l_i s} \langle e^{(2l_i +\frac{3}{2}) s} z^{[i]} \rangle^{-\frac{1}{4}}
	+
	T^{\frac{3}{8}}
	R^{-\frac{1}{4}}
	e^{-l_i s }
	\langle |z^{[i]}|^{2l_i} \rangle
	+
	e^{-(l_i+\frac{1}{2})s} |z^{[i]}|^{2l_i+2}
\mbox{ \ for \ }
|z^{[i]} | \le e^{\frac{l_i s}{2l_i+2} } .
\end{equation*}
Integrating $|\Phi_i| \lesssim R^{-\frac{1}{4}}$ in $\overline{\mathbb{R}_+^5}\times (s_0,\infty)$, we attain the conclusion.
\end{proof}

\begin{prop}\label{prop-phi0}
Under all assumptions in Lemma \ref{G1G2est-lem},
then for $\mathcal{T}^{\rm {out}}[ \psi,\bm{\phi},\bm{\mu},\bm{\xi}]$ given in \eqref{qd24Jan01-2}, there exist $\varphi_0(x)\in C^\infty_{\rm c}(\overline{\mathbb{R}^5_+})$ in \eqref{outer-1} and a constant $C>0$ independent of $T, \delta_0$ such that
\begin{equation*}
\begin{aligned}
&
	|\mathcal{T}^{\rm {out}}[ \psi,\bm{\phi},\bm{\mu},\bm{\xi}](x,t)|
	\le C
	R^{-\frac{1}{4}}
	\Big[
	\1_{\cap_{i=1}^{\mathfrak{o}} \big\{ |z^{[i]} | > e^{\frac{l_i s}{2l_i+2} }
		\big\} }
	+
	\sum_{i=1}^{\mathfrak{o}}
	(T-t)^{l_i}
	\langle |z^{[i]}|^{2l_i+2} \rangle
	\1_{ |z^{[i]} | \le e^{\frac{l_i s}{2l_i+2} } }
	\Big] ,
\\
&  \mathcal{T}^{\rm {out}}[ \psi,\bm{\phi},\bm{\mu},\bm{\xi}](x,0)=\sum\limits_{i=1}^{\mathfrak{o}}{\bm{b}}_i\cdot {\bm{ \tilde{e}} }_i\big( T^{-\frac{1}{2}}(x-q^{[i]}) \big)+ \sum_{i=1}^{\mathfrak{o}} \sum_{ \mathbf{p}\in \mathbb{N}^5, \| \mathbf{p}\|_{\ell_1} \le 4l_{\rm{max}}+4, p_5\in 2 \mathbb{N}}
C_{q^{[i]}, \mathbf{p}}
\varphi_{q^{[i]},\mathbf{p},0}(x) ,
\end{aligned}
\end{equation*}
where $\bm{b}_i = {\bm{b}}_i[ \psi,\bm{\phi}, \bm{\mu}, \bm{\xi} ] $ are constant vectors and
$| {\bm{b}}_i | \le C |\ln T| T^{ \frac{5}{3}l_i+\frac{1}{2} }$; ${\bm{ \tilde{e}} }_i$ are given in \eqref{24Dec01-1};
$C_{q^{[i]}, \mathbf{p}} = C_{q^{[i]}, \mathbf{p}} [ \psi,\bm{\phi},\bm{\mu}, \bm{\xi} ]$ are constants and $|C_{q^{[i]}, \mathbf{p}}| \le C e^{-\frac{9 \delta^2}{22 T} }$;
$\varphi_{q^{[i]},\mathbf{p},0}(x) \in C^\infty_{\rm c}(\overline{\mathbb{R}^5_+})$ and $\varphi_{q^{[i]},\mathbf{p},0}(x) = 0$ in $\overline{\mathbb{R}^{5}_+  \backslash B_5^+(q^{[i]} ,2\delta)}$.
In particular, for $\mathbf{m} \in \mathbb{N}^5$, there exists a constant $C_{\mathbf{m}}>0$ independent of $T, \delta_0$ such that
\begin{equation*}
|D_x^{\mathbf{m}} \mathcal{T}^{\rm {out}}[ \psi,\bm{\phi},\bm{\mu},\bm{\xi}](x,0)| \le
C_{\mathbf{m}} |\ln T| T^{ \frac{5}{3}l_i+\frac{1}{2} -\frac{\| \mathbf{m}\|_{\ell_1}}{2} }
\sum_{i=1}^{\mathfrak{o}} \1_{|x-q^{[i]}| \le 2\delta}.
\end{equation*}

\end{prop}

\begin{proof}
Recalling \eqref{qd23Dec21-2}, \eqref{qd23Dec21-1}, we have
 \begin{equation*}
\begin{aligned}
&
\psi_i(x,t) ={\bm{d}}_i(s)\cdot\bm{\tilde e}_i(z^{[i]})+\Phi_{i}(z^{[i]},s)
\quad
\mbox{ \ with \ } z^{[i]} = (T-t)^{-\frac 1 2 } \left(x -q^{[i]}\right) , \quad s=- \ln (T-t) ,
\\
&
\psi_i(x,0) ={\bm{b}}_i \cdot\bm{\tilde e}_i\big( T^{-\frac{1}{2}} (x-q^{[i]} ) \big)
\end{aligned}
\end{equation*}
with $\bm{b}_i = {\bm{b}}_i[ \psi,\bm{\phi}, \bm{\mu}, \bm{\xi} ] =\bm{d}_i(s_0)$ given in \eqref{gi10}.
Lemma \ref{g1g2-L2rho-lem} deduces that $| {\bm{b}}_i | \lesssim |\ln T| T^{ \frac{5}{3}l_i+\frac{1}{2} }$ and $\big| {\bm{d}}_i(s)\cdot\bm{\tilde e}_i(z^{[i]}) \big| \lesssim
\1_{|z^{[i]}|\le 2 C_{\tilde{e}}} s e^{-(\frac{5}{3}l_i+\frac{1}{2})s} $.
Combining \eqref{esti-8.23} and $T^{\frac{3}{8}} \ll R^{-\frac{1}{4}}$ due to $T\ll 1$, we have
\begin{equation}\label{outer200}
|\psi_i(x,t)|
\lesssim
R^{-\frac{1}{4}}
\Big[
(T-t)^{l_i}
\langle |z^{[i]}|^{2l_i+2} \rangle
\1_{ |z^{[i]} | \le e^{\frac{l_i s}{2l_i+2} } }
+
\1_{|z^{[i]} | > e^{\frac{l_i s}{2l_i+2} }}
\Big]
,
\quad
i=1,2,\dots, \mathfrak{o}.
\end{equation}

By uniqueness, we set $\psi_{\mathfrak{o}+1} =\psi_{\mathfrak{o}+1}^{(1)}(x,t)+\psi_{\mathfrak{o}+1}^{(2)}(x,t)$ and decompose the equation \eqref{qd23Dec23-3} of $\psi_{\mathfrak{o}+1}$ into the following two equations.
\begin{equation}\label{esti-8.37}
\begin{aligned}
&
\partial_t \psi^{(1)}_{\mathfrak{o}+1}=\Delta_x \psi^{(1)}_{\mathfrak{o}+1}
\mbox{ \ in \ }
\mathbb{R}_+^5 \times (0,T) ,
\quad
    -\partial_{x_5}\psi_{\mathfrak{o}+1}^{(1)}=\mathcal{G}_{2, \mathfrak{o}+1}
    \mbox{ \ on \ } \partial\mathbb{R}_+^5 \times (0,T),
    \quad
    \psi_{\mathfrak{o}+1}^{(1)}(x,0)=0
    \mbox{ \ in \ } \mathbb{R}_+^5 ,
\\
&
        \partial_t \psi_{\mathfrak{o}+1}^{(2)}=\Delta_x \psi_{\mathfrak{o}+1}^{(2)}
        \mbox{ \ in \ } \mathbb{R}_+^5 \times (0,T),
\quad
    -\partial_{x_5}\psi_{\mathfrak{o}+1}^{(2)}=0
    \mbox{ \ on \ } \partial\mathbb{R}_+^5 \times (0,T),
    \quad
    \psi_{\mathfrak{o}+1}^{(2)}(x,0)=\varphi_0(x)
    \mbox{ \ in \ } \mathbb{R}_+^5.
\end{aligned}
\end{equation}
Set
\begin{equation}
\hat \psi_i(x,t):=\psi^{(1)}_{\mathfrak{o}+1}(x,t)+\sum_{j=1, j\not=i}^{\mathfrak{o}}\psi_j(x,t)
\quad \mbox{ \ for \ }
i=1,2,\dots, \mathfrak{o}
.
\end{equation}
By \eqref{outeri}, \eqref{esti-8.37}, $\hat \psi_i$ satisfies
 \begin{equation}\label{qd24Feb21-1}
    \begin{cases}
        \partial_t \hat{\psi}_i=\Delta_x \hat\psi_i+ h_{1,i}(x,t)
        \quad
        \mbox{ \ in \ } \mathbb{R}_+^5 \times (0,T),
 \quad
    -\partial_{x_5}\hat{\psi}_i=h_{2,i}(\tilde{x}, t)
      \quad
    \mbox{ \ on \ } \partial\mathbb{R}_+^5 \times (0,T),
    \\
    \hat{\psi}_i(x,0)=
    \sum_{j=1, j\not=i}^{\mathfrak{o}}
    {\bm{b}}_j \cdot {\bm{\tilde e }}_j\big( T^{-\frac{1}{2}}(x-q^{[j]}) \big)
     \quad
      \mbox{ \ in \ } \mathbb{R}_+^5 ,
    \end{cases}
\end{equation}
where
\begin{equation*}
h_{1,i}(x,t):=\sum_{j=1, j\not=i}^{\mathfrak{o}}\mathcal{G}_{1,j}(x,t) ,
\quad
h_{2,i}(\tilde x, t):=\sum_{j=1, j\not=i}^{\mathfrak{o}+1}
\mathcal{G}_{2,j}(\tilde{x},t).
\end{equation*}
By the properties of ${\bm{\tilde e }}_j$ given in \eqref{24Dec01-1}, \eqref{qd23Dec11-1}, $|\hat{\psi}_i(x,0) | \lesssim
|\ln T| T^{\frac{1}{2}} \sum_{j=1, j\not=i}^{\mathfrak{o}} \1_{|x-q^{[j]}|\le 2C_{\tilde{e}} T^{\frac{1}{2}} } $, $\hat{\psi}_i(x,0) \in C_c^{\infty} \big( \cup_{j=1, j\not=i}^{\mathfrak{o}} \big\{ x \in \overline{\mathbb{R}_+^5} \ \big| \ |x-q^{[j]}|\le 2C_{\tilde{e}} T^{\frac{1}{2}} \big\} \big) $.
By Lemma \ref{G1G2est-lem}, one sees that
\begin{equation*}
\begin{aligned}
&
\left| h_{1,i}(x,t) \right|
\lesssim
R^{-\frac{1}{4}} \left(T-t\right)^{-4-3l_{\rm{max}}}
\1_{ \cup_{j=1, j\not=i}^{\mathfrak{o}} \{|x-q^{[j]}|\le 4\delta \}},
\\
&
\left| h_{2,i}(\tilde x, t) \right|
\lesssim
\1_{\cap_{j=1}^{\mathfrak{o}} \{ |\tilde{x}-\tilde{q}^{[j]}|\ge 4\delta \} }
+
R^{-\frac{1}{4}}
\left(T-t\right)^{-2-l_{\rm{max}}}
\1_{ \cup_{j=1, j\not=i}^{\mathfrak{o}} \{ |\tilde{x}-\tilde{q}^{[i]}|\le 8\delta \} },
\mbox{ \ with \ }
\quad l_{\rm{max}} = \max\limits_{j=1,2,\dots, \mathfrak{o}} l_j  .
\end{aligned}
\end{equation*}

$\hat{\psi}_i$ can be represented by the formula
\eqref{qd2023Dec03-3} with $n=5$ and $t_0=0$.
Since $h_{1,i}(x,t)= \hat{\psi}_i(x,0) = 0$ for $|x-q^{[i]}|\le 3\delta$ and $h_{2,i}(\tilde x, t)= 0$ for $|\tilde{x}-\tilde{q}^{[i]}|\le 3\delta$, applying Lemma \ref{lem-smooth} with $r=3\delta$, we have
\begin{equation*}
\hat{\psi}_i \in C^\infty( \overline{B_5^+(q^{[i]} , 3\delta/2 ) } \times [0,T])
\mbox{ \ and  \ }
|\partial_t^{\iota} D_{x}^{\mathbf{m}}\hat{\psi}_i | \lesssim_{\iota, \mathbf{m}} e^{-\frac{9 \delta^2}{22 T} }
\mbox{ \ in \ } \overline{B_5^+(q^{[i]} , 3\delta/2 ) } \times [0,T]
\mbox{ \ for all \ }
\iota\in \mathbb{N}, \mathbf{m}\in \mathbb{N}^5,
\end{equation*}
where $\hat{\psi}_i$ can be defined by extension naturally for $t=T$, $x\in \overline{B_5^+(q^{[i]} , 3\delta/2 ) }$.
Under the additional assumption $m_5\in 2\mathbb{N} +1$, we have  $\partial_t^{\iota} D_{x}^{\mathbf{m}}\hat{\psi}_i((\tilde{x},0),t) = 0 $ in $\overline{B_4(\tilde{q}^{[i]} , 3\delta/2 )} \times [0,T]$.

By Proposition \ref{lemma-cons1},
for $j=1,2,\dots,\mathfrak{o}$, $\mathbf{p}\in \mathbb{N}^5$, $\| \mathbf{p}\|_{\ell_1} \le 4l_{\rm{max}}+4$, $p_5 \in 2\mathbb{N}$, we take the vanishing adjustment functions $\varphi_{ q^{[j]}, \mathbf{p} }(x,t)$ to satisfy
\begin{equation*}
	\pp_t \varphi_{ q^{[j]}, \mathbf{p} } =\Delta_x \varphi_{q^{[j]}, \mathbf{p} }
	\mbox{ \ in \ }  \mathbb{R}^{5}_+\times (0,T] ,
	\quad
	-\pp_{x_5} \varphi_{q^{[j]}, \mathbf{p} }=0
	\mbox{ \ on \ } \pp \mathbb{R}^{5}_+ \times (0,T] ,
\quad
\varphi_{q^{[j]},\mathbf{p}}(x,0)
=
\varphi_{q^{[j]},\mathbf{p},0}(x)
\mbox{ \ in \ }  \mathbb{R}^{5}_+ ,
\end{equation*}
and the following properties hold:
\begin{enumerate}
\item
$\varphi_{q^{[j]}, \mathbf{p},0}(x)$
is smooth in $\overline{\mathbb{R}_+^5}$ and $\varphi_{q^{[j]}, \mathbf{p},0}(x) = 0$ in $\overline{\mathbb{R}^{5}_+  \backslash B_5^+(q^{[j]} ,2\delta)}$.
	
\item
$\partial_t^\iota D_{x}^{\mathbf{m}} \varphi_{q^{[j]}, \mathbf{p} }((\tilde{x},0),t) = 0$ for $\tilde{x}\in \mathbb{R}^{4}$, $t\in [0,T]$, and $\iota\in \mathbb{N}$, $\mathbf{m}\in \mathbb{N}^5$, $m_5\in 2\mathbb{N}+1$.

\item
$D_x^{\mathbf{m}} \varphi_{q^{[j]} , \mathbf{p}}(q^{[k]},T) = \delta_{\mathbf{p}, \mathbf{m}} \delta_{q^{[j]}, q^{[k]}}$ for $\mathbf{m}\in \mathbb{N}^5$, $\| \mathbf{m}\|_{\ell_1} \le 4 l_{\rm{max}} +4$, $k=1,2,\dots,\mathfrak{o}$.

\item
$\|\partial_t^{\iota} D_x^{\mathbf{m}} \varphi_{ q^{[j]}, \mathbf{p}}   \|_{L^\infty(\mathbb{R}_+^5 \times [0,T] ) } \lesssim 1$ for $\iota\in\mathbb{N}$, $\mathbf{m}\in \mathbb{N}^5$, $2\iota + \| \mathbf{m} \|_{\ell_1} \le 4l_{\rm{max}}+4$.

\end{enumerate}
We take
\begin{equation}\label{qd23Dec30-3}
\psi_{\mathfrak{o}+1}^{(2)} =
\sum_{j=1}^{\mathfrak{o}} \sum_{ \mathbf{p}\in \mathbb{N}^5, \| \mathbf{p}\|_{\ell_1} \le 4l_{\rm{max}}+4, p_5\in 2 \mathbb{N}}
-\left(  D_x^{\mathbf{p} }
\hat{\psi}_j \right)(q^{[j]},T)
\varphi_{ q^{[j]}, \mathbf{p} }(x,t).
\end{equation}
Then for all $i=1,2,\dots, \mathfrak{o} $, $\mathbf{m}\in \mathbb{N}^5$, $\| \mathbf{m}\|_{\ell_1} \le 4l_{\rm{max}}+4$, $m_5 \in 2\mathbb{N}$, we have $D_x^{\mathbf{m} } \big( \hat{\psi}_i + \psi_{\mathfrak{o}+1}^{(2)} \big)(q^{[i]},T) = 0$. One using \eqref{esti-8.37} and the support of $h_{1,i}$ in \eqref{qd24Feb21-1}, it follows that
\begin{equation*}
\partial_t^{\iota} D_x^{\mathbf{m} } \big( \hat{\psi}_i + \psi_{\mathfrak{o}+1}^{(2)} \big)(q^{[i]},T) = \Delta_x^{\iota} D_x^{\mathbf{m} } \big( \hat{\psi}_i + \psi_{\mathfrak{o}+1}^{(2)} \big)(q^{[i]},T) = 0
\mbox{ \ for \ }
\iota\in \mathbb{N}, \mathbf{m}\in \mathbb{N}^5, 2\iota+ \| \mathbf{m}\|_{\ell_1} \le 4l_{\rm{max}}+4, m_5 \in 2\mathbb{N}.
\end{equation*}
Besides, we have
\begin{equation*}
\begin{aligned}
&
\partial_t^{\iota} D_x^{\mathbf{m} } \big( \hat{\psi}_i + \psi_{\mathfrak{o}+1}^{(2)} \big)(q^{[i]},T)  = 0
\mbox{ \ for \ }
\iota\in \mathbb{N}, \mathbf{m}\in \mathbb{N}^5, m_5 \in 2\mathbb{N}+1;
\\
&
\big|\partial_t^{\iota} D_x^{\mathbf{m} } \big( \hat{\psi}_i + \psi_{\mathfrak{o}+1}^{(2)} \big) \big| \lesssim e^{-\frac{9 \delta^2}{22 T} }
\mbox{ \ for \ }
\iota\in \mathbb{N}, \mathbf{m}\in \mathbb{N}^5, 2\iota + \| \mathbf{m} \|_{\ell_1} \le 4l_{\rm{max}}+4
\mbox{ \ in \ }
\overline{B_5^+(q^{[i]} , 3\delta/2 ) } \times [0,T].
\end{aligned}
\end{equation*}
Applying Taylor expansion to $\hat{\psi}_i + \psi_{\mathfrak{o}+1}^{(2)}$ at the point $(q^{[i]},T)$, for $(x,t) \in \overline{B_5^+(q^{[i]} , 3\delta/2 ) } \times [0,T]$, we have
\begin{equation}\label{qd23Dec30-4}
\big|\big(  \hat{\psi}_i + \psi_{\mathfrak{o}+1}^{(2)} \big)(x,t)\big|
\lesssim
	e^{-\frac{9 \delta^2}{22 T} }
	\left[ (T-t)^{2 l_{\rm{max}}+2}+|x-q^{[i]}|^{2l_{\rm{max}}+2}
	\right]
	\lesssim e^{-\frac{9 \delta^2}{22 T} } (T-t)^{l_i+1} \langle |z^{[i]}|^{2l_i+2} \rangle .
\end{equation}

By \eqref{Gk+12estimate}, $\left| \mathcal{G}_{2,\mathfrak{o}+1}
\right|
\le 1 $. Recall the equation of $\psi^{(1)}_{\mathfrak{o}+1}$ in \eqref{esti-8.37}.
By the Green's formula \eqref{qd2023Dec03-3}, we have
$| \psi^{(1)}_{\mathfrak{o}+1} | \lesssim t^{\frac{1}{2}}$.
By \eqref{qd23Dec30-3}, $|\psi_{\mathfrak{o}+1}^{(2)}| \lesssim e^{-\frac{9 \delta^2}{22 T} }$. By \eqref{outer200}, $|\psi_k|
\lesssim R^{-\frac{1}{4}}$, $k=1,2,\dots, \mathfrak{o}$. Hence,
\begin{equation*}
\big| \hat{\psi}_i + \psi_{\mathfrak{o}+1}^{(2)} \big|
\lesssim R^{-\frac{1}{4}}
\mbox{ \ in \ } \overline{\mathbb{R}_+^5} \times (0,T).
\end{equation*}
Combining \eqref{qd23Dec30-4}, we get
\begin{equation}\label{qd23Dec30-6}
	\big| \hat{\psi}_i + \psi_{\mathfrak{o}+1}^{(2)} \big|
	\lesssim e^{-\frac{9 \delta^2}{22 T} } (T-t)^{l_i+1} \langle |z^{[i]}|^{2l_i+2} \rangle \1_{|x-q^{[i]}|\le 3\delta/2}
+
R^{-\frac{1}{4}}
\1_{|x-q^{[i]}| > 3\delta/2} .
\end{equation}

Integrating \eqref{outer200}, \eqref{qd23Dec30-6}, and arbitrary choice of $i=1,2,\dots, \mathfrak{o}$, we complete the proof.
\end{proof}

\begin{lemma}\label{lem-8.12}
Under all assumptions in Lemma \ref{G1G2est-lem},
then for $\mathcal{T}^{\rm {out}}[ \psi,\bm{\phi},\bm{\mu},\bm{\xi}]$ given in \eqref{qd24Jan01-2},
	there exists a constant $C>0$ independent of $T, \delta_0$ such that
	\begin{equation}\label{barrier20}
		|\mathcal{T}^{\rm {out}}[ \psi,\bm{\phi},\bm{\mu},\bm{\xi}](x,t)|\le C  \max\{ R^{-\frac{1}{4}} , \delta_0^{\frac{5}{3}} \} \langle x\rangle^{-\frac{7}{3}}
		\quad
\mbox{ \ for \ }
|x|\ge 99  \max_{i = 1,2,\dots, \mathfrak{o}} |q^{[i]}|
.
	\end{equation}
\end{lemma}

\begin{proof}
	
Recall the equation \eqref{outer-1} of $\psi$, and $\mathcal{G}_{1} =\sum_{i=1}^{\mathfrak{o}} \mathcal{G}_{1,i}$, $\mathcal{G}_{2}=\sum_{i=1}^{\mathfrak{o}+1} \mathcal{G}_{2,i}$.
Denote $r_0 = 99  \max_{i = 1,2,\dots, \mathfrak{o}} |q^{[i]}| $.
By Lemma \ref{G1G2est-lem}  and Proposition \ref{prop-phi0},
\begin{equation*}
\begin{aligned}
&
\mathcal{G}_{1},
\mathcal{T}^{\rm {out}}[ \psi,\bm{\phi},\bm{\mu},\bm{\xi}](x,0) \equiv 0 \mbox{ \ for \ }
x\in \mathbb{R}_+^5\backslash B_5(0,r_0),
\quad
| \mathcal{G}_{2} | =
| \mathcal{G}_{2,\mathfrak{o}+1} |
\le
\delta_0^{\frac{5}{3}}
\langle \tilde{x} \rangle^{-\frac{10}{3}}
\mbox{ \ for \ }
\tilde{x} \in \mathbb{R}^4\backslash B_4(0,r_0) ,
\\
&
|\mathcal{T}^{\rm {out}}[ \psi,\bm{\phi},\bm{\mu},\bm{\xi}](x,t)| \lesssim R^{-\frac{1}{4}}
\mbox{ \ for \ }
x\in \mathbb{R}_+^5 \cap \partial B_5(0,r_0) .
\end{aligned}
\end{equation*}
By Lemma \ref{lemma-barrier1}, set $P(x)= \big|(\tilde{x},x_5+1+\vartheta_1 |\tilde{x}|)\big|^{-\frac{7}{3}}$ with a constant $\vartheta_1>0$ sufficiently small. Then $P(x)$ satisfies
$
-\pp_{x_5}P
\sim
\langle x \rangle^{-\frac{10}{3}} $,
$
 \Delta P \lesssim  - \langle x \rangle^{-\frac{13}{3}} $.
Thus, $C_1 \max\{ R^{-\frac{1}{4}} , \delta_0^{\frac{5}{3}} \} P(x)$ with a large constant $C_1>0$ is a barrier function of \eqref{outer-1} in $\mathbb{R}_+^5\backslash B_5(0,r_0)$.
\end{proof}	

\section{Completion of the proof of Theorem \ref{main}}\label{section-gluing}

Recall $T_{\sigma_0} =T-\sigma_0, \sigma_0\in (0, T)$.
In order to avoid the difficulties of the singularity of the right-hand sides at $t=T$ (see \eqref{Gi1estimate} for example) and
achieve the compactness for the mappings by the regularity theory,
we solve the outer problem in $B_5^+(0,\sigma_0^{-1}) \times (0, T_{\sigma_0})$,
the orthogonal equations in $(0, T_{\sigma_0})$, and the inner problems in $B_5^+(0,2R) \times (0, T_{\sigma_0})$ in order. Then we get a solution $u_{\sigma_0}$ of \eqref{fractextend} in $B_5^+(0,\sigma_0^{-1}) \times (0,T_{\sigma_0})$. Finally, we take $\sigma_0\downarrow 0$ and use the compactness argument to conclude Theorem \ref{main}.

\subsection{Solving the outer problem in $B_5^+(0,\sigma_0^{-1}) \times (0,T_{\sigma_0})$}

\begin{lemma}\label{24Jan01-5-lem}
	
	Suppose that $\delta_0=R^{-\frac{1}{5}}$, $\sigma_0\in (0,T)$,
$ |\mu_{i,1}| \le  (9 C_{\mu_{i,0}})^{-1} (T-t)^{2l_i+2},
|\dot{\mu}_{i,1}| \le  (9 C_{\mu_{i,0}})^{-1}
		 (T-t)^{2l_i+1}  $ in $(0,T_{\sigma_0})$
for $i=1,2,\dots,\mathfrak{o}$, $\bm{\xi}, \dot{\bm{\xi}}$ satisfy \eqref{outer-assumption2} in $(0,T_{\sigma_0})$,
 $\bm{\phi} \in B_{{\rm in},\sigma_0}$,
	then for $T\ll 1$, there exists
	$\psi_{\sigma_0} = \psi_{\sigma_0}[\bm{\phi},\bm{\mu}_{,1},\bm{\xi}] \in \mathcal{X}_{\delta_0,\sigma_0} $
	solving the outer problem \eqref{cyl-outer} in $B_5^+(0,\sigma_0^{-1}) \times (0,T_{\sigma_0})$ with an initial value $\psi_{\sigma_0}(x,0) \in C_c^{\infty}(\overline{\mathbb{R}_+^5})$  satisfying
	\begin{equation*}
		\psi_{\sigma_0}(x,0) =\sum\limits_{i=1}^{\mathfrak{o}}
		{\bm{b}}_{i,\sigma_0}
		\cdot {\bm{ \tilde{e}} }_i\big( T^{-\frac{1}{2}}(x-q^{[i]}) \big)+ \sum_{i=1}^{\mathfrak{o}} \sum_{ \mathbf{p}\in \mathbb{N}^5, \| \mathbf{p}\|_{\ell_1} \le 4l_{\rm{max}}+4, p_5\in 2 \mathbb{N}}
		C_{q^{[i]}, \mathbf{p}, \sigma_0} \varphi_{ q^{[i]}, \mathbf{p},0}(x)  ,
	\end{equation*}
	where
	${\bm{b}}_{i,\sigma_0} = {\bm{b}}_{i,\sigma_0}[ \psi_{\sigma_0} , \bm{\phi} ,\bm{\mu}_{,1}, \bm{\xi} ] $ are constant vectors and $| {\bm{b}}_{i,\sigma_0} | \le C |\ln T| T^{ \frac{5}{3}l_i+\frac{1}{2} }$ with a constant $C>0$ independent of $T, \sigma_0$; ${\bm{ \tilde{e}} }_i$ are given in \eqref{24Dec01-1};
 $C_{q^{[i]}, \mathbf{p},\sigma_0} = C_{q^{[i]}, \mathbf{p},\sigma_0} [ \psi_{\sigma_0} , \bm{\phi} ,\bm{\mu}_{,1}, \bm{\xi} ]$ are constants
and $|C_{q^{[i]}, \mathbf{p},\sigma_0}| \le C e^{-\frac{9 \delta^2}{22 T} }$;
$\varphi_{q^{[i]}, \mathbf{p},0 }(x)\in C_c^\infty(\overline{\mathbb{R}_+^5})$ and $\varphi_{q^{[i]}, \mathbf{p},0 }(x) = 0$ in $\overline{\mathbb{R}^{5}_+  \backslash B_5^+(q^{[i]} ,2\delta)}$.
	
For any compact set $K\subset \overline{\mathbb{R}_+^5} \times [0,T)$, there exists $\sigma_{K}\in (0,T)$ sufficiently small such that for all $\sigma_0\in (0,\sigma_{K})$, $\psi_{\sigma_0}$ are well-defined and uniformly H\"older continuous in $K$.

Moreover, there exists a universal constant $ \varsigma_0 \in (0,1)$ independent of any parameters such that if $\varsigma_1 \in (0,\varsigma_0]$, there exists a constant $C_1>0$ independent of $T, \sigma_0$ such that for all
	\begin{equation*}
		t_* \in (0,T_{\sigma_0}),
		\quad x_*\in \overline{B_5^+(q^{[j]} , R^{-1} (T-t_*)^{1/2 } ) },
		\quad
  j=1,2,\dots,\mathfrak{o},
  \quad
		\rho \in (0,R^{-1} (T-t_*)^{1/2 } ] ,
	\end{equation*}
	we have
	\begin{equation}\label{qd24Jan09-2}
		\left[ \psi_{\sigma_0} \right]_{C^{\varsigma_1,\varsigma_1/2} \big( \overline{B_5^+(x_*,\rho) }
			\times \big( \max\{0, t_*-\rho^2 \}, t_* \big]
			\big) }
		\le C_1
		\rho^{-\varsigma_1}
		\left[ \delta_0 (T-t_*)^{l_j}
		+ \rho^2 R^{-\frac{1}{4}} (T-t_*)^{-3 l_j -4}
		\right].
	\end{equation}
	
\end{lemma}

\begin{remark}

Recall $\mathcal{H}_{2,i}$ given in \eqref{24Jan04-1}. The quantitative H\"older estimate of $\psi_{\sigma_0}$ is required for the application of
Proposition \ref{prop-23Oct24-1} with $n=5$ in solving the inner problems.
\end{remark}

\begin{proof}

Given a function $f$ defined in a domain of $\mathbb{R}^5_+\times (0,T)$, $(0,T)$ or $\ B_5^+(0,2R) \times  (0,T)$,
denote $f_{\star 1}$, $f_{\star 2}$ or $f_{\star 3}$ as the zero extension of $f$ in $ \mathbb{R}^5_+\times (0,T)$, $(0,T)$ or $\ B_5^+(0,2R) \times (0,T)$ respectively. Denote $\xi^{[i]}_{\star \tilde{q}^{[i]}}(t) = \xi^{[i]}(t) \1_{t\in (0, T_{\sigma_0})} + \tilde{q}^{[i]} \1_{t\in [T_{\sigma_0}, T)}$ and $\bm{\xi}_{\star \bm{\tilde{q} }  } = (\xi^{[1]}_{\star  \tilde{q}^{[1]}}, \xi^{[2]}_{\star \tilde{q}^{[2]}},\dots, \xi^{[\mathfrak{o}]}_{\star  \tilde{q}^{[\mathfrak{o}]}})$. Denote $\bm{A} = \big(\bm{\mu}_{,0} + (\bm{\mu}_{,1})_{\star 2}$, $\dot{\bm{\mu}}_{,0} + (\dot{\bm{\mu}}_{,1})_{\star 2}$, $\bm{\xi}_{\star \bm{\tilde{q} } }$, $\dot{\bm{\xi}}_{*2} \big)$. It follows that $\bm{A}$
satisfies \eqref{outer-assumption2} in $(0,T)$.

By Remark \ref{qd24Feb23-1rk}, we make a distinction between $\bm{\mu} (\bm{\xi})$ and $\dot{\bm{\mu}} (\dot{\bm{\xi}})$ in Proposition \ref{prop-phi0} and Lemma \ref{lem-8.12}.
For $g \in \mathcal{X}_{\delta_0,\sigma_0} $,
we denote $\mathcal{T}^{\rm {out}}[g_{\star 1}] = \mathcal{T}^{\rm {out}}[ g_{\star 1} , \bm{\phi}_{\star 3} ,\bm{A} ]$ for brevity in this proof and the following properties hold
	\begin{equation*}
		\begin{aligned}
			&
			|\mathcal{T}^{\rm {out}}[g_{\star 1}](x,t)|
			\lesssim
			\max\{ R^{-\frac{1}{4}} , \delta_0^{\frac{5}{3}} \} \langle x\rangle^{-\frac{7}{3}}   \1_{\cap_{i=1}^{\mathfrak{o}} \big\{ |z^{[i]} | > e^{\frac{l_i s}{2l_i+2} }
				\big\} }
			+
			R^{-\frac{1}{4}}  \sum_{i=1}^{\mathfrak{o}}
			(T-t)^{l_i}
			\langle |z^{[i]}|^{2l_i+2} \rangle
			\1_{ |z^{[i]} | \le e^{\frac{l_i s}{2l_i+2} } }
			,
			\\
			&
			 |D_x^{\mathbf{m}} \mathcal{T}^{\rm {out}}[ g_{\star 1} ](x,0)| \lesssim_{\mathbf{m}}
			|\ln T| T^{ \frac{5}{3}l_i+\frac{1}{2} -\frac{\| \mathbf{m}\|_{\ell_1}}{2} } \sum_{i=1}^{\mathfrak{o}} \1_{|x-q^{[i]}| \le 2\delta}
			\quad \mbox{ \ for \ } \mathbf{m} \in \mathbb{N}^5 ,
			\\
			&  \mathcal{T}^{\rm {out}}[ g_{\star 1}](x,0)=\sum\limits_{i=1}^{\mathfrak{o}}{\bm{b}}_i\cdot {\bm{ \tilde{e}} }_i\big( T^{-\frac{1}{2}}(x-q^{[i]}) \big)+ \sum_{i=1}^{\mathfrak{o}} \sum_{ \mathbf{p}\in \mathbb{N}^5, \| \mathbf{p}\|_{\ell_1} \le 4l_{\rm{max}}+4, p_5\in 2 \mathbb{N}}
			C_{q^{[i]}, \mathbf{p}}
			\varphi_{ q^{[i]}, \mathbf{p},0 }(x) ,
		\end{aligned}
	\end{equation*}
	where
	${\bm{b}}_i = {\bm{b}}_i[ g_{\star 1} , \bm{\phi}_{\star 3} , \bm{A}] $ are constant vectors and  $C_{q^{[i]}, \mathbf{p}} = C_{q^{[i]}, \mathbf{p}} [  g_{\star 1} , \bm{\phi}_{\star 3} ,\bm{A}]$ are constants, which satisfy
	$| {\bm{b}}_i | \lesssim |\ln T| T^{ \frac{5}{3}l_i+\frac{1}{2} }$ and $|C_{q^{[i]}, \mathbf{p}} | \lesssim e^{-\frac{9 \delta^2}{22 T} }$.   	
For $\delta_0=R^{-\frac{1}{5}}$ and $T \ll 1$, then
	$\mathcal{T}^{\rm {out}}[g_{\star 1}] \in \mathcal{X}_{\delta_0, 0} \cap \mathcal{X}_{\delta_0,\sigma_0}$.
	
By $\mathcal{G}_1=\sum_{i=1}^{\mathfrak{o}} \mathcal{G}_{1,i}, \mathcal{G}_2=\sum_{i=1}^{\mathfrak{o} +1} \mathcal{G}_{2,i}$, Lemma \ref{G1G2est-lem},
$\mathcal{G}_{1}, \mathcal{G}_{2}$ are uniformly bounded for $t\in (0,T_{\sigma_0})$. One combining the uniform boundedness of $ |D_x^{\mathbf{m}} \mathcal{T}^{\rm {out}}[ g_{\star 1} ](x,0)|$ with $\mathbf{m} \in \mathbb{N}^5, \| \mathbf{m}\|_{\ell_1} \le 1 $, by the parabolic regularity theory, $\{ \mathcal{T}^{\rm {out}}[g_{\star 1}] \ | \ g\in \mathcal{X}_{\delta_0,\sigma_0}  \}$ is uniformly H\"older continuous in $\overline{B_5^+(0,\sigma_0^{-1}) } \times [0,T_{\sigma_0}) $.
	Hence, $g \mapsto \mathcal{T}^{\rm {out}}[g_{\star 1}]$ is a compact operator from $\mathcal{X}_{\delta_0,\sigma_0} $ to $\mathcal{X}_{\delta_0,\sigma_0} $.The Schauder fixed-point theorem then yields the existence of a fixed point   $\psi_{\sigma_0} = \mathcal{T}^{\rm {out}}[\psi_{\sigma_0, \star 1 } ]$ in $\overline{B_5^+(0,\sigma_0^{-1})} \times (0, T_{\sigma_0}) $ and $\psi_{\sigma_0} \in \mathcal{X}_{\delta_0,\sigma_0} $, which implies that $\psi_{\sigma_0}$
	solves \eqref{cyl-outer} in $B_5^+(0,\sigma_0^{-1}) \times (0 ,T_{\sigma_0}) $. Set $\bm{b}_{i,\sigma_0}:= {\bm{b}}_i[ \psi_{\sigma_0,\star 1} , \bm{\phi}_{\star 3} , \bm{A}]$ and $C_{q^{[i]}, \mathbf{p},\sigma_0}:= C_{q^{[i]}, \mathbf{p}} [  \psi_{\sigma_0,\star 1} , \bm{\phi}_{\star 3} ,\bm{A}]$. Then the initial value has the desired form.
	
The uniform H\"older continuity of $\psi_{\sigma_0}$ in compact sets of
$\overline{\mathbb{R}_+^5} \times [0,T)$ for $\sigma_0$ sufficiently small is straightforward by the parabolic estimate.
	
Finally, we will give the quantitative H\"older estimate of $\psi_{\sigma_0}$ around the blow-up points.
We always assume $T\ll 1$ to make $B_5^+(0,\sigma_0^{-1})$ sufficiently large to make the following estimates well-defined.
By $\mathcal{G}_1=\sum_{i=1}^{\mathfrak{o}} \mathcal{G}_{1,i}, \mathcal{G}_2=\sum_{i=1}^{\mathfrak{o} +1} \mathcal{G}_{2,i}$, Lemma \ref{G1G2est-lem}, and $\mathcal{T}^{\rm {out}}[\psi_{\sigma_0, \star 1 } ] \in \mathcal{X}_{\delta_0, 0} $, for $j=1,2,\dots, \mathfrak{o}$, we have
	\begin{equation*}
		\begin{aligned}
			&
			|\mathcal{G}_1| \lesssim R^{-\frac{1}{4}} (T-t)^{-3 l_j -4}
			& & 	\mbox{ \ in \ }
			Q_{\mathcal{G}_1, j}:=
			\left\{ (x,t) \ | \ t\in (0,T),
			x\in B_5^+(q^{[j]}, (T-t)^{1/2} )  \right\} ,
			\\
			&
			|\mathcal{G}_2| \lesssim R^{-\frac{1}{4}} (T-t)^{-l_j-2}
			& &
			\mbox{ \ in \ }
			Q_{\mathcal{G}_2, j} :=
			\left\{ (\tilde{x}, t) \ | \
			t\in (0, T),
			\tilde{x}\in
			B_4( \tilde{q}^{[j]}, (T-t)^{1/2})
			\right\}  ,
			\\
			&
			|\mathcal{T}^{\rm {out}}[\psi_{\sigma_0, \star 1 } ]| \lesssim \delta_0 (T-t)^{l_j}
			& &
			\mbox{ \ in \ } Q_{\mathcal{G}_1, j} .
		\end{aligned}
	\end{equation*}
	Recall the definition of $\mathcal{T}^{\rm {out}}[\psi_{\sigma_0, \star 1 } ]$ in \eqref{qd24Jan01-2}. Given
	\begin{equation*}
		t_* \in (0,T),
		\quad
		x_*\in \overline{B_5^+(q^{[j]} , R^{-1} (T-t_*)^{1/2 } ) },
		\quad
		\rho \in (0,R^{-1} (T-t_*)^{1/2 } ] ,
	\end{equation*}
	notice that for any $t_1 \in \big( \max\{ 0,t_*-4\rho^2\} , t_* \big]$, by $R^{-1} \ll 1$, we have $T-t_1 \in \left[ T-t_*, 2(T-t_*) \right]$. Then for any $w\in B_5^+(x_*,2\rho)$, it holds that $|w-q^{[j]}| \le |w-x_*| + |x_*-q^{[j]}| \le 3 R^{-1} (T-t_*)^{1/2 } \le (T-t_1)^{1/2}$, which implies $B_5^+(x_*,2\rho)  \times \big( \max\{ 0,t_*-4\rho^2\} , t_* \big] \subset Q_{\mathcal{G}_1, j}$. Similarly,
	$ B_4(\tilde{x}_*,2 \rho ) \times \big( \max\{ 0,t_*-4\rho^2\} , t_* \big] \subset Q_{\mathcal{G}_2, j}$ holds. Hence, for any $(w,t_1) \in   B_5^+(x_*,2\rho)  \times \big( \max\{ 0,t_*-4\rho^2\} , t_* \big] $, we have $| \mathcal{T}^{\rm {out}}[\psi_{\sigma_0, \star 1 } ] (w,t_1) | \lesssim \delta_0 (T-t_1)^{l_j} \sim \delta_0 (T-t_*)^{l_j}$, that is,
	\begin{equation*}
		\| \mathcal{T}^{\rm {out}}[\psi_{\sigma_0, \star 1 } ] \|_{L^\infty
			\big(  B_5^+(x_*,2\rho)  \times \big( \max\{ 0,t_*-4\rho^2\} , t_* \big] \big) }
		\lesssim
		\delta_0 (T-t_*)^{l_j}.
	\end{equation*}
Similarly, we can deduce that
	\begin{equation*}
		\begin{aligned}
			&
			\rho^2 \|\mathcal{G}_1\|_{L^\infty
				\big( B_5^+(x_*,2\rho) \times \big( \max\{ 0,t_*-4\rho^2\} , t_* \big] \big) }
			\lesssim
			\rho^2 R^{-\frac{1}{4}} (T-t_*)^{-3 l_j -4} ,
			\\
			&
			\rho
			\| \mathcal{G}_2 \|_{L^\infty  \big( B_{4}(\tilde{x}_*,2 \rho ) \times \big( \max\{ 0,t_*-4\rho^2\} , t_* \big] \big) }
			\lesssim
			\rho  R^{-\frac{1}{4}} (T-t_*)^{-l_j-2}  .
		\end{aligned}
	\end{equation*}
	Since for any $w\in B_5^+(x_*,2\rho)$, $|w-q^{[j]}| \le 3 R^{-1} (T-t_*)^{1/2 } $, then for any $\alpha\in (0,1)$,
	\begin{equation*}
		\begin{aligned}
			&
			\| \mathcal{T}^{\rm {out}}[\psi_{\sigma_0, \star 1 } ] (\cdot, 0) \|_{L^{\infty} ( B_5^+(x_*,2 \rho ) ) }
			\lesssim
			|\ln T| T^{ \frac{5}{3}l_j+\frac{1}{2} } ,
			\\
			&
			\rho^\alpha [  \mathcal{T}^{\rm {out}}[\psi_{\sigma_0, \star 1 } ] (\cdot, 0) ]_{C^{\alpha } (B_5^+(x_*,2\rho )) }
			\\
			= \ &
			\rho^\alpha
			\Big[ {\bm{b}}_{j,\sigma_0}\cdot {\bm{ \tilde{e}} }_j\big( T^{-\frac{1}{2}}(x-q^{[j]}) \big)+  \sum_{ \mathbf{p}\in \mathbb{N}^5, \| \mathbf{p}\|_{\ell_1} \le 4l_{\rm{max}}+4, p_5\in 2 \mathbb{N}}
			C_{q^{[j]}, \mathbf{p},\sigma_0}
			\varphi_{ q^{[j]}, \mathbf{p},0 }(x) \Big]_{C^{\alpha } (B_5^+(x_*,2\rho )) }
			\\
			\lesssim \ &
			\rho^\alpha
			\big( |\ln T| T^{ \frac{5}{3}l_j+\frac{1}{2} }
			T^{-\frac{1}{2}} \rho^{1-\alpha}
			+
			e^{-\frac{9 \delta^2}{22 T} }
			\rho^{1-\alpha} \big)
			\sim
			\rho
			|\ln T| T^{ \frac{5}{3}l_j } \le R^{-1} |\ln T| T^{ \frac{5}{3}l_j +\frac{1}{2}}  .
		\end{aligned}
	\end{equation*}
	By Lemma \ref{24Jan08-1-lem},
there exist positive universal constants $\varsigma_0\in (0,1), K_1$ independent of any parameters such that if $\varsigma_1 \in (0,\varsigma_0]$, we have
	\begin{equation*}
		\begin{aligned}
			&
			\left[\mathcal{T}^{\rm {out}}[\psi_{\sigma_0, \star 1 } ] \right]_{C^{\varsigma_1,\varsigma_1/2} \big( \overline{B_5^+(x_*,\rho) }
				\times \big( \max\{0, t_*-\rho^2 \}, t_* \big]
				\big) }
			\\
			\le \ & K_1
			\rho^{-\varsigma_1}
			\Big[ \delta_0 (T-t_*)^{l_j}
			+ \rho^2 R^{-\frac{1}{4}} (T-t_*)^{-3 l_j -4}
			+
			\1_{x_{*5} \le 4\rho}
			\
			\rho  R^{-\frac{1}{4}} (T-t_*)^{-l_j-2}
			+
			\1_{\sqrt{t_*} \le 4\rho} \
			|\ln T| T^{ \frac{5}{3}l_j+\frac{1}{2} }	
			\Big].
		\end{aligned}
	\end{equation*}
Therein, by the inequality $c_1+c_2\ge 2\sqrt{c_1 c_2}$ for $c_1, c_2 \ge 0$, it holds that
	\begin{equation*}
		\delta_0 (T-t_*)^{l_j}
		+ \rho^2 R^{-\frac{1}{4}} (T-t_*)^{-3 l_j -4}
		\ge 2 (\delta_0 R^{-\frac{1}{4}})^{\frac{1}{2}} \rho (T-t_*)^{-l_j -2} \ge
		\rho  R^{-\frac{1}{4}} (T-t_*)^{-l_j-2}.
	\end{equation*}
Besides, $\sqrt{t_*} \le 4\rho$,
	$\rho \in (0,R^{-1} (T-t_*)^{1/2 } ]$ deduce $t_*\le 16 R^{-2} T$. Combining $R^{-1}\ll 1$, then $\1_{\sqrt{t_*} \le 4\rho} \
	|\ln T| T^{ \frac{5}{3}l_j+\frac{1}{2} } \lesssim \delta_0 (T-t_*)^{l_j}$.
Thus, we get
	\begin{equation*}
		\left[\mathcal{T}^{\rm {out}}[\psi_{\sigma_0, \star 1 } ] \right]_{C^{\varsigma_1,\varsigma_1/2} \big( \overline{B_5^+(x_*,\rho) }
			\times \big( \max\{0, t_*-\rho^2 \}, t_* \big]
			\big) }
		\lesssim
		\rho^{-\varsigma_1}
		\left[ \delta_0 (T-t_*)^{l_j}
		+ \rho^2 R^{-\frac{1}{4}} (T-t_*)^{-3 l_j -4}
		\right].
	\end{equation*}
Using $\psi_{\sigma_0} = \mathcal{T}^{\rm {out}}[\psi_{\sigma_0, \star 1 } ]$ in $B_5^+(0,\sigma_0^{-1}) \times [0,T_{\sigma_0}) $, we finally achieve \eqref{qd24Jan09-2}. \end{proof}

\subsection{Solving the orthogonal equations in $(0,T_{\sigma_0})$}\label{subsec-orth}

In order to apply Proposition \ref{prop-23Oct24-1} with $n=5$ to solve the inner problems \eqref{cyl-inner} in $B_5^+(0,2R) \times (0,T_{\sigma_0})$, we need to choose suitable $\bm{\mu}, \bm{\xi} $ to attain the orthogonality conditions \eqref{gh-ortho-Oct20}.  In other words, we need to solve the following formal orthogonal equations.
\begin{equation}\label{orth-1}
\begin{aligned}
&
	\int_{\mathbb{R}^n_+} \eta\Big(\frac{y}{4R}\Big)
	\Big(
	\dot{\mu}_i
	\mu_i  Z_n(y )
	+
	\mu_i
	\dot{\xi}^{[i]} \cdot  \left( \nabla_{\tilde{y} } U \right)(\tilde{y},y_n)
	\Big)Z_j(y ) dy
\\
&
	+
	\int_{\mathbb{R}^{n-1}}
\frac{n}{n-2}
\mu_i^{\frac{n}{2}-1}
U^{\frac{2}{n-2}}(\tilde{y},0)
\eta\Big(\frac{\tilde{y} }{4R} , 0 \Big)
\Big(
\psi_{\sigma_0}\big(
(\mu_i \tilde{y} +\xi^{[i]},0),t\big)
+
\Theta_{l_i}\big(
(\mu_i \tilde{y} +\xi^{[i]}- \tilde{q}^{[i]} , 0) ,t \big)
\Big) Z_j(\tilde{y},0) d \tilde{y}
 =0
\end{aligned}
\end{equation}
for $j=1,2,\dots,n$, $t\in (0, T_{\sigma_0})$,
where $\psi_{\sigma_0} = \psi_{\sigma_0}[\bm{\phi},\bm{\mu}_{,1},\bm{\xi}] \in \mathcal{X}_{\delta_0,\sigma_0} $ is given by Lemma \ref{24Jan01-5-lem} with $\bm{\mu}_{,1}, \bm{\xi}$ in a suitable space
to meet the assumption in Lemma \ref{24Jan01-5-lem}.
Here we used $y$ instead of $y^{[i]}$ as the integral variable.

\begin{lemma}\label{lem-9.2}

Suppose that $\delta_0=R^{-\frac{1}{5}}$, $\sigma_0\in (0,T)$,
$\bm{\phi} \in B_{{\rm in},\sigma_0}$,
then for $T\ll 1$, there exist $ \bm{\mu}_{,1,\sigma_0} = \bm{\mu}_{,1,\sigma_0}[\bm{\phi}], \bm{\xi}_{\sigma_0} = \bm{\xi}_{\sigma_0}[\bm{\phi}] $  solving  \eqref{orth-1} with $n=5$ in $(0,T_{\sigma_0})$ and satisfying
	\begin{equation}\label{orth-101}
|\mu_{i,1,\sigma_0}| + |\xi_{\sigma_0}^{[i]}-\tilde{q}^{[i]}| \le  \delta_0^{\frac13} (T-t)^{2l_i+2},
		\quad
|\dot{\mu}_{i,1,\sigma_0}| + |\dot{\xi}_{\sigma_0}^{[i]}|  \le  \delta_0^{\frac13}
		 (T-t)^{2l_i+1}
	\end{equation}
for $t\in [0,T_{\sigma_0})$, $i=1,2,\dots,\mathfrak{o}$. In particular, the ansatz \eqref{outer-assumption2} holds in $[0,T_{\sigma_0})$.

Moreover, for any compact set $K\subset  [0,T)$, there exists $\sigma_{K}\in (0,T)$ sufficiently small such that for all $\sigma_0\in (0,\sigma_{K})$, $\dot{\mu}_{i,1,\sigma_0}$ $(\dot{\xi}_{\sigma_0}^{[i]})$ are well-defined and uniformly H\"older continuous in $K$.

\end{lemma}

\begin{proof}
By the radial property of $\eta(x)$ and the parity of $U(y)$ and $Z_j(y)$ given in \eqref{qd24Jan14-U}, \eqref{qd2023Nov28-1} respectively,
we have
\begin{equation*}
\int_{\mathbb{R}^n_+} \eta (\frac{y}{4R} ) Z_j(y) Z_k(y) dy =0
\mbox{ \ for \ } j,k =1,2,\dots, n, j \ne k;
\
\int_{\mathbb{R}^{n-1}}
U^{\frac{2}{n-2}}(\tilde{y},0)
\eta(\frac{\tilde{y} }{4R} , 0)  Z_j(\tilde{y},0) d \tilde{y}
=0  \mbox{ \ for \ } j=1,2,\dots, n-1.
\end{equation*}
Hence, \eqref{orth-1} is equivalent to
	\begin{align}
& \label{orth-2}
\begin{aligned}
	\dot{\mu}_i = \ & -\frac{n}{n-2}
	\Big(\int_{\mathbb{R}^n_+}  Z_n^2(y) \eta\Big(\frac{y}{4R}\Big) d y\Big)^{-1}\mu_{i}^{\frac{n-4}{2}}\times
	\\ &\int_{\mathbb{R}^{n-1} }U^{\frac{2}{n-2}}(\tilde y,0)\eta \Big(\frac{\tilde{y}}{4R} , 0 \Big)
	\Big[
	\psi_{\sigma_0}\big( (\mu_i\tilde y+ \xi^{[i]},0),t \big)
	+
	\Theta_{l_i}\big((\mu_i \tilde {y}+ \xi^{[i]}-\tilde q^{[i]},0),t\big)  \Big]Z_n(\tilde y,0)d\tilde{y} ,
\end{aligned}
\\
& 	\label{orth-3}
	 \dot{\xi}^{[i]}= \mathcal{S}^{[i]}[\bm{\mu}_{,1},  \bm{\xi} ]:=
		\left(
		\mathcal{S}_1^{[i]}[\bm{\mu}_{,1},  \bm{\xi} ],
		\mathcal{S}_2^{[i]}[\bm{\mu}_{,1},  \bm{\xi}],
		\dots,\mathcal{S}_{n-1}^{[i]}[\bm{\mu}_{,1},  \bm{\xi} ]
		\right)
	\end{align}
for $i=1,2,\dots, \mathfrak{o}$, where for $j=1,2,\dots, n-1$,
	\begin{equation}\label{orth-300}
		\begin{aligned}	
			\mathcal{S}_j^{[i]}[\bm{\mu}_{,1},  \bm{\xi}] := \ & -\frac{n}{n-2}
			\Big(\int_{\mathbb{R}^n_+}  Z_j^2(y) \eta\Big(\frac{y}{4R}\Big) d y\Big)^{-1}\mu_{i}^{\frac{n-4}{2}} \int_{\mathbb{R}^{n-1}}U^{\frac{2}{n-2}}(\tilde y,0)\eta \Big(\frac{\tilde y}{4R} ,0 \Big)
			\Big[\psi_{\sigma_0}\big( (\mu_i\tilde{y} + \xi^{[i]},0),t \big)\\
			&
			+
			\Theta_{l_i}\big((\mu_i \tilde y+ \xi^{[i]}
			-
			\tilde{q}^{[i]},0),t\big)-\Theta_{l_i}(0,t)\Big]Z_j(\tilde{y},0)d\tilde{y} .
		\end{aligned}
	\end{equation}

Recalling \eqref{qd23Jan01-2} and $\Theta_{l_i}(0,t)=-(T-t)^{l_i}$, we have
\begin{equation}\label{qd24Jan01-6}
\dot{\mu}_{i,0}
=
-
\frac{n}{n-2}
\Big(\int_{\mathbb{R}_+^n}
 Z_n^2(y) \eta\Big(\frac{y}{4R}\Big) dy \Big)^{-1}
\Theta_{l_i}(0,t) \mu_{i,0}^{\frac{n-4}{2}  }  \int_{\mathbb{R}^{n-1}}
U^{\frac{2}{n-2}}(\tilde{y},0)
\eta\Big(\frac{\tilde{y}}{4R} , 0 \Big)  Z_n(\tilde{y},0) d\tilde{y}  .
\end{equation}
\eqref{orth-2} minus \eqref{qd24Jan01-6} implies
\begin{equation}\label{qd24Feb24-1}
	\dot{\mu}_{i,1} =   \mathcal{F}_i[\bm{\mu}_{,1}, \bm{\xi} ](t)
-\frac{n}{n-2}
\Big(\int_{\mathbb{R}^n_+}Z_n^2(y) \eta\Big(\frac{y}{4R}\Big) d y\Big)^{-1}
\Theta_{l_i} (0,t )
\frac{n-4}{2}
\mu_{i,0}^{\frac{n-6}{2} }
\mu_{i,1}
\int_{\mathbb{R}^{n-1} }U^{\frac{2}{n-2}}(\tilde y,0)\eta \Big(\frac{\tilde{y}}{4R} , 0 \Big)
Z_n(\tilde y,0)d\tilde{y}
,
\end{equation}
where
\begin{equation}\label{orth-500}
\begin{aligned}
&
\mathcal{F}_i[\bm{\mu}_{,1}, \bm{\xi} ](t)
:=
\frac{-n}{n-2}
\Big(\int_{\mathbb{R}^n_+}Z_n^2(y) \eta\Big(\frac{y}{4R}\Big) d y\Big)^{-1}
\bigg\{
\mu_{i}^{\frac{n-4}{2}}
\int_{\mathbb{R}^{n-1} }U^{\frac{2}{n-2}}(\tilde y,0)\eta \Big(\frac{\tilde{y}}{4R} , 0 \Big)
\psi_{\sigma_0}\big( (\mu_i\tilde y+ \xi^{[i]},0),t \big) Z_n(\tilde y,0)d\tilde{y}
\\
& +
\mu_{i}^{\frac{n-4}{2}}
\int_{\mathbb{R}^{n-1} }U^{\frac{2}{n-2}}(\tilde y,0)\eta \Big(\frac{\tilde{y}}{4R} , 0 \Big)
\big[
\Theta_{l_i}\big((\mu_i \tilde {y}+ \xi^{[i]}-\tilde q^{[i]},0),t\big)
-
\Theta_{l_i} (0,t )
\big]  Z_n(\tilde y,0)d\tilde{y}
\\
& +
\Theta_{l_i} (0,t )
\Big(
\mu_{i}^{\frac{n-4}{2}}
-
\mu_{i,0}^{\frac{n-4}{2} }
-
\frac{n-4}{2}
\mu_{i,0}^{\frac{n-6}{2} }
\mu_{i,1}
\Big)
\int_{\mathbb{R}^{n-1} }U^{\frac{2}{n-2}}(\tilde y,0)\eta \Big(\frac{\tilde{y}}{4R} , 0 \Big)
Z_n(\tilde y,0)d\tilde{y}
\bigg\} .
\end{aligned}
\end{equation}
By \eqref{24Jan01-6}, \eqref{24Jan01-7}, then \eqref{qd24Feb24-1} is equivalent to
\begin{equation}\label{orth-5}
\dot{\mu}_{i,1}
+
\beta(t)  \mu_{i,1} =   \mathcal{F}_i[\bm{\mu}_{,1}, \bm{\xi} ](t)
\quad
\mbox{ \ with \ }
\quad
\beta(t) := \frac{n-4}{6-n} (l_i +1) (T-t)^{-1}.
\end{equation}

To obtain a solution of \eqref{orth-3}, \eqref{orth-5}, it suffices to solve the following fixed-point problem   about $\dot{\bm{\mu}}_{,1}, \dot{\bm{\xi} }$.
	\begin{equation}\label{orth-6}
		\begin{aligned}
&
( \dot{\bm{\mu}}_{,1} ,\dot{\bm{\xi} } ) =
\mathcal{T}^{\rm{ort}}[\bm{\mu}_{,1},\bm{\xi} ] := \big( (\mathcal{S}_1[\bm{\mu}_{,1},\bm{\xi} ] ,\mathcal{S}_2[\bm{\mu}_{,1},\bm{\xi} ] ,\dots, \mathcal{S}_{\mathfrak{o}}[\bm{\mu}_{,1},\bm{\xi} ] ), ( \mathcal{S}^{[1]}[\bm{\mu}_{,1},\bm{\xi} ] , \mathcal{S}^{[2]}[\bm{\mu}_{,1},\bm{\xi} ] ,\dots, \mathcal{S}^{[\mathfrak{o}]}[\bm{\mu}_{,1},\bm{\xi} ] ) \big) ,
\\
&
\mathcal{S}_i[\bm{\mu}_{,1},\bm{\xi} ]
:=\frac{d}{dt}\Big(\int_{T_{\sigma_0}}^te^{\int_t^s\beta(a)da}\mathcal{F}_i[\bm{\mu}_{,1}, \bm{\xi} ](s)ds\Big)
=
-\beta(t)\int_{T_{\sigma_0}}^te^{\int_t^s\beta(a)da}\mathcal{F}_i[\bm{\mu}_{,1}, \bm{\xi} ](s)ds
			+
			\mathcal{F}_i[\bm{\mu}_{,1}, \bm{\xi} ](t) ,
\\
&
\mu_{i,1}=\mu_{i,1}[\dot{\mu}_{i,1}](t)=\int_{T_{\sigma_0}}^t\dot{\mu}_{i,1}(a)da,
\quad
\xi^{[i]}=  \xi^{[i]}[\dot{\xi}^{[i]}](t)=\int_{T_{\sigma_0}}^t\dot{\xi}^{[i]}(a)da+\tilde q^{[i]} ,
\quad
i =1,2,\dots, \mathfrak{o} .
		\end{aligned}
	\end{equation}
We will solve $\dot{\bm{\mu}}_{,1}, \dot{\bm{\xi} }$ in suitable space to make the above integrals and $\psi_{\sigma_0} =\psi_{\sigma_0}[\bm{\phi},\bm{\mu}_{,1},
\bm{\xi}]\in \mathcal{X}_{\delta_0,\sigma_0} $ given by Lemma \ref{24Jan01-5-lem} well-defined.
Set the norm
\begin{equation*}
	\|f\|_{l_i,\sigma_0}:=\sup_{s\in (0,T_{\sigma_0} )}  (T-s)^{-(2l_i+1)} |f(s)| .
\end{equation*}
We will find a solution $( \dot{\bm{\mu}}_{,1} ,\dot{\bm{\xi} } )$ of \eqref{orth-6} in the space
$B_{*1,\sigma_0} \times B_{*2,\sigma_0}$, where
\begin{equation*}
	\begin{aligned}
		&
		B_{*1,\sigma_0}:=
		\Big\{ (f_1,f_2,\dots, f_{\mathfrak{o}})  \ \big| \
		f_i \in C\big( (0,T_{\sigma_0} ),\mathbb{R} \big), i=1,2,\dots, \mathfrak{o} ,
		\
		\max_{i=1,2,\dots, \mathfrak{o}} \|f_i \|_{l_i,\sigma_0}
		\le \delta_0^{\frac{1}{2}}
		\Big\},
		\\
		&
		B_{*2,\sigma_0} :=
		\Big\{
		(\mathbf{f}^{[1]}, \mathbf{f}^{[2]}, \dots, \mathbf{f}^{[\mathfrak{o}]})
		\ \big| \
		\mathbf{f}^{[i]} \in C\big( (0,T_{\sigma_0} ),\mathbb{R}^{n-1} \big), i=1,2,\dots, \mathfrak{o} ,
		\
		\max_{i=1,2,\dots, \mathfrak{o}} \|\mathbf{f}^{[i]} \|_{l_i,\sigma_0}
		\le \delta_0^{\frac{1}{2}}
		\Big\} .
	\end{aligned}
\end{equation*}

Hereafter, we plug in $n=5$.
Given any $(\dot{\bm{\mu}}_1,\dot{\bm{ \xi}}  )\in B_{*1,\sigma_0} \times B_{*2,\sigma_0}$, for $t\in (0,T_{\sigma_0} )$, we have
	\begin{equation}\label{orth-602}
		|\mu_{i,1}|
\le
\int_t^{T_{\sigma_0}} |\dot{\mu}_{i,1}(a)|da
		 \lesssim \delta_0^\frac12(T-t)^{2l_i+2}   ,
		\quad
		|\xi^{[i]}-\tilde{q}^{[i]}|
		\le
		\int_t^{T_{\sigma_0}}
		|\dot{\xi}^{[i]}(a)| da
		\lesssim \delta_0^\frac12(T-t)^{2l_i+2}.
	\end{equation}
In particular, the integrals about $\mu_{i,1}, \xi^{[i]}$ in \eqref{orth-6} are well-defined.
Due to the small quantity $\delta_0^{\frac{1}{2}}$ and $T\ll 1$, $\bm{\mu}, \dot{\bm{\mu}}, \bm{\xi}, \dot{\bm{\xi}} $ satisfy \eqref{outer-assumption2} in $(0,T_{\sigma_0} )$ and the assumption in Lemma \ref{24Jan01-5-lem} holds, which implies that $\psi_{\sigma_0} =\psi_{\sigma_0}[\bm{\phi},\bm{\mu}_{,1},
\bm{\xi}]\in \mathcal{X}_{\delta_0,\sigma_0} $ given by Lemma \ref{24Jan01-5-lem} is well-defined.
For $|\tilde{y}|\le 8R$, $t\in (0,T_{\sigma_0} )$,
	\begin{equation*}
		|(\mu_i\tilde y+ \xi^{[i]},0)- q^{[i]}|
		\le |\mu_i \tilde{y}|+
		 |\xi^{[i]}-\tilde {q}^{[i]}|
		 \lesssim  R (T-t)^{2l_i+2}\ll(T-t)^{\frac{1}{2l_i+2}}.
	\end{equation*}
Combining $\psi_{\sigma_0} \in \mathcal{X}_{\delta_0, \sigma_0}$ and $\Theta_{l_i} $ given in \eqref{eq-heat-1} with $n=5$, we have
	\begin{align}
& \label{orth-7}
		|\psi_{\sigma_0} \big( (\mu_i\tilde y+ \xi^{[i]},0),t \big)|\le \delta_0(T-t)^{l_i}
		\big\langle
		(T-t)^{-\frac{1}{2}}
		|\mu_i\tilde{y}+  \xi^{[i]}-\tilde{q}^{[i]}|
		\big\rangle^{2l_i+2}
		\sim \delta_0(T-t)^{l_i} ,
\\
&
\label{orth-8}
| \Theta_{l_i}\big((\mu_i \tilde y+ \xi^{[i]}
-
\tilde{q}^{[i]},0),t\big)-\Theta_{l_i}(0,t) |\lesssim(T-t)^{l_i}\sum_{j=1}^{l_i}\left[
(T-t)^{-1} |\mu_i\tilde{y}+ \xi^{[i]}-\tilde{ q}^{[i]}|^2 \right]^j
\lesssim
R^2 (T-t)^{5 l_i +3}.
	\end{align}
It then follows from \eqref{outer-assumption2}, \eqref{orth-7} and \eqref{orth-8} that $\mathcal{S}_j^{[i]}[\bm{\mu}_{,1},  \bm{\xi}]$ given in \eqref{orth-300} has the upper bound
	\begin{equation}\label{orth-9}
		|\mathcal{S}_j^{[i]}[\bm{\mu}_{,1},  \bm{\xi}] | \lesssim \delta_0(T-t)^{2l_i+1}.
	\end{equation}

By \eqref{qd24Jan02-1}, \eqref{orth-602}, and $\delta_0 \ll 1$,
$$|\mu_{i}^{\frac{1}{2}}-\mu_{i,0}^{\frac{1}{2}}- 2^{-1} \mu_{i,0}^{-\frac{1}{2}}\mu_{i,1}|
	\lesssim
\mu_{i,0}^{-\frac{3}{2}} \mu_{i,1}^2	 \lesssim \delta_0(T-t)^{l_i+1}.$$
Combining \eqref{outer-assumption2}, \eqref{orth-7}, \eqref{orth-8}, $\mathcal{F}_i[\bm{\mu}_{,1}, \bm{\xi} ](t)$ given in \eqref{orth-500} has the upper bound
	\begin{equation}\label{orth-10}
|\mathcal{F}_i[\bm{\mu}_{,1}, \bm{\xi} ](t)| \lesssim \delta_0(T-t)^{2l_i+1} .
	\end{equation}
	This together with $\beta(t)$ given in \eqref{orth-5} implies that for $t\in (0,T_{\sigma_0} )$,
\begin{equation}\label{orth-11}
\Big|\beta(t)\int_{T_{\sigma_0}}^te^{\int_t^s\beta(a)da}\mathcal{F}_i[\bm{\mu}_{,1}, \bm{\xi} ](s)ds \Big|
\lesssim
(T-t)^{-1}
\int_t^{T_{\sigma_0}} e^{ (l_i +1) \int_t^s (T-a)^{-1} da}
\delta_0(T-s)^{2l_i+1}
ds
\lesssim \delta_0 (T-t)^{2l_i +1} .
\end{equation}
	
By \eqref{orth-9}, \eqref{orth-10}, \eqref{orth-11}, for $\delta_0\ll 1$, we deduce that
$ \mathcal{T}^{\rm{ort}}[\bm{\mu}_{,1},\bm{\xi} ] \in B_{*1,\sigma_0} \times B_{*2,\sigma_0} $.

For all $(\dot{\bm{\mu}}_{,1},\dot{\bm{ \xi}}  )\in B_{*1,\sigma_0} \times B_{*2,\sigma_0}$, $\psi_{\sigma_0}[\bm{\phi}, \bm{\mu}_{,1}, \bm{\xi}  ] $ are uniformly H\"older continuous
in $\overline{B_5^+(0,\sigma_0^{-1}/2) } \times [0,T_{\sigma_0} )$. By \eqref{orth-6},  $\mathcal{T}^{\rm{ort}}[\bm{\mu}_{,1},\bm{\xi} ]$
is uniformly H\"older continuous
in $[0,T_{\sigma_0} )$. Hence, $(\bm{\mu}_{,1},\bm{\xi}) \mapsto \mathcal{T}^{\rm{ort}}[\bm{\mu}_{,1},\bm{\xi}]$ is a compact mapping from $B_{*1,\sigma_0} \times B_{*2,\sigma_0}$ to itself by the Arzel\`a-Ascoli theorem. By the Schauder fixed-point theorem, we find a solution $(\dot{\bm{\mu}}_{,1,\sigma_0}, \dot{\bm{\xi}}_{\sigma_0} )$  of \eqref{orth-6} in $B_{*1,\sigma_0} \times B_{*2,\sigma_0}$. Combining \eqref{orth-602}, $\delta_0=R^{-\frac{1}{5}}$, $T\ll1$, and
uniform H\"older continuity of $\mathcal{T}^{\rm{ort}}[\bm{\mu}_{,1},\bm{\xi} ]$
in $[0,T_{\sigma_0} )$, we conclude the final estimate.
\end{proof}

\subsection{Solving the inner problems in $B_5^+(0,2R) \times (0,T_{\sigma_0})$}

Now the inner problems \eqref{cyl-inner} for $t\in (0, T_{\sigma_0})$ with $n=5$ can be rewritten as the following form formally.
For $i=1,2,\dots, \mathfrak{o}$,
\begin{equation}\label{z24Feb24-3}
\begin{cases}
(\mu_{i,\sigma_0}[\bm{\phi}])^2  \pp_t \phi_i
=
\Delta_{y_{\sigma_0}^{[i]}} \phi_i
+
\mathcal{H}_{1,i}[\bm{\phi} ](y_{\sigma_0}^{[i]},t)
& \mbox{ \ for \ }
t\in \left(0,T_{\sigma_0}\right) ,
y_{\sigma_0}^{[i]}  \in B_5^+(0,2R) ,
\\
-\pp_{y_{5,\sigma_0}^{[i]}} \phi_i
=
\frac{5}{3}
U^{\frac{2}{3}}(y_{\sigma_0}^{[i]}) \phi_i
+
\mathcal{H}_{2,i}[\bm{\phi} ](\tilde{y}_{\sigma_0}^{[i]},t)
& \mbox{ \ for \ }
t\in \left(0,T_{\sigma_0} \right) ,
y_{\sigma_0}^{[i]}  \in B_{4}(0,2R) \times \{ 0 \},
\end{cases}
\end{equation}
where $ \bm{\mu}_{,1,\sigma_0}[\bm{\phi}]$, $\bm{\xi}_{\sigma_0}[\bm{\phi}] $, $ \psi_{\sigma_0}[\bm{\phi}]=\psi_{\sigma_0}[\bm{\phi}, \bm{\mu}_{,1,\sigma_0}[\bm{\phi}] , \bm{\xi}_{\sigma_0}[\bm{\phi}] ]$ are given by Lemma \ref{lem-9.2} and Lemma \ref{24Jan01-5-lem} respectively, $\mu_{i,\sigma_0}[\bm{\phi}] = \mu_{i,0}+ \mu_{i,1,\sigma_0}[\bm{\phi}]$,
$ y_{\sigma_0}^{[i]} =\frac{x-(\xi_{\sigma_0}^{[i]}[\bm{\phi}] ,0) }{\mu_{i,\sigma_0}[\bm{\phi}] }$,
$\tilde{y}_{\sigma_0}^{[i]}$ is the first four components of $y_{\sigma_0}^{[i]}$, $y_{5,\sigma_0}^{[i]}$ is the fifth component of $y_{\sigma_0}^{[i]}$,
\begin{equation}
	\begin{aligned}
		\mathcal{H}_{1,i}[\bm{\phi} ](y_{\sigma_0}^{[i]},t)
		= \ &
		\eta\Big(\frac{y_{\sigma_0}^{[i]} }{4R}\Big)
		\Big(
		\dot{\mu}_{i,\sigma_0}[\bm{\phi} ]
		\mu_{i,\sigma_0}[\bm{\phi} ]  Z_5(y_{\sigma_0}^{[i]})
		+
		\mu_{i,\sigma_0}[\bm{\phi} ]
		\dot{\xi}_{i,\sigma_0}^{[i]}[\bm{\phi} ] \cdot  \left( \nabla_{\tilde{y}_{\sigma_0}^{[i]}} U \right)(\tilde{y}_{\sigma_0}^{[i]},y_{5,\sigma_0}^{[i]})
		\Big)
		,
		\\
  \mathcal{H}_{2,i}[\bm{\phi} ](\tilde{y}_{\sigma_0}^{[i]},t) = \ &
		\frac{5}{3}
		\mu_{i,\sigma_0}^{\frac{3}{2}}[\bm{\phi}]
		U^{\frac{2}{3}}(\tilde{y}_{\sigma_0}^{[i]}, 0)
		\eta\Big(\frac{\tilde{y}_{\sigma_0}^{[i]}}{4R} , 0 \Big)
		\Big(
		\Theta_{l_i}\big( ( \mu_{i,\sigma_0}[\bm{\phi}] \tilde{y}_{\sigma_0}^{[i]} + \xi_{\sigma_0}^{[i]}[\bm{\phi}] -\tilde{q}^{[i]} ,0),t \big)
		\\
		&
		+\psi_{\sigma_0}[\bm{\phi}]\big( (\mu_{i,\sigma_0}[\bm{\phi}] \tilde{y}_{\sigma_0}^{[i]} + \xi^{[i]}[\bm{\phi}] ,0),t \big) \Big).
	\end{aligned}
\end{equation}

\begin{lemma}\label{24Jan05-1-lem}

Suppose that $\delta_0=R^{-\frac{1}{5}}$, $\sigma_0\in (0,T)$,
then for $T\ll 1$, there exists a solution $\bm{\phi}_{\sigma_0} = (\phi_{1,\sigma_0}, \phi_{2,\sigma_0},\dots, \phi_{\mathfrak{o},\sigma_0}) \in B_{{\rm in},\sigma_0}$ of \eqref{z24Feb24-3}. Moreover, for $i=1,2,\dots, \mathfrak{o}$, the initial value $\phi_{i,\sigma_0}(\cdot, 0) = C_{{\rm in}, i,\sigma_0} \tilde{Z}_0$ in $B_5^+(0,2R)$, where $\tilde{Z}_0 \in C^{\infty}(\overline{B_5^+(0,2R)})$ and a constant $C_{{\rm in}, i,\sigma_0}$ satisfies $|C_{{\rm in}, i,\sigma_0}| \le  C T^{4l_i+3} R^{\frac{3}{2}}$ with a constant $C>0$ independent of $T, \sigma_0$.

Furthermore, for any compact set $K\subset [0,T)$, there exists $\sigma_{K}\in (0,T)$ sufficiently small such that for all $\sigma_0\in (0,\sigma_{K})$, $i=1,2,\dots, \mathfrak{o}$, $\phi_{i,\sigma_0}$ are well-defined and uniformly $C^{1+\varsigma_1,(1+\varsigma_1)/2}$ bounded in $\overline{B_5^+(0,2R)} \times K$ with a constant $\varsigma_1\in(0,1/10)$ independent of $T, \sigma_0$.

\end{lemma}

\begin{proof}

For any $\bm{\phi} \in B_{{\rm in},\sigma_0}$, $ \bm{\mu}_{,1,\sigma_0}[\bm{\phi}],  \bm{\xi}_{\sigma_0}[\bm{\phi}], \psi_{\sigma_0}[\bm{\phi} ]$ are well-defined.
For $i=1,2,\dots, \mathfrak{o}$, set the new time variable
\begin{equation}\label{qd24Jan04-5}
	\tau_{i,\sigma_0} = \tau_{i,\sigma_0}(t) := \int_0^t \left(\mu_{i,\sigma_0}[\bm{\phi}](s) \right)^{-2} ds
	+
	C_{\mu_{i,0}}^2 T^{-4 l_i -3} \sim (T-t)^{-4l_i-3}
	\mbox{ \ for \ }
	t\in (0,T_{\sigma_0}),
\end{equation}
where we used \eqref{outer-assumption2} for the last step. For brevity, we use $y^{[i]}, \tau_i$ to denote $y_{\sigma_0}^{[i]}, \tau_{i, \sigma_0}$ in this proof. We set the corresponding inverse function about $\tau_i$ as
\begin{equation*}
t=t_i(\tau_i) \mbox{ \ for \ } \tau_i \in \left(\tau_i(0), \tau_i(T_{\sigma_0}) \right),
\end{equation*}
which is monotonically increasing in $\tau_i$.
Since $(\mu_{i,\sigma_0}[\bm{\phi}])^2  \pp_t \phi_i = \partial_{\tau_i} \phi_i $, we consider
\begin{equation}
	\begin{cases}
		\partial_{\tau_i} \phi_i
		=
		\Delta_{y^{[i]}} \phi_i
		+   \mathcal{H}_{1,i} [\bm{\phi}](y^{[i]},t_i(\tau_i))
		& \mbox{ \ for \ }
		\tau_i \in \left(\tau_i(0), \tau_i(T_{\sigma_0}) \right) ,
		y^{[i]}  \in \mathbb{R}_+^5 ,
		\\
		-\pp_{y_{5}^{[i]}} \phi_i
		=
		\frac{5}{3}
		U^{\frac{2}{3}}(y^{[i]})    \phi_i
		+ \mathcal{H}_{2,i} [\bm{\phi} ](\tilde{y}^{[i]},t_i(\tau_i))
		& \mbox{ \ for \ }
		\tau_i \in \left(\tau_i(0), \tau_i(T_{\sigma_0}) \right) ,
		y^{[i]}  \in \partial \mathbb{R}_+^5.
	\end{cases}
\end{equation}

By \eqref{qd24Jan04-5},
\eqref{qd24Jan02-1}, \eqref{orth-101}, we make the following preparation for estimates about $\mathcal{H}_{1,i}[\bm{\phi}], \mathcal{H}_{2,i}[\bm{\phi}]$.
\begin{equation}\label{qd24Jan09-3}
\begin{aligned}
&
T-t_i(\tau_i) \sim \tau_i^{-\frac{1}{4l_i +3}} ;
\quad
T-t_i(a_1) \sim T-t_i(a_2)
\quad
\mbox{ \ for \ }  a_1, a_2 \in \left(\tau_i(0), \tau_i(T_{\sigma_0}) \right), \ a_1 \sim  a_2 ;
\\
&
\frac{d }{d \tau_i} t_i(\tau_i)
=\mu_{i,\sigma_0}^2[\bm{\phi}](t_i(\tau_i) ) ;
\quad
\mu_{i,\sigma_0}[\bm{\phi}]( t_i(\tau_i))
\sim
\left(T-t_i(\tau_i) \right)^{2 l_i+2}
\sim \tau_i^{-\frac{2 l_i+2}{4l_i +3}}
;
\\
&
\Big|
\frac{d}{d \tau_i}
\Big(
\mu_{i,\sigma_0}[\bm{\phi}](t_i(\tau_i)) \Big) \Big|
=
\Big|
\frac{d}{d t}
\Big(
\mu_{i,\sigma_0}[\bm{\phi}] \Big) (t_i(\tau_i))
\frac{d}{d \tau_i} t_i(\tau_i)
 \Big|
 \lesssim
(T-t_i(\tau_i))^{6 l_i +5 }
\sim
\tau_i^{-\frac{6 l_i +5}{4l_i+3 } } ;
\\
&
|\xi_{\sigma_0}^{[i]}[\bm{\phi}](t_i(\tau_i))  -\tilde{q}^{[i]}| \le  \delta_0^{\frac13} (T-t_i(\tau_i))^{2l_i+2}
\sim
\delta_0^{\frac13}
\tau_i^{-\frac{2 l_i+2}{4l_i +3}}  ;
\\
&
\Big|
\frac{d}{d \tau_i}
\Big(
\xi_{\sigma_0}^{[i]}[\bm{\phi}](t_i(\tau_i)) \Big) \Big|
=
\Big|
\frac{d}{d t}
\Big( \xi_{\sigma_0}^{[i]}[\bm{\phi}] \Big) (t_i(\tau_i))
\frac{d}{d \tau_i} t_i(\tau_i)
\Big|
\lesssim
\delta_0^{\frac13}
(T-t_i(\tau_i))^{6 l_i +5 }
\sim
\delta_0^{\frac13}
\tau_i^{-\frac{6 l_i +5}{4l_i+3 } } .
\end{aligned}
\end{equation}

For $\mathcal{H}_{1,i} [\bm{\phi}]$,
\begin{equation}\label{24Jan04-7}
	\big|\mathcal{H}_{1,i} [\bm{\phi}] (y^{[i]},t_i(\tau_i)) \big|
	\lesssim
	(T-t(\tau_i))^{4l_i+3}  \langle y^{[i]} \rangle^{-3} \eta\Big(\frac{y^{[i]}}{4R} \Big)
\lesssim
R^{\frac{3}{2} }
\tau_i^{-1}  \langle y^{[i]} \rangle^{-\frac{9}{2}} \eta\Big(\frac{y^{[i]}}{4R} \Big),
\end{equation}
where we used the trick in \cite[p.19]{finite-frac2020} for the last step.

For the application of Proposition \ref{prop-23Oct24-1}, pointwise and H\"older estimates are required for $\mathcal{H}_{2,i} [\bm{\phi} ](\tilde{y}^{[i]},t_i(\tau_i))$.
For $|\tilde{y}^{[i]}| \le \tau_i^{1/2}$, denote $Q_{\tilde{y}^{[i]}, \tau_i}  :=
\left\{ (\tilde{v},s)\in \mathbb{R}^{4} \times \left(\tau_i(0), \tau_i(T_{\sigma_0}) \right) \ \big| \ |\tilde{v}-\tilde{y}^{[i]}|\le \frac{|\tilde{y}^{[i]}|}{2} , \tau_i- \frac{|\tilde{y}^{[i]}|^2}{4}\le s \le \tau_i \right\}$.
For any $(\tilde{v}^{[j]} , s_j )\in Q_{\tilde{y}^{[i]}, \tau_i}$,  $j=1,2$ with $s_2\le s_1$, it holds that $t_i(s_2) \le t_i(s_1)$,
$|\tilde{v}^{[j]}| \in \big[ |\tilde{y}^{[i]}|/2, 3 |\tilde{y}^{[i]}|/2 \big]$,
and
$s_j \in [\tau_i- \frac{|\tilde{y}^{[i]}|^2}{4}, \tau_i] \subset [3 \tau_i/4 , \tau_i]$.
Due to $\eta(\frac{\tilde{y}^{[i]}}{4R} , 0 )$, $[\mathcal{H}_{2,i} \big[\bm{\phi} ](\cdot,t_i(\cdot)) \big]_{C^{\alpha,\alpha/2} (Q_{\tilde{y}^{[i]}, \tau_i}) } = 0$ for $|\tilde{y}^{[i]}|>16R$ and $\mathcal{H}_{2,i} [\bm{\phi} ](\tilde{y}^{[i]},t_i(\tau_i))=0$ for $|\tilde{y}^{[i]}|>8R$. Hence, we always assume $|\tilde{y}^{[i]}|\le 16R$
in this proof.
By \eqref{qd24Jan09-3},
\begin{equation}\label{qd24Jan-2}
	\begin{aligned}
		&
		\mu_{i,\sigma_0}^{\frac{3}{2}}[\bm{\phi}]( t_i(s_j))
		\sim
		\tau_i^{-\frac{3 l_i+3}{4l_i +3}}
		,
		\\
		&
		\big| \mu_{i,\sigma_0}^{\frac{3}{2}}[\bm{\phi}]( t_i(s_1))
		-
		\mu_{i,\sigma_0}^{\frac{3}{2}}[\bm{\phi}]( t_i(s_2))
		\big|
		=
		\Big| \frac{d}{d s}
		\big(
		\mu_{i,\sigma_0}^{\frac{3}{2}}[\bm{\phi}]( t_i(s)) \big)\big|_{s= \theta s_1 + (1-\theta ) s_2 } (s_1-s_2)
		\Big|
		\\
		\lesssim \ &
		\left[ \theta s_1 + (1-\theta ) s_2 \right]^{-\frac{l_i+1}{4l_i +3}}
		\left[ \theta s_1 + (1-\theta ) s_2 \right]^{-\frac{6 l_i +5}{4l_i+3 } } |s_1-s_2|
		\sim
		\tau_i^{-\frac{7 l_i +6}{4l_i+3 } } |s_1-s_2|
	\end{aligned}
\end{equation}
with some $\theta\in [0,1]$. In this proof, $\theta\in [0,1]$ will be used repetitively and will vary from line to line.

Recall $U(x)$ given in \eqref{qd24Jan14-U} with $n=5$.
\begin{equation}\label{qd24Jan-3}
	\begin{aligned}
		&
		U^{\frac{2}{3}}(\tilde{v}^{[j]}, 0)
		\sim \langle \tilde{v}^{[j]} \rangle^{-2}
		\sim \langle \tilde{y}^{[i]} \rangle^{-2} ,
		\\
		&
		\big|
		U^{\frac{2}{3}}(\tilde{v}^{[1]}, 0)
		-
		U^{\frac{2}{3}}(\tilde{v}^{[2]}, 0)
		\big|
		= 3  \left[ \theta |\tilde{v}^{[1]}|^2 + (1-\theta) |\tilde{v}^{[2]}|^2
		+1 \right]^{-2}
		\left| |\tilde{v}^{[1]}|^2  - |\tilde{v}^{[2]}|^2 \right|
		\lesssim
		\langle \tilde{y}^{[i]} \rangle^{-3}
		|\tilde{v}^{[1]} - \tilde{v}^{[2]}| .
\\
&
\Big| \eta\Big(\frac{\tilde{v}^{[1]}}{4R} , 0\Big)
- \eta\Big(\frac{\tilde{v}^{[2]}}{4R} , 0\Big) \Big|
\lesssim R^{-1} |\tilde{v}^{[1]} - \tilde{v}^{[2]}| \1_{|\tilde{y}^{[i]}| \le 16 R}.
	\end{aligned}
\end{equation}

Denote $w^{[j]} = \big( \mu_{i,\sigma_0}[\bm{\phi}](t_i(s_j)) \tilde{v}^{[j]} + \xi_{\sigma_0}^{[i]}[\bm{\phi}](t_i(s_j))  -\tilde{q}^{[i]} ,0 \big)$.  By \eqref{qd24Jan09-3}, for $|\tilde{y}^{[i]}| \le 16 R$,
\begin{equation}\label{qd24Jan10-4}
	\begin{aligned}
		&
		|w^{[j]} |
		\lesssim
		\left(T-t_i(s_j) \right)^{2 l_i+2} \langle \tilde{y}^{[i]} \rangle
		\sim
		\left(T-t_i(\tau_i) \right)^{2 l_i+2}
		\langle \tilde{y}^{[i]} \rangle
		\ll R^{-1} (T-t_i(s_j) )^{1/2} ,
		\\
		&
		|w^{[1]} -w^{[2]} |
		\lesssim
		(T-t_i(\tau_i))^{6 l_i +5 }
		\langle \tilde{y}^{[i]} \rangle   |s_1-s_2|
		+
		\left(T-t_i(\tau_i) \right)^{2 l_i+2}
		|\tilde{v}^{[1]} - \tilde{v}^{[2]}|  ,
\\
&
|t_i(s_1) - t_i(s_2)| \lesssim (T-t_i(\tau_i))^{4l_i +4} |s_1-s_2| ,
	\end{aligned}
\end{equation}
where in the second line, $(T-t_i(\tau_i))^{6 l_i +5 }
\langle \tilde{y}^{[i]} \rangle   |s_1-s_2| \ll
T \left(T-t_i(\tau_i) \right)^{2 l_i+2}$.
Recall $\Theta_{l_i}$ given in \eqref{eq-heat-1} with $n=5$.
By \eqref{qd24Jan10-4}, \eqref{qd24Jan09-3},
\begin{equation}\label{qd24Jan-4}
	\begin{aligned}
		&
		\big|
		\Theta_{l_i}\big( w^{[j]},t_i(s_j) \big) \big|
		\sim (T- t_i(s_j) )^{l_i}
		\sim
		\tau_i^{-\frac{l_i}{4l_i +3}}  ,
		\\
		&
		\big|
		\Theta_{l_i}\big( w^{[1]},t_i(s_1) \big)
		-
		\Theta_{l_i}\big( w^{[2]},t_i(s_2) \big)
		\big|
		=
		\left(L_{l_i}^{\frac{3}{2}}(0)\right)^{-1}
		\bigg|
		\left[
		(T-t_i(s_1))^{l_i}
		-
		(T-t_i(s_2))^{l_i}
		\right]
		L_{l_i}^{\frac{3}{2}} \left(\frac{|w^{[1]}|^2}{4(T-t_i(s_1))} \right)
		\\
		&
		\quad
		+
		(T-t_i(s_2))^{l_i}
		\bigg[
		L_{l_i}^{\frac{3}{2}} \left(\frac{|w^{[1]}|^2}{4(T-t_i(s_1))} \right)
		-
		L_{l_i}^{\frac{3}{2}} \left(\frac{|w^{[2]}|^2}{4(T-t_i(s_2))} \right)
		\bigg]
		\bigg|
		\\
		\lesssim \ &
		\Big|
		\frac{d}{ds}
		(T-t_i(s))^{l_i} \big|_{s=\theta s_1 +(1-\theta)s_2 }
		(s_1-s_2) \Big|
		+
		(T-t_i(s_2))^{l_i}
		\bigg|
		\frac{|w^{[1]}|^2}{T-t_i(s_1)}
		-
		\frac{|w^{[2]}|^2}{T-t_i(s_2)}
		\bigg|
		\\
		\lesssim \ &
		\big(T-t_i(\theta s_1 +(1-\theta)s_2) \big)^{5l_i +3 }
		|s_1-s_2|
		\\
		&
		+
		(T-t_i(s_2))^{l_i}
		\bigg|
		\frac{|w^{[1]}|^2 (t_i(s_1) - t_i(s_2) ) }{(T-t_i(s_1) ) (T-t_i(s_2) ) }
		+
		\frac{ (|w^{[1]}| + |w^{[2]}| )
			(|w^{[1]}| - |w^{[2]}| ) }{T-t_i(s_2)}
		\bigg|
		\\
		\lesssim \ &
		\big(T-t_i(\tau_i) \big)^{5l_i +3 }
		|s_1-s_2|
		+
		(T-t_i(\tau_i))^{l_i}
		\Big\{
		\left(T-t_i(\tau_i) \right)^{8 l_i+6} \langle \tilde{y}^{[i]} \rangle^2
		|s_1-s_2|
		\\
		&
		+
		\left(T-t_i(\tau_i) \right)^{2 l_i+1}
		\langle \tilde{y}^{[i]} \rangle
		\left[
		(T-t_i(\tau_i))^{6 l_i +5 }
		\langle \tilde{y}^{[i]} \rangle   |s_1-s_2|
		+
		\left(T-t_i(\tau_i) \right)^{2 l_i+2}
		|\tilde{v}^{[1]} - \tilde{v}^{[2]}|
		\right]
		\Big\}
		\\
		\sim \ &
		\big(T-t_i(\tau_i) \big)^{5l_i +3 }
		\left(|s_1-s_2|
		+
		\langle \tilde{y}^{[i]} \rangle
		|\tilde{v}^{[1]} - \tilde{v}^{[2]}|
		\right)
		\sim
		\tau_i^{-\frac{5l_i +3 }{4l_i+3}}
		\left(|s_1-s_2|
		+
		\langle \tilde{y}^{[i]} \rangle
		|\tilde{v}^{[1]} - \tilde{v}^{[2]}|
		\right).
	\end{aligned}
\end{equation}

By \eqref{qd24Jan10-4}, $T\ll 1$ makes $w^{[j]} + q^{[i]} \in \overline{B_5^+(0,\sigma_0^{-1}) } $.
Using
$ \psi_{\sigma_0}[\bm{\phi}] \in \mathcal{X}_{\delta_0,\sigma_0} $
and $|w^{[j]} |
\ll R^{-1} (T-t_i(s_j) )^{1/2} $ in \eqref{qd24Jan10-4},
we have
\begin{equation}\label{qd24Jan-5}
	\big| \psi_{\sigma_0}[\bm{\phi}]\big( w^{[j]} + q^{[i]},t_i(s_j) \big) \big|
	\lesssim \delta_0 (T-t_i(s_j))^{l_i}
	\sim
	\delta_0 \tau_i^{-\frac{l_i}{4l_i +3}}.
\end{equation}

If $\max\{ |\tilde{v}^{[1]} - \tilde{v}^{[2]}|, |s_1-s_2| \} > 1$,
by \eqref{qd24Jan-5}, $\delta_0=R^{-\frac{1}{5}}$, $|\tilde{y}^{[i]}|\le 16R$, for any $\varsigma_1\in (0,1/10)$, we have
\begin{equation}\label{qd24Jan-7}
	\frac{\big| \psi_{\sigma_0}[\bm{\phi}]\big( w^{[1]} + q^{[i]},t_i(s_1) \big)
		-
		\psi_{\sigma_0}[\bm{\phi}]\big( w^{[2]} + q^{[i]},t_i(s_2) \big)
		\big|}{\left(
		\max\left\{
		|\tilde{v}^{[1]} - \tilde{v}^{[2]}| ,
		|s_1-s_2|^{1/2} \right\}
		\right)^{\varsigma_1} }
	\lesssim  \delta_0 \tau_i^{-\frac{l_i}{4l_i +3}}
	\lesssim
	R^{-\frac{1}{10}}
	\langle \tilde{y}^{[i]} \rangle^{-\varsigma_1}
	\tau_i^{-\frac{l_i}{4l_i +3}} .
\end{equation}

If $\max\{ |\tilde{v}^{[1]} - \tilde{v}^{[2]}|, |s_1-s_2| \} \le 1$, by \eqref{qd24Jan10-4} and $T- t_i(s_j) \sim T- t_i(\tau_i)$, we have
$|w^{[1]}| \le R^{-1} (T-t_i(s_1))^{1/2}$,
$|w^{[1]} -w^{[2]} | + |t_i(s_1) - t_i(s_2)|^{1/2} \lesssim \left(T-t_i(s_1) \right)^{2 l_i+2}$.
Hence,
by the H\"older estimate \eqref{qd24Jan09-2} with $(x_*,t_*)=(w^{[1]}+q^{[i]}, t_i(s_1))$, $\rho =  \left(T-t_i(s_1) \right)^{2 l_i+2} \ln R$, where obviously $\left(T-t_i(s_1) \right)^{2 l_i+2} \ll \rho \le R^{-1} (T- t_i(s_1) )^{1/2 }$, then there exists a constant $\varsigma_1\in(0,1/10)$  independent of $T, \sigma_0$
such that
\begin{equation}\label{qd24Jan-6}
	\begin{aligned}
		&
		\big| \psi_{\sigma_0}[\bm{\phi}]\big( w^{[1]} + q^{[i]},t_i(s_1) \big)
		-
		\psi_{\sigma_0}[\bm{\phi}]\big( w^{[2]} + q^{[i]},t_i(s_2) \big)
		\big|
		\\
		\lesssim \ &
		\rho^{-\varsigma_1}
		\left[ \delta_0 (T-t_i(s_1))^{l_i}
		+ \rho^2 R^{-\frac{1}{4}} (T-t_i(s_1))^{-3 l_i -4}
		\right]
		\left(
		\max\left\{ |w^{[1]} - w^{[2]}|,
		|t_i(s_1) - t_i(s_2)|^{1/2} \right\}
		\right)^{\varsigma_1}
		\\
		\lesssim \ &
		\rho^{-\varsigma_1}
		\left[ \delta_0 (T-t_i(\tau_i))^{l_i}
		+ \rho^2 R^{-\frac{1}{4}} (T-t_i(\tau_i))^{-3 l_i -4}
		\right]
		\left(
		\left(T-t_i(\tau_i) \right)^{2 l_i+2}
		\max\left\{
		|\tilde{v}^{[1]} - \tilde{v}^{[2]}| ,
		|s_1-s_2|^{1/2} \right\}
		\right)^{\varsigma_1}
\\
\sim \ &
(\ln R)^{-\varsigma_1}
\left( \delta_0 + (\ln R)^2 R^{-\frac{1}{4}} \right)
\left(T-t_i(\tau_i)\right)^{l_i}
\left(
\max\left\{
|\tilde{v}^{[1]} - \tilde{v}^{[2]}| ,
|s_1-s_2|^{1/2} \right\}
\right)^{\varsigma_1}
\\
\lesssim \ &
R^{-\frac{1}{10}}
\langle \tilde{y}^{[i]} \rangle^{-\varsigma_1}
\tau_i^{-\frac{l_i}{4l_i +3}}
\left(
\max\left\{
|\tilde{v}^{[1]} - \tilde{v}^{[2]}| ,
|s_1-s_2|^{1/2} \right\}
\right)^{\varsigma_1} ,
	\end{aligned}
\end{equation}
where we used \eqref{qd24Jan10-4} for the second ``$\lesssim$''.

Combining \eqref{qd24Jan-2}, \eqref{qd24Jan-3}, \eqref{qd24Jan-4},
\eqref{qd24Jan-5}, \eqref{qd24Jan-7}, \eqref{qd24Jan-6},
we obtain
\begin{equation}\label{24Jan04-8}
\begin{aligned}
&
	\big| \mathcal{H}_{2,i} [\bm{\phi}](\tilde{y}^{[i]},t_i(\tau_i)) \big| \lesssim
\tau_i^{-1}
	\langle \tilde{y}^{[i]}  \rangle^{-2}
	\eta\Big(\frac{\tilde{y}^{[i]}}{4R}, 0 \Big)
	\lesssim
	R^{\frac{3}{2}}
	\tau_i^{-1}
	\langle \tilde{y}^{[i]}  \rangle^{-\frac{7}{2}}
	\eta\Big(\frac{\tilde{y}^{[i]}}{4R}, 0\Big) ,
\\
&
[\mathcal{H}_{2,i} \big[\bm{\phi} ](\cdot,t_i(\cdot)) \big]_{C^{\varsigma_1,\varsigma_1/2} (Q_{\tilde{y}^{[i]}, \tau_i}) }
\lesssim
\tau_i^{-1}
\langle \tilde{y}^{[i]}  \rangle^{-2}
\1_{|\tilde{y}^{[i]}| \le 16R }
\Big[
\tau_i^{-1} |\tilde{y}^{[i]}|^{2-\varsigma_1}
+
\left(
\langle \tilde{y}^{[i]}  \rangle^{-1}
+
\tau_i^{-1} \langle \tilde{y}^{[i]}  \rangle
\right)
|\tilde{y}^{[i]}|^{1-\varsigma_1}
\\
&
\qquad
+
R^{-\frac{1}{10}}
\langle \tilde{y}^{[i]}  \rangle^{-\varsigma_1}
\Big]
\lesssim
\tau_i^{-1}
\langle \tilde{y}^{[i]}  \rangle^{-2-\varsigma_1}
\1_{|\tilde{y}^{[i]}| \le 16R }
\lesssim
R^{\frac{3}{2}}
\tau_i^{-1}
\langle \tilde{y}^{[i]}  \rangle^{-\frac{7}{2}-\varsigma_1}
\1_{|\tilde{y}^{[i]}| \le 16R }.
\end{aligned}
\end{equation}

Recalling the norms given in \eqref{qd24Jan25-10},
by \eqref{24Jan04-7}, \eqref{24Jan04-8}, we have
$\|\mathcal{H}_{1,i} [\bm{\phi}](\cdot,t_i(\cdot))\|_{-1,\frac{9}{2},\tau_i^p,\mathbb{R}^5_+,\tau_i(0),\tau_i(T_{\sigma_0})} \lesssim R^{\frac{3}{2}}$, \\
$\|\mathcal{H}_{2,i} [\bm{\phi}](\cdot,t_i(\cdot))\|_{-1,\frac{7}{2},\tau_i^p,\varsigma_1,\mathbb{R}^4,\tau_i(0),\tau_i(T_{\sigma_0})} \lesssim R^{\frac{3}{2}}$ provided $\tau_i^p>16R$.
The choice of $\bm{\mu}_{,1,\sigma_0}[\bm{\phi}],  \bm{\xi}_{\sigma_0}[\bm{\phi}]$ in Lemma \ref{lem-9.2} meets the orthogonal equations \eqref{orth-1}. Namely, $\mathcal{H}_{1,i} [\bm{\phi}] (y^{[i]},t_i(\tau_i)), \mathcal{H}_{2,i} [\bm{\phi}](\tilde{y}^{[i]},t_i(\tau_i))$ satisfy the orthogonality conditions \eqref{gh-ortho-Oct20} with $n=5$ for $\tau_i \in \left(\tau_i(0), \tau_i(T_{\sigma_0}) \right)$.
Thus, for $T\ll 1$, Proposition \ref{prop-23Oct24-1} with $n=5,
\tau=\tau_i \in \big( \tau_i(0), \tau_i(T_{\sigma_0}) \big)$,
$\sigma=-1, a=\frac{5}{2}, \ell(\tau) = \tau_i^p, p\in (\frac{2}{5}, \frac{3}{5}), \iota \in \big(\frac{1}{10}(\frac{5}{2} p +1) , \frac{1}{4} \big)$, $\varsigma=\varsigma_1$, $g=\mathcal{H}_{1,i} [\bm{\phi}](y^{[i]},t_i(\tau_i))$,
$h=\mathcal{H}_{2,i} [\bm{\phi}](\tilde{y}^{[i]},t_i(\tau_i))$,
and $\tilde{Z}_0 \in C_c^\infty( \overline{\mathbb{R}_+^5} )$ satisfying \eqref{til-Z0-Oct20} (see Remark \ref{qd24Jan14-9-rem} for the existence of $\tilde{Z}_0$)
yields a mapping $\mathcal{T}_i^{{\rm in}}[\bm{\phi}] = \mathcal{T}_i^{{\rm in}}[\bm{\phi}](y^{[i]},\tau_i)$ with the estimate
\begin{equation*}
 |\mathcal{T}_i^{{\rm in}}[\bm{\phi}]|+\langle y^{[i]} \rangle  |\nabla_{y^{[i]}} \mathcal{T}_i^{{\rm in}}[\bm{\phi}]|
\lesssim
R^{\frac{3}{2}}
\tau_i^{-1}
\langle y^{[i]}  \rangle^{-\frac{5}{2}}
\sim
R^{-\frac{1}{2}}
R^2
(T-t_i(\tau_i))^{4l_i+3}
\langle y^{[i]}  \rangle^{-\frac{5}{2}}
\end{equation*}
for $ (y^{[i]}, t_i(\tau_i)) \in B_5^+(0,9R) \times (0,T_{\sigma_0})$.
For solving \eqref{z24Feb24-3}, it suffices to solve the fixed-point problem
\begin{equation*}
	\bm{\phi} = \mathcal{T}^{{\rm in}}[\bm{\phi}] := \big( \mathcal{T}_1^{{\rm in}}[\bm{\phi}](y^{[1]},\tau_1(t)),
	\mathcal{T}_2^{{\rm in}}[\bm{\phi}](y^{[2]},\tau_2(t) ),
	\dots, \mathcal{T}_{\mathfrak{o}}^{{\rm in}}[\bm{\phi}](y^{[\mathfrak{o}]},\tau_{\mathfrak{o}}(t) ) \big)
\end{equation*}
for $t\in \left(0,T_{\sigma_0}\right) ,
y^{[i]}  \in B_5^+(0,2R)$, $i=1,2,\dots, \mathfrak{o}$.
Due to the small quantity $R^{-\frac{1}{2}}$,  $\mathcal{T}^{{\rm in}}[\bm{\phi}] \in B_{{\rm in},\sigma_0}$.

For all $\bm{\phi} \in B_{{\rm in},\sigma_0}$,
the uniform H\"older continuity of $\nabla_{y^{[i]}} \mathcal{T}_i^{{\rm in}}[\bm{\phi}]$
in $t\in \left(0,T_{\sigma_0}\right) ,
y^{[i]}  \in \overline{B_5^+(0,8R)}$, $i=1,2,\dots, \mathfrak{o}$
follows from the parabolic regularity theory. Hence $\mathcal{T}^{\rm in}[\cdot]$ is a compact mapping from $B_{{\rm in},\sigma_0}$ to itself. The Schauder fixed-point theorem gives a fixed point $\bm{\phi}_{\sigma_0} = (\phi_{1,\sigma_0}, \phi_{2,\sigma_0},\dots, \phi_{\mathfrak{o},\sigma_0}) \in B_{{\rm in},\sigma_0}$. Moreover, Proposition \ref{prop-23Oct24-1} gives  $\phi_{i,\sigma_0}(\cdot, 0) = C_{{\rm in}, i,\sigma_0} \tilde{Z}_0$ in $B_5^+(0,2R)$ with a constant $C_{{\rm in}, i,\sigma_0}$ satisfying $|C_{{\rm in}, i,\sigma_0}| \lesssim (\tau_i(0))^{-1} R^{\frac{3}{2}} \sim T^{4l_i+3} R^{\frac{3}{2}}$.
The last uniform $C^{1+\varsigma_1,(1+\varsigma_1)/2}$ boundedness in $\overline{B_5^+(0,2R)} \times K$ follows from \eqref{24Jan04-7}, \eqref{24Jan04-8}, and the parabolic regularity theory.
\end{proof}

\begin{remark}
We did not apply the H\"older estimate \eqref{qd24Jan09-2} to $\psi_{\sigma_0}[\bm{\phi}]$ with $(x_*,t_*)=(q^{[i]}, t_i(\tau_i))$. Since due to \eqref{qd24Jan10-4}, we require $\rho \gtrsim \left(T-t_i(\tau_i) \right)^{2 l_i+2}
\langle \tilde{y}^{[i]} \rangle $. Then the term $\rho^{2-\varsigma_1} R^{-\frac{1}{4}} (T-t_i(\tau_i))^{-3 l_i -4} $ will lead to additional spatial growth $\langle \tilde{y}^{[i]} \rangle^2$, which can not be eliminated by $R^{-\frac{1}{4}}$.
\end{remark}

\subsection{Proof of Theorem \ref{main}}\label{qd24Jan29-1-sec}

\begin{proof}[Proof of Theorem \ref{main}]

Combining Lemma \ref{24Jan01-5-lem}, Lemma \ref{lem-9.2}, Lemma \ref{24Jan05-1-lem}, for $\sigma_0\in (0,T)$, we set
\begin{equation}\label{qd24Feb24-6}
	v_{\sigma_0}(x,t) :=
	\psi_{\sigma_0}(x,t)
	+
	\sum_{i=1}^{\mathfrak{o}}\bigg(
	U_{\mu_{i,\sigma_0},\xi_{\sigma_0}^{[i]}} (x)
	\eta\Big(\frac{x^{[i]}}{2\delta} \Big)
	+
	\Theta_{l_i}(x^{[i]},t)\eta\Big(\frac{x^{[i]} }{\delta} \Big)+
	\mu_{i,\sigma_0}^{-\frac{3}{2}} \phi_{i,\sigma_0}\Big(\frac{x-\xi_{\sigma_0}^{[i]}}{\mu_{i,\sigma_0}},t \Big) \eta\Big(\frac{x-\xi_{\sigma_0}^{[i]}}{\mu_{i,\sigma_0} R } \Big) \bigg)
\end{equation}
with $\mu_{i,\sigma_0} = \mu_{i,0} + \mu_{i,1,\sigma_0}$, and then $v_{\sigma_0}$ satisfies
\begin{equation}\label{qd24Jan15-1}
\left\{
\begin{aligned}
&
\partial_t v_{\sigma_0}=\Delta v_{\sigma_0}
\mbox{ \ in \ }  B_5^+(0,\sigma_0^{-1}) \times (0,T_{\sigma_0}) ,
\quad
\left(-\partial_{x_5} v_{\sigma_0} \right)(\tilde{x},0) =
\big(|v_{\sigma_0}|^{\frac{2}{3}} v_{\sigma_0} \big)(\tilde{x},0)   \mbox{ \ on \ } B_4(0,\sigma_0^{-1}) \times (0,T_{\sigma_0})  ,
\\
&
v_{\sigma_0}(x,0) =
\sum\limits_{i=1}^{\mathfrak{o}}
\bigg[
{\bm{b}}_{i,\sigma_0}
\cdot {\bm{ \tilde{e}} }_i\big( T^{-\frac{1}{2}} x^{[i]} \big)
+
\sum_{ \mathbf{p}\in \mathbb{N}^5, \| \mathbf{p}\|_{\ell_1} \le 4l_{\rm{max}}+4, p_5\in 2 \mathbb{N}}
C_{q^{[i]}, \mathbf{p}, \sigma_0} \varphi_{ q^{[i]}, \mathbf{p},0 }(x)
	+
	U_{\mu_{i,\sigma_0}(0),\xi_{\sigma_0}^{[i]}(0) } (x)
	\eta\Big(\frac{x^{[i]} }{2\delta} \Big)
\\
& \quad
	+
	\Theta_{l_i}(x^{[i]},0)\eta\Big(\frac{x^{[i]} }{\delta} \Big)
	+
	\big(\mu_{i,\sigma_0}(0) \big)^{-\frac{3}{2}}
	C_{{\rm in}, i,\sigma_0} \tilde{Z}_0\Big(\frac{x-\xi_{\sigma_0}^{[i]}(0) }{\mu_{i,\sigma_0}(0) } \Big) \eta
	\Big(\frac{x-\xi_{\sigma_0}^{[i]}(0)}{\mu_{i,\sigma_0}(0) R } \Big) \bigg]
\mbox{ \ in \ }	 B_5^+(0,\sigma_0^{-1}) ,
\end{aligned}
\right.
\end{equation}
where
${\bm{b}}_{i,\sigma_0} $ are constant vectors and  $C_{q^{[i]}, \mathbf{p},\sigma_0}$, $C_{{\rm in}, i,\sigma_0}$ are constants satisfying
$| {\bm{b}}_{i,\sigma_0} | \le C |\ln T| T^{ \frac{5}{3}l_i+\frac{1}{2} }$, $|C_{q^{[i]}, \mathbf{p},\sigma_0}| \le C e^{-\frac{9 \delta^2}{22 T} }$, and $|C_{{\rm in}, i,\sigma_0}| \le  C T^{4l_i+3} R^{\frac{3}{2}}$ with a constant $C$ independent of $T, \sigma_0$; ${\bm{ \tilde{e}} }_i$ are given in \eqref{24Dec01-1}; $\varphi_{q^{[i]}, \mathbf{p},0 } \in C_c^{\infty}(\overline{\mathbb{R}_+^5})$ and $\varphi_{q^{[i]}, \mathbf{p},0 } = 0$ in $\overline{\mathbb{R}^{5}_+  \backslash B_5^+(q^{[i]} ,2\delta)}$;
$\tilde{Z}_0 \in C^{\infty}(\overline{B_5^+(0,2R)})$.

Up to a subsequence, there exist constant vectors ${\bm{b}}_{i,0}$ and constants $C_{q^{[i]}, \mathbf{p},0}, C_{{\rm in}, i,0}$ such that ${\bm{b}}_{i,\sigma_0} \to {\bm{b}}_{i,0}$, $C_{q^{[i]}, \mathbf{p}, \sigma_0} \to C_{q^{[i]}, \mathbf{p},0}, C_{{\rm in}, i,\sigma_0} \to C_{{\rm in}, i,0}$ as $\sigma_0\downarrow 0$. For simplicity of exposition, we will often take a subsequence with $\sigma_0 \downarrow 0$ but will not state it.
By Lemma \ref{24Jan01-5-lem}, by the diagonal method, there exists a function $\psi_0$ such that $\psi_{\sigma_0} \to \psi_0$ in $L_{\rm{loc}}^{\infty} (\overline{\mathbb{R}_+^5} \times [0, T) )$, and then $\psi_0 \in \mathcal{X}_{\delta_0,0}$. Similarly,
by Lemma \ref{lem-9.2},  $\mu_{i,1,\sigma_0} \to \mu_{i,1,0}, \xi_{\sigma_0}^{[i]} \to \xi_0^{[i]}$ in $C_{\rm{loc}}^{1} ( [0,T) )$,
and then $\mu_{i,1,0}, \xi_0^{[i]}$ satisfy \eqref{orth-101} with $\sigma_0 =0$. Denote $\mu_{i,0} = \mu_{i,0} + \mu_{i,1,0}$.
By Lemma \ref{24Jan05-1-lem},
$\phi_{i,\sigma_0} \to \phi_{i,0}$ and $\nabla \phi_{i,\sigma_0} \to \nabla \phi_{i,0}$
in
$L^{\infty}\big(\overline{B_5^+(0,2R)} \times K \big)$ for any compact set $K\subset [0,T)$, and  then $\phi_{i,0}\in B_{{\rm in},0}$.

Now we can extend the definition \eqref{qd24Feb24-6} of $v_{\sigma_0}$ at $\sigma_0=0$ naturally. Then $v_{0}(x,0)\in C_c^{\infty}(\overline{\mathbb{R}_+^5})$. By the convergence argument above, $v_{\sigma_0} \to v_{0}$ in $L_{\rm{loc}}^{\infty} (\overline{\mathbb{R}_+^5} \times [0,T) )$ and
$v_{\sigma_0}(x,0) \to v_{0}(x,0)$ in $C_{\rm{loc}}^{k} (\overline{\mathbb{R}_+^5})$ for any $k \in \mathbb{N}$. One testing \eqref{qd24Jan15-1} with arbitrary functions in $C_c^{\infty}(\overline{\mathbb{R}_+^5} \times [0,T) )$ with $\sigma_0$ sufficiently small, then taking $\sigma_0\downarrow 0$ deduces that $v_{0}$ is a weak solution of
\eqref{fractextend}. Set $u=v_0, \mu_i=\mu_{i,0}, \xi^{[i]}=\xi_0^{[i]}$, and then we get Theorem \ref{main}.
\end{proof}

\begin{remark}
We can not use contraction mapping like \cite{Green16JEMS} to solve $\dot{\bm{\mu}}_{,1}, \dot{\bm{\xi} }$ under the current topology since we do not have gradient estimate of $\psi$ and can not get the Lipschitz continuity about $\dot{\bm{\mu}}_{,1}, \dot{\bm{\xi} }$ in \eqref{orth-6} with $T_{\sigma_0} = T$.
\end{remark}

\section*{Acknowledgements}

We appreciate the discussions about the vanishing adjustment method for the heat equation with Yifu Zhou, the distribution of the right-hand side of the outer problem with Jianfeng Zhao, and the linear theory for inner problems with Youquan Zheng. Qidi Zhang wrote part of this paper when visiting Wuhan University and appreciates
the hospitality of Yifu Zhou. Xiaoyu Zeng appreciates the warm hospitality from the University of British Columbia and the Department of Mathematics, and he is supported by China Scholarship Council and NSFC (Grant Nos. 12322106, 11931012, 12171379).

\appendix

\section{A type I solution of \texorpdfstring{\eqref{qd2023Nov19-1}}{$(\ref{qd2023Nov19-1})$}}

\begin{prop}\label{2023Nov23-3-prop}
	Suppose $\alpha \in (0,1)$, $p>1$, then
	$u_1$ given in \eqref{qd2023Nov21-u1-def} satisfies \eqref{qd2023Nov19-1} with $\left(t_0,t_1\right) = \left(-\infty,T\right)$ and \eqref{qd2023Nov21-2}.
\end{prop}

\begin{proof}
	By \cite[Corollary 1.4]{stinga2017regularity}, for $t<T$,
	\begin{equation*}
		\begin{aligned}
			&
			\left(\partial_t -\Delta \right)^\alpha \left(T-t\right)^a =
			\left(\partial_t \right)^\alpha \left(T-t\right)^a
			\\
			= \ &
			\frac{1}{\left|\Gamma(-\alpha)\right|}
			\int_{-\infty}^t
			\frac{\left(T-t\right)^a -\left(T-\tau\right)^a }{\left(t-\tau\right)^{1+\alpha} } d\tau
			=
			\frac{1}{\left|\Gamma(-\alpha)\right|}
			\left(T-t\right)^{a-\alpha}
			\int_1^\infty \frac{1-z^a}{\left(z-1\right)^{1+\alpha} } dz  ,
		\end{aligned}
	\end{equation*}
	where we changed the variable $z=\frac{T-\tau}{T-t}$, and $\max\left\{ 0,a\right\} <\alpha <1$ is required for the integrability.

	Given $u\left(\tilde{x},0,t\right)$ of the form $ C\left(T-t\right)^a$, by \cite[Theorem 1.7, (1.7)]{stinga2017regularity}, in order to make
	\begin{equation*}
		\frac{\left|\Gamma(-\alpha)\right|}{4^\alpha \Gamma(\alpha)}
		\left(\partial_t -\Delta_{\tilde{x} } \right)^\alpha
		\left[ C \left(T-t\right)^a \right]
		=
		\left| C \left(T-t\right)^a \right|^{p-1} C \left(T-t\right)^a ,
	\end{equation*}
	we take $a=\frac{-\alpha}{p-1}$ with $p>1$, and $C= \pm C_{\alpha,p}$ with $C_{\alpha,p}>0$ given in \eqref{qd2023Nov21-u1-def}. Thus
	$ u\left(\tilde{x},0,t\right) = \pm C_{\alpha,p} \left(T-t\right)^{\frac{-\alpha}{p-1}}  $.
By \cite[Theorem 1.7, (1.5)]{stinga2017regularity}, for $x_n>0$,
	\begin{equation*}
		\begin{aligned}
			&
			u_1\left(\tilde{x},x_n,t\right) =
			\frac{x_n^{2\alpha}}{4^\alpha \Gamma(\alpha)}
			\int_0^\infty e^{-\frac{x_n^2}{4\tau}}
			\int_{\mathbb{R}^{n-1}} \left(4\pi \tau\right)^{-\frac{n-1}{2}} e^{-\frac{|z|^2}{4\tau}}
			\left(\pm C_{\alpha,p} \right)
			\left(T-t+\tau\right)^{\frac{-\alpha}{p-1}} dz \tau^{-1-\alpha} d\tau
			\\
			= \ &
			\frac{x_n^{2\alpha}}{4^\alpha \Gamma(\alpha)}
			\int_0^\infty e^{-\frac{x_n^2}{4\tau}}
			\left(\pm C_{\alpha,p} \right)
			\left(T-t+\tau\right)^{\frac{-\alpha}{p-1}}   \tau^{-1-\alpha} d\tau
			=
			\frac{ \pm C_{\alpha,p} }{\Gamma(\alpha)}
			\int_0^\infty e^{-s}
			\left(T-t+\frac{x_n^2}{4s}\right)^{\frac{-\alpha}{p-1}}
			s^{\alpha-1} ds
			\\
			= \ &
			\frac{\pm C_{\alpha,p}}{\Gamma(\alpha)}
			\left( \frac{x_n^2}{4}\right)^{\frac{-\alpha}{p-1}}
			\bigg(
			\int_0^{\frac{x_n^2}{T-t}}
			+
			\int_{\frac{x_n^2}{T-t}}^\infty
			\bigg) e^{-s} s^{\alpha-1}
			\left[\frac{4(T-t)}{x_n^2}+s^{-1} \right]^{\frac{-\alpha}{p-1}}
			ds
			\\
			\sim \ &
			\frac{\pm C_{\alpha,p}}{\Gamma(\alpha)}
			\left( \frac{x_n^2}{4}\right)^{\frac{-\alpha}{p-1}}
			\bigg[
			\int_0^{\frac{x_n^2}{T-t}}
			e^{-s} s^{\alpha-1+\frac{\alpha}{p-1}}
			ds
			+
			\left(\frac{T-t}{x_n^2}  \right)^{\frac{-\alpha}{p-1}}
			\int_{\frac{x_n^2}{T-t}}^\infty
			e^{-s} s^{\alpha-1}
			ds
			\bigg]
			\sim
			\pm \left(
			\max\left\{	T-t , x_n^2 \right\} \right)^{\frac{-\alpha}{p-1}}  .
		\end{aligned}
	\end{equation*}
\end{proof}

\section{Convolution estimates}\label{convo-sec}

\subsection{Inequalities toolkit}

	\begin{lemma}\label{23Dec23-3-lem}	
\begin{enumerate}
\item\label{23Oct10-2}
Given $a>0$, $b>0$, then for any $r>0$, $r^a e^{-b r} \le b \int_r^\infty  x^a e^{-b x} dx $.
Given $a < 0$, $b>0$, then for any $r>0$, $b \int_r^\infty  x^a e^{-bx} dx
		\le
		r^a e^{-br} $.
Given $a \in \mathbb{R}$, $b>0$, $r_0>0$, then for any $r>r_0$,
	\begin{equation}\label{23Oct07-1}  \int_r^\infty  x^a e^{-bx} dx
		\sim  C(a,b,r_0)  r^a  e^{-br} ,
	\end{equation}
	where $C(a,b,r_0)>0$ is a constant depending on $a, b, r_0$, and ``$\sim$'' does not depend on any parameters.
	
\item\label{23Oct10-3}
Given $a\in \mathbb{R}$, $r_0>0$, then for any $A>B\ge r_0$,
\begin{equation}\label{23Oct09-5}
	\int_B^A  r^a \left(1+\ln A -\ln r\right) e^{-r} dr
	\le
	\left(1+ \ln \left(B^{-1}A\right) \right)
	\int_B^A  r^a  e^{-r} dr
	\le C(a,r_0) B^a \left(1+\ln \left(B^{-1}A \right) \right) e^{-B} .
\end{equation}

\item\label{23Oct10-4}
Given $b>0$, $a\ge 0$, then $r^{a} e^{b r}$ is monotone increasing in $r>0$. Given $b>0$, $a<0$, $r_0>0$, then for any $r_2 > r_1 \ge r_0$,
\begin{equation}\label{23Oct09-1}
	r_1^{a} e^{b r_1} \le
	C(a,b,r_0)  r_2^{a} e^{b r_2}
\end{equation}
with a constant $C(a,b,r_0)>0$ depending on $a, b, r_0$.

Given $b<0$, $a\le 0$, then $r^{a} e^{b r}$ is monotone decreasing in $r>0$. Given $b<0$, $a> 0$, $r_0>0$, then for any $r_2 > r_1 \ge r_0$,
\begin{equation}
	r_1^{a} e^{b r_1} \ge
	C(a,b,r_0)  r_2^{a} e^{b r_2}
\end{equation}
with a constant $C(a,b,r_0)>0$ depending on $a, b, r_0$.

\item\label{23Oct01-6} 	Given $a<0$, $r_0>0$, then for any $r_2>r_1\ge r_0$,
\begin{equation}\label{23Oct09-3}
	r_1^a \left(1+ \left|\ln r_1 \right| \right) \ge
	C(a,r_0)
	r_2^a \left(1+ \left|\ln r_2 \right| \right) ,
\end{equation}
\begin{equation}\label{23Oct09-4}
	r_1^{a} \left(1+ \ln \left( r_0^{-1} r_1 \right) \right)
	\ge
	C(a)
	r_2^{a} \left(1+ \ln \left(r_0^{-1} r_2 \right)\right)
\end{equation}
with a constant $C(a,r_0)>0$ depending on $a, r_0$ and a constant $C(a)>0$ depending on $a$.

\item\label{23Oct10-7}
Given $a<-1$, $b\in \mathbb{R}$, $r_0>0$, then for any $r\ge r_0$,
\begin{equation}\label{23Oct10-1}
	\int_{r}^\infty x^a \left(1+ \left|\ln x\right|\right)^b dx \le C(a,b,r_0)
	r^{a+1} \left(1+ \left|\ln r\right|\right)^b
\end{equation}
with a constant $C(a,b,r_0)>0$ depending on $a,b,r_0$.
	
\end{enumerate}
\end{lemma}

\begin{proof}

\eqref{23Oct10-2}. The first two inequalities are straightforward.
Given $a \in \mathbb{R}$, $b=1$, $r_0>0$, for $M>1$,
	\begin{equation*}
		\int_r^{Mr}  x^a e^{-x} dx
		\sim C(M,a) r^a  \int_r^{Mr} e^{-x} dx
		=
		C(M,a) r^a e^{-r}
		\left[
		1-e^{-(M-1)r}
		\right]
		.
	\end{equation*}
One taking $M=M(r_0)$ large to make
	$(M-1)r_0 \ge 1$, then
$ \int_r^{Mr}  x^a e^{-x} dx
		\sim
		C(M,a) r^a e^{-r} $.
	For the other part,
	\begin{equation*}
		\int_{Mr}^\infty  x^a e^{-x} dx
		=
		(Mr)^a e^{-Mr}
		+
		\int_{Mr}^\infty a x^{a-1} e^{-x} dx ,
\mbox{ \ and \ }
\left|\int_{Mr}^\infty a x^{a-1} e^{-x} dx\right|
\le
\frac{|a|}{M r_0}  \int_{Mr}^\infty x^a e^{-x} dx .
	\end{equation*}
One taking $M=M(|a|,r_0)$ large to make $\frac{|a|}{M r_0} < 9^{-1}$, then
	\begin{equation*}
	\int_{Mr}^\infty  x^a e^{-x} dx
		\sim
		(Mr)^a e^{-Mr}
		\le
		e^{(1-M)r_0} M^a r^a e^{-r}
		.
	\end{equation*}
	Thus, we conclude \eqref{23Oct07-1} for $b=1$, which implies the general case $b>0$ by
\begin{equation*}
\int_r^\infty  x^a e^{-bx} dx
=
b^{-a-1} \int_{br}^\infty z^a e^{-z} dz
\sim C(a,b,r_0) r^a e^{-br} .
\end{equation*}
	
	\medskip
	
\eqref{23Oct10-3} is deduced by \eqref{23Oct10-2}.

	\medskip

\eqref{23Oct10-4}. 	For $b>0$, $a<0$, $r_0>0$, denote
$f(r) = r^{a} e^{b r}$, then $
f'(r) = r^a e^{b r} \left( b + r^{-1} a\right) $ and $f'(r) > 0$ for $r\ge 9|a| b^{-1}$.
For $r_0\le r \le \max\left\{ 9|a| b^{-1}, r_0 \right\}$, we have $ r^{a} e^{b r}
\sim C(a,b,r_0) r_0^{a} e^{b r_0}$. Thus, \eqref{23Oct09-1} holds.

The inequality for the case $b<0$, $a>0$, $r_0>0$ is deduced by applying the above result to $r^{-a} e^{-b r}$.

	\medskip

\eqref{23Oct01-6}. Denote $g(r) = r^a \left(1+ \left|\ln r\right| \right)$. For $r>1$,
$ g'(r) =
r^{a-1} \left[ a\left(1+\ln r\right)+1 \right] $ and $g'(r) <0$
for $r \ge C_1$ with a large constant $C_1=C_1(a) >1$.
For $r \in \left[r_0,C_1 \right]$, $g(r) \sim C(C_1,r_0)$. Thus, \eqref{23Oct09-3} holds.

\eqref{23Oct09-4} is deduced by changing the variable $z= r_0^{-1} r$ and the above result.

\medskip

\eqref{23Oct10-7}. 	For any $M>\max\left\{ r_0,1\right\}$, $r\in \left[r_0,M\right]$, we have
\begin{equation*}
	\int_{r}^\infty x^a \left(1+ \left|\ln x\right|\right)^b dx \le C(a,b,r_0,M)
	r^{a+1} \left(1+ \left|\ln r\right|\right)^b .
\end{equation*}
For $r>M$,
\begin{equation*}
	\int_{r}^\infty x^a \left(1+ \ln x \right)^b dx
	=
	\frac{-1}{a+1} r^{a+1} \left(1+\ln r\right)^b - \frac{b}{a+1}
	\int_r^\infty x^a \left(1+\ln x\right)^{b-1} dx ,
\end{equation*}
where
\begin{equation*}
	\left| \frac{b}{a+1}
	\int_r^\infty x^a \left(1+\ln x\right)^{b-1} dx \right|
	\le \frac{1}{9 }
	\int_{r}^\infty x^a \left(1+ \ln x \right)^b dx
\end{equation*}
for $M=M(a,b)$ sufficiently large. Thus, we conclude \eqref{23Oct10-1}.

\end{proof}

\subsection{Beyond Neumann boundary value \texorpdfstring{$v(t) |\tilde{x}|^{-b} \1_{ l_1(t) \le |\tilde{x}| \le l_2(t) }$ and $v(t) \left(|\tilde{x}|+l_1(t)\right)^{-b} \1_{ |\tilde{x}| \le l_2(t) }$}{in the self-similar region}}

 \begin{lemma}\label{bd-annu-Lem}
 	
 	Let  $n>2$ be an integer, $t>t_0 \ge 0$, $b\in \mathbb{R}$. Suppose that
 	$v(s)\ge 0$ for $s\in [t_0,t]$;  $0\le l_1(s) \le l_2(s) \le C_{*} s^{\frac{1}{2}}$ for $s\in [t_0,t]$,
 	$C_l^{-1} l_1(t) \le l_1(s)$, and $l_2(s) \le C_l l_2(t)  $, for all $s\in[\max\{t_0,\frac{t}{2} \},t] $,  where $C_*>0,C_l\ge1$ are constants. Given $x=(\tilde{x},x_n)$, $\tilde{x}\in \mathbb{R}^{n-1}$, $x_n \ge 0$, $C_0>0$, we define
 	\begin{equation*}
 \mathcal{T}\left[f\right](x,t):=
 		\int_{t_0}^t\int_{\mathbb{R}^{n-1}}
 		(t-s)^{-\frac{n}{2}}
 		e^{-C_0 \frac{|\tilde{x}-y|^2+ x_n^2 }{t-s}}
 		f(y,s)
 		dy ds
 	\end{equation*}
for an admissible function $f$.
Then for any $\epsilon\in (0,1)$, we have
\begin{equation*}
\mathcal{T}\left[v(t) |\tilde{x}|^{-b} \1_{ l_1(t) \le |\tilde{x}| \le l_2(t) } \right](x,t) \le C
\Big(
w_1 + \sup\limits_{ t_1 \in \left[\max\{t_0,t/2 \},t \right] }
v(t_1)
w_2 \Big),
\end{equation*}
where $C$ is a constant only depending on $n,b,C_*,C_l,C_0$,
\begin{equation*}
\begin{aligned}
w_{1} := \ &
	t^{-\frac n2}
	e^{-C_0 \frac{x_n^2 }{t-t_0}}
	\left(\1_{|\tilde{x}|\le C_* \left[1+\left(1-\epsilon\right)^{\frac{1}{2}}\right] \epsilon^{-1} t^{\frac{1}{2}}} + e^{-C_0\left(1-\epsilon\right) \frac{|\tilde{x}|^2 }{t-t_0}} \1_{|\tilde{x}|> C_* \left[1+\left(1-\epsilon\right)^{\frac{1}{2}}\right] \epsilon^{-1} t^{\frac{1}{2}}} \right)
\\
&\times		
	\int_{t_0} ^{\max\{t_0,\frac{t}{2} \}}
	v(s)
	\begin{cases}
		l_2^{n-1-b}(s) ,
		&
		b<n-1
		\\
		\ln (\frac{l_2(s)}{ l_1(s) } ) ,
		&
		b=n-1
		\\
		l_1^{n-1-b}(s) ,
		&
		b>n-1
	\end{cases}
	d s,
\end{aligned}
\end{equation*}
\begin{equation*}
	w_2:=
	\begin{cases}
		\begin{cases}
			l_2^{1-b}(t)
			,
			&b< 1
			\\
			\langle 	\ln (\frac{l_2(t)}{l_1(t)} )
			\rangle
			,
			&b = 1
			\\
			l_1^{1-b}(t)
			,
			& b >1
		\end{cases}
		,
		&
		|x|\le l_1(t)
		\\
		\begin{cases}
			l_2^{1-b}(t) , &b<1
			\\
			\langle \ln (\frac{l_2(t)}{|x|}) \rangle
			, &b=1
			\\
			|x|^{1-b}
			, &1<b<n-1
			\\
			|x|^{2-n} \langle \ln(\frac{|x|}{l_1(t)}) \rangle
			, &b=n-1
			\\
			|x|^{2-n} l_1^{n-1-b}(t)
			,
			&b>n-1
		\end{cases}
		,&  l_1(t) < |x| \le l_2(t)
		\\
		\begin{cases}
			l_2^{n-1-b}(t)
			,
			&b<n-1
			\\
			\langle \ln (  \frac{l_2(t)}{l_1(t)} ) \rangle
			,
			&b=n-1
			\\
			l_1^{n-1-b}(t)
			,
			&b>n-1
		\end{cases}
		\begin{cases}
			|x|^{2-n} ,&  l_2(t) < |x| \le C_* t^{\frac{1}{2}}
			\\
			|x|^{-2} t^{2-\frac{n}{2}}
			e^{-2C_0 \frac{
					x_n^2  + \frac{64}{81}|\tilde{x}|^2 }{t}} ,
			& |x|> C_*  t^{\frac{1}{2}}
		\end{cases}
	\end{cases}
\end{equation*}
with the convention $\frac{l_2(s)}{ l_1(s) }=1$ $\left(\frac{l_2(t)}{ l_1(t) }=1\right)$ if $l_1(s) = l_2(s)=0$ $\left(l_1(t) = l_2(t)=0\right)$.

Under the additional assumption that $l_1(s) \le C_l l_1(t)  $ for all $s\in[\max\{t_0,\frac{t}{2} \},t] $, one replacing $\ln (\frac{l_2(s)}{ l_1(s) } )$ by $\langle\ln (\frac{l_2(s)}{ l_1(s) } ) \rangle$ in the definition of $w_1$, then the same upper bound is held for
$\mathcal{T}\left[
v(t) \left(|\tilde{x}|+l_1(t)\right)^{-b} \1_{ |\tilde{x}| \le l_2(t) }
 \right]$.

 \end{lemma}

 \begin{proof} In this proof, we assume $\int_{t_1}^{t_2}\cdots ds=0$ if $t_1\ge t_2$ and $\int_{\mathbb{R}^{n-1}} \1_{c_1 \le |y| \le c_2} \cdots dy =0$ if $c_1\ge c_2$.
 	We emphasize that all ``$\lesssim$'', ``$\sim$'' in this proof are independent of $\epsilon$, $t_0$.	
 	Denote $t_*:=\max\{t_0,\frac{t}{2} \}$.	
 	\begin{equation*}
 		\begin{aligned}
 			&	\mathcal{T}\left[v(t) |\tilde{x}|^{-b} \1_{ l_1(t) \le |\tilde{x}| \le l_2(t) } \right]
 			\lesssim
 			t^{-\frac{n}{2}}
 			e^{-C_0 \frac{x_n^2 }{t-t_0}}
 			\int_{t_0}^{t_*}
 			\int_{\mathbb{R}^{n-1}}
 			e^{-C_0 \frac{|\tilde{x}-y|^2 }{t-t_0}}
 			v(s) |y|^{-b} \1_{ l_1(s) \le |y| \le l_2(s) }
 			dy ds
 			\\
 			& +
 			\sup\limits_{ t_1 \in [t_*,t] } v(t_1)
 			\int_{t_*}^t
 			(t-s)^{-\frac{n}{2}}
 			e^{-C_0 \frac{x_n^2 }{t-s}}
 			\int_{\mathbb{R}^{n-1}}
 			e^{-C_0 \frac{|\tilde{x}-y|^2 }{t-s}}
 			|y|^{-b} \1_{ l_1(s) \le |y| \le l_2(s)}
 			dy ds
 			:=
 			u_{1} + \sup\limits_{ t_1 \in [t_*,t] } v(t_1) \tilde{u}_{2} .
 		\end{aligned}
 	\end{equation*}
Obviously,
\begin{equation*}
\tilde{u}_2 \le u_2:=
\int_{t_*}^t
(t-s)^{-\frac{n}{2}}
e^{-C_0 \frac{x_n^2 }{t-s}}
\int_{\mathbb{R}^{n-1}}
e^{-C_0 \frac{|\tilde{x}-y|^2 }{t-s}}
|y|^{-b} \1_{ C_l^{-1} l_1(t) \le |y| \le C_l l_2(t)}
dy ds.
\end{equation*}
 	
 	For $u_{1}$, notice $|y|\le C_* t^{\frac 12}$. For $|\tilde{x}|\le M t^{\frac 12}$, we use $e^{-C_0 \frac{|\tilde{x}-y|^2 }{t-t_0}} \le 1$;
 Given a constant $\epsilon\in (0,1)$, for $|\tilde{x}| > M t^{\frac 12}$ with $M= C_* \left[1+\left(1-\epsilon\right)^{\frac{1}{2}}\right] \epsilon^{-1}$, one has  $|\tilde{x}-y|\ge \left(1-M^{-1} C_*\right) |\tilde{x}| \ge \left(1-\epsilon\right)^{\frac{1}{2}}|\tilde{x}|$. Then, $u_1\lesssim w_1$.

 	Let us estimate $u_2$ in different regions.
 	
 	For $|\tilde{x}| \le 2^{-1} C_l^{-1} l_1(t)$,
 	we have $|\tilde{x}-y|^2 \ge 4^{-1} |y|^2$, and then
 	\begin{equation*}
 		\begin{aligned}
 			&	u_{2}
 			\le
 			\int_{t_*}^t
 			(t-s)^{-\frac{n}{2}}
 			e^{-C_0 \frac{x_n^2 }{t-s}}
 			\int_{\mathbb{R}^{n-1}}
 			e^{-4^{-1} C_0  \frac{|y|^2 }{t-s}}
 			|y|^{-b} \1_{ C_l^{-1} l_1(t) \le |y| \le C_l l_2(t)}
 			dy ds
 			\\
 			\sim \ &
 			\left(
 			\int_{t_*}^{t-\frac{l_2^2(t)}{(9C_*)^2} }
 			+
 			\int_{t-\frac{l_2^2(t)}{(9C_*)^2}}^{t-\frac{l_1^2(t)}{(9C_*)^2}}
 			+
 			\int_{t- \frac{l_1^2(t)}{(9C_*)^2} }^{t}
 			\right)
 			e^{-C_0 \frac{x_n^2 }{t-s}}
 			(t-s)^{-\frac{b+1}{2}}
 			\int_{4^{-1} C_0 C_{l}^{-2} \frac{l_1^2(t)}{t-s} }^{4^{-1} C_0  C_{l}^{2} \frac{l_2^2(t)}{t-s}}
 			e^{-z} z^{\frac{n-b}{2}-\frac{3}{2}}
 			dz ds
 			:=    u_{21} + u_{22} + u_{23} .
 		\end{aligned}
 	\end{equation*}

 	For $u_{21} $, using $e^{-z}\sim 1$ and $t_*\ge \frac{t}{2}$, we have
 	\begin{equation*}
 		u_{21}
 		\lesssim
 		x_n^{2-n}
 		\int_{2C_0 \frac{x_n^2}{t}}^{C_0 \frac{{(9C_*)^2} x_n^2}{l_2^2(t)}}
 		e^{-r} r^{\frac{n}{2} -2} dr
 		\begin{cases}
 			l_2^{n-b-1}(t)
 			,
 			& b<n-1
 			\\
 			\ln\left(C_l^2 \frac{l_2(t)}{l_1(t)} \right)  ,
 			& b=n-1
 			\\
 			l_1^{n-b-1}(t) ,
 			& b>n-1 ,
 		\end{cases}
 	\end{equation*}
 	where we always use the convention $\frac{l_2(t)}{l_1(t)}=1$ if $l_1(t) = l_2(t)=0$.
 	Since $n>2$,  we have $\frac{n}{2} -2>-1$. Then
 	\begin{equation*}
 		u_{21}\lesssim
 		\begin{cases}
 			l_2^{n-b-1}(t)
 			,
 			& b<n-1
 			\\
 			\ln\left(C_l^2 \frac{l_2(t)}{l_1(t)} \right)  ,
 			& b=n-1
 			\\
 			l_1^{n-b-1}(t) ,
 			& b>n-1
 		\end{cases}
 		\begin{cases}
 			l_2^{2-n}(t)
 			,&x_n\le l_2(t)
 			\\
 			x_n^{2-n} ,
 			& l_2(t)< x_n\le C_* t^{\frac{1}{2}}
 			\\
 			x_n^{-2} t^{2-\frac{n}{2}} e^{-2C_0 \frac{x_n^2}{t}} ,
 			&
 			x_n> C_* t^{\frac{1}{2}} ,
 		\end{cases}
 	\end{equation*}
 	where we used \eqref{23Oct07-1} for the case $x_n> C_* t^{\frac{1}{2}}$.

 	For $u_{22}$, since $\frac{l_1^2(t)}{(9C_*)^2} \le t-s \le \frac{l_2^2(t)}{(9C_*)^2}$, then
 	\begin{equation*}
 		\begin{aligned}
 			u_{22}	
 			\lesssim \ &
 			\int_{t-\frac{l_2^2(t)}{(9C_*)^2}}^{t-\frac{l_1^2(t)}{(9C_*)^2}}
 			e^{-C_0 \frac{x_n^2 }{t-s}}
 			(t-s)^{-\frac{b+1}{2}}
 			\begin{cases}
 				1,
 				&b<n-1
 				\\
 				\langle \ln ( \frac{ t-s}{l_1^2(t)})  \rangle
 				,
 				&b=n-1
 				\\
 				( \frac{l_1^2(t)}{t-s})^{\frac{n-b}{2} -\frac{1}{2} } ,
 				&b>n-1
 			\end{cases}
 			ds
 			\\
 			\sim \ &
 			\begin{cases}
 				x_n^{1-b}
 				\int_{C_0 \frac{(9C_*)^2 x_n^2}{l_2^2(t)}}^{C_0 \frac{(9C_*)^2 x_n^2}{l_1^2(t)}}
 				e^{-r} r^{\frac{b}{2}-\frac{3}{2}} dr
 				,
 				&b<n-1
 				\\
 				x_n^{2-n}
 				\int_{C_0 \frac{(9C_*)^2 x_n^2}{l_2^2(t)}}^{C_0 \frac{(9C_*)^2 x_n^2}{l_1^2(t)}}
 				e^{-r} r^{\frac{n}{2} -2}
 				\langle \ln (\frac{C_0 x_n^2}{r l_1^2(t)}) \rangle dr
 				,
 				&b=n-1
 				\\
 				l_1^{n-b-1}(t)
 				x_n^{2-n}
 				\int_{C_0 \frac{(9C_*)^2 x_n^2}{l_2^2(t)}}^{C_0 \frac{(9C_*)^2 x_n^2}{l_1^2(t)}}
 				e^{-r} r^{\frac{n}{2}-2} dr
 				,
 				&b>n-1 .
 			\end{cases}
 		\end{aligned}
 	\end{equation*}	
 	Since $n>2$, we have
 	\begin{equation*}
 		u_{22} \lesssim
 		\begin{cases}
 			\begin{cases}
 				l_2^{1-b}(t)
 				,
 				&b< 1
 				\\
 				\ln (\frac{l_2(t)}{l_1(t)} )
 				,
 				&b = 1
 				\\
 				l_1^{1-b}(t)
 				,
 				& b >1
 			\end{cases}
 			,& x_n\le l_1(t)
 			\\
 			\begin{cases}
 				l_2^{1-b}(t)
 				,
 				&b< 1
 				\\
 				\langle \ln (\frac{l_2(t)}{x_n}) \rangle
 				,
 				&b=1
 				\\
 				x_n^{1-b}
 				,
 				&1<b<n-1
 				\\
 				x_n^{2-n}
 				\langle \ln (\frac{x_n}{ l_1(t)}) \rangle
 				,
 				&b=n-1
 				\\
 				x_n^{2-n} l_1^{n-b-1}(t)
 				,
 				&b>n-1
 			\end{cases}
 			,& l_1(t) < x_n\le l_2(t)
 			\\
 			\begin{cases}
 				x_n^{-2} l_2^{3-b}(t)
 				e^{-C_0 \frac{(9C_*)^2 x_n^2}{l_2^2(t)}}
 				,
 				&b<n-1
 				\\
 				x_n^{-2} l_2^{4-n}(t)
 				\langle \ln (\frac{l_2(t)}{l_1(t)}) \rangle
 				e^{-C_0 \frac{(9C_*)^2 x_n^2}{l_2^2(t)}}
 				,
 				&b=n-1
 				\\
 				x_n^{-2}
 				l_2^{4-n}(t)
 				l_1^{n-b-1}(t)
 				e^{-C_0 \frac{(9C_*)^2 x_n^2}{l_2^2(t)}}
 				,
 				&b>n-1
 			\end{cases}
 			, &x_n>l_2(t) ,
 		\end{cases}
 	\end{equation*}
 	where we used \eqref{23Oct09-5}, \eqref{23Oct10-1} for the case $l_1(t) < x_n\le l_2(t)$, $b=n-1$,
 	and \eqref{23Oct07-1}, \eqref{23Oct09-5} for the case $x_n>l_2(t)$.

 	For $u_{23}$, by \eqref{23Oct07-1}, we have
 	\begin{equation*}
 		\begin{aligned}
 			&
 			u_{23} \lesssim
 			l_1^{n-b-3}(t)
 			\int_{t-\frac{l_1^2(t)}{(9C_*)^2}}^{t}
 			e^{-C_0
 				\left(
 				x_n^2 + 4^{-1} C_{l}^{-2}
 				l_1^2(t) \right) \frac{1}{t-s}}
 			(t-s)^{1-\frac{n}{2}}
 			ds
 			\\
 			= \ &
 			l_1^{n-b-3}(t)
 			\left[
 			C_0
 			\left(
 			x_n^2 + 4^{-1} C_{l}^{-2}
 			l_1^2(t) \right)
 			\right]^{2-\frac{n}{2}}
 			\int_{ C_0 (9C_*)^2
 				\frac{ x_n^2 + 4^{-1} C_l^{-2} l_1^2(t) }{l_1^2(t)}
 			}^\infty
 			e^{-r} r^{\frac{n}{2}-3} dr
 			\\
 			\lesssim \ &
 			l_1^{3-b}(t)
 			\left(
 			x_n^2 + 4^{-1} C_{l}^{-2}
 			l_1^2(t) \right)^{-1}
 			e^{- C_0 (9C_*)^2
 				\frac{ x_n^2 +4^{-1}C_l^{-2} l_1^2(t) }{l_1^2(t)}
 			}
 			\sim
 			l_1^{3-b}(t)
 			\left( x_n^2 + l_1^2(t) \right)^{-1}
 			e^{- C_0 \frac{ (9C_*)^2 x_n^2 }{l_1^2(t) }
 			}
 			.
 		\end{aligned}
 	\end{equation*}

 	In sum, when $n>2$, for $|\tilde{x}| \le 2^{-1} C_l^{-1} l_1(t)$, we have
 \begin{equation*}
u_2\lesssim w_{21}:=
\begin{cases}
	\begin{cases}
		l_2^{1-b}(t)
		,
		&b< 1
		\\
		\langle 	\ln (\frac{l_2(t)}{l_1(t)} )
		\rangle
		,
		&b = 1
		\\
		l_1^{1-b}(t)
		,
		& b >1
	\end{cases}
	,
	&
	x_n\le l_1(t)
	\\
	\begin{cases}
		l_2^{1-b}(t)
		,
		&b< 1
		\\
		\langle \ln (\frac{l_2(t)}{x_n}) \rangle
		,
		&b=1
		\\
		x_n^{1-b}
		,
		&1<b<n-1
		\\
		x_n^{2-n}
		\langle \ln (\frac{x_n}{ l_1(t)}) \rangle
		,
		&b=n-1
		\\
		x_n^{2-n} l_1^{n-1-b}(t)
		,
		&b>n-1
	\end{cases}
	,
	&
	l_1(t) < x_n\le l_2(t)
	\\
	\begin{cases}
		l_2^{n-1-b}(t)
		,
		& b<n-1
		\\
		\langle \ln( \frac{l_2(t)}{l_1(t)} )  \rangle,
		& b=n-1
		\\
		l_1^{n-1-b}(t) ,
		& b>n-1
	\end{cases}
	\begin{cases}
		x_n^{2-n} , & l_2(t) <x_n \le C_* t^{\frac{1}{2}}
		\\
		x_n^{-2} t^{2-\frac{n}{2}} e^{-2C_0 \frac{x_n^2}{t}} ,
		& x_n > C_* t^{\frac{1}{2}} ,
	\end{cases}
\end{cases}
 \end{equation*}
 	where we used \eqref{23Oct09-4} for the case $l_1(t) < x_n\le l_2(t)$, $b=n-1$, and Lemma \ref{23Dec23-3-lem} \eqref{23Oct10-4} for the case $x_n > C_* t^{\frac{1}{2}}$.
 	
 	\medskip
 	
 	For $2^{-1} C_l^{-1} l_1(t) < |\tilde{x}| \le \min\left\{9C_l l_2(t), 9C_* t^{\frac{1}{2}} \right\}$,
 	\begin{equation*}
 		\begin{aligned}
 			u_{2}
 			\le \ &  \int_{\frac{t}{2}}^t
 			(t-s)^{-\frac{n}{2}}
 			e^{-C_0 \frac{x_n^2 }{t-s}}
 			\int_{\mathbb{R}^{n-1}}
 			e^{-C_0 \frac{|\tilde{x}-y|^2 }{t-s}}
 			|y|^{-b}
 			\left(
 			\1_{ C_l^{-1} l_1(t) \le |y| < \frac{|\tilde{x}|}{2} }
 			+
 			\1_{\frac{|\tilde{x}|}{2} \le |y| \le 2|\tilde{x}| }
 			+
 			\1_{ 2|\tilde{x}| < |y| \le C_l l_2(t) }
 			\right)
 			dy ds
 			\\
 			:= \ &
 			u_{21} + u_{22} + u_{23} .
 		\end{aligned}
 	\end{equation*}
 	
 	For $u_{21}$, since $  |\tilde{x} -y|\ge \frac{|\tilde{x}|}{2}$, we have
 	\begin{equation*}
 		\begin{aligned}
 			&
 			u_{21}
 			\lesssim
 			\int_{\frac{t}{2}}^t
 			(t-s)^{-\frac{n}{2}}
 			e^{-C_0\left( x_n^2 + \frac{|\tilde{x}|^2}{4}\right) \frac{1}{t-s}}
 			ds
 			\begin{cases}
 				|\tilde{x}|^{n-1-b}, &
 				b<n-1
 				\\
 				\langle\ln ( \frac{|\tilde{x}|}{l_1(t)} ) \rangle,
 				&
 				b=n-1
 				\\
 				l_1^{n-1-b}(t), &
 				b>n-1
 			\end{cases}
 			\\
 			= \ &
 			\left[
 			C_0 \left(x_n^2+\frac{|\tilde{x}|^2}{4}\right)
 			\right]^{1-\frac{n}{2}}
 			\int_{2C_0 \frac{ x_n^2+\frac{|\tilde{x}|^2}{4} }{t}}^\infty
 			e^{-r} r^{\frac{n}{2} -2} dr
 			\begin{cases}
 				|\tilde{x}|^{n-1-b}, &
 				b<n-1
 				\\
 				\langle\ln ( \frac{|\tilde{x}|}{l_1(t)} ) \rangle,
 				&
 				b=n-1
 				\\
 				l_1^{n-1-b}(t), &
 				b>n-1  .
 			\end{cases}
 		\end{aligned}
 	\end{equation*}
 	Since $n>2$, using \eqref{23Oct07-1} for the case $x_n > t^{\frac{1}{2}}$,
 	we have
 	\begin{equation*}
 		u_{21}
 		\lesssim
 		\begin{cases}
 			|\tilde{x}|^{n-1-b}, &
 			b<n-1
 			\\
 			\langle\ln ( \frac{|\tilde{x}|}{l_1(t)} ) \rangle,
 			&
 			b=n-1
 			\\
 			l_1^{n-1-b}(t), &
 			b>n-1
 		\end{cases}
 		\begin{cases}
 			|x|^{2-n}
 			,
 			& x_n\le t^{\frac{1}{2}}
 			\\
 			|x|^{-2} t^{2-\frac{n}{2}}
 			e^{- 2C_0 \frac{x_n^2 }{t} }
 			,
 			& x_n > t^{\frac{1}{2}} .
 		\end{cases}
 	\end{equation*}
 	
 	For $u_{22}$, using $n>1$, we have
 	\begin{equation*}
 		\begin{aligned}
 			& u_{22} \lesssim
 			|\tilde{x}|^{-b}   \int_{\frac{t}{2}}^t
 			(t-s)^{-\frac{n}{2}}
 			e^{-C_0 \frac{x_n^2 }{t-s}}
 			\int_{\mathbb{R}^{n-1}}
 			e^{-C_0 \frac{|\tilde{x}-y|^2 }{t-s}}
 			\1_{|\tilde{x}-y| \le 3|\tilde{x}| }
 			dy ds
 			\\
 			\sim \ &
 			|\tilde{x}|^{-b}   \int_{\frac{t}{2}}^t
 			(t-s)^{-\frac{1}{2}}
 			e^{-C_0 \frac{x_n^2 }{t-s}}
 			\int_0^{9C_0 \frac{|\tilde{x}|^2}{t-s}}
 			e^{-z} z^{\frac{n}{2} -\frac{3}{2}}
 			dz ds
 \\
 \lesssim \ &
 |\tilde{x}|^{-b}
 \left[
 |\tilde{x}|^{n-1}
 \int_{\frac{t}{2}}^{t-\frac{|\tilde{x}|^2}{(99 C_*)^2}}
 (t-s)^{-\frac{n}{2}} e^{-C_0 \frac{x_n^2}{t-s}} ds
 +
 \int_{t-\frac{|\tilde{x}|^2}{(99 C_*)^2}}^t
 (t-s)^{-\frac{1}{2}}
 e^{-C_0 \frac{x_n^2 }{t-s}}
 ds
 \right]
 \\
 = \ &
 |\tilde{x}|^{-b}
 \left[
 |\tilde{x}|^{n-1}
 \left( C_0 x_n^2 \right)^{1-\frac{n}{2}}
 \int_{2C_0 \frac{ x_n^2}{t}}^{C_0 \frac{(99 C_*)^2 x_n^2}{|\tilde{x}|^2}}
 e^{-r} r^{\frac{n}{2}-2} dr
 +
 \left(C_0 x_n^2\right)^{\frac{1}{2}}
 \int_{C_0 \frac{(99 C_*)^2 x_n^2}{|\tilde{x}|^2}}^\infty
 e^{-r} r^{-\frac{3}{2}} dr
 \right].
 		\end{aligned}
 	\end{equation*}
 	Since $n>2$, using \eqref{23Oct07-1} to the case $|\tilde{x}|<x_n\le t^{\frac{1}{2}}$, and \eqref{23Oct07-1}, \eqref{23Oct09-1} to the case $x_n > t^{\frac{1}{2}} $, we have
 	\begin{equation*}
 		u_{22}
 		\lesssim
 		\begin{cases}
 			|\tilde{x}|^{1-b}, & x_n\le |\tilde{x}|
 			\\
 			x_n^{2-n} |\tilde{x}|^{n-1-b}
 			, &  |\tilde{x}|<x_n\le t^{\frac{1}{2}}
 			\\
 			x_n^{-2}
 			|\tilde{x}|^{n-1-b}
 			t^{2-\frac{n}{2}}
 			e^{-2C_0 \frac{x_n^2}{t}}
 			, &  x_n > t^{\frac{1}{2}} ,
 		\end{cases}
 	\end{equation*}
 where we assume that if $|\tilde{x}|\ge t^{\frac{1}{2}}$, the case $|\tilde{x}|<x_n\le t^{\frac{1}{2}}$ is vacuum, and the cases $x_n\le |\tilde{x}|$ and $x_n > t^{\frac{1}{2}}$ have the common part $t^{\frac{1}{2}} < x_n \le |\tilde{x}|$ with the same upper bound up to a multiplicity of a constant. The same convention is used for the other similar conditions.
 	
 	For $u_{23} $,
 	\begin{equation*}
 		\begin{aligned}
 			&
 			u_{23}
 			\le  \int_{\frac{t}{2}}^t
 			(t-s)^{-\frac{n}{2}}
 			e^{-C_0 \frac{x_n^2 }{t-s}}
 			\int_{\mathbb{R}^{n-1}}
 			e^{-\frac{C_0}{4} \frac{|y|^2 }{t-s}}
 			|y|^{-b}
 			\1_{ 2|\tilde{x}| < |y| \le C_l l_2(t) }
 			dy ds
 			\\
 			\sim \ &
 			\left(
 			\int_{\frac{t}{2}}^{t-\frac{l_2^2(t)}{(9C_*)^2}}
 			+
 			\int_{t-\frac{l_2^2(t)}{(9 C_*)^2}}^{t-\frac{|\tilde{x}|^2}{(99C_l C_*)^2}}
 			+
 			\int_{t-\frac{|\tilde{x}|^2}{{(99C_l C_*)^2}}}^t
 			\right)
 			(t-s)^{-\frac{b}{2}-\frac{1}{2}}
 			e^{-C_0 \frac{x_n^2 }{t-s}}
 			\int_{C_0 \frac{|\tilde{x}|^2}{t-s}}^{\frac{C_0 C_l^2}{4} \frac{l_2^2(t)}{t-s}}
 			e^{-z} z^{\frac{n-b}{2} -\frac{3}{2}}
 			dz
 			ds
 			:=
 			u_{231} + u_{232} + u_{233}  ,
 		\end{aligned}
 	\end{equation*}
 where $\frac{|\tilde{x}|^2}{(99C_l C_*)^2}\le \frac{l_2^2(t)}{(9C_*)^2} \le \frac{t}{2}$.
 	
 	For $u_{231}$, using $e^{-z}\sim 1$ for this part, we have
 	\begin{equation*}
 		\begin{aligned}
 			&
 			u_{231}
 			\lesssim
 			\int_{\frac{t}{2}}^{t-\frac{l_2^2(t)}{(9C_*)^2}}
 			(t-s)^{-\frac{n}{2}}
 			e^{-C_0 \frac{x_n^2 }{t-s}}
 			ds
 			\begin{cases}
 				l_2^{n-1-b}(t)
 				,
 				&b<n-1
 				\\
 				\langle \ln \left( |\tilde{x}|^{-1} l_2(t)\right) \rangle
 				,
 				&b=n-1
 				\\
 				|\tilde{x}|^{n-1-b}
 				,
 				&b>n-1
 			\end{cases}
 			\\
 			= \ &
 			\left( C_0 x_n^2\right)^{1-\frac{n}{2}}
 			\int_{2C_0 \frac{x_n^2}{t}}^{C_0 \frac{(9C_*)^2 x_n^2}{l_2^2(t)}}
 			e^{-r} r^{\frac{n}{2}-2} dr
 			\begin{cases}
 				l_2^{n-1-b}(t)
 				,
 				&b<n-1
 				\\
 				\langle \ln \left( |\tilde{x}|^{-1} l_2(t)\right) \rangle
 				,
 				&b=n-1
 				\\
 				|\tilde{x}|^{n-1-b}
 				,
 				&b>n-1 .
 			\end{cases}
 		\end{aligned}
 	\end{equation*}
 	Since $n>2$, using \eqref{23Oct07-1} for the case $x_n > C_* t^{\frac{1}{2}}$, we have
 	\begin{equation*}
 		u_{231} \lesssim
 		\begin{cases}
 			l_2^{n-1-b}(t)
 			,
 			&b<n-1
 			\\
 			\langle \ln (  \frac{l_2(t)}{|\tilde{x}|} ) \rangle
 			,
 			&b=n-1
 			\\
 			|\tilde{x}|^{n-1-b}
 			,
 			&b>n-1
 		\end{cases}
 		\begin{cases}
 			l_2^{2-n}(t), & x_n\le l_2(t)
 			\\
 			x_n^{2-n}
 			, &  l_2(t) < x_n \le C_* t^{\frac{1}{2}}
 			\\
 			x_n^{-2} t^{2-\frac{n}{2}}
 			e^{-2C_0 \frac{x_n^2}{t}}
 			, &  x_n > C_*  t^{\frac{1}{2}}  .
 		\end{cases}
 	\end{equation*}

For $u_{232}$, since $\frac{|\tilde{x}|^2}{(99C_l C_*)^2}  \le t-s \le \frac{l_2^2(t)}{(9 C_*)^2}$, we have
 	\begin{equation*}
 		\begin{aligned}
 			& u_{232}
 			\lesssim
 			\begin{cases}
 				\int_{t-\frac{l_2^2(t)}{(9 C_*)^2}}^{t-\frac{|\tilde{x}|^2}{(99C_l C_*)^2}}
 				(t-s)^{-\frac{b}{2}-\frac{1}{2}}
 				e^{-C_0 \frac{x_n^2 }{t-s}} ds, &b<n-1
 				\\
 				\int_{t-\frac{l_2^2(t)}{(9 C_*)^2}}^{t-\frac{|\tilde{x}|^2}{(99C_l C_*)^2}}
 				(t-s)^{-\frac{n}{2}}
 				e^{-C_0 \frac{x_n^2 }{t-s}}	\langle \ln \left( |\tilde{x}|^{-2} (t-s)\right) \rangle ds
 				, &b=n-1
 				\\
 				|\tilde{x}|^{n-1-b}
 				\int_{t-\frac{l_2^2(t)}{(9 C_*)^2}}^{t-\frac{|\tilde{x}|^2}{(99C_l C_*)^2}}
 				(t-s)^{-\frac{n}{2}}
 				e^{-C_0 \frac{x_n^2 }{t-s}}	
 				ds,
 				&b>n-1
 			\end{cases}
 			\\
 			= \ &
 			\begin{cases}
 				\left(C_0 x_n^2\right)^{\frac{1-b}{2}}
 				\int_{C_0 \frac{(9 C_*)^2 x_n^2}{l_2^2(t)}}^{C_0 \frac{(99C_l C_*)^2x_n^2}{|\tilde{x}|^2}}
 				e^{-r} r^{\frac{b}{2} -\frac{3}{2}} dr
 				, &b<n-1
 				\\
 				(C_0 x_n^2)^{1-\frac{n}{2}}
 				\int_{C_0 \frac{(9 C_*)^2 x_n^2}{l_2^2(t)}}^{C_0 \frac{(99C_l C_*)^2 x_n^2}{|\tilde{x}|^2}}
 				e^{-r} r^{\frac{n}{2} - 2}
 				\langle \ln(\frac{C_0 x_n^2}{r|\tilde{x}|^2}) \rangle dr
 				, &b=n-1
 				\\
 				|\tilde{x}|^{n-1-b}
 				(C_0 x_n^2)^{1-\frac{n}{2}}
 				\int_{C_0 \frac{(9 C_*)^2 x_n^2}{l_2^2(t)}}^{C_0 \frac{(99C_l C_*)^2 x_n^2}{|\tilde{x}|^2}}
 				e^{-r} r^{\frac{n}{2}-2} dr
 				,
 				&b>n-1 .
 			\end{cases}
 		\end{aligned}
 	\end{equation*}
 	Since $n>2$, we have
 	\begin{equation*}
 			u_{232}
 			\lesssim
 			\begin{cases}
 				\begin{cases}
 					l_2^{1-b}(t)
 					, &b<1
 					\\
 					\left|	\ln (\frac{l_2(t)}{|\tilde{x}|}) \right|
 					, &b=1
 					\\
 					|\tilde{x}|^{1-b}
 					, &b>1
 				\end{cases}
 				,& x_n \le |\tilde{x}|
 				\\
 				\begin{cases}
 					l_2^{1-b}(t) , &b<1
 					\\
 					\langle \ln (\frac{l_2(t)}{x_n}) \rangle
 					, &b=1
 					\\
 					x_n^{1-b}
 					, &1<b<n-1
 					\\
 					x_n^{2-n} \langle \ln(\frac{x_n}{|\tilde{x}|}) \rangle
 					, &b=n-1
 					\\
 					x_n^{2-n} |\tilde{x}|^{n-1-b}
 					,
 					&b>n-1
 				\end{cases}
 				,&  |\tilde{x}| < x_n \le l_2(t)
 				\\
 				\begin{cases}
 					x_n^{-2} l_2^{3-b}(t)
 					e^{-C_0 \frac{(9 C_*)^2 x_n^2}{l_2^2(t)} }
 					, &b<n-1
 					\\
 					x_n^{-2} l_2^{4-n}(t) \langle \ln (\frac{l_2(t)}{|\tilde{x}|})\rangle e^{-C_0 \frac{(9 C_*)^2 x_n^2}{l_2^2(t)}}
 					, &b=n-1
 					\\
 					x_n^{-2} |\tilde{x}|^{n-1-b}
 					l_2^{4-n}(t)
 					e^{-C_0 \frac{(9 C_*)^2 x_n^2}{l_2^2(t)}}
 					,
 					&b>n-1
 				\end{cases}
 				,&  x_n > l_2(t) ,
 			\end{cases}
 	\end{equation*}
 	where \eqref{23Oct07-1}, \eqref{23Oct09-5} are utilized for the case $x_n > l_2(t)$;
 	for the case $|\tilde{x}| < x_n \le l_2(t)$, $b=n-1$, we used $n>2$, \eqref{23Oct10-1}, \eqref{23Oct09-5} to get
 	\begin{equation*}
 	\begin{aligned}
 &
 		\bigg(
 		\int_{C_0 \frac{(9 C_*)^2 x_n^2}{l_2^2(t)}}^{1}
 		+
 		\int_{1}^{C_0 \frac{(99C_l C_*)^2 x_n^2}{|\tilde{x}|^2}}
 		\bigg)
 		e^{-r} r^{\frac{n}{2} - 2}
 		\langle \ln(\frac{C_0 x_n^2}{r|\tilde{x}|^2}) \rangle dr
 		\\
 		\lesssim \ &
 		\left( C_0 \frac{x_n^2}{|\tilde{x}|^2}\right)^{\frac{n}{2}-1}
 		\int_{C_0 \frac{x_n^2}{|\tilde{x}|^2}}^{\frac{l_2^2(t)}{(9 C_*)^2 |\tilde{x}|^2}} z^{-\frac{n}{2}} \langle \ln z\rangle dz
 		+
 		\langle \ln (\frac{x_n}{|\tilde{x}|})\rangle
 		\sim
 		\langle \ln (\frac{x_n}{|\tilde{x}|})\rangle .
 		\end{aligned}
 	\end{equation*}
 	
 	For $u_{233}$, by \eqref{23Oct07-1}, then
 	\begin{equation*}
 	\begin{aligned}
 		u_{233} \lesssim \ &
 		|\tilde{x}|^{n-3-b}
 		\int_{t-\frac{|\tilde{x}|^2}{(99C_l C_*)^2 } }^t
 		(t-s)^{1-\frac{n}{2}}
 		e^{-C_0 \frac{ |x|^2 }{t-s}}
 		ds
 		=
 		|\tilde{x}|^{n-3-b}
 		\left( C_0 |x|^2\right)^{2-\frac{n}{2}}
 		\int_{C_0 \frac{(99C_l C_*)^2 |x|^2}{|\tilde{x}|^2}}^\infty
 		e^{-r} r^{\frac{n}{2}-3} dr
 		\\
 		\lesssim \ &
 		|x|^{-2}
 		|\tilde{x}|^{3-b}
 		e^{-C_0 \frac{(99C_l C_*)^2 |x|^2}{|\tilde{x}|^2}}
 	\sim
 	|x|^{-2}
 	|\tilde{x}|^{3-b}
 	e^{-C_0 \frac{(99C_l C_*)^2 x_n^2}{|\tilde{x}|^2}}
 		.
 			\end{aligned}
 	\end{equation*}
 	
 	In sum, when $n>2$, for $2^{-1} C_l^{-1} l_1(t) < |\tilde{x}| \le \min\left\{9C_l l_2(t), 9C_* t^{\frac{1}{2}} \right\}$,
 	through tedious comparison, we have
 \begin{equation*}
 u_2\lesssim w_{22}:= 	
 \begin{cases}
 	\begin{cases}
 		l_2^{1-b}(t)
 		, &b<1
 		\\
 		\langle	\ln (\frac{l_2(t)}{|\tilde{x}|}) \rangle
 		, &b=1
 		\\
 		|\tilde{x}|^{1-b}
 		, &1<b<n-1
 		\\
 		|\tilde{x}|^{2-n}  \langle\ln ( \frac{|\tilde{x}|}{l_1(t)} ) \rangle,
 		&
 		b=n-1
 		\\
 		|\tilde{x}|^{2-n} 	l_1^{n-1-b}(t), &
 		b>n-1
 	\end{cases}
 	,& x_n \le |\tilde{x}|
 	\\
 	\begin{cases}
 		l_2^{1-b}(t) , &b<1
 		\\
 		\langle \ln (\frac{l_2(t)}{x_n}) \rangle
 		, &b=1
 		\\
 		x_n^{1-b}
 		, &1<b<n-1
 		\\
 		x_n^{2-n} \langle \ln(\frac{x_n}{l_1(t)}) \rangle
 		, &b=n-1
 		\\
 		x_n^{2-n} l_1^{n-1-b}(t)
 		,
 		&b>n-1
 	\end{cases}
 	,&  |\tilde{x}| < x_n \le l_2(t)
 	\\
 	\begin{cases}
 		l_2^{n-1-b}(t)
 		,
 		&b<n-1
 		\\
 		\langle \ln (  \frac{l_2(t)}{l_1(t)} ) \rangle
 		,
 		&b=n-1
 		\\
 		l_1^{n-1-b}(t)
 		,
 		&b>n-1
 	\end{cases}
 	\begin{cases}
 		x_n^{2-n} ,&  l_2(t) < x_n \le C_* t^{\frac{1}{2}}
 		\\
 		x_n^{-2} t^{2-\frac{n}{2}}
 		e^{-2C_0 \frac{x_n^2}{t}} ,
 		& x_n> C_* t^{\frac{1}{2}} ,
 	\end{cases}
 \end{cases}
 \end{equation*}
 	where we used Lemma \ref{23Dec23-3-lem} \eqref{23Oct10-4} for the case $x_n > C_* t^{\frac{1}{2}}$;
 	\eqref{23Oct09-4} for the case $x_n \le |\tilde{x}|$, $b=n-1$, and the case $|\tilde{x}| < x_n \le l_2(t)$, $b=n-1$.
 	
 	\medskip
 	
 	For $|\tilde{x}| > \min\left\{9C_l l_2(t), 9C_* t^{\frac{1}{2}} \right\} $, we estimate $\tilde{u}_2$ instead of $u_2$. For $\tilde{u}_2$, note that  $ |\tilde{x}-y|^2 \ge \frac{64}{81}|\tilde{x}|^2 $, then
 	\begin{equation*}
 		\begin{aligned}
 			&
 			\tilde{u}_2 \le
 			\int_{\frac{t}{2}}^t
 			(t-s)^{-\frac{n}{2}}
 			e^{-C_0\left(
 				x_n^2  + \frac{64}{81} |\tilde{x}|^2
 				\right) \frac{1}{t-s}}
 			\int_{\mathbb{R}^{n-1}}
 			|y|^{-b} \1_{ C_l^{-1} l_1(t) \le |y| \le C_l l_2(t)}
 			dy ds
 			\\
 			\lesssim \ &
 			\left[
 			C_0\left(
 			x_n^2  + \frac{64}{81} |\tilde{x}|^2
 			\right)
 			\right]^{1-\frac{n}{2}}
 			\int_{2C_0\frac{
 					x_n^2  + \frac{64}{81}|\tilde{x}|^2 }{t}}^\infty
 			e^{-r} r^{\frac{n}{2}-2} dr
 			\begin{cases}
 				l_2^{n-1-b}(t) ,
 				&  b<n-1
 				\\
 				\ln  (C_l^2 \frac{l_2(t)}{l_1(t)} ) ,
 				&  b=n-1
 				\\
 				l_1^{n-1-b}(t) ,
 				&  b>n-1 .
 			\end{cases}
 		\end{aligned}
 	\end{equation*}
 	Since $n>2$, using \eqref{23Oct07-1} for the case $|x| > t^{\frac{1}{2}}$, we have
 \begin{equation*}
\tilde{u}_2\lesssim w_{23}:= \begin{cases}
	l_2^{n-1-b}(t) ,
	&  b<n-1
	\\
	\langle	\ln (  \frac{l_2(t)}{l_1(t)} ) \rangle,
	&  b=n-1
	\\
	l_1^{n-1-b}(t) ,
	&  b>n-1
\end{cases}
\begin{cases}
	|x|^{2-n}, &|x|\le t^{\frac{1}{2}}
	\\
	|x|^{-2} t^{2-\frac{n}{2}} e^{-2C_0 \frac{
			x_n^2  + \frac{64}{81}|\tilde{x}|^2 }{t}}
	,
	&  |x| > t^{\frac{1}{2}} .
\end{cases}
 \end{equation*}

Since $w_{21} \sim w_{22}$ when $|\tilde{x}|\sim l_1(t)$, $\min\left\{9C_l l_2(t), 9C_* t^{\frac{1}{2}} \right\} \sim l_2(t)$ due to the assumption $l_2(t)\le C_* t^{\frac{1}{2}}$, and $w_{22} \sim w_{23}$ when $|\tilde{x}|\sim l_2(t)$, we get the first conclusion
\begin{equation*}
	\mathcal{T}\left[v(t) |x|^{-b} \1_{ l_1(t) \le |x| \le l_2(t) } \right] \lesssim
	w_1 + \sup\limits_{ t_1 \in \left[\max\{t_0,t/2 \},t \right] }
	v(t_1)
	\left(w_{21}\1_{|\tilde{x}|\le l_1(t)}
	+
	w_{22}\1_{ l_1(t)< |\tilde{x}|\le l_2(t)}
	+
	w_{23}\1_{|\tilde{x}|> l_2(t)}
	\right),
\end{equation*}
where $ w_{21}\1_{|\tilde{x}|\le l_1(t)}
+
w_{22}\1_{ l_1(t)< |\tilde{x}|\le l_2(t)}
+
w_{23}\1_{|\tilde{x}|> l_2(t)}
 \sim w_2$.

The upper bound of $\mathcal{T}\left[
v(t) \left(|\tilde{x}|+l_1(t)\right)^{-b} \1_{ |\tilde{x}| \le l_2(t) }
\right]$ is deduced by the first conclusion and
 \begin{equation*}
v(t) \left(|\tilde{x}|+l_1(t)\right)^{-b} \1_{ |\tilde{x}| \le l_2(t) } \sim
v(t) l_1^{-b}(t) \1_{|\tilde{x}|\le l_1(t)}
+
v(t)|\tilde{x}|^{-b} \1_{ l_1(t) \le |\tilde{x}| \le l_2(t) } .
 \end{equation*}

 \end{proof}

\subsection{Exponential decay estimates for Neumann boundary \texorpdfstring{$t^a |\tilde{x}|^{-b} \1_{ l_1(t) \le |\tilde{x}| \le l_2(t) } $ and $ t^a \left(|\tilde{x}|+l_1(t)\right)^{-b} \1_{ |\tilde{x}| \le l_2(t) }$}{in the self-similar region}}

 \begin{lemma}\label{bd-ann-23Lem}
 	Let  $n>2$ be an integer, $t>t_0 \ge 1$, $x=(\tilde{x},x_n)$, $\tilde{x}\in \mathbb{R}^{n-1}$, $x_n \ge 0$. For an admissible function $f$, we define
 \begin{equation*}
 		\mathcal{T}_1\left[f\right](x,t):=
 		2 \int_{t_0}^t\int_{\mathbb{R}^{n-1}}
 		\left[4\pi(t-s)\right]^{-\frac{n}{2}}
 		e^{- \frac{|\tilde{x}-y|^2+ x_n^2 }{4(t-s)}}
 		f(y,s)
 		dy ds,
 	\end{equation*}
 	which satisfies the equation
 	\begin{equation*}
 		\pp_t u =
 		\Delta u
 		\mbox{ \ in \ } \mathbb{R}_+^n \times (t_0,\infty),
 		\quad
 		-\pp_{x_n} u =
 		f(\tilde{x},t)
 		\mbox{ \ on \ } \pp\mathbb{R}_+^n \times (t_0,\infty) ,
 		\quad
 		u(\cdot,t_0)=0
 		\mbox{ \ in \ } \mathbb{R}_+^{n} .
 	\end{equation*}
 	
 	Suppose that $0\le l_1(s) \le l_2(s) \le C_{*} s^{\frac{1}{2}}$ for $s\in [t_0,t]$,
 	$C_l^{-1} l_1(t) \le l_1(s)$ for all $s\in[\max\{t_0,\frac{t}{2} \},t] $, $C_l^{-1} t^{p_2} \le l_2(t) \le C_l t^{p_2} $, where $C_*>0,C_l\ge1$ are constants,
 	\begin{equation*}
 		b<n-1,
 		\quad
 		\begin{cases}
 			p_2 \le \frac{1}{2},
 			a+p_2(1-b)+\frac{n}{2} \ge 0 ,
 			&
 			\mbox{ \ if \ }
 			a+p_2(n-1-b)\ne -1
 			\\
 			p_2<\frac{1}{2}
 			,
 			&
 			\mbox{ \ if \ }
 			a+p_2(n-1-b)=-1  ,
 		\end{cases}
 	\end{equation*}
 	then
 	\begin{equation*}
 		\mathcal{T}_1\left[t^a |\tilde{x}|^{-b} \1_{ l_1(t) \le |\tilde{x}| \le l_2(t) } \right](x,t) \le C
 		t^a
 		\begin{cases}
 			\begin{cases}
 				l_2^{1-b}(t)
 				,
 				&b< 1
 				\\
 				\langle 	\ln (\frac{l_2(t)}{l_1(t)} )
 				\rangle
 				,
 				&b = 1
 				\\
 				l_1^{1-b}(t)
 				,
 				& b >1
 			\end{cases}
 			,
 			&
 			|x|\le l_1(t)
 			\\
 			\begin{cases}
 				l_2^{1-b}(t) , &b<1
 				\\
 				\langle \ln (\frac{l_2(t)}{|x|}) \rangle
 				, &b=1
 				\\
 				|x|^{1-b}
 				, & b>1
 			\end{cases}
 			,&  l_1(t) < |x| \le l_2(t)
 			\\
 			l_2^{1-b}(t) e^{-  \frac{|\tilde{x}|^2+|x_n+1|^2}{4 t}}
 			,&  |x| > l_2(t)
 		\end{cases}
 	\end{equation*}
 	with the convention $\frac{l_2(t)}{ l_1(t) }=1 $ if $l_1(t) = l_2(t)=0$, where $C$ is a constant only depending on $n, a, b, p_2, C_*, C_l$.

 	Under the additional assumption that $l_1(s) \le C_l l_1(t)  $ for all $s\in[\max\{t_0,\frac{t}{2} \},t] $, the same conclusion is held for
 	$\mathcal{T}_1\left[ t^a \left(|\tilde{x}|+l_1(t)\right)^{-b} \1_{ |\tilde{x}| \le l_2(t) }
 	\right]$.

 \end{lemma}

 \begin{proof}
 	By Lemma \ref{bd-annu-Lem}, the upper bound estimate of $\mathcal{T}_1\left[f\right](x,t)$ has two parts. For $b<n-1$, $p_2\le \frac{1}{2}$, in order to make the second part dominate the first part in the range $|x|\le l_2(t)$,
 	it suffices to guarantee
 	\begin{equation*}
 	t^{-\frac n2}	\int_{t_0}^{\max\{t_0,\frac{t}{2} \} }
 		s^{a+p_2(n-1-b)}
 		d s
 		\lesssim
 		t^{a+p_2(1-b)   }  ,
 	\end{equation*}
 	which can be deduced under the assumption
 	\begin{equation*}
 		\begin{cases}
 			p_2 \le \frac{1}{2},
 			a+p_2(1-b)+\frac{n}{2} \ge 0 ,
 			&
 			a+p_2(n-1-b)\ne -1
 			\\
 			p_2<\frac{1}{2}
 			,
 			&
 			a+p_2(n-1-b)=-1  .
 		\end{cases}
 	\end{equation*}
Then $\mathcal{T}_1\left[t^a |\tilde{x}|^{-b} \1_{ l_1(t) \le |\tilde{x}| \le l_2(t) } \right] \lesssim t^{a+p_2(1-b)}$ on $|x|=l_2(t)$. Denote $f:= t^{a+p_2(1-b)} e^{-\kappa \frac{|\tilde{x}|^2+|x_n+1|^2}{t}}$.
 	Then
 	\begin{equation*}
 		\left( \pp_t -\Delta\right) f =
 		e^{-\kappa \frac{|\tilde{x}|^2+|x_n+1|^2}{t}}
 		t^{-1+a+p_2(1-b) }
 		\left[
 		\left( \kappa -4\kappa^2\right) t^{-1} \left( |\tilde{x}|^2+|x_n+1|^2 \right)
 		+2\kappa n
 		+ a+
 		p_2(1-b)
 		\right],
 	\end{equation*}
 	which is non-negative in $\mathbb{R}^n\times \left(t_0,\infty\right)$ if
 	$\kappa\in[0,\frac{1}{4}]$ and $2\kappa n
 	+ a+
 	p_2(1-b)\ge 0$.
 	\begin{equation*}
 		\left(-\pp_{x_n} f\right)\left( \tilde{x},0,t\right)
 		=
 		2\kappa
 		t^{a+p_2(1-b)-1} e^{-\kappa \frac{|\tilde{x}|^2+1}{t}} .
 	\end{equation*}
 	
 	Take $\kappa=\frac{1}{4}$. As a result, $Cf$ is a barrier function in the range $|x|>l_2(t)$ when $C$ is sufficiently large and the first conclusion holds.

The pointwise upper bound of $\mathcal{T}_1\left[ t^a \left(|\tilde{x}|+l_1(t)\right)^{-b} \1_{ |\tilde{x}| \le l_2(t) }
 	\right]$ is deduced similarly.
 	
 \end{proof}

\subsection{Neumann boundary value \texorpdfstring{$ t^a |\tilde{x}|^{-b} e^{-\ell \frac{|\tilde{x}|^2}{t}} \1_{|\tilde{x}| \ge t^{\frac{1}{2}}  }$}{outside the self-similar region}}

\begin{lemma}\label{lem:far-23Oct16}

Let  $n\ge 2$ be an integer, $\ell>0$, $\kappa\in (0,\frac{1}{4})$, $\frac{1}{2}+a-\frac{b}{2} + 2 \kappa n \ge 0 $, $t> t_0\ge 1$, $x=(\tilde{x},x_n)$, $\tilde{x}\in \mathbb{R}^{n-1}$, $x_n \ge 0$, $\mathcal{T}_1\left[ \cdot\right]$ defined in Lemma \ref{bd-ann-23Lem}, then
	\begin{equation*}
	\mathcal{T}_1\left[ t^a |\tilde{x}|^{-b} e^{-\ell \frac{|\tilde{x}|^2}{t}} \1_{|\tilde{x}| \ge t^{\frac{1}{2}}  } \right](x,t) \le
	t^{a-\frac{b}{2}+\frac{1}{2}} e^{-\kappa \frac{|x|^2}{t}}
		\begin{cases}
			C(\kappa)
			 , & \mbox{ \ if \ } b\ge -1, \kappa \in \left(0,\ell \right]\cap (0,\frac{1}{4})
			\\
			C(b,\ell,\kappa)
			, & \mbox{ \ if \ }  b< -1, \kappa \in  (0,\min\{\ell,\frac{1}{4}\} )  ,
		\end{cases}
	\end{equation*}
	where $C(\kappa)>0$ is a constant depending on $\kappa$ and $C(b,\ell,\kappa)>0$ is a constant depending on $b,\ell,\kappa$.
	
\end{lemma}

\begin{proof}
	
	Set $ f_1(x,t) = t^{a-\frac{b}{2}+\frac{1}{2}} e^{-\kappa \frac{|\tilde x|^2+ x_n^2}{t}}$,
$f_2(x,t) = t^{a-\frac{b}{2}+\frac{1}{2}} e^{-\kappa \frac{\left(|\tilde{x}|+x_n\right)^2}{t}} $
	 with a constant $\kappa$ to be determined later.
	Direct calculation gives
	\begin{equation*}
		\begin{aligned}
			&
			\left( \pp_t -\Delta \right) f_1
			=
			e^{-\kappa\frac{|\tilde{x}|^2+ x_n^2}{t}}
			t^{a-\frac{b}{2}-\frac{3}{2}}
			\left[
		\left(\frac{1}{2}+a-\frac{b}{2}+2\kappa n\right)t
		+\kappa\left(1-4\kappa\right)
		\left( |\tilde{x}|^2+ x_n^2\right)	\right]
			,
			\\
			&
			\left(-\pp_{x_n} f_1 \right)\left(\tilde{x},0,t\right)
			= 0
			,
\\
&
\left( \pp_t -\Delta \right) f_2
=
e^{-\kappa \frac{\left(|\tilde{x}|+x_n\right)^2}{t}}
t^{a-\frac{b}{2}-\frac{3}{2}}
\left[
\left(\frac{1}{2}+a-\frac{b}{2}+2\kappa n\right)t + 2\kappa\left(n-2\right) t
\frac{x_n}{|\tilde{x}|}
+
\kappa \left(1-8\kappa\right)
\left(|\tilde{x}|+x_n\right)^2
\right],
\\
&
\left(-\pp_{x_n} f_2 \right)\left(\tilde{x},0,t\right)
=2\kappa t^{a-\frac{b}{2}-\frac{1}{2}}
|\tilde{x}| e^{-\kappa \frac{|\tilde{x}|^2}{t}} .
		\end{aligned}
	\end{equation*}
Since $n\ge 2$, $\kappa\in (0,\frac{1}{4})$, $\frac{1}{2}+a-\frac{b}{2} + 2 \kappa n \ge 0 $, $x_n\ge 0$, there exists $C_2(\kappa)>0$ small such that $\left( \pp_t -\Delta \right) \left(f_1 + C_2(\kappa) f_2 \right)\ge 0$.  Additionally,
	\begin{equation*}
	\begin{aligned}
&
		\left[-\pp_{x_n} \left(f_1 + C_2(\kappa) f_2 \right) \right]\left(\tilde{x},0,t\right)
		\ge
	C_2(\kappa) 2\kappa t^{a-\frac{b}{2}-\frac{1}{2}}
	|\tilde{x}| e^{-\kappa \frac{|\tilde{x}|^2}{t}}
	\\
		\ge \ &
		t^a  |\tilde{x}|^{-b} e^{-\ell \frac{|\tilde{x}|^2}{t}} \1_{|\tilde{x}| \ge t^{\frac{1}{2}}  }
		\begin{cases}
			C(\kappa) , &
			\mbox{ \ if \ } b\ge -1, \kappa \in (0,\ell]
			\\
			C(b,\ell,\kappa) , &
			\mbox{ \ if \ }
			 b< -1, \kappa \in (0,\ell)
		\end{cases}
	\end{aligned}
	\end{equation*}
for some positive constants $C(\kappa), C(b,\ell,\kappa)$.
Thus, $C \left(f_1 + C_2(\kappa) f_2 \right) $ is a barrier function with $C$ sufficiently large.
 	
\end{proof}

\subsection{Right-hand sides \texorpdfstring{$t^a |x|^{-b} \1_{ l_1(t) \le |x| \le l_2(t) }$ and $t^a \left(|x|+l_1(t)\right)^{-b} \1_{|x|\le l_2(t)} $}{in the self-similar region}  }

\begin{lemma}\label{2023Oct01-1-lem}
	Let  $n>2$ be an integer, $t>t_0 \ge 1$, $x \in \mathbb{R}^n$. For an admissible function $f$, we define
	\begin{equation}\label{zz24Jan02-3}
		\mathcal{T}_{\mathbb{R}^n }\left[f\right](x,t):=  \int_{t_0}^t\int_{\mathbb{R}^n }
		\left[4\pi(t-s)\right]^{-\frac{n}{2}}
		e^{- \frac{|x-y|^2}{4(t-s)}}
		f(y,s)
		dy ds,
	\end{equation}
	which satisfies the equation $
		\pp_t u =
		\Delta u  +f $  in $\mathbb{R}^n \times (t_0,\infty)$,
		$ u(\cdot,t_0)=0 $
		in $\mathbb{R}^{n}$.
	
	Suppose that $0\le l_1(s) \le l_2(s) \le C_{*} s^{\frac{1}{2}}$ for $s\in [t_0,t]$,
	$C_l^{-1} l_1(t) \le l_1(s)$ for all $s\in[\max\{t_0,\frac{t}{2} \},t] $, $C_l^{-1} t^{p_2} \le l_2(t) \le C_l t^{p_2} $, where $C_*>0,C_l\ge1$ are constants,
\begin{equation*}
	b<n,
	\quad
	\begin{cases}
		p_2 \le \frac{1}{2},
		a+p_2(2-b)+\frac{n}{2} \ge 0 ,
		& \mbox{ \ if \ }
		a+p_2(n-b)\ne -1
		\\
		p_2<\frac{1}{2}
		,
		& \mbox{ \ if \ }
		a+p_2(n-b)=-1  ,
	\end{cases}
\end{equation*}
then
\begin{equation*}
	\mathcal{T}_{\mathbb{R}^n }\left[t^a |x|^{-b} \1_{ l_1(t) \le |x| \le l_2(t) }\right](x,t)
	\le C
	t^a
	\begin{cases}
		\begin{cases}
			l_2^{2-b}(t)
			& \mbox{ \ if \ }
			b<2
			\\
			\langle	\ln (\frac{l_2(t)}{l_1(t)}) \rangle
			& \mbox{ \ if \ }
			b=2
			\\
			l_1^{2-b}(t)
			& \mbox{ \ if \ }
			b>2
		\end{cases} ,
		&  |x| \le l_1(t)
		\\
		\begin{cases}
			l_2^{2-b}(t)
			&
			\mbox{ \ if \ } b<2
			\\
			\langle\ln(\frac{l_2(t)}{|x|} )  \rangle
			&
			\mbox{ \ if \ } b=2
			\\
			|x|^{2-b}
			&
			\mbox{ \ if \ } b>2
		\end{cases} ,
		&
		l_1(t) < |x| \le l_2(t)
		\\
		l_2^{2-b}(t) e^{- \frac{|x|^2}{4 t}} ,
		&
		|x| > l_2(t)
	\end{cases}
\end{equation*}
with the convention $\frac{l_2(t)}{ l_1(t) }=1 $ if $l_1(t) = l_2(t)=0$, where $C$ is a constant only depending on $n, a, b, p_2, C_*, C_l$.

Under the additional assumption that $l_1(s) \le C_l l_1(t)  $ for all $s\in[\max\{t_0,\frac{t}{2} \},t] $, the same conclusion is held for
$\mathcal{T}_{\mathbb{R}^n }\left[t^a \left(|x|+l_1(t)\right)^{-b} \1_{|x|\le l_2(t)}
\right]$.

\end{lemma}

\begin{proof}
Similar to the proof of Lemma \ref{bd-annu-Lem}, the requirement ``$C_l^{-1} l_i(t) \le l_i(s) \le C_l l_i(t)$, $i=1,2$, for all $\frac{t}{2} \le s \le t$, $t\ge t_0$'' in \cite[Lemma A.1]{infi4D} can be relaxed to ``$C_l^{-1} l_1(t) \le l_1(s)$, and $l_2(s) \le C_l l_2(t)  $, for all $s\in[\max\{t_0,\frac{t}{2} \},t] $'' and the ``$\lesssim$'' in \cite[Lemma A.1]{infi4D} is indeed independent of $t_0$.

One recalling \cite[Lemma A.1]{infi4D}, the estimate of $\mathcal{T}_{\mathbb{R}^n}\left[ t^a |x|^{-b} \1_{ l_1(t) \le |x| \le l_2(t) }\right]$ is split into two parts. We call the term including $\int_{t_0/2}^{t/2}$ as the first part and the other term as the second part. For $b<n$, $p_2\le \frac{1}{2}$, in order to make the second part dominate the first part in the range $|x|\le l_2(t)$,
it suffices to make
\begin{equation*}
t^{-\frac n2}
\int_{t_0}^{t/2 }
s^{a+p_2(n-b)}
d s
\le C
t^{a+p_2(2-b) }
\end{equation*}
with a constant $C$ independent of $t_0$. This can be deduced under the assumption
\begin{equation*}
\begin{cases}
	p_2 \le \frac{1}{2},
	a+p_2(2-b)+\frac{n}{2} \ge 0 ,
	& \mbox{ \ if \ }
	a+p_2(n-b)\ne -1
\\
p_2<\frac{1}{2}
,
& \mbox{ \ if \ }
a+p_2(n-b)=-1  .
\end{cases}
\end{equation*}
Then $ \mathcal{T}_{\mathbb{R}^n }\left[t^a |x|^{-b} \1_{ l_1(t) \le |x| \le l_2(t) }\right]  \lesssim t^{a+p_2(2-b)}$ on $|x|=l_2(t)$. Denote $f:= t^{a+p_2(2-b)} e^{-\kappa \frac{|x|^2}{t}}$ with a constant $\kappa$ to be determined later. Then
\begin{equation*}
	\left( \pp_t -\Delta\right) f =
e^{-\kappa \frac{|x|^2}{t}}
t^{-1+a+p_2(2-b) }
\left[
\left( \kappa -4\kappa^2\right) t^{-1} |x|^2
+2\kappa n
+ a+
p_2(2-b)
\right],
\end{equation*}
which is non-negative in $\mathbb{R}^n\times \left(t_0,\infty\right)$ if
$\kappa\in[0,\frac{1}{4}]$ and $2\kappa n
+ a+
p_2(2-b)\ge 0$. For $b<n$, $p_2\le \frac{1}{2}$, we take $\kappa=\frac{1}{4}$ and $Cf$ is a barrier function in the range $|x|>l_2(t)$ when $C$ is sufficiently large.

The estimate about $\mathcal{T}_{\mathbb{R}^n }\left[t^a \left(|x|+l_1(t)\right)^{-b} \1_{|x|\le l_2(t)}
\right] $ is similar.
\end{proof}

\section{Sobolev-type lemmas}

\begin{lemma}\label{z24Jan15-8}

		Given a domain $\Omega\subset\mathbb{R}^n$ (possibly unbounded), suppose that
  $2\le p <\infty $ if $n =1,2$, $2\le p \le \frac{2n}{n-2}$ if $n\ge 3$, and a sequence $(u_k)_{k \ge 1} \subset H^1(\Omega)$ satisfies
		\begin{equation*}
			u_k \rightharpoonup v_1 \mbox{ \ in \ } H^1(\Omega) ,
			\quad
			\nabla u_k \rightharpoonup \bm{v_2} = (v_{21},v_{22},\dots, v_{2n}) \mbox{ \ in \ } L^2(\Omega) ,
			\quad
			u_k  \rightharpoonup v_3 \mbox{ \ in \ } L^p(\Omega),
		\end{equation*}
		then  $\nabla v_1 = \bm{v_2} $ in $L^2(\Omega)$, $v_1 = v_3$ in $L^p(\Omega)$.
		
\end{lemma}

\begin{proof}
	
	Given a function $f\in L^2(\Omega)$, $\int_{\Omega} f\pp_{x_i} u_k \rightarrow  \int_{\Omega} f  v_{2i} $, $i=1,2,\dots, n$.
	For any $g \in H^1(\Omega)$, $
		g	\mapsto \int_{\Omega} f \pp_{x_i} g $
	is a bounded linear mapping on $H^1(\Omega)$. By $u_k \rightharpoonup v_1$ in $H^1(\Omega)$, we have
	$ \int_{\Omega} f \pp_{x_i} u_k \rightarrow
		\int_{\Omega} f \pp_{x_i} v_1 $.
	Due to the arbitrary choice of $f\in L^2(\Omega)$, then $\pp_{x_i} v_1 = v_{2i}  $ in $L^2(\Omega)$.

	For the other part, given a function $f_1\in L^{p'}(\Omega)$, where $\frac{1}{p} + \frac{1}{p'} =1$, $
		\int_{\Omega} f_1 u_k \rightarrow  \int_{\Omega} f_1 v_3 $.
	For any $g\in H^1(\Omega)$, $
		g	\mapsto	\int_{\Omega} f_1 g $
	is also a bounded linear mapping on $H^1(\Omega)$ by the Sobolev embedding theorem and the choice of $p$. By $u_k \rightharpoonup v_1$ in $H^1(\Omega)$, we have $
		\int_{\Omega}  f_1 u_k \rightarrow
		\int_{\Omega}  f_1 v_1 $.
	Due to the arbitrary choice of $f_1\in L^{p'}(\Omega)$, then $v_1 = v_3$ in $L^p(\Omega) $.
\end{proof}

\begin{lemma}\label{z24Jan15-6}
	
Let $V(\tilde{x}) \in L^\infty( \mathbb{R}^{n-1})$, $V(\tilde{x}) \to 0$ as $|\tilde{x}| \to \infty$. If a sequence $(u_k)_{k\ge 1}$ satisfies
$u_k \rightharpoonup u_0$ in $H^1(\mathbb{R}_+^n)$ as $k\to \infty$, then up to a subsequence, $\int_{\mathbb{R}^{n-1}} |V(\tilde{x})| \left|(u_k-u_0)(\tilde{x},0) \right|^2 d\tilde{x} \to 0$ and $\int_{\mathbb{R}^{n-1}} V(\tilde{x}) \left|u_k(\tilde{x},0) \right|^2 d\tilde{x} \to \int_{\mathbb{R}^{n-1}} V(\tilde{x}) \left|u_0(\tilde{x},0) \right|^2 d\tilde{x}$ as $k\to \infty$.

\end{lemma}

\begin{proof}
	
Without loss of generality, assume $u_k \rightharpoonup 0$.
Since the Sobolev embedding $H^1(\mathbb{R}^n_+) \hookrightarrow L^2(\pp \mathbb{R}_+^n)$ is continuous, we have $\sup_{k\ge 0} \| u_k(\cdot,0) \|_{L^2(\mathbb{R}^{n-1})} \le C$. Thus, for any $\epsilon>0$, there exists  $R$ sufficiently large such that
$ \int_{|\tilde{x}|\ge R} \left| V(\tilde{x}) \right| \left|u_k(\tilde{x},0) \right|^2 d\tilde{x}$ $ <\epsilon $.
Since for any compact set $\Omega \subset \pp \mathbb{R}^n_+$, the Sobolev embedding $H^1(\mathbb{R}_+^n) \hookrightarrow L^2(\Omega)$ is compact, up to a subsequence, $
\lim\limits_{k\to \infty}	\int_{|\tilde{x}| < R} V(\tilde{x}) \left|u_k(\tilde{x},0) \right|^2 d \tilde{x}  =0 $. Thus, $\int_{\mathbb{R}^{n-1}} |V(\tilde{x})| \left|(u_k-u_0)(\tilde{x},0) \right|^2 d\tilde{x} \to 0$ as $k\to \infty$. Combining $\sup_{k\ge 0}$$ \| u_k(\cdot,0) \|_{L^2(\mathbb{R}^{n-1})}$ $ \le C$, we have the second convergence result.
\end{proof}

Recall $L^2_{\rho}(\mathbb{R}^{n-1}), H^1_\rho(\mathbb{R}^n_+)$ defined in \eqref{z24Jan25-8}, \eqref{z24Jan25-7} respectively.
\begin{lemma}\label{Sob-lem1}
Given an integer $n\ge2$, for $u\in H^1_\rho(\mathbb{R}^n_+)$, we have
	$\int_{\mathbb{R}^{n-1}} u^2(\tilde{x},0) e^{-\frac{|\tilde x|^2}{4}} d\tilde x \le \frac{n+4}{4}\|u\|_{H_\rho^1(\mathbb{R}_+^n)}^2$.
	
\end{lemma}

\begin{proof}
Set the norm
 $\|f\|_{H_V^1(\mathbb{R}^n_+)}:=\big[ \int_{\mathbb{R}^n_+}(|\nabla f|^2+V(x)f^2)dx \big]^{1/2} $ with $V(x):=\frac{|x|^2}{16}+1$.
	For any $u\in C_c^\infty(\overline{\mathbb{R}^n_+})$, let $u(x)=e^\frac{|x|^2}{8}v(x)$, then direct calculation gives that
	\begin{equation*}
 \begin{aligned}
     &
		\int_{\mathbb{R}^n_+}|u|^2e^{-\frac{|x|^2}{4}}dx=\int_{\mathbb{R}^n_+}v^2dx, \quad
		\int_{\mathbb{R}^n_+}|\nabla u|^2e^{-\frac{|x|^2}{4}}dx=\int_{\mathbb{R}^n_+}\left[|\nabla v|^2+\left(\frac{|x|^2}{16}-\frac{n}{4}\right)v^2\right] dx,
  \\
  &
  \|v\|_{H_V^1(\mathbb{R}^n_+)}^2=\int_{\mathbb{R}^n_+}\left(|\nabla u|^2+  \frac{n+4}{4} u^2 \right)e^{-\frac{|x|^2}{4}}dx,
  \end{aligned}
  \end{equation*}
	which indicates
	$ \|u\|_{H_\rho^1(\mathbb{R}^n_+)}^2\leq \|v\|_{H_V^1(\mathbb{R}^n_+)}^2\leq \frac{n+4}{4}\|u\|_{H_\rho^1(\mathbb{R}^n_+)}^2$.
	Thus,
	\begin{equation}\label{trace-4}
		\begin{aligned}
			&
			\int_{\mathbb{R}^{n-1}} u^2(\tilde{x},0) e^{-\frac{|\tilde x|^2}{4}} d\tilde x=\int_{ \mathbb{R}^{n-1}} v^2(\tilde{x},0) d\tilde{x}
			=
			-
			\int_{\mathbb{R}_+^n} \pp_{x_n}( v^2) dx
			=
			- 2
			\int_{\mathbb{R}_+^n} v \pp_{x_n} v dx
			\\
			\le \ & 2
			\| v \|_{L^2(\mathbb{R}_+^n)}
			\| \pp_{x_n} v \|_{L^2(\mathbb{R}_+^n)}
			\le \
			\| v \|_{L^2(\mathbb{R}_+^n)}^2
			+
			\| \nabla v \|_{L^2(\mathbb{R}_+^n)}^2\le \frac{n+4}{4}\|u\|_{H_\rho^1(\mathbb{R}^n_+)}^2 ,
		\end{aligned}
	\end{equation}
	which implies the lemma by approximating general functions in $H^1_\rho(\mathbb{R}^n_+)$ by $C_c^\infty(\overline{\mathbb{R}^n_+})$.
\end{proof}

\section{Negative eigenvalue and exponential decay of the eigenfunction}\label{negative-eigenvalue}

This section is devoted to the eigenvalue problem about the linearized equation around the ground state $U(x)$. The main result is Proposition \ref{z24Jan25-6-prop}.

Define
$D^{1,2}(\mathbb{R}_+^n)$ as the completion of $C_c^{\infty}(\overline{\mathbb{R}_+^n})$ ($f(\tilde{x},0)$ can be nonzero) under the norm $\left\| \nabla f \right\|_{L^2(\mathbb{R}_+^n)}$. Obviously, $H^1(\mathbb{R}^n_+) \subsetneqq D^{1,2}(\mathbb{R}_+^n)$.
Given an integer $n\ge 3$, $p =\frac{n}{n-2}$,
for $v, f, g\in D^{1,2}(\mathbb{R}_+^n)$, set functionals
\begin{equation}\label{z24Jan23-1}
\begin{aligned}
&
J[v] :=  \int_{\mathbb{R}_+^n} \frac{\left| \nabla v \right|^2}{2}  dx -
\int_{ \mathbb{R}^{n-1} } \frac{ \left| v\right|^{p+1}(\tilde{x},0) }{p+1}  d\tilde{x} ,
\quad
I[v] :=
 \int_{\mathbb{R}_+^n} \left| \nabla v \right|^2 dx
 -
\int_{\mathbb{R}^{n-1} } \left| v\right|^{p+1}(\tilde{x},0) d\tilde{x} ,
\\
&
B_v[f,g] := \int_{\mathbb{R}_+^n} \nabla f\cdot \nabla g dx - p
\int_{\mathbb{R}^{n-1}}
\left( |v|^{p-1}  f g\right)(\tilde{x},0) d \tilde{x}  .
\end{aligned}
\end{equation}
Direct calculation deduces that $J[\cdot], I[\cdot] \in C^2\left( D^{1,2}(\mathbb{R}_+^n) ,\mathbb{R} \right)$ and $I[v]=
\langle J'[v], v\rangle$, $B_v[f,g] := \langle J''[v] f,g \rangle$.
Define
\begin{equation}\label{z24Jan24-6}
	\begin{aligned}
		&
		D_*^{1,2}(\mathbb{R}_+^n)
		:=
		\left\{ f \in D^{1,2}(\mathbb{R}_+^n) \ | \
		f(\tilde{x},0) \in L^{p+1}(\mathbb{R}^{n-1}) \backslash \{ 0 \}
		\right\} ,
		\\
		&
		Q(f) :=
		\Big( \int_{ \mathbb{R}^{n-1} }
		|f|^{p+1}(\tilde{x},0) d\tilde{x} \Big)^{- \frac{2}{p+1} }
		\int_{\mathbb{R}_+^n} |\nabla f|^2 dx
		,
		\quad
		f\in D_*^{1,2}(\mathbb{R}_+^n) .
	\end{aligned}
\end{equation}
Define the Nehari manifold
\begin{equation}\label{qd24Jan20-9}
\mathbf{Ne}:= \left\{ f \in D^{1,2}(\mathbb{R}_+^n) \backslash \{ 0 \} \ | \ I[f]=0 \right\} .
\end{equation}

By \cite[Theorem 1]{Escobar88},
for all
$f\in D^{1,2}(\mathbb{R}_+^n) $,
\begin{equation}\label{z24Jan30-1}
	(n-2)^{\frac{1}{2}} 2^{-\frac{1}{2}} |S^{n-1}|^{\frac{1}{2(n-1)}} \| f (\cdot, 0) \|_{L^{p+1}(\mathbb{R}^{n-1}) } \le  \| \nabla f \|_{L^2(\mathbb{R}_+^{n}) }
	\quad
	\mbox{ \ with \ }
	p+1 = \frac{2(n-1)}{n-2}, \ n\ge 3 ,
\end{equation}
where $|S^{n-1}|$ is the volume of the $n-1$ dimensional unit sphere.
The equality sign is attained only by
\begin{equation}\label{minima-func}
	\varphi(x) = c_1 c_2^{\frac{n-2}{2}}
	\left[   |\tilde{x}-\tilde{v}|^2 + \left( x_n + c_2 \right)^2  \right]^{-\frac{n-2}{2}}
	, \quad  x=(\tilde{x},x_n)\in \mathbb{R}_+^n
\end{equation}
with constants
$c_1 \ne 0$, $c_2>0$, and a constant vector $\tilde{v} \in \mathbb{R}^{n-1}$. It follows that \eqref{minima-func} attains $ \inf\limits_{f \in D_*^{1,2}(\mathbb{R}_+^n) } Q(f)$.
By \eqref{z24Jan30-1}, we have
\begin{equation}\label{z24Jan28-1}
\int_{\mathbb{R}_+^n} \left| \nabla f \right|^2 dx
=
\int_{\mathbb{R}^{n-1} } \left| f\right|^{p+1}(\tilde{x},0) d\tilde{x} \ge \big[ (n-2)^{\frac{1}{2}} 2^{-\frac{1}{2}} |S^{n-1}|^{\frac{1}{2(n-1)}} \big]^{\frac{2(p+1)}{p-1} }
\mbox{ \ for \ }
f\in \mathbf{Ne};
\quad
\mathbf{Ne} \subsetneqq D_*^{1,2}(\mathbb{R}_+^n) .
\end{equation}

Define the tangent space at $v \in \mathbf{Ne}$  as
\begin{equation}
	T_{\mathbf{Ne}} v :=
	\Big\{
	f\in D^{1,2}(\mathbb{R}_+^n) \ | \ \langle I'[v],f \rangle =
2\int_{\mathbb{R}_+^n} \nabla v \cdot \nabla f dx -
\left(p+1\right) \int_{ \mathbb{R}^{n-1}} \left(\left| v \right|^{p-1} v f \right)(\tilde{x},0) d \tilde{x} = 0
	\Big\}  .
\end{equation}

\begin{lemma}\label{z24Jan20-9-lem}

Suppose that $n\ge 3$ is an integer, $p =\frac{n}{n-2}$, and $v \in \mathbf{Ne}$ attains $\inf\limits_{f \in \mathbf{Ne} } J[f] \in \mathbb{R}$, then
\begin{equation*}
	\int_{\mathbb{R}_+^n}
	\nabla v \cdot \nabla g  dx - \int_{\mathbb{R}^{n-1}}
	\left( | v |^{p-1} v g \right)(\tilde{x},0) d \tilde{x} = 0 \quad \mbox{ \ for \ } g\in D^{1,2}(\mathbb{R}_+^n) ;
\quad
B_v[\varphi,\varphi] \ge 0
\quad
\mbox{ \ for \ }
\varphi\in T_{\mathbf{Ne}} v .
\end{equation*}
	
\end{lemma}

\begin{proof}
	
The first result is deduced by the Lagrange multiplier method.
For the second result, there exists $\psi \in D^{1,2}(\mathbb{R}_+^n)$ such that $\langle I'[v],\psi \rangle \ne 0$. Indeed, since $v \in \mathbf{Ne} $ and $\langle I'[v],v \rangle = \left(1-p\right) \int_{ \mathbb{R}^{n-1}} \left| v \right|^{p+1}(\tilde{x},0)  d \tilde{x} \ne 0$, we can take $\psi=v$.
Given $\varphi\in T_{\mathbf{Ne}} v$,	set $H(t,s):= I\left[ v+ t\varphi + s\psi \right] $.
	Then $H(0,0)=0$ and $\pp_s H(t,s) \big|_{(t,s)=(0,0)} = \langle I'[v],\psi \rangle \ne 0$. For $|t|, |s|$ sufficiently small, $H(t,s)$ is smooth about $t,s$.
	By the implicit function theorem, there exists $\delta>0$ small such that
	\begin{equation*}
		H\left( t,s(t) \right) \equiv 0 \quad
		\forall t\in (-\delta,\delta);
		\quad
		s=s(t) \in C^2(-\delta,\delta);
		\quad s(0) = 0;
		\quad
		s'(0) = -\frac{\langle I'[v], \varphi\rangle}{ \langle I'[v], \psi \rangle }  =0 .
	\end{equation*}
	Set $h(t) :=t\varphi +s(t) \psi $, and then $h(t)$ satisfies
	\begin{equation*}
		h(0)=0,
		\quad h'(0)=\varphi,
		\quad
		h(t) \in C^2((-\delta,\delta), D^{1,2}(\mathbb{R}_+^n) ),
		\quad
		v+h(t) \in \mathbf{Ne}
		\mbox{ \ for \ } t\in (-\delta,\delta), 0< \delta \ll 1 ,
	\end{equation*}
where we take $\delta\ll 1$ to make $v+h(t) \ne 0$.
The fact that $v \in \mathbf{Ne}$ attains $\inf\limits_{f \in \mathbf{Ne} } J[f]$ implies $
		\frac{d^2}{dt^2} J(v+h(t))\big|_{t=0} \ge 0 $. Therein,
	\begin{equation*}
		\begin{aligned}
			\frac{d^2}{dt^2} J(v+h(t)) \Big|_{t=0}
			= \ &
			\int_{\mathbb{R}_+^n}
			\left|\nabla h'(0) \right|^2 dx
			+
			\int_{\mathbb{R}_+^n}
			\nabla v \cdot \nabla h''(0) dx
			-  p
			\int_{\mathbb{R}^{n-1}}
			\left(
			\left| v \right|^{p-1}
			\left(h'(0) \right)^2\right)(\tilde{x},0) d \tilde{x}
			\\
			&
			-
			\int_{\mathbb{R}^{n-1}}
			\left(
			\left| v \right|^{p-1} v h''(0) \right)(\tilde{x},0) d \tilde{x}
			=
			B_v[\varphi,\varphi] ,
		\end{aligned}
	\end{equation*}
	where we used the first result for the last step.  Hence the second result holds.
\end{proof}

\begin{lemma}\label{z24Jan20-5-lem}

Assume that $n\ge 3$ is an integer, $p =\frac{n}{n-2}$.
For $f \in \mathbf{Ne}$, we have
\begin{equation}\label{qd24Jan20-10}
J[f]  =
\left(
\frac{1}{2}
- \frac{1}{p+1}
\right)
\int_{\mathbb{R}_+^n} \left| \nabla f \right|^2 dx
=
\left(
\frac{1}{2}
- \frac{1}{p+1}
\right)
\left( Q(f) \right)^{\frac{p+1}{p-1}} .
\end{equation}
Given $f \in D_*^{1,2}(\mathbb{R}_+^n)$, for a constant $c_f>0$,
\begin{equation}\label{qd24Jan20-7}
\begin{aligned}
&
	I[c_f f ]
	=
	c_f^2
	\int_{\mathbb{R}_+^n} \left| \nabla f \right|^2 dx
	-
	c_f^{p+1}
	\int_{\mathbb{R}^{n-1} } \left| f\right|^{p+1}(\tilde{x},0) d\tilde{x}  = 0 \Leftrightarrow
	c_f =
	\Bigg(
	\frac{\int_{\mathbb{R}_+^n} \left| \nabla f \right|^2 dx}{\int_{\mathbb{R}^{n-1} } \left| f\right|^{p+1}(\tilde{x},0) d\tilde{x}}
	\Bigg)^{\frac{1}{p-1}} ,
\\
&
\mbox{and then \ }
c_f f \in \mathbf{Ne},
\quad
		J[c_f f ]
		=
		\left(
		\frac{1}{2}
		- \frac{1}{p+1}
		\right)
		\left( Q(f) \right)^{\frac{p+1}{p-1}} .
\end{aligned}
\end{equation}
Moreover,
\begin{equation*}
	\inf\limits_{f \in \mathbf{Ne} } J[f]
	= \left(
	\frac{1}{2}
	- \frac{1}{p+1}
	\right)
	\Big( \inf\limits_{f \in D_*^{1,2}(\mathbb{R}_+^n) } Q(f) \Big)^{\frac{p+1}{p-1}}  .
\end{equation*}
In particular, if $f_1 \in \mathbf{Ne}$ attains $\inf\limits_{f \in \mathbf{Ne} } J[f]$, then $f_1$ attains $ \inf\limits_{f \in D_*^{1,2}(\mathbb{R}_+^n) } Q(f)$;
conversely, if $f_2 \in D_*^{1,2}(\mathbb{R}_+^n)$ attains $ \inf\limits_{f \in D_*^{1,2}(\mathbb{R}_+^n) } Q(f)$, then $c_{f_2} f_2 \in \mathbf{Ne}$  and $c_{f_2} f_2 $ attains $\inf\limits_{f \in \mathbf{Ne} } J[f]$.
\end{lemma}

\begin{proof}
\eqref{qd24Jan20-10} and \eqref{z24Jan28-1} imply
\begin{equation*}
	\inf\limits_{f \in \mathbf{Ne} } J[f]
	=
	\left(
	\frac{1}{2}
	- \frac{1}{p+1}
	\right)
	\Big( \inf\limits_{f \in \mathbf{Ne} } Q(f) \Big)^{\frac{p+1}{p-1}}
	\ge
	\left(
	\frac{1}{2}
	- \frac{1}{p+1}
	\right)
	\Big( \inf\limits_{f \in D_*^{1,2}(\mathbb{R}_+^n) } Q(f) \Big)^{\frac{p+1}{p-1}}  .
\end{equation*}
Given $f \in D_*^{1,2}(\mathbb{R}_+^n)$, the choice of $c_f$ yields $c_f f \in \mathbf{Ne}$. One combining \eqref{qd24Jan20-10}, then
$ J[c_f f ]
  =
		\left(
		\frac{1}{2}
		- \frac{1}{p+1}
		\right)
		\left( Q(f) \right)^{\frac{p+1}{p-1}} $,
which implies
\begin{equation*}
	\left(
	\frac{1}{2}
	- \frac{1}{p+1}
	\right)
	\Big( \inf\limits_{f \in D_*^{1,2}(\mathbb{R}_+^n) } Q(f) \Big)^{\frac{p+1}{p-1}}
	=
	\inf\limits_{f \in D_*^{1,2}(\mathbb{R}_+^n) }
	J[c_f f ]
	\ge
	\inf\limits_{c_f f \in \mathbf{Ne} }
	J[c_f f ]
	\ge
	\inf\limits_{ f \in \mathbf{Ne} }
	J[f].
\end{equation*}
\end{proof}

\begin{lemma}\label{qd24Jan24-2-lem}
	
	Given an integer $n\ge 2$ and constants $\lambda = -m^2$, $m> 0$, then for $u\in C^2(\mathbb{R}_+^n) \cap C^1(\overline{\mathbb{R}_+^n})$ and $x\in \overline{\mathbb{R}_+^n}$, we have
	\begin{equation}\label{qd24Jan25-1}
		\begin{aligned}
			u(x) = \ &
			\int_{\mathbb{R}_+^n} \left(\Delta u + \lambda u \right)(y) \left[ -E_n^m(|x-y|) -E_n^m ( |x-(\tilde{y}, -y_n) | ) \right] dy
			\\
			&
			-
			\int_{\mathbb{R}^{n-1}} (-\partial_{y_n} u)(\tilde{y},0) \left[
			-2 E_n^m ( |x-(\tilde{y}, 0) | )
			\right]   d\tilde{y} ,
		\end{aligned}
	\end{equation}
	where $
	E_n^m(r)
	=
	m^{\frac{n}{2}-1} (2\pi)^{-\frac{n}{2}} r^{1-\frac{n}{2}} K_{\frac{n}{2}-1}(mr) $ for  $r> 0 $ and $K_{\frac{n}{2}-1}$ is the second kind modified Bessel function of order $\frac{n}{2}-1$. Moreover,
	\begin{equation}\label{qd24Jan29-2}
		\left| E_n^m(r) \right|
		\le
		C(m,n)
		\bigg(
		\1_{r<1}
		\begin{cases}
			r^{2-n},  & n>2
			\\
			\langle \ln r \rangle ,
			& n=2
		\end{cases}
		+
		\1_{r\ge 1}	e^{-m r}
		\bigg)
		\quad
		\mbox{ \ for \ }  r>0 .
	\end{equation}
\end{lemma}

\begin{remark}
Similar to \eqref{qd2023Dec03-3}, \eqref{qd24Jan25-1} can be used to represent solutions with some rough data.
\end{remark}

\begin{proof}
	
By \cite[Theorem 2.5]{stuart1998}, the fundamental solution of $\Delta + \lambda$ in $\mathbb{R}^n$ is given by $-E_n^m(|x|)$, namely $\left(\Delta + \lambda \right) (-E_n^m(|x|)) = \delta(x) $ in $\mathbb{R}^n$. The properties of $-E_n^m$ are given in \cite[pp.9-10]{stuart1998}.
	
	Given $\tilde{x}  \in \mathbb{R}^{n-1}$, $x_n>0$, $x=(\tilde{x}, x_n)$,
	take $v(y) = -E_n^m(|x-y|) -E_n^m ( |x-(\tilde{y}, -y_n) | ) $.
	Direct calculation yields $\partial_{y_n} v\big|_{y_n=0} = 0$,
$\left(\Delta + \lambda \right) v
=
\left(\Delta_y + \lambda \right)
\left[
-E_n^m(|y-x|) -E_n^m  \big( |y-(\tilde{x}, -x_n) | \big)
\right]
=
\delta_{y-x} + \delta_{y-(\tilde{x}, -x_n)} $ in $\mathbb{R}^n$. Integration by parts yields
	\begin{equation*}
		\begin{aligned}
&
\int_{\mathbb{R}_+^n} \left(\Delta u + \lambda u \right) v dy
=
\int_{ \mathbb{R}^{n-1}} \left[ (-\partial_{y_n} u) v - (-\partial_{y_n} v) u \right] (\tilde{y},0) d\tilde{y} +
\int_{\mathbb{R}_+^n} u \left(\Delta v + \lambda v \right) dy
\\
 =  \ &
			\int_{\mathbb{R}^{n-1}} \left[(-\partial_{y_n} u) v \right] (\tilde{y},0) d\tilde{y} +
			\int_{\mathbb{R}^n} \left(u(y)\1_{y_n> 0} + 0 \1_{y_n\le 0} \right) \left( \delta_{y-x} + \delta_{y-(\tilde{x}, -x_n)} \right) dy
			\\
			= \ &
			\int_{\mathbb{R}^{n-1}} \left[(-\partial_{y_n} u) v \right] (\tilde{y},0) d\tilde{y} +
			u(x) .
		\end{aligned}
	\end{equation*}
Taking $x_n\downarrow 0$ deduces
the case $x\in \partial \mathbb{R}_+^n$.
Plugging $v(y)$, then we complete the proof.
\end{proof}

\medskip

Given $V_1(x)\in L^\infty(\mathbb{R}_+^n)$ and $V_2(\tilde{x})\in L^\infty(\mathbb{R}^{n-1})$, we say that $f\in H^1(\mathbb{R}_+^n)$ satisfies the equation
\begin{equation*}
	-\Delta f = V_1(x)  f \mbox{ \ in \ } \mathbb{R}_+^n ,
	\quad
	-\partial_{x_n} f = V_2(\tilde{x}) f \mbox{ \ on \ } \pp \mathbb{R}_+^n
\end{equation*}
in the weak sense if
\begin{equation*}
	\int_{\mathbb{R}_+^n} \nabla f\cdot \nabla g dx - \int_{\mathbb{R}^{n-1}} V_2(\tilde{x}) (f g)(\tilde{x},0) d \tilde{x} = \int_{\mathbb{R}_+^n} V_1(x) f g dx
	\quad
	\mbox{ \ holds for all \ }
	g\in H^1(\mathbb{R}_+^n) .
\end{equation*}

\begin{lemma}\label{z24Jan24-3-lem}

Suppose that $n\ge 2$ is an integer, $\lambda<0$, $V(\tilde x)$ satisfies
$V(\tilde x)\in L^\infty(\mathbb{R}^{n-1}) $ and $ \lim_{|\tilde x|\to\infty}V(\tilde x)=0 $,
	let $\phi \in H^1(\mathbb{R}^n_+)$ satisfy
	\begin{equation}\label{inner-eigen00}
		-\Delta \phi =\lambda \phi
		\quad \mbox{ \ in \ }   \mathbb{R}^{n}_+ ,
		\quad
		-\pp_{x_n} \phi   =
		V(\tilde x) \phi
		\quad \mbox{ \ on \ }  \pp \mathbb{R}^{n}_+
	\end{equation}
	in the weak sense.
	Then for all  $\nu \in [0,\sqrt{-\lambda})$,
	we have $ |\phi(x)|\le C e^{-\nu|x|} $ in $\overline{\mathbb{R}_+^n}$ with a constant $C$ depending on $n, \lambda, \nu, \| \phi \|_{L^2(\mathbb{R}^n_+)}, V(\tilde{x})$.
	
\end{lemma}
\begin{proof}
Denote $m=\sqrt{-\lambda}$. Since $\| \phi \|_{L^2(\mathbb{R}^n_+)}<\infty$, by \cite[Theorem 5.36, Theorem 5.45]{lieberman2013oblique-book}, $\phi \in C(\overline{\mathbb{R}_+^n })$ and $\left|\phi(x)\right| \to 0$ as $|x|\to \infty$.
	By the representation formula \eqref{qd24Jan25-1} in Lemma \ref{qd24Jan24-2-lem} and uniqueness of the weak solution of \eqref{inner-eigen00} in $H^1(\mathbb{R}^n_+)$, $\phi$ can be written as
	\begin{equation}\label{qd24Jan25-3}
		\phi(x) =   2
		\int_{\mathbb{R}^{n-1}} V(w) \phi(w,0)
		E_n^m ( |x-(w, 0) | )    dw ,
	\end{equation}
	where $E_n^m$ is given in Lemma \ref{qd24Jan24-2-lem} and  satisfies \eqref{qd24Jan29-2}.

The following argument is in the same spirit of \cite[Proof of Theorem 2.1]{hislop2000}. Given $\nu \in [0,m)$, denote
	\begin{equation*}
		A(x):= \sup\limits_{w\in \mathbb{R}^{n-1} }
		\left| \phi(w,0)\right| e^{-\nu \left|x- (w,0)\right|}
		,
		\quad
		B(x) :=
		\int_{\mathbb{R}^{n-1}}
		2 \left|E_n^m(|x-(w,0)|) \right|
		e^{\nu \left|x- (w,0)\right|} \left|V(w) \right| dw  .
	\end{equation*}
If $A(x)=0$ for some $x\in \overline{\mathbb{R}_+^n}$, then $\phi(\cdot,0)\equiv 0$ in $\mathbb{R}^{n-1}$, which implies that $\phi\equiv 0$ in $\overline{\mathbb{R}_+^n}$ by \eqref{qd24Jan25-3} and the conclusion holds. Hereafter, we always assume $A(x)>0$ in $\overline{\mathbb{R}_+^n}$.
Obviously, $\left|\phi(x) \right| \le A(x) B(x) $. By \eqref{qd24Jan29-2}, $\nu<m$, the properties of $V(x)$, and Lebesgue's dominated convergence theorem,
\begin{equation*}
		B(x)
		=
		\int_{\mathbb{R}^{n-1}}
		2 \left|E_n^m(|(z,x_n)|) \right|
		e^{\nu \left|(z,x_n)\right|} \left|V(\tilde{x} -z) \right| dz
		\to 0
		\mbox{ \ as \ }
		|x|\to \infty   ,
	\end{equation*}
which implies that there exists $R_1 = R_1(n,m,V(\tilde{x}))>0$ sufficiently large such that
	\begin{equation}\label{qd24Jan25-6}
		\left|\phi (x) \right| \le 2^{-1} A(x) \quad
		\mbox{ \ for \ }
		x\in \overline{\mathbb{R}_+^n},
		\  \left|x\right|\ge R_1 .
	\end{equation}
	Note that for $\nu\ge 0$,
	\begin{equation*}
	\begin{aligned}
		A(x)
		= \ &
		 \sup\limits_{w\in \mathbb{R}^{n-1} }
		\left| \phi(w,0)\right|
		\sup\limits_{z\in \mathbb{R}^{n-1} }
		e^{-\nu \left|x- (z,0)\right|}
		e^{-\nu \left|(z,0)-(w,0)\right|}
		\\
		= \ &
		\sup\limits_{z\in \mathbb{R}^{n-1} }
		\sup\limits_{w\in \mathbb{R}^{n-1} }
		\left| \phi(w,0)\right|
		e^{-\nu \left|x- (z,0)\right|}
		e^{-\nu \left|(z,0)-(w,0)\right|}
		=
		\sup\limits_{z\in \mathbb{R}^{n-1} }
		A(z,0)
		e^{-\nu \left|x- (z,0)\right|}  .
		\end{aligned}
	\end{equation*}
Combining \eqref{qd24Jan25-6}, we have
	\begin{equation*}
		\sup\limits_{|w|> R_1 }
		\left| \phi(w,0)\right| e^{-\nu \left|x- (w,0)\right|}
		\le
		2^{-1}
		\sup\limits_{|w|> R_1 }
		A(w,0) e^{-\nu \left|x- (w,0)\right|}
		\le
		2^{-1} A(x) < A(x) .
	\end{equation*}
By the definition of $A(x)$, it implies
	\begin{equation*}
			A(x)
			=
			\sup\limits_{|w|\le R_1 }
			\left| \phi(w,0)\right| e^{-\nu \left|x- (w,0)\right|}
			\le
			e^{-\nu \left|x \right|}
			\sup\limits_{|w|\le R_1 }
			\left| \phi(w,0)\right| e^{\nu \left|w\right|}
			.
	\end{equation*}
Combining \eqref{qd24Jan25-6} and the local $L^\infty$ estimate \cite[Theorem 5.36]{lieberman2013oblique-book} in $\overline{B_n^+(0,2R_1) }$, we conclude this lemma.
\end{proof}

\begin{prop}\label{z24Jan25-6-prop}
	Given an integer $n\ge 3$, $p=\frac{n}{n-2}$, and $U(x)$ given in \eqref{qd24Jan14-U}, there exists only one negative eigenvalue $\lambda_0$ for the following eigenvalue problem in $H^1(\mathbb{R}_+^n)$,
	\begin{equation}\label{z24Jan20-1}
		-\Delta f =\lambda_0  f \mbox{ \ in \ } \mathbb{R}_+^n ,
		\quad
		-\partial_{x_n} f =pU^{p-1} f \mbox{ \ on \ } \pp \mathbb{R}_+^n  .
	\end{equation}
The eigenvalue $\lambda_0$ is simple with an eigenfunction $Z_0(x) \in C^\infty(\overline{\mathbb{R}_+^n}) \cap H^1(\mathbb{R}_+^n)$ satisfying $\| Z_0 \|_{L^2(\mathbb{R}^n_+)} =1$,
$0<Z_0(x) \le C e^{-\nu|x|} $ in $\overline{\mathbb{R}_+^n}$ for all  $\nu \in \left[0, \sqrt{-\lambda_0} \right)$
with a constant $C$ depending on $n, \lambda_0, \nu$.
Moreover,
$ \lambda_0 = \inf\limits_{0\not\equiv f\in H^{1}(\mathbb{R}^n_+)}
\| f\|_{L^2 (\mathbb{R}^n_+) }^{-2}
B_U[f,f] $.

\end{prop}

\begin{proof}

\textbf{Step 1.}
For any $f \in H^1(\mathbb{R}_+^n)$ satisfying \eqref{z24Jan20-1} in the weak sense, by parabolic regularity theorem, $f\in C^\infty(\overline{\mathbb{R}_+^n})$.
By Lemma \ref{z24Jan24-3-lem},
for all  $\nu \in \left[ 0, \sqrt{-\lambda_0} \right)$,
we have $ |f(x)|\le C e^{-\nu|x|} $ in $\overline{\mathbb{R}_+^n}$ with a constant $C$ depending on $n, \lambda_0, \nu, \| f \|_{L^2(\mathbb{R}^n_+)}$.
	
	\textbf{Step 2.} Denote $ \lambda_* = \inf\limits_{0\not\equiv f\in H^{1}(\mathbb{R}^n_+)}
	\| f\|_{L^2 (\mathbb{R}^n_+) }^{-2}
	B_U[f,f] $, where $B_U[f,f]$ is well-defined since $U(x)\in D^{1,2}(\mathbb{R}_+^n)$ for $n>2$.
	Notice that
	$
	B_U[U,U]
	=
	\left(1- p\right) \int_{ \mathbb{R}^{n-1}} U^{p+1} (\tilde{x},0)  d \tilde{x} < 0
	$. Since $U \not\in L^2(\mathbb{R}_+^n)$ when $n\le 4$, instead, applying $\| U \eta(x/R) \|_{L^2(\mathbb{R}_+^n)}^{-2} B_U[U \eta(x/R),U \eta(x/R) ] $ with $R$ sufficiently large implies
	$\lambda_* <0$ .
For any $f\in H^1(\mathbb{R}_+^n)$,
$ \int_{\mathbb{R}^{n-1}} f^2(\tilde{x},0) d\tilde{x}
	\le
	2
	\| f \|_{L^2(\mathbb{R}_+^n)}
	\| \pp_{x_n} f \|_{L^2(\mathbb{R}_+^n)} $,
	then for any $\epsilon>0$,
	\begin{equation*}
		p
		\int_{\mathbb{R}^{n-1}} \left( U^{p-1} f^2\right)(\tilde{x},0) d \tilde{x}
		\le
		\epsilon
		\| \nabla f \|_{L^2(\mathbb{R}_+^n)}^2
		+
		C(\epsilon) \| f \|_{L^2(\mathbb{R}_+^n)}^2 ,
	\end{equation*}
	 which implies  $\lambda_*>-\infty$.

\textbf{Step 3.}
Take a sequence $(f_k)_{k\ge 1} \subset H^1(\mathbb{R}^n_+)$ such that $\| f_k \|_{L^2 (\mathbb{R}^n_+) }=1$ and $B_U[f_k,f_k] =\lambda_* + o(1)$ with $o(1)\to 0$ as $k\to \infty$. It implies
	\begin{equation*}
		\epsilon
		\| \nabla f_k \|_{L^2(\mathbb{R}_+^n)}^2
		+
		C(\epsilon) \| f_k \|_{L^2(\mathbb{R}_+^n)}^2\geq p\int_{  \mathbb{R}^{n-1}} \left(U^{p-1} f_k^2\right)(\tilde{x},0) d \tilde{x} =
		\int_{\mathbb{R}_+^n} \left|\nabla f_k \right|^2 dx -\lambda_* +o(1) \geq -\frac{\lambda_*}{2}>0
	\end{equation*}
	when $k$ is sufficiently large.
	Thus
	\begin{equation*}
		\left( 1-\epsilon \right)\| \nabla f_k \|_{L^2(\mathbb{R}_+^n)}^2 \le C(\epsilon) + \lambda_* + o(1) ,
		\quad
		p\int_{\mathbb{R}^{n-1}}
		\left(U^{p-1} f_k^2\right)(\tilde{x},0) d \tilde{x}
		\ge
		-\frac{\lambda_*}{2} .
	\end{equation*}
	It follows that $\sup\limits_{k \ge 1} \|f_k \|_{H^1(\mathbb{R}_+^n)} <\infty$. By Lemma \ref{z24Jan15-8}, \ref{z24Jan15-6}, up to a subsequence,
	\begin{equation*}
		\begin{aligned}
			&
			f_k  \rightharpoonup f_*
			\mbox{ \ in \ } H^1(\mathbb{R}_+^n),
			\quad
			\nabla f_k  \rightharpoonup \nabla  f_*
			\mbox{ \ in \ } L^2(\mathbb{R}_+^n),
			\quad
			f_k  \rightharpoonup f_*
			\mbox{ \ in \ } L^2(\mathbb{R}_+^n),
			\\
			&
			p\int_{\mathbb{R}^{n-1}} \left(U^{p-1} f_k^2\right)(\tilde{x},0) d \tilde{x}
			\to
			p\int_{\mathbb{R}^{n-1}} \left( U^{p-1} f_*^2 \right)(\tilde{x},0) d \tilde{x} \geq -\frac{\lambda_*}{2} ,
		\end{aligned}
	\end{equation*}
	which implies
	\begin{equation*}
		B_U[f_*,f_*] \le \lambda_*<0 ,
		\quad
		0<\| f_* \|_{L^2 (\mathbb{R}^n_+) } \le 1
	\end{equation*}
	and then
	\begin{equation*}
		\| f_* \|_{L^2 (\mathbb{R}^n_+) }^{-2} B_U[f_*,f_*]\le B_U[f_*,f_*] \le \lambda_* .
	\end{equation*}
	By the definition of $\lambda_*$, we know that $\| f_* \|_{L^2 (\mathbb{R}^n_+) }^{-2} B_{U}[f_*,f_*] \ge \lambda_*$. Thus $\| f_* \|_{L^2 (\mathbb{R}^n_+) }=1$ and $B_U[f_*,f_*]=\lambda_*$. Since $|\nabla |f|| \le |\nabla f|$ a.e., by the definition of $\lambda_*$, we also have $B_U[|f_*|,|f_*|]= \lambda_*$.
	
Denote $S_{<}:= \big\{ f\in H^1(\mathbb{R}_+^n) \ | \ $ $f$ satisfies \eqref{z24Jan20-1} with some $\lambda_0 <0$ in the weak sense $\big\}$. By the Lagrange multiplier method, $f_*, |f_*| \in S_{<}$ with the eigenvalue $\lambda_*$. Step 1 shows that the elements in $S_{<}$ are smooth with exponential decay. By strong maximum principle and Hopf theorem, $|f_*|>0$ in $\overline{\mathbb{R}_+^n}$.

\textbf{Step 4.}
Claim: the dimension of $S_{<}$ is $1$.

Assume the opposite that there exists two linearly independent functions $f_1, f_2 \in S_{<}$ with negative eigenvalues $\lambda_1, \lambda_2$ respectively.
Note that $\lambda_1 (f_1,f_2)_{L^2(\mathbb{R}_+^n)} = B_U[f_1,f_2] =  \lambda_2 (f_1,f_2)_{L^2(\mathbb{R}_+^n)}$. We can assume $(f_1,f_2)_{L^2(\mathbb{R}_+^n)} =0$ since if $\lambda_1\ne \lambda_2$,
it holds automatically, and if $\lambda_1 = \lambda_2$,
we replace $f_1$ by $f_1 - \|f_2\|_{L^2(\mathbb{R}_+^n) }^{-2} (f_1,f_2)_{L^2(\mathbb{R}_+^n)} f_2  \not\equiv 0$ by the linear independence of $f_1, f_2$. Then
$ B_U[f_1,f_2] = 0$. For $i=1,2$, $f_i \ne 0$ deduces $B_U[f_i,f_i] = \lambda_i \|f_i\|_{L^2(\mathbb{R}_+^n)}^2 <0$.

Since $J[\cdot]\in C^2\left( D^{1,2}(\mathbb{R}_+^n) ,\mathbb{R} \right)$, $\langle J''[g] g, g \rangle = (1-p)\int_{\mathbb{R}_+^n} |\nabla g|^2 dx \ne 0$ for $g\in \mathbf{Ne}$, and \eqref{z24Jan28-1} holds,
by \cite[Proposition 5.75]{2014Topological-Book}, $\mathbf{Ne}$ is a complete $C^1$-Banach submanifold of $D^{1,2}(\mathbb{R}_+^n)$ of
codimension $1$.  Since $U(x) \in \mathbf{Ne}$ for $n>2$, in particular,
$T_{\mathbf{Ne}} U$ is codimension $1$ in $D^{1,2}(\mathbb{R}_+^n)$.

Thus, there exists $w_0 \in D^{1,2}(\mathbb{R}_+^n) \backslash T_{\mathbf{Ne}} U$, such that $f_1= a_1 w_0 + a_2 w_1$, $f_2 = b_1 w_0 + b_2 w_2$ for some constants $a_1, b_1, a_2, b_2\in \mathbb{R}$ and $w_1,w_2 \in T_{\mathbf{Ne}}U$.
Taking
$c_1=c_2=1, \tilde{v}=0$ in \eqref{minima-func} and then
using Lemma \ref{z24Jan20-5-lem}, \eqref{qd24Jan20-7}, we get that $U(x)$ given in \eqref{qd24Jan14-U}
attains $\inf\limits_{f \in \mathbf{Ne} } J[f]$.
By Lemma \ref{z24Jan20-9-lem} and $B_U[f_i,f_i]<0$, $i=1,2$, we have $a_1\ne 0, b_1\ne 0$.
So we can set $f_3:= f_1-  b_1^{-1} a_1 f_2 \in T_{\mathbf{Ne}}U$. One using $ B_U[f_1,f_2] = 0$, then $B_U[f_3,f_3] = B_U[f_1,f_1] + (b_1^{-1} a_1)^2 B_U[f_2,f_2] <0$, which contradicts with
Lemma \ref{z24Jan20-9-lem}.

Taking $\lambda_0 = \lambda_*$ and $Z_0 = |f_*|$, we complete the proof.
\end{proof}


\begin{thebibliography}{10}

\bibitem{BourgainWang97}
Jean Bourgain and W.~Wang.
\newblock Construction of blowup solutions for the nonlinear {S}chr\"{o}dinger
  equation with critical nonlinearity.
\newblock volume~25, pages 197--215 (1998). 1997.
\newblock Dedicated to Ennio De Giorgi.

\bibitem{finite-frac2020}
Guoyuan Chen, Juncheng Wei, and Yifu Zhou.
\newblock Finite time blow-up for the fractional critical heat equation in
  {$\Bbb R^n$}.
\newblock {\em Nonlinear Anal.}, 193:111420, 23, 2020.

\bibitem{ChenLiOu2006}
Wenxiong Chen, Congming Li, and Biao Ou.
\newblock Classification of solutions for an integral equation.
\newblock {\em Comm. Pure Appl. Math.}, 59(3):330--343, 2006.

\bibitem{Collot2017}
Charles Collot.
\newblock Nonradial type {II} blow up for the energy-supercritical semilinear
  heat equation.
\newblock {\em Anal. PDE}, 10(1):127--252, 2017.

\bibitem{Green16JEMS}
Carmen Cort\'{a}zar, Manuel del Pino, and Monica Musso.
\newblock Green's function and infinite-time bubbling in the critical nonlinear
  heat equation.
\newblock {\em J. Eur. Math. Soc. (JEMS)}, 22(1):283--344, 2020.

\bibitem{Daners-Koch1992}
Daniel Daners and Pablo Koch~Medina.
\newblock {\em Abstract evolution equations, periodic problems and
  applications}, volume 279 of {\em Pitman Research Notes in Mathematics
  Series}.
\newblock Longman Scientific \& Technical, Harlow; copublished in the United
  States with John Wiley \& Sons, Inc., New York, 1992.

\bibitem{18type2Yamabeflow}
Panagiota Daskalopoulos, Manuel del Pino, and Natasa Sesum.
\newblock Type {II} ancient compact solutions to the {Y}amabe flow.
\newblock {\em J. Reine Angew. Math.}, 738:1--71, 2018.

\bibitem{2dEuler2020}
Juan Davila, Manuel Del~Pino, Monica Musso, and Juncheng Wei.
\newblock Gluing methods for vortex dynamics in {E}uler flows.
\newblock {\em Arch. Ration. Mech. Anal.}, 235(3):1467--1530, 2020.

\bibitem{3dEuler-filament2022}
Juan D\'{a}vila, Manuel del Pino, Monica Musso, and Juncheng Wei.
\newblock Travelling helices and the vortex filament conjecture in the
  incompressible {E}uler equations.
\newblock {\em Calc. Var. Partial Differential Equations}, 61(4):Paper No. 119,
  30, 2022.

\bibitem{frac-nondege2013}
Juan D\'{a}vila, Manuel del Pino, and Yannick Sire.
\newblock Nondegeneracy of the bubble in the critical case for nonlocal
  equations.
\newblock {\em Proc. Amer. Math. Soc.}, 141(11):3865--3870, 2013.

\bibitem{2020HMF}
Juan D\'{a}vila, Manuel del Pino, and Juncheng Wei.
\newblock Singularity formation for the two-dimensional harmonic map flow into
  {$S^2$}.
\newblock {\em Invent. Math.}, 219(2):345--466, 2020.

\bibitem{type25D}
Manuel del Pino, Monica Musso, and Jun~Cheng Wei.
\newblock Type {II} {B}low-up in the 5-dimensional {E}nergy {C}ritical {H}eat
  {E}quation.
\newblock {\em Acta Math. Sin. (Engl. Ser.)}, 35(6):1027--1042, 2019.

\bibitem{delMussoWei2021}
Manuel del Pino, Monica Musso, and Juncheng Wei.
\newblock Geometry driven type {II} higher dimensional blow-up for the critical
  heat equation.
\newblock {\em J. Funct. Anal.}, 280(1):Paper No. 108788, 49, 2021.

\bibitem{ni4D}
Manuel del Pino, Monica Musso, Juncheng Wei, and Yifu Zhou.
\newblock Type {II} finite time blow-up for the energy critical heat equation
  in {$\Bbb R^4$}.
\newblock {\em Discrete Contin. Dyn. Syst.}, 40(6):3327--3355, 2020.

\bibitem{du2021finiteArx}
Shi-Zhong Du.
\newblock Finite time blowup and type II rate for harmonic heat flow from
  Riemannian manifolds.
\newblock {\em arXiv preprint arXiv:2104.00408}, 2021.

\bibitem{Escobar88}
Jos\'{e}~F. Escobar.
\newblock Sharp constant in a {S}obolev trace inequality.
\newblock {\em Indiana Univ. Math. J.}, 37(3):687--698, 1988.

\bibitem{fila1991blow}
Marek Fila and Pavol Quittner.
\newblock The blow-up rate for the heat equation with a non-linear boundary
  condition.
\newblock {\em Mathematical Methods in the Applied Sciences}, 14(3):197--205,
  1991.

\bibitem{FHV00}
Stathis Filippas, Miguel~A. Herrero, and Juan J.~L. Vel\'azquez.
\newblock Fast blow-up mechanisms for sign-changing solutions of a semilinear
  parabolic equation with critical nonlinearity.
\newblock {\em R. Soc. Lond. Proc. Ser. A Math. Phys. Eng. Sci.},
  456(2004):2957--2982, 2000.

\bibitem{Fujita66}
Hiroshi Fujita.
\newblock On the blowing up of solutions of the {C}auchy problem for
  {$u_{t}=\Delta u+u^{1+\alpha }$}.
\newblock {\em J. Fac. Sci. Univ. Tokyo Sect. I}, 13:109--124 (1966), 1966.

\bibitem{giga1985asymptotically}
Yoshikazu Giga and Robert~V Kohn.
\newblock Asymptotically self-similar blow-up of semilinear heat equations.
\newblock {\em Communications on pure and applied mathematics}, 38(3):297--319,
  1985.

\bibitem{giga1987characterizing}
Yoshikazu Giga and Robert~V Kohn.
\newblock Characterizing blowup using similarity variables.
\newblock {\em Indiana University Mathematics Journal}, 36(1):1--40, 1987.

\bibitem{Giga04Indiana}
Yoshikazu Giga, Shin'ya Matsui, and Satoshi Sasayama.
\newblock Blow up rate for semilinear heat equations with subcritical
  nonlinearity.
\newblock {\em Indiana Univ. Math. J.}, 53(2):483--514, 2004.

\bibitem{Giga04-MathMethod}
Yoshikazu Giga, Shin'ya Matsui, and Satoshi Sasayama.
\newblock On blow-up rate for sign-changing solutions in a convex domain.
\newblock {\em Math. Methods Appl. Sci.}, 27(15):1771--1782, 2004.

\bibitem{Harada17system}
Junichi Harada.
\newblock Construction of type {II} blow-up solutions for a semilinear
  parabolic system with higher dimension.
\newblock {\em Calc. Var. Partial Differential Equations}, 56(4):Paper No. 121,
  36, 2017.

\bibitem{harada2019higher}
Junichi Harada.
\newblock A higher speed type {II} blowup for the five dimensional energy
  critical heat equation.
\newblock {\em Ann. Inst. H. Poincar\'{e} Anal. Non Lin\'{e}aire},
  37(2):309--341, 2020.

\bibitem{harada6D}
Junichi Harada.
\newblock A type {II} blowup for the six dimensional energy critical heat
  equation.
\newblock {\em Ann. PDE}, 6(2):Paper No. 13, 63, 2020.

\bibitem{HV1993}
M.A. Herrero and J..L. Velazquez.
\newblock A blow up result for semilinear heat equations in the supercritical
  case.
\newblock {\em Unpublished paper}, 1993.

\bibitem{hislop2000}
P.~D. Hislop.
\newblock Exponential decay of two-body eigenfunctions: A review.
\newblock In {\em Proceedings of the Symposium on Mathematical Physics and
  Quantum Field Theory (Berkeley, CA, 1999)}, volume~4, pages 265--288.
  Citeseer, 2000.

\bibitem{parabolicBook1968}
O.~A. Ladyženskaja, V.~A. Solonnikov, and N.~N. Ural'ceva.
\newblock {\em Linear and quasilinear equations of parabolic type}, volume Vol.
  23 of {\em Translations of Mathematical Monographs}.
\newblock American Mathematical Society, Providence, RI, 1968.
\newblock Translated from the Russian by S. Smith.

\bibitem{Li-Li06-nonlinearYama}
Aobing Li and Yan~Yan Li.
\newblock A fully nonlinear version of the {Y}amabe problem on manifolds with
  boundary.
\newblock {\em J. Eur. Math. Soc. (JEMS)}, 8(2):295--316, 2006.

\bibitem{li2022slow}
Tongtong Li, Liming Sun, and Shumao Wang.
\newblock A slow blow up solution for the four dimensional energy critical semi
  linear heat equation.
\newblock {\em arXiv preprint arXiv:2204.11201}, 2022.

\bibitem{03Li-Zhang-LiouHarn}
YanYan Li and Lei Zhang.
\newblock Liouville-type theorems and {H}arnack-type inequalities for
  semilinear elliptic equations.
\newblock {\em J. Anal. Math.}, 90:27--87, 2003.

\bibitem{LiZhu1995}
Yanyan Li and Meijun Zhu.
\newblock Uniqueness theorems through the method of moving spheres.
\newblock {\em Duke Math. J.}, 80(2):383--417, 1995.

\bibitem{Lieberman1986}
Gary~M. Lieberman.
\newblock Mixed boundary value problems for elliptic and parabolic differential
  equations of second order.
\newblock {\em J. Math. Anal. Appl.}, 113(2):422--440, 1986.

\bibitem{lieberman1996second}
Gary~M Lieberman.
\newblock {\em Second order parabolic differential equations}.
\newblock World scientific, 1996.

\bibitem{Lieberman2001-nonsmooth}
Gary~M. Lieberman.
\newblock Pointwise estimates for oblique derivative problems in nonsmooth
  domains.
\newblock {\em J. Differential Equations}, 173(1):178--211, 2001.

\bibitem{Lieberman2012CPAA}
Gary~M. Lieberman.
\newblock Oblique derivative problems for elliptic and parabolic equations.
\newblock {\em Commun. Pure Appl. Anal.}, 12(6):2409--2444, 2013.

\bibitem{lieberman2013oblique-book}
Gary~M Lieberman.
\newblock {\em Oblique derivative problems for elliptic equations}.
\newblock World Scientific, 2013.

\bibitem{MatanoMerle04}
Hiroshi Matano and Frank Merle.
\newblock On nonexistence of type {II} blowup for a supercritical nonlinear
  heat equation.
\newblock {\em Comm. Pure Appl. Math.}, 57(11):1494--1541, 2004.

\bibitem{Merle09JFA}
Hiroshi Matano and Frank Merle.
\newblock Classification of type {I} and type {II} behaviors for a
  supercritical nonlinear heat equation.
\newblock {\em J. Funct. Anal.}, 256(4):992--1064, 2009.

\bibitem{Mizoguchi04}
Noriko Mizoguchi.
\newblock Type-{II} blowup for a semilinear heat equation.
\newblock {\em Adv. Differential Equations}, 9(11-12):1279--1316, 2004.

\bibitem{Mizoguchi-Souplet19-typeILp}
Noriko Mizoguchi and Philippe Souplet.
\newblock Optimal condition for blow-up of the critical {$L^q$} norm for the
  semilinear heat equation.
\newblock {\em Adv. Math.}, 355:106763, 24, 2019.

\bibitem{2014Topological-Book}
Dumitru Motreanu, Viorica~Venera Motreanu, and Nikolaos Papageorgiou.
\newblock {\em Topological and variational methods with applications to
  nonlinear boundary value problems}.
\newblock Springer, New York, 2014.

\bibitem{2019-fracInfinite}
Monica Musso, Yannick Sire, Juncheng Wei, Youquan Zheng, and Yifu Zhou.
\newblock Infinite time blow-up for the fractional heat equation with critical
  exponent.
\newblock {\em Math. Ann.}, 375(1-2):361--424, 2019.

\bibitem{quittner2020optimal}
Pavol Quittner.
\newblock An optimal Liouville theorem for the linear heat equation with a
  nonlinear boundary condition.
\newblock {\em Journal of Dynamics and Differential Equations}, pages 1--11,
  2020.

\bibitem{Quittner2021-Duke}
Pavol Quittner.
\newblock Optimal {L}iouville theorems for superlinear parabolic problems.
\newblock {\em Duke Math. J.}, 170(6):1113--1136, 2021.

\bibitem{Quittner-Souplet12}
Pavol Quittner and Philippe Souplet.
\newblock Blow-up rate of solutions of parabolic problems with nonlinear
  boundary conditions.
\newblock {\em Discrete Contin. Dyn. Syst. Ser. S}, 5(3):671--681, 2012.

\bibitem{Souplet19book}
Pavol Quittner and Philippe Souplet.
\newblock {\em Superlinear parabolic problems}.
\newblock Birkh\"{a}user Advanced Texts: Basler Lehrb\"{u}cher. [Birkh\"{a}user
  Advanced Texts: Basel Textbooks]. Birkh\"{a}user/Springer, Cham, 2019.
\newblock Blow-up, global existence and steady states, Second edition of [
  MR2346798].

\bibitem{Schweyer12JFA}
R\'emi Schweyer.
\newblock Type {II} blow-up for the four dimensional energy critical semi
  linear heat equation.
\newblock {\em J. Funct. Anal.}, 263(12):3922--3983, 2012.

\bibitem{Seki18}
Yukihiro Seki.
\newblock Type {II} blow-up mechanisms in a semilinear heat equation with
  critical {J}oseph-{L}undgren exponent.
\newblock {\em J. Funct. Anal.}, 275(12):3380--3456, 2018.

\bibitem{17halfHMF}
Yannick Sire, Juncheng Wei, and Youquan Zheng.
\newblock Infinite time blow-up for half-harmonic map flow from {$\Bbb R$} into
  {$\Bbb S^1$}.
\newblock {\em Amer. J. Math.}, 143(4):1261--1335, 2021.

\bibitem{FreeHMF2023}
Yannick Sire, Juncheng Wei, and Youquan Zheng.
\newblock Singularity formation in the harmonic map flow with free boundary.
\newblock {\em Amer. J. Math.}, 145(4):1273--1314, 2023.

\bibitem{Souplet19-simplyproof}
Philippe Souplet.
\newblock A simplified approach to the refined blowup behavior for the
  nonlinear heat equation.
\newblock {\em SIAM J. Math. Anal.}, 51(2):991--1013, 2019.

\bibitem{stinga2017regularity}
Pablo~Ra{\'u}l Stinga and Jos{\'e}~L Torrea.
\newblock Regularity theory and extension problem for fractional nonlocal
  parabolic equations and the master equation.
\newblock {\em SIAM Journal on Mathematical Analysis}, 49(5):3893--3924, 2017.

\bibitem{stuart1998}
C.~A. Stuart.
\newblock An introduction to elliptic equations on {$\mathbb{R}^N$}.
\newblock {\em Nonlinear Functional Analysis And Applications To Differential
  Equations: Proceedings Of The Second School}, page 237, 1998.

\bibitem{sun2021bubble}
Liming Sun, Jun-cheng Wei, and Qidi Zhang.
\newblock Bubble towers in the ancient solution of energy-critical heat
  equation.
\newblock {\em Calc. Var. Partial Differential Equations}, 61(6):Paper No. 200,
  47, 2022.

\bibitem{wang2021refined}
Kelei Wang and Juncheng Wei.
\newblock Refined blowup analysis and nonexistence of type {II} blowups for an
  energy critical nonlinear heat equation.
\newblock {\em arXiv preprint arXiv:2101.07186}, 2021.

\bibitem{2022-gluingsurvey}
Juncheng Wei, Qidi Zhang, and Yifu Zhou.
\newblock On the parabolic gluing method and singularity formation.
\newblock {\em C. R. Math. Acad. Sci. Soc. R. Can.}, 44(4):69--121 [69--87 on
  first page], 2022.

\bibitem{infi4D}
Juncheng Wei, Qidi Zhang, and Yifu Zhou.
\newblock On {F}ila-{K}ing conjecture in dimension four.
\newblock {\em J. Differential Equations}, 398:38--140, 2024.

\bibitem{Zhang-Zhao2023}
Liqun Zhang and Jianfeng Zhao.
\newblock Construction of blow-up solution for 5-dimensional critical
  {F}ujita-type equation with different blow-up speed.
\newblock {\em J. Fixed Point Theory Appl.}, 25(3):Paper No. 67, 30, 2023.

\end{thebibliography}

\end{document}